\documentclass [12pt]{amsart}
\usepackage[utf8]{inputenc}
\usepackage[square]{natbib}
\setcitestyle{numbers}

\usepackage{amsmath,amssymb,amsthm,amsfonts}
\usepackage{mathtools}

\usepackage{mathrsfs}

\usepackage{tikz,graphicx,color}
\usetikzlibrary{arrows}
\usetikzlibrary{decorations.markings}

\usepackage[arrow]{xy}
\usepackage{MnSymbol}

\usepackage[margin=0.9in]{geometry}

\newtheorem{theorem}{Theorem}[section]

\newtheorem{lemma}[theorem]{Lemma}
\newtheorem{proposition}[theorem]{Proposition}
\newtheorem{corollary}[theorem]{Corollary}
\theoremstyle{definition}
\newtheorem{remark}[theorem]{Remark}
\theoremstyle{definition}
\newtheorem{definition}[theorem]{Definition}
\newtheorem{conjecture}[theorem]{Conjecture}

\theoremstyle{definition}

\theoremstyle{definition}
\newtheorem{example}[theorem]{Example}

% ...

% ...

% ...

\def\<{{\langle}}
\def\>{{\rangle}}

\def\lam{{\lambda}}

\def\det{{ \operatorname{det}}}

\begin{document}

\title{Tensor diagrams and cluster combinatorics at punctures}

\author{Chris Fraser}
\address{\hspace{-.3in} School of Mathematics, University of Minnesota, Minneapolis, MN, USA}
\email{cfraser@umn.edu}

\author{Pavlo Pylyavskyy}
\email{ppylyavs@umn.edu}


\thanks{P.~P. was partially supported by NSF grants  DMS-1148634, DMS-1351590, and Sloan Fellowship. C.~F. was supported by NSF grant DMS-1745638 and a Simons Travel Fellowship.}

\subjclass{
Primary 
13F60, 
Secondary
05E99, 
57M50  
}

\keywords{Cluster algebras, skein algebra, web, tensor diagram.}

\begin{abstract}
Fock and Goncharov introduced a family of cluster algebras associated with the moduli of 
${\rm SL}_k$-local systems on a marked surface with extra decorations at marked points. We study this family from an algebraic and combinatorial perspective, emphasizing the structures which arise when the surface has punctures. When $k=2$, these structures are the tagged arcs and tagged triangulations of Fomin, Shapiro, and Thurston. For higher $k$, the tagging of arcs is replaced by a Weyl group action at punctures discovered by Goncharov and Shen. We pursue a higher analogue of a tagged triangulation in the language of tensor diagrams, extending work of Fomin and the second author, and we formulate skein-algebraic tools for calculating in these cluster algebras. We analyze the finite mutation type examples in detail. 
\end{abstract}

\maketitle


\section{Introduction}
The aim of {\sl cluster combinatorics} is to identify, in a particular cluster algebra of interest, a labeling set for all of the cluster variables and a rule which specifies when two such cluster variables are {\sl compatible} in the sense of residing in the same cluster. With these rules uncovered, one knows the {\sl cluster monomials}, a certain subset of a ``canonical'' basis for the cluster algebra. Contrary to the name, there are several constructions of canonical bases valid in various settings, e.g. the dual canonical, dual semicanonical, generic, Satake, and theta bases. These bases do not all coincide with one another, but it is believed, and in some instances proved, that each of these bases contains the cluster monomials.

The {\sl cluster algebras from surfaces} \cite{CATSI} have been very widely studied, perhaps especially because their cluster combinatorics is completely understood. The story is very clean when the underlying marked surface~$\mathbb{S}$ has no punctures. In this case, cluster variables in the associated cluster algebra $\mathscr{A}(\mathbb{S})$ are labeled by isotopy classes of arcs connecting marked points in $\mathbb{S}$. Two cluster variables are compatible exactly when their corresponding arcs are noncrossing. The clusters are the maximal sets of noncrossing arcs, i.e. the ideal triangulations of $\mathbb{S}$.

When $\mathbb{S}$ has punctures, the above description captures certain clusters in $\mathscr{A}(\mathbb{S})$, but not all of them. Fomin, Shapiro and Thurston identified the combinatorial wrinkle closing this gap. Each end of an arc which is at a puncture, i.e. not at a marked point on the boundary, 
needs to be {\sl tagged} as either a {\sl plain} or {\sl notched} end. There is a compatibility notion for tagged arcs -- the arcs themselves should be noncrossing and their taggings must satisfy some extra conditions. A {\sl tagged triangulation} is a maximal set of pairwise compatible tagged arcs. With these notions in place, one has that the cluster variables in $\mathscr{A}(\mathbb{S})$ are exactly the tagged arcs and the clusters are exactly the tagged triangulations. (We exclude here the exceptional case of once-punctured closed surfaces.)

The above cluster algebras $\mathscr{A}(\mathbb{S})$ fit into a broader family of cluster algebras $\mathscr{A}(G,\mathbb{S})$ indexed by a choice of surface $\mathbb{S}$ and simple Lie group~$G$. Clusters in $\mathscr{A}(G,\mathbb{S})$ provide rational coordinate charts on the moduli space of {\sl decorated}  $G$-local systems on~$\mathbb{S}$ and are part of an approach to Higher Teichm\"uller theory of $\mathbb{S}$. These cluster algebras were introduced by Fock and Goncharov in type $A$ \cite{FGMod} and extended to other Lie types in \cite{Le,GoncharovShenQuantum}. They include the aforementioned $\mathscr{A}(\mathbb{S})$ in the special case $\mathscr{A}(\mathbb{S}) = \mathscr{A}({\rm SL}_2,\mathbb{S})$. 

The cluster algebras $\mathscr{A}({\rm SL}_k,\mathbb{S})$  (and their dual ``$\mathcal{X}$-space'') recover important structures in representation theory, for example the Knutson-Tao hive polytopes, the Sch\"utzenberger involution on semistandard Young tableaux, and the quantum group with its braid group action, see~e.g. \cite{GoncharovShenCanonical, GoncharovShenDT, GoncharovShenQuantum, ILP,SchraderShapiroQuantum,Magee}. The cluster combinatorics of these cluster algebras is much more wild once $k>2$ and is not yet well understood. 

This paper presents conjectures and theorems concerning cluster combinatorics of $\mathscr{A}({\rm SL}_k,\mathbb{S})$. We focus on the structures which arise when $\mathbb{S}$ has punctures, i.e. on the higher rank generalizations of tagged arcs and tagged triangulations. We now summarize our results and conjectures in a different order than they appear in the body of the paper. We use the following standard notations: $T$, $P$, and $W$ for a maximal torus in $G$ with its weight lattice and Weyl group; $[k] := \{1,\dots,k\}$ and $\binom{[k]}a$ for the collection of $a$-subsets drawn from $[k]$; $f^\lambda$ for the number of standard Young tableaux of shape $\lambda$; $S_{g,h}$ for the closed genus~$g$ surface with $h$ punctures.

One expects a connection between the cluster algebra $\mathscr{A}(G,\mathbb{S})$ and a related commutative algebra ${\rm Sk}(G,\mathbb{S})$, the $G$-{\sl skein algebra on} $\mathbb{S}$. For example, when $k=2$ and $\mathbb{S}$ has no punctures, one has $\mathscr{A}(\mathbb{S}) \subseteq {\rm Sk}({\rm SL}_2,\mathbb{S})$ and this containment is typically an equality \cite{Muller}. As a vector space, ${\rm Sk}(G,\mathbb{S})$ consists of formal linear combinations of $G$-{\sl tensor diagrams} modulo the $G$-skein relations. A tensor diagram is a certain graph drawn on $\mathbb{S}$ with edges labeled by fundamental $G$-weights. A contractible piece of a $G$-tensor diagram encodes a map between tensor products of fundamental $G$-representations and $G$-skein relations capture the linear relations between such maps. Multiplication in ${\rm Sk}(G,\mathbb{S})$ is induced by superposition of diagrams. The {\sl invariant} $[T]$ encoded by a diagram $T$ is its class in the skein algebra. 

Tensor diagram calculus when $G = {\rm SL}_2$ is classical and involves arcs and Ptolemy relations. The calculus for rank two groups $G$ was pioneered in \cite{Kup} and is very rich while still being tractable. For $\mathbb{S}$ with no punctures, Fomin and the second author put forth precise and as yet unproved conjectures formulating cluster combinatorics of $\mathscr{A}({\rm SL}_3,\mathbb{S})$ in terms of ${\rm SL}_3$~tensor diagram calculus \cite{FPII,FP}. See \cite{Zamo,FraserBraid,FrohmanSikora,HyunKyu,DouglasSun,IY} for subsequent work furthering this viewpoint. 

Any element of the skein algebra is a linear combination of planar diagram invariants. The conjectures from \cite{FPII,FP} predict more strongly that each cluster monomial is an invariant of a planar tensor diagram. This way of encoding a cluster variable is more compact than a Laurent polynomial expansion with respect to a particular seed, and moreover is independent of such a choice of seed. The conjectures also predict which planar invariants are cluster monomials. Tensor diagram calculus becomes more cumbersome when $k>3$, but we expect that variations on these recipes still hold. At a first pass, one can think that the higher rank generalization of an ``arc'' is a planar tensor diagram. 

Goncharov and Shen \cite{GoncharovShenDT} identified the higher rank generalization of tagging cluster variables at punctures: if $\mathbb{S}$ has $h$ many punctures, then there is an action of the $h$-fold product of Weyl groups $\prod_{i=1}^h W$ on the cluster algebra $\mathscr{A}({\rm SL}_k,\mathbb{S})$ by cluster automorphisms. We define a {\sl tagged tensor diagram invariant} as the pullback of a diagram invariant $[T]$ along this action. These serve as the higher rank generalization of tagged arcs. When $k=2$, one has $W = \mathbb{Z} / 2 \mathbb{Z}$ and the action interchanges plain and notched tags at punctures. Not every tagged tensor diagram invariant is a cluster variable although we expect that the converse of this statement is true. In this paper, we address the question of when two tagged diagram invariants, 
each of which is a cluster variable in $\mathscr{A}({\rm SL}_k,\mathbb{S})$, are compatible.

We introduce the class of pseudotagged tensor diagrams $(T,\varphi)$, a pair consisting of an ${\rm SL}_k$ tensor diagram $T$ and a pseudotagging function $\varphi$. The edges $e$ in a tensor diagram are labeled by fundamental weights, or equivalently by integers ${\rm wt}(e)$ drawn from $[k-1]$. The function $\varphi$ is a choice of subset $\varphi(e)  \in \binom{[k]}{{\rm wt}(e)}$ (that is, of a point in the $W$-orbit of the fundamental weight $\omega_{{\rm wt}(e)}$) for every edge $e$ which is incident to a puncture. When $k=2$ every edge has weight~1 and the two possible values of $\varphi$ correspond to plain or notched tagging. 

By appropriately incorporating the $\prod_i W$ action at punctures, a pseudotagged diagram $(T,\varphi)$ encodes an {\sl invariant} $[(T,\varphi)]$ in the fraction field of the skein algebra. A pseudotagged diagram is {\sl tagged} provided it satisfies the following additional condition: for any edges $e$ and $e'$ ending at the same puncture, one has $\varphi(e) \subseteq \varphi(e')$ or $\varphi(e') \subseteq \varphi(e)$. This extra condition is equivalent to requiring that $[(T,\varphi)]$ is the pullback of~$[T] $ along the Weyl group action at punctures. When $k=2$, an example of a diagram which is pseudotagged but not tagged is given by a loop based at a puncture with one of its ends plain and the other end notched. Another example is a union of two arcs whose taggings disagree at a puncture.

The seemingly exotic class of pseudotagged tensor diagram invariants is natural for the following reason. Multiplication of invariants is implemented by superposition of diagrams in the sense that $[T_1][T_2] = [T_1 \cup T_2]$, and such a superposition is typically pseudotagged but not tagged. 

Our first result (the ``flattening theorem'' Theorem~\ref{thm:flattening}) clarifies the relationship between these two classes of diagram invariants: any pseudotagged diagram invariant is a $\mathbb{Q}$-linear combination of tagged diagram invariants. In particular, the vector space spanned by tagged diagram invariants is an algebra. We expect that this algebra coincides with the (upper) cluster algebra $\mathscr{A}^{{\rm up}}({\rm SL}_k,\mathbb{S})$ in most cases.

The proof of the flattening theorem is constructive: we identify an algebraic relation between pseudotagged diagram invariants and show that repeated application of this algebraic relation ``flattens'' any pseudotagged diagram invariant into a linear combination of tagged diagram invariants. The $\mathbb{Q}$-coefficients appearing in the flattening relation have binomial coefficients in their  numerator and the $f^\lambda$'s in their denominator.  Certain exchange relations in $\mathscr{A}({\rm SL}_k,\mathbb{S})$ are instances of the flattening relation, and we expect that {\sl every} exchange relation is a consequence of the skein relations together with the flattening relations. Figure~\ref{fig:flattenproduct} illustrates a typical application of the flattening relation in terms of tensor diagrams.

Our first main cluster combinatorics conjecture (Conjecture~\ref{conj:yespunctures}) asserts that every cluster monomial in $\mathscr{A}({\rm SL}_k,\mathbb{S})$ is the invariant of 
a planar tagged diagram and also of a tagged diagram with no cycles on interior vertices. This conjecture extends those from \cite{FPII,FP} to higher rank, and more novelly to surfaces with punctures. 

Our second main cluster combinatorics conjecture (Conjecture~\ref{conj:clusterconjecture}) posits a necessary and sufficient condition for a product of tagged tensor diagram invariants, each of which is a cluster variable, to determine a cluster monomial. For example, consider a set of cluster variables given by taggings $(T,\varphi_i)$ of the same underlying tensor diagram $T$ and whose tagging functions $\varphi_i$ only disagree on a single edge $e$ of $T$ incident to some puncture $p$. Then we conjecture that the product of these variables is a cluster monomial. This conjecture is a higher rank version of the Fomin-Shapiro-Thurston compatibility rule for tagged arcs.

Our second result (the ``spiral theorem'' Theorem~\ref{thm:clusterflattening}) is a skein-algebraic sense in which our two conjectures are consistent with one another. Consider a set of variables $[(T,\varphi_i)]$ as in the preceding paragraph. The union of these diagrams is pseudotagged and not tagged. Conjecture~\ref{conj:clusterconjecture} predicts that the corresponding pseudotagged invariant $\prod_i [(T,\varphi_i)]$ is a cluster monomial. Hence, by Conjecture~\ref{conj:yespunctures}, it should be the invariant of a tagged diagram with no interior cycles and also of a tagged planar diagram. The spiral theorem confirms the first of these expectations, and establishes the second expectation in a weak sense: we show that $\prod_i [(T,\varphi_i)]$ can be given by a tagged diagram which is ``planar nearby the puncture $p$'' in a certain sense. We give a stronger planarity statement when $k=3$ in Proposition~\ref{prop:kis3}. We note that while the spiral theorem is motivated by our cluster combinatorics conjectures, the theorem itself is purely skein-theoretic. Figure~\ref{fig:twoforms} illustrates the pictures underlying the spiral theorem.

Heading in a different direction, we turn our attention from the delicate setting of tensor diagram calculus towards a coarser setting which nonetheless captures some essential features of $\mathscr{A}({\rm SL}_k,\mathbb{S})$. When $k=2$, this means forgetting the topological information encoded in a tagged arc and remembering only how many ends it has at each puncture and how the ends are tagged. This data is the weight of the corresponding cluster variable with respect to a natural algebraic torus action on $\mathscr{A}(\mathbb{S})$. One can parse clusters in a similar way. At an even coarser level, each cluster comes in one of three flavors at any puncture $p$: either all arcs are plain at~$p$, all are notched at $p$, or there is exactly one plain and one notched arc at $p$. 

One might think that the tagging formalism is an unfortunate departure of cluster algebra theory from more aesthetically pleasing notions such as arcs and triangulations. To the contrary, tagging is the key ingredient in three notable phenomena in cluster algebra theory \cite[Sections 9 and 10]{CATSI}. Namely:
\begin{itemize}
\item Each algebra $\mathscr{A}(S_{g,1})$ does not admit a reddening sequence;  the $g$-vector fan is contained in a half-space.
\item The exchange graph of each $\mathscr{A}(S_{g,2})$ admits a ``long cycle,'' i.e. an element of its fundamental group which is not a product of the 4- and 5-cycles coming from finite type rank 2 cluster subalgebras.
\item The cluster complex of each $\mathscr{A}(S_{g,h})$ is homotopy equivalent to an $(h-1)$-sphere, providing an example of an infinite cluster type whose cluster complex is not contractible. 
\end{itemize}
These three phenomena rely heavily on the coarse structures alluded to above.
The first relies on the fact that in order to pass from a tagged triangulation which is plain at $p$ to one that is notched at $p$ by mutations, we must pass through one that is both plain and notched at $p$ as an intermediate step. The second relies on this, and also on the fact that when $\mathbb{S} = S_{g,2}$ (or more generally when $\mathbb{S} = S_{g,h}$), a tagged triangulation cannot have a plain and notched arc at every puncture. For the third, consider a sign vector $\varepsilon = \{\pm 1\}^h$ indicating a choice of plain or notched at each puncture. It determines a subcomplex $X_\varepsilon$ of the cluster complex, namely the subcomplex on the tagged arcs whose tagging at every puncture weakly agrees with~$\varepsilon$. Then $X_\varepsilon$, moreover any intersection of these, is contractible, so that the cluster complex is homotopy equivalent to the $h$-cube by the nerve lemma from topology.

We propose higher analogues of these coarse structures with a view towards understanding the topology of the cluster complex and exchange graph of $\mathscr{A}({\rm SL}_k,\mathbb{S})$. We introduce {\sl decorated} ordered set partitions of~$[k]$, namely an ordered set partition with a choice of sign for each block of cardinality at least three. We call them dosp's for short. We propose that a dosp at $p$ is 
the higher rank generalization of the three flavors of tagged triangulation at 
$p$ from above.

We also study the possible weights of clusters in $\mathscr{A}({\rm SL}_k,\mathbb{S})$ with respect to the natural $T$-action at each puncture. We call such a set of weights a {\sl $P$-cluster}. The mutation rule for $P$-clusters is linear algebra with integer vectors and can easily be experimented with by hand or on a computer. 

Our third main conjecture (Conjecture~\ref{conj:Pclusterconjecture})  proposes a necessary condition which any $P$-cluster must satisfy.  It is a combinatorial shadow of our second cluster combinatorics conjecture from above. By the {\sl good part of the exchange graph}, we mean those clusters which can be reached from the initial seed by mutations which never violate Conjecture~\ref{conj:Pclusterconjecture}. Our expectation is that this good part is in fact the entire exchange graph.

Our third main result (developed in Sections~\ref{secn:CtoPiEps} and \ref{secn:contraction}) associates a dosp at each puncture to any cluster in the good part of the exchange graph. The construction depends only on the $P$-cluster, i.e. not on the cluster variables themselves. We identify a mutation operation on dosp's which models mutation of clusters: performing a mutation in the good part of the exchange graph changes 
at most one of the dosps at the punctures and does so by a dosp mutation. When $k=2$, this mutation operation captures the first fact from above, i.e. that we cannot pass from plain at $p$ to notched at $p$ by a mutation. We show moreover that {\sl every} dosp mutation  ``comes from'' a mutation of clusters when $\partial \mathbb{S}$ carries at least two boundary points. 

To broaden our family of examples, we introduce a ``Grassmannian version'' of the Fock-Goncharov moduli space. In this version, marked points on boundary components 
are decorated by vectors rather than by affine flags. We use primes $\mathscr{A}'({\rm SL}_k,\mathbb{S})$ to denote the Grassmannian version of the cluster algebra. All results and conjectures in this paper work equally well for both versions $\mathscr{A}({\rm SL}_k,\mathbb{S})$ and $\mathscr{A}'({\rm SL}_k,\mathbb{S})$ of the cluster algebra.

As a fourth result, we explore our conjectures for those $\mathscr{A}({\rm SL}_k,\mathbb{S})$ and $\mathscr{A}'({\rm SL}_k,\mathbb{S})$ of finite mutation type. Building on prior work and some coincidences, there are only three cluster algebras which need to be understood. Each of these is an $\mathscr{A}'({\rm SL}_k,\mathbb{S})$, either when $\mathbb{S}$ is a once-punctured digon and $k=3$ or $4$ or when $\mathbb{S}$ is a once-punctured triangle and $k=3$. In each case, we prove that any cluster can be moved to a finite set of clusters modulo the action of some explicit quasi cluster automorphisms  using the ideas from \cite[Section 10]{FraserBraid}. We carry out the needed finite check in the two 
$k=3$ examples, verifying all of our conjectures in these cases. We view the success of these calculations in the $k=3$ case as further evidence in support of the conjectures from \cite{FPII,FP}. The finite check required in the $k=4$ case is lengthy and we did not carry it out.

Fifth and finally, we prove that the Weyl group action on $\mathscr{A}({\rm SL}_k,S_{g,1})$ is cluster when $k>2$. This case was left open in \cite{GoncharovShenDT,GoncharovShenQuantum}. Modulo some missing extra details, one can conclude that the Donaldson-Thomas transformation on $\mathscr{A}({\rm SL}_k,S_{g,1})$ is cluster when $k>2$. 

The paper is organized as follows. 

Sections~\ref{secn:background} through \ref{secn:WActs} provide background with some new definitions and results as needed to fill in gaps. Section \ref{secn:background} collects standard background on type $A$ combinatorics, cluster algebras, and cluster algebras from surfaces. It ends with a discussion of $P$-clusters, which is not a new notion although our terminology is. Section~\ref{secn:moduli} recalls the Fock-Goncharov moduli space, introduces its Grassmannian cousin, and discusses cluster algebras asssociated to both moduli spaces. Section~\ref{secn:WActs} recalls the Weyl group action at punctures and proves that this action is cluster in the case of $\mathscr{A}'$ and also when $\mathbb{S} = S_{g,1}$ and $k>2$.

Sections~\ref{secn:Dosps} through \ref{secn:contraction} concern the possible weights of cluster variables at punctures. Section~\ref{secn:Dosps} introduces a compatibility notion for weight vectors and states our main conjecture concerning $P$-clusters. Section~\ref{secn:CtoPiEps} explains that when this $P$-cluster conjecture holds, then one can associate a dosp at every puncture to a cluster. It also shows that mutation of clusters induces a mutation of dosps. Section~\ref{secn:contraction} studies the extent to which {\sl every} dosp mutation is induced by a mutation of clusters. The answer is sensitive to the chosen surface $\mathbb{S}$. 

Sections~\ref{secn:Webs} and \ref{secn:flattening} concern tensor diagram calculus and its pseudotagged version. Section~\ref{secn:Webs} introduces tensor diagrams and skein relations, pseuotaggings and taggings, and states our two main cluster combinatorics conjectures. Section~\ref{secn:flattening}
states and proves the flattening and spiral theorems and discusses the special case $k=3$. 

Section~\ref{secn:fmt} discusses the finite mutation type examples. 

{\bf Acknowledgements.} We thank Ian Le, Greg Muller, Linhui Shen, and David Speyer for conversations which inspired parts of this work. We thank the referees for helfpul suggestions and for spotting errors in our original proofs of Lemma~\ref{lem:flipconnectedness} and Proposition~\ref{prop:goodimpliessymmetrical}.

\section{Background}\label{secn:background}
We collect background on cluster algebras, aspects of type~$A$ combinatorics, cluster algebras from surfaces, and $P$-clusters. All but the last of these is standard.

{\bf Convention.} Throughout the paper we use the notation~$\mathcal{M}$ when we have fixed in mind one of the two possible versions of moduli space of decorated $G$-local systems. As we will explain, the symbol~$\mathcal{M}$ implicitly requires that one has made a choice of marked surface $\mathbb{S}$,
type $A$ complex simple Lie group $G = {\rm SL}_k$ for some $k$, and
style of decoration at boundary points, either by affine flags (the ``${\rm Fock-Goncharov}$ version'' of the moduli space) or by vectors (``the {\rm Grassmannian} version'').

With such an $\mathcal{M}$ fixed, we use the notation $\mathscr{A}(\mathcal{M})$ for the associated cluster algebra defined below. We use primes $\mathscr{A}(G,\mathbb{S})$ versus  $\mathscr{A}'(G,\mathbb{S})$ to distinguish between the  Fock-Goncharov and Grassmannian versions of the cluster algebra when this is needed. 

\subsection{Cluster algebras}
We assume familiarity with the construction of a {\sl cluster algebra} and {\sl upper cluster algebra} from an initial seed inside an ambient field of rational functions, and with the terminology of mutable and frozen variables, clusters, cluster monomials, cluster complex, and exchange graph. Two cluster variables are {\sl compatible} when they reside in the same cluster, i.e. when their product is a cluster monomial. The {\sl cluster type} is the underlying mutation pattern of mutable subquivers. 

A {\sl cluster automorphism} is an algebra automorphism which separately permutes the mutable and frozen variables while preserving the clusters. The {\sl cluster modular group} can be defined as the cluster automorphism group of the cluster algebra once frozen variables are specialized to~$1$. A {\sl quasi cluster automorphism} is a looser notion of automorphism of cluster algebra which allows for ``rescalings'' of cluster and frozen variables by Laurent monomials in frozen variables \cite{FraserQH}. 

We assume familiarity with the related constructs of reddening sequences, maximal green sequences, and cluster Donaldson-Thomas transformations. The first of these is a mutation sequence which reaches the cluster variables whose $g$-vectors are the negative of the $g$-vectors of the initial cluster variables, the second is a restricted class of such sequences in which each mutation is ``positive'' in a certain sense, and the third is a quasi cluster automorphism which can be computed by a reddening mutation sequence.

\subsection{Marked surfaces and triangulations}\label{subsec:S}
Let $\mathbb{S} = (\mathbf{S},\mathbb{M})$ be an oriented marked surface~\cite{CATSI}. The set of {\sl marked points} $\mathbb{M}$ decomposes into the set of  {\sl punctures}
$\mathbb{M}_{\circ} := \mathbb{M} \cap \text{int }\mathbf{S}$ and the set of  {\sl boundary points} $\mathbb{M}_{\partial} := \mathbb{M} \cap \partial\mathbf{S}$. Denote by $S_{g,h}$  the oriented closed genus~$g$ surface with $h$~punctures and by $D_{n,h}$ the $n$-gon with $h$ punctures.

We henceforth exclude the following surfaces which have no ideal triangulations:  the sphere with 2 or fewer punctures and the unpunctured monogon or digon. We also exclude the once-punctured monogon $D_{1,1}$. This last exclusion is not essential -- it is possible to assign a cluster algebra to the Fock-Goncharov version of $D_{1,1}$, at least when $k>3$. We thank L.~Shen for explaining this to us. However, explaining the details of this construction and accordingly modifying our proofs adds extra technicalities to our presentation which we have chosen to avoid. 

An {\sl arc} on $\mathbb{S}$ is an immersion $\gamma \colon [0,1] \to \mathbf{S}$, such that $\gamma(0),\gamma(1) \in \mathbb{M}$ and $\gamma$ restricts to an injection $(0,1) \hookrightarrow \mathbf{S} \setminus (\mathbb{M} \cup \partial \mathbf{S})$. Such arcs are considered up to isotopy fixing $\mathbb{M}$. A {\sl loop} is an arc whose endpoints coincide. The subset of $\partial \mathbf{S}$ which connects adjacent boundary points along a boundary component is a {\sl boundary interval}. An arc which is isotopic to a union of (at least two) consecutive boundary intervals is a {\sl boundary arch}. 
An arc which is not a boundary arch is a {\sl spanning arc} -- such an arc either has its two endpoints on two different boundary components or has at least one end at a puncture.

Two arcs are {\sl noncrossing} if one can find representatives for their isotopy classes which do not intersect except perhaps at their endpoints. A {\sl triangulation} $\Delta$ of $\mathbb{S}$ is a maximal set of noncrossing arcs. Such a $\Delta$ decomposes $\mathbb{S}$ into ``triangles'' whose sides are either arcs or boundary intervals and whose 
vertices are marked points. 

A triangulation is {\sl regular} if it contains no loop $\gamma$ which encloses a single puncture $p$ on either of its two sides. In a regular triangulation, each triangle has three distinct sides. Because we exclude once-punctured monogons, any $\mathbb{S}$ considered in this paper has a regular triangulation. 

A regular triangulation is {\sl taut} if each of its triangles has at most one side which is a boundary interval. Equivalently, a taut triangulation is one which has no boundary arches.  

In the $n$-gon, {\sl every} arc is a boundary arch, so there are no taut triangulations. On the other hand, provided $\mathbb{S}$ is not the $n$-gon, then it admits a taut triangulation. For example, the unique taut triangulation of a once-punctured $n$-gon is the one in which each boundary point is connected to the puncture by an arc.

\subsection{Cluster algebras from surfaces}\label{secn:catsI}
Fomin, Shapiro, and Thurston associated to any marked surface $\mathbb{S}$ a cluster algebra $\mathscr{A}(\mathbb{S})$ \cite{CATSI}. Initial seeds for $\mathscr{A}(\mathbb{S})$ are the $k=2$ case of the construction presented in 
Section~\ref{secn:moduli}. In this section, we summarize the cluster combinatorics underlying $\mathscr{A}(\mathbb{S})$. 

To {\sl tag} an arc $\gamma$ is to make a binary choice (either {\sl plain} or {\sl notched} tagging) at every endpoint of $\gamma$ which is a puncture. There is no binary choice for endpoints which are boundary points. We indicate notched tagging using the bowtie symbol $\bowtie$. If $\gamma$ is a loop based at $p$, we require that both ends of $\gamma$ are tagged the same way. We also require that the underlying arc $\gamma$ is not contractible to a puncture or a boundary interval and does not cut out a once-punctured monogon.

Distinct tagged arcs are {\sl compatible} if the arcs which underlie them are noncrossing and the following additional conditions on their taggings are satisfied. If the underlying arcs do not coincide, then their taggings agree at any puncture which they have in common. If the underlying arcs do coincide, then the taggings disagree at exactly one end. For example, if $\gamma$ is a loop based a puncture $p$ and tagged plain at~$p$, then any tagged arcs which are compatible with $\gamma$ are also plain at $p$ if they have an endpoint at $p$. A {\sl tagged triangulation} of $\mathbb{S}$ is a maximal set of pairwise compatible tagged arcs. 

Suppose that $\mathbb{S} \neq S_{g,1}$ for some $g \geq 1$. Then cluster variables (resp. clusters) in $\mathscr{A}(\mathbb{S})$ are in bijection with tagged arcs (resp. tagged triangulations) in $\mathbb{S}$. The case $\mathbb{S} = S_{g,1}$ is exceptional: in this case the cluster variables (resp. clusters) in $\mathscr{A}(\mathbb{S})$ are in bijection with the arcs (resp. triangulations) of $\mathbb{S}$. 

In Section~\ref{secn:fmt}, we assume familiarity with the following nested sequence of subgroups associated to a given $\mathbb{S}$: 
the {\sl pure mapping class group} ${\rm PMCG}(\mathbb{S})$,
the mapping class group ${\rm MCG}(\mathbb{S})$, and the {\sl tagged mapping class group} ${\rm MCG}^\bowtie(\mathbb{S})$. The first of these consists of mapping classes which fix each puncture pointwise, the second allows mapping classes to permute the punctures, and the third allows compositions of mapping classes and tag-changing transformations (interchanging plain tags with notched tags) at punctures. These groups act on the cluster algebra $\mathscr{A}(\mathbb{S})$ by cluster automorphisms. In particular, the reader should be familiar with {\sl half-twists} dragging one puncture over another puncture \cite[Section 9]{Farb}. The square of such a half-twist is a Dehn twist about the simple closed curve enclosing the two punctures.

\subsection{Type $A$ combinatorics}\label{subsec:G}
We use the following standard Lie theoretic terminology in the special case of $G = {\rm SL}(V)$ for a complex $k$-dimensional vector space~$V$, 
with hopes that our constructions might also be valid in other Lie types. 

The group $G$ has Weyl group $W = W_k$ the symmetric group on $k$ letters. It has a Coxeter presentation generated by the simple transpositions $s_i = (i,i+1)$ for $i \in [k-1]$ modulo the relations 
\begin{equation}\label{eq:braidrelations}
s_i  s_j = s_j s_i \text{ when $|i-j|>1$ and } s_is_{i+1}s_i = s_{i+1}s_is_{i+1}. \end{equation}
A {\sl reduced expression} for an element $w \in W$ is a minimal-length factorization of $w$ as a product of simple transpositions, i.e. and expression $w = s_{i_1} \cdots s_{i_\ell}$ with $\ell$ minimal.

A {\sl weight} is an algebraic group homomorphism $T \to \mathbb{C}^*$, where $T\subset G$ is a choice of maximal torus. We can choose a basis $e_1,\dots,e_k$ for $V$ so that $T \subset G$ is identified with the set of unimodular diagonal matrices. Then a weight is  a Laurent monomial map ${\rm diag}(t_1,\dots,t_k) \mapsto t^\lambda$, for a {\sl weight vector} $\lambda \in \mathbb{Z}^k / (1,1,\dots,1)$. We denote by
$$P = \mathbb{Z}^k / {\rm span}(1,1,\dots,1)$$ 
the {\sl weight lattice}, inside the ambient real vector space $P_{\mathbb{R}}:= P \otimes_{\mathbb{Z}} \mathbb{R}$. We have {\sl fundamental weights} $\omega_i  = e_1+\cdots +e_i \in P$. For $S \subset [k]$, we have an {\sl indicator  vector} $\iota_S := \sum_{s \in S}e_s \in P$. 

In examples, we often choose to encode weight vectors in ``multiplicative notation,'' i.e. as (Laurent) monomials in the letters $a,b,c,\dots,$. For example, we would encode the weight vector $e_1+2e_3+e_4$ as the monomial $ac^2d$.

An {\sl ordered set partition} of $[k]$ (an {\sl osp} for short) is a tuple $\Pi = B_1 | B_2 | \cdots |B_\ell$ of nonempty disjoint subsets with union $[k]$. The $B_i$ are the {\sl blocks} of the osp. Block of cardinality one and two are called {\sl singletons} and {\sl doubletons} respectively. 

An osp $\Pi = B_1| \cdots |B_\ell$ encodes a {\sl Weyl region} $C(\Pi) \subset P$ consisting of vectors whose coordinates satisfy the inequalities determined by the blocks of $\Pi$. That is,  $(\lambda_1,\dots,\lambda_k) \in C(\Pi)$ means that 
$\lambda_a = \lambda_b$ whenever $a$ and $b$ are in the same block of $\Pi$, and $\lambda_a > \lambda_b$ whenever $a \in B_i$, $b \in B_j$, and $i < j$. Given $\lambda \in P$, we write $\Pi(\lambda)$ for the osp which encodes its coordinate inequalities: $\lambda \in C_{\Pi(\lambda)}$.

Thinking instead of $C(\Pi)$ as a subset of $P_{\mathbb{R}}$, the Weyl regions are the faces of a geometric realization of a flag simplicial complex (the {\sl Coxeter complex}). The partial order which encodes the closure of Weyl regions is the {\sl coarsening order} on ordered set partitions. That is, $C(\Pi')$ is in the closure of $C(\Pi)$ if each block of $\Pi'$ is a union of several consecutive blocks of $\Pi$. We also say that $\Pi$ {\sl refines} $\Pi'$ in this case. For example, $12|3|57|68|9$ is a refinement of $12|35678|9$. We use overlines to denote closures of Weyl regions, writing e.g. $\overline{C(\Pi)}$ for the union of the Weyl regions indexed by $\Pi$ and those osp's which coarsen $\Pi$. 

We write $\vee$ for the operation of taking the least upper bound of several  ordered set partitions in the refinement partial order when this least upper bound exists. For example, $13|245|67 \vee 1|3|24567 \vee 13|24|567 =  1|3|24|5|67$. The Weyl region closure of the right hand side contains the three regions on the left hand side and is the minimal region with this property.

The group $W$ acts on the weight lattice $P$ by coordinate permutation:
$$w \cdot (\lambda_1,\dots,\lambda_k) = (\lambda_{w(1)},\dots,\lambda_{w(k)}).$$
We say that distinct weight vectors $\lambda$ and $ \mu$ are {\sl $w$-conjugate} if $w \cdot \lambda = \mu$. We typically use this notion when $w$ is an involution, in which case the notion is symmetric in $\lambda$ and $\mu$. They are {\sl conjugate} if they are  {\sl $w$-conjugate} for some~$w$. For example, every indicator vector $\iota_S$ is conjugate to a fundamental weight $\omega_{|S|}$.

The {\sl root system} $\Phi$ consists of the vectors (themselves known as {\sl roots}) $e_a-e_b \in P$ for $a \neq b \in [k]$. Such a root determines a {\sl root hyperplane} in $P_{\mathbb{R}}$ defined by the coordinate equality $\lambda_a = \lambda_b$. The action of the transposition $(ab) \in W$ on $P$ is by reflection about this root hyperplane.

A subset $S \in \binom{[k]}a$  determines a {\sl Grassmannian permutation} $w_S \in W$ whose first $a$ symbols (resp. last $k-a$ symbols) in one-line notation are the elements of $S$ (resp. $[k]\setminus S$) written in sorted order. We abbreviate $\ell(S) := \ell(w_S)$ where the latter is Coxeter length. 
The subset $S$ also determines a lattice walk in the $a \times (k-a)$ rectangle consisting of unit steps in the $(0,1)$ direction at times $S$ and in the 
$(1,0)$ direction at times $[k] \setminus S$. This lattice walk determines a Young diagram $\lambda(S)$, namely the boxes which are in the $a \times (k-a)$ rectangle and which lie weakly northwest of the lattice walk. It is known that the number of reduced expressions for the permutation $w_S$ equals the number $f^{\lambda(S)}$ of standard Young Tableaux (SYT's) of shape $\lambda(S)$. Moreover, this set of reduced expressions is connected by the commutation moves  $s_is_j = s_js_i$ only, see e.g. \cite{BJS}.

\subsection{$P$-clusters}\label{secn:pseeds}
Cluster algebras often come with a grading by an abelian group such that each cluster variable is a homogeneous element. See \cite{Grabowski} for a general discussion. The cluster algebras considered in this paper are graded by the direct sum $P^{\oplus \mathbb{M}_\circ}$ of copies of the $G$-weight lattice indexed by the punctures in $\mathbb{S}$.

\begin{definition}\label{defn:Pseed}
A {\sl $P$-seed} is a pair $(Q,{\rm wt}(v)_{v \in V(Q)})$ consisting of a quiver $Q$ (with no frozen vertices) and a weight vector ${\rm wt}(v) \in P^{\oplus \mathbb{M}_\circ}$ for each vertex $v$ of $Q$, subject to the {\sl balancing constraint} 
\begin{equation}\label{eq:exchangemonomials}
\sum_{u \to v \text{ in } Q} {\rm wt}(u) =
\sum_{v \to u \text{ in } Q} {\rm wt}(u) \text{ for all vertices $v$}.
\end{equation}

The $P$-{\sl cluster} is the multiset of weights at the vertices $({\rm wt}(v)_{v \in V(Q)})$. It consists of $P$-{\sl cluster variables}. 

To mutate a $P$-seed at a mutable vertex $v$, we perform quiver mutation at $v$ and we modify the weight at $v$ according to 
$${\rm wt}(v) \mapsto -{\rm wt}(v)+\sum_{u \to v \text{ in } Q} {\rm wt}(u)$$ leaving the other weights unchanged.  

\end{definition}

It is known that mutation of a $P$-seed again yields a $P$-seed \cite{Grabowski}. The $P$-{\sl exchange graph} is the graph whose vertices are the $P$-seeds and whose edges indicate mutations. 

In what follows, the weight vector ${\rm wt}(v)$ represents the grading of a cluster variable $x(v)$ with respect to a natural $\prod_{\mathbb{M}_\circ} T$-action. The common weight of the two terms in \eqref{eq:exchangemonomials} is the weight of the right hand side of the exchange relation describing mutation at the cluster variable $x(v)$ out of the current seed. Thus, there is a map from clusters to $P$-clusters which commutes with mutations.

\section{Cluster nature of the moduli spaces}\label{secn:moduli}
Fock and Goncharov associated a cluster algebra to the moduli space of decorated $G$-local systems. We review this construction and introduce its ``Grassmannian cousin.''

\subsection{Fock-Goncharov version of the moduli space}\label{subsec:FG}
\begin{definition}
Let $V$ be a complex vector space of dimension $k$ with a fixed volume form $\xi \in \bigwedge^k(V) \cong \mathbb{C}$. An {\sl affine flag} $F$ in $V$ is a choice of tensors $F_{(i)} \in \bigwedge^i(V)$ subject to the normalization $F_{(k)} = \xi$ and also the {\sl flag condition}: for each $i \in [k-1]$, there exists $v \in V$ such that $F_{(i)} \wedge v = F_{(i+1)}$. 
\end{definition}

Due to the normalization condition, these are sometimes called affine flags for~${\rm SL}_k$.

The left $G$-action on $V$, hence on $\bigwedge^i(V)$, induces an action on affine flags. The stabilizer of any affine flag is a maximal unipotent subgroup in $G$. In particular, a matrix which stabilizes an affine flag is unipotent. (All of its eigenvalues equal~1.) 

We have a right $T$-action on affine flags by rescaling the steps of the affine flag: 
\begin{equation}\label{eq:Taction}
 (F \cdot t)_{(i)}:= t_1 \cdots t_i F_{(i)} \in \bigwedge^{(i)}V \text{ for } t = {\rm diag}(t_1,\dots,t_k) \in T.
\end{equation}

\begin{remark}\label{rmk:representatives}
One can specify an affine flag $F$ by setting $F_{(i)} := v_1 \wedge \cdots \wedge v_i$ where 
$v_1,\dots,v_{k-1}$ is a linearly independent sequence of vectors. Each $v_i$ is determined modulo span$\{v_1,\dots,v_{i-1}\}$. The sequence of subspaces ${\rm span}\, F_{(i)} := {\rm span}(v_1,\dots,v_i)$ is a complete flag of subspaces depending only on $F$, not on the chosen~$v_i$'s.  Two affine flags give rise to the same complete flag if and only if they are related by the right $T$-action on affine flags \eqref{eq:Taction}. 
\end{remark}

Suppose for the moment that $k$ is odd. By a {\sl $G$-local system} on $\mathbb{S}$ we will mean a copy $V_p$ of the vector space $V$ with its volume form $\xi$ at every point $p \in \mathbb{S}$, and a {\sl transport  isomorphism} $\Xi_\alpha \colon V_p \to V_q$ for every path $\alpha \colon p \to q$ in $\mathbb{S}$, preserving the form $\xi$ and depending only on the isotopy class of~$\alpha$.

Once we have chosen a choice of $G$-local system, every closed curve $\alpha$ based at a point $p \in \mathbb{S}$ determines a {\sl monodromy matrix} $M_\alpha \in {\rm Aut}(V_p) \cong G$ given by the transport isomorphism along~$\alpha$. The local system is {\sl unipotent} if the monodromy $u_p$ of a simple closed curve contractible to $p$ is a unipotent matrix. 

A {\sl decorated} $G$-local system on $\mathbb{S}$ is a $G$-local system on $\mathbb{S}$ together with a choice of affine flag $F_p \in V_p$ at every marked point $p \in \mathbb{M}$, subject to the extra constraint that the monodromy $u_p$ around every puncture $p \in \mathbb{M}_\circ$ stabilizes the corresponding affine flag $F_p$ at~$p$. (The affine flags are ``decorations.'') 


When $k$ is even one makes essentially the same definitions, except that rather than considering local systems on $\mathbb{S}$, one considers local systems on the {\sl punctured tangent bundle} of $\mathbb{S}$ whose mondromy in the fiber direction equals $-{\rm Id} \in G$. Fock and Goncharov call these {\sl twisted local systems}. The space of twisted decorated $G$-local systems is isomorphic to the space of decorated $G$-local systems but the isomorphism involves making certain choices. Making a different choice of isomorphism will perhaps change certain cluster variables and tensor diagram invariants we discuss below by a sign. The details of these signs do not strongly affect the combinatorial results presented here. One can typically change the sign of a tensor diagram invariant as needed by ``migrating hairs,'' i.e. the constructions below are typically flexible enough to capture any needed sign ambiguity.

We denote by $\mathcal{A}_{G,\mathbb{S}}$ the {\sl moduli space} of (twisted) decorated $G$-local systems on~$\mathbb{S}$, i.e. the set of decorated $G$-local 
systems considered up to simultaneous $G$-action.

Let $g$ be the genus of the closed surface obtained from $\mathbf{S}$ by filling in all of the boundary components with disks, and let $b$ be the number of such boundary components. Then 
\begin{equation}\label{eq:flagparametercount}
\dim \mathcal{A}_{G,\mathbb{S}} = (2g-2+b+|\mathbb{M}|)|G|-|\mathbb{M}_\partial|\dim U.
\end{equation}

\subsection{Grassmannian version of the moduli space}\label{subsec:Gr}
Elements in the moduli space just defined come with an affine flag at every marked point. We now introduce a variation in which boundary points carry a different type of decoration. 

\begin{definition}\label{defn:Grspace}
Denote by $\mathcal{A}'_{G,\mathbb{S}}$ the space of twisted $G$-local systems on~$\mathbf{S}$ together with an affine flag $F_p$ in $V_p$ at every puncture $p$ and a vector $v_p \in V_p$ at every boundary marked point $p$, considered up to simultaneous $G$-action. 
\end{definition}

When $k=3$, this version of the moduli space matches the perspective of \cite{FPII,FP}. We introduce it here in hopes that it will appear in future contexts. More immediately, it provides us with a richer set of explicit examples in which our cluster combinatorics conjectures can be tested.

We have 
\begin{equation}\label{eq:vectorparametercount}
\dim \mathcal{A}'_{G,\mathbb{S}} =(2g-2+b+|M_\circ|)|G|+k|M_\partial|.
\end{equation}


\begin{remark}\label{rmk:Grkn}
The moduli space $\mathcal{A}'_{{\rm SL}_k,D_{n,0}}$ is the space of $n$-tuples of vectors in $\mathbb{C}^k$ modulo simultaneous ${\rm SL}_k$ action. Its algebra of regular functions is the homogeneous coordinate ring of the Grassmannian ${\rm Gr}(k,n)$ of $k$-subspaces in $\mathbb{C}^n$ in the Pl\"ucker embedding. More generally, for unpunctured surfaces~$\mathbb{S}$, 
a point in $\mathcal{A}'_{{\rm SL}_k,\mathbb{S}}$ consists of a tuple of vectors $v \in V$ together with a tuple of endomorphisms $M \in {\rm End}(V)$, with such tuples considered up to simultaneous $G$-action. This is a classical object of study in invariant theory, see \cite{FPII} for some discussion. 
\end{remark}

\begin{remark}\label{rmk:subtle}
We have the following relationships between the two flavors $\mathcal{A}$ versus $\mathcal{A}'$ of moduli space. First, the two moduli spaces coincide when $G = {\rm SL}_2$ or when $\mathbb{S}$ has no boundary. 

Second, we will shortly associate a cluster algebra $\mathscr{A}({\rm SL}_k,\mathbb{S})$ (resp. $\mathscr{A}'({\rm SL}_k,\mathbb{S})$) to the moduli space $\mathcal{A}_{{\rm SL}_k,\mathbb{S}}$ (resp. $\mathcal{A}'_{{\rm SL}_k,\mathbb{S}}$). There is a projection $\mathcal{A}_{{\rm SL}_k,\mathbb{S}} \twoheadrightarrow \mathcal{A}'_{{\rm SL}_k,\mathbb{S}}$ induced by forgetting all but the first step of each affine flag. We conjecture that this map is a {\sl cluster fibration} \cite[Definition 3.2]{FLTropical}, i.e. that the inclusion of algebras $\mathscr{A}'({\rm SL}_k,\mathbb{S}) \hookrightarrow \mathscr{A}({\rm SL}_k,\mathbb{S})$ sends a cluster in the former to a partial cluster in the latter while moreover preserving the mutable part of the quiver. This holds when $\mathbb{S} = D_{n,0}$ is an $n$-gon \cite[Proposition 3.4.(3)]{FLTropical}.

Third, we expect the following generalization of \cite[Section 7]{FraserBraid} from $n$-gons to arbitrary surfaces. Suppose that every boundary component of $\mathbb{S}$ has an even number of marked points. Associate to it a new marked surface $\mathbb{S}'$ in which we change neither the underlying surface 
$\mathbf{S}$ nor the number of punctures but we increase the number of boundary points on each boundary component by a factor of $\frac{k}{2}$. We expect that the cluster algebras $\mathscr{A}({\rm SL}_k,\mathbb{S})$ and $\mathscr{A}'({\rm SL}_k,\mathbb{S'})$ are related by a quasi cluster isomorphism, namely the one already proposed in \cite{FraserBraid}. One needs to verify that this candidate map indeed sends a seed in the former cluster algebra to one in the latter in the appropriate sense. 
\end{remark}

\subsection{A family of quivers}\label{subsec:Qks}
Before introducing initial seeds for the moduli spaces $\mathcal{A}_{G,\mathbb{S}}$ and $\mathcal{A}'_{G,\mathbb{S}}$, we take a detour to introduce a family of quivers $Q_k^s$ where $k \in \mathbb{Z}_{\geq 2}$ and $s \in [k-1]$. This wider family shows up in certain of our proofs. 

We let $\hat{\mathbb{H}}_k = \{(a,b,c) \in \mathbb{N}^3 \colon a+b+c = k\}$ and 
$\mathbb{H}_k = \hat{\mathbb{H}}_k \setminus \{(k,0,0),(0,k,0),(0,0,k)\}$. We think of $\mathbb{H}_k$ as the vertex set of a quiver $Q_k$ whose arrows are obtained by drawing small counterclockwise triangles as indicated in the examples of  $Q_4$ and $Q_5$ below: 
\begin{center}
\begin{tikzpicture}
\node at (0,0) {\begin{xy} 0;<.2pt,0pt>:<0pt,-.2pt>:: 
(174,0) *+{301} ="0",
(339,5) *+{310} ="1",
(86,145) *+{202} ="2",
(257,142) *+{211} ="3",
(424,144) *+{220} ="4",
(0,282) *+{103} ="5",
(171,284) *+{112} ="6",
(342,284) *+{121} ="7",
(513,284) *+{130} ="8",
(84,427) *+{013} ="9",
(259,428) *+{022} ="10",
(427,423) *+{031} ="11",
"1", {\ar"0"},
"0", {\ar"3"},
"3", {\ar"1"},
"3", {\ar"2"},
"2", {\ar"6"},
"4", {\ar"3"},
"6", {\ar"3"},
"3", {\ar"7"},
"7", {\ar"4"},
"6", {\ar"5"},
"5", {\ar"9"},
"7", {\ar"6"},
"9", {\ar"6"},
"6", {\ar"10"},
"8", {\ar"7"},
"10", {\ar"7"},
"7", {\ar"11"},
"11", {\ar"8"},
\end{xy}};
\node at (0,-2.05) {$Q_4$};

\node at (5.25,0)
{\begin{xy} 0;<.25pt,0pt>:<0pt,-.2pt>:: 
(211,4) *+{401} ="0",
(351,0) *+{410} ="1",
(142,118) *+{302} ="2",
(281,119) *+{311} ="3",
(420,123) *+{320} ="4",
(71,236) *+{203} ="5",
(210,236) *+{212} ="6",
(351,236) *+{221} ="7",
(492,238) *+{230} ="8",
(0,351) *+{104} ="9",
(140,353) *+{113} ="10",
(281,353) *+{122} ="11",
(421,353) *+{131} ="12",
(564,355) *+{140} ="13",
(65,469) *+{014} ="14",
(218,469) *+{023} ="15",
(355,473) *+{032} ="16",
(488,472) *+{041} ="17",
"1", {\ar"0"},
"0", {\ar"3"},
"3", {\ar"1"},
"3", {\ar"2"},
"2", {\ar"6"},
"4", {\ar"3"},
"6", {\ar"3"},
"3", {\ar"7"},
"7", {\ar"4"},
"6", {\ar"5"},
"5", {\ar"10"},
"7", {\ar"6"},
"10", {\ar"6"},
"6", {\ar"11"},
"8", {\ar"7"},
"11", {\ar"7"},
"7", {\ar"12"},
"12", {\ar"8"},
"10", {\ar"9"},
"9", {\ar"14"},
"11", {\ar"10"},
"14", {\ar"10"},
"10", {\ar"15"},
"12", {\ar"11"},
"15", {\ar"11"},
"11", {\ar"16"},
"13", {\ar"12"},
"16", {\ar"12"},
"12", {\ar"17"},
"17", {\ar"13"},
\end{xy}};
\node at (5.25,-2.15) {$Q_5$};
\end{tikzpicture}
\end{center}

The larger family of quivers $Q_k^s$ for $s \in [k-1]$ is obtained from $Q_k$ by identifying the vertices $(a,b,c)\sim (a,b+c,0) \in \mathbb{H}_k$ when $b=s$ and $a \neq 0$ and deleting the vertices  $(a,b,c)$ when $b \in [s+1,k-1]$. We draw below this family of quivers when $k=5$:

\begin{tikzpicture}
\node at (0,0) {\scalebox{1}{\begin{xy} 0;<.18pt,0pt>:<0pt,-.18pt>:: 
(215,0) *+{\bullet} ="0",
(357,4) *+{\bullet} ="1",
(152,117) *+{\bullet} ="2",
(287,118) *+{\bullet} ="3",
(429,121) *+{\bullet} ="4",
(75,228) *+{\bullet} ="5",
(216,235) *+{\bullet} ="6",
(357,235) *+{\bullet} ="7",
(497,235) *+{\bullet} ="8",
(0,357) *+{\bullet} ="9",
(146,352) *+{\bullet} ="10",
(287,352) *+{\bullet} ="11",
(427,352) *+{\bullet} ="12",
(568,353) *+{\bullet} ="13",
(79,468) *+{\bullet} ="14",
(209,467) *+{\bullet} ="15",
(360,471) *+{\bullet} ="16",
(498,471) *+{\bullet} ="17",
"1", {\ar"0"},
"0", {\ar"3"},
"3", {\ar"1"},
"3", {\ar"2"},
"2", {\ar"6"},
"4", {\ar"3"},
"6", {\ar"3"},
"3", {\ar"7"},
"7", {\ar"4"},
"6", {\ar"5"},
"5", {\ar"10"},
"7", {\ar"6"},
"10", {\ar"6"},
"6", {\ar"11"},
"8", {\ar"7"},
"11", {\ar"7"},
"7", {\ar"12"},
"12", {\ar"8"},
"10", {\ar"9"},
"9", {\ar"14"},
"11", {\ar"10"},
"14", {\ar"10"},
"10", {\ar"15"},
"12", {\ar"11"},
"15", {\ar"11"},
"11", {\ar"16"},
"13", {\ar"12"},
"16", {\ar"12"},
"12", {\ar"17"},
"17", {\ar"13"},
\end{xy}}};
\node at (4.25,0) {\scalebox{1}{
\begin{xy} 0;<.18pt,0pt>:<0pt,-.18pt>:: 
(215,0) *+{\bullet} ="0",
(357,4) *+{\bullet} ="1",
(152,117) *+{\bullet} ="2",
(287,118) *+{\bullet} ="3",
(429,121) *+{\bullet} ="4",
(75,228) *+{\bullet} ="5",
(216,235) *+{\bullet} ="6",
(357,235) *+{\bullet} ="7",
(497,235) *+{\bullet} ="8",
(0,357) *+{\bullet} ="9",
(146,352) *+{\bullet} ="10",
(287,352) *+{\bullet} ="11",
(540,347) *+{\bullet} ="12",
(79,468) *+{\bullet} ="13",
(209,467) *+{\bullet} ="14",
(360,471) *+{\bullet} ="15",
"1", {\ar"0"},
"0", {\ar"3"},
"3", {\ar"1"},
"3", {\ar"2"},
"2", {\ar"6"},
"4", {\ar"3"},
"6", {\ar"3"},
"3", {\ar"7"},
"7", {\ar"4"},
"6", {\ar"5"},
"5", {\ar"10"},
"7", {\ar"6"},
"10", {\ar"6"},
"6", {\ar"11"},
"8", {\ar"7"},
"11", {\ar"7"},
"7", {\ar"12"},
"12", {\ar"8"},
"10", {\ar"9"},
"9", {\ar"13"},
"11", {\ar"10"},
"13", {\ar"10"},
"10", {\ar"14"},
"12", {\ar"11"},
"14", {\ar"11"},
"11", {\ar"15"},
"15", {\ar"12"},
\end{xy}
}};
\node at (8.5,0) {\begin{xy} 0;<.18pt,0pt>:<0pt,-.18pt>:: 
(215,0) *+{\bullet} ="0",
(357,4) *+{\bullet} ="1",
(152,117) *+{\bullet} ="2",
(288,121) *+{\bullet} ="3",
(432,120) *+{\bullet} ="4",
(75,228) *+{\bullet} ="5",
(216,235) *+{\bullet} ="6",
(505,241) *+{\bullet} ="7",
(0,357) *+{\bullet} ="8",
(151,352) *+{\bullet} ="9",
(538,347) *+{\bullet} ="10",
(79,468) *+{\bullet} ="11",
(219,471) *+{\bullet} ="12",
"1", {\ar"0"},
"0", {\ar"3"},
"3", {\ar"1"},
"3", {\ar"2"},
"2", {\ar"6"},
"4", {\ar"3"},
"6", {\ar"3"},
"3", {\ar"7"},
"7", {\ar"4"},
"6", {\ar"5"},
"5", {\ar"9"},
"7", {\ar"6"},
"9", {\ar"6"},
"6", {\ar"10"},
"10", {\ar"7"},
"9", {\ar"8"},
"8", {\ar"11"},
"10", {\ar"9"},
"11", {\ar"9"},
"9", {\ar"12"},
"12", {\ar"10"},
\end{xy}};
\node at (12.7,0)
{\begin{xy} 0;<.18pt,0pt>:<0pt,-.18pt>:: 
(215,0) *+{\bullet} ="0",
(351,1) *+{\bullet} ="1",
(156,117) *+{\bullet} ="2",
(434,113) *+{\bullet} ="3",
(77,228) *+{\bullet} ="4",
(506,241) *+{\bullet} ="5",
(0,353) *+{\bullet} ="6",
(544,349) *+{\bullet} ="7",
(215,466) *+{\bullet} ="8",
"1", {\ar"0"},
"0", {\ar"3"},
"3", {\ar"1"},
"3", {\ar"2"},
"2", {\ar"5"},
"5", {\ar"3"},
"5", {\ar"4"},
"4", {\ar"7"},
"7", {\ar"5"},
"7", {\ar"6"},
"6", {\ar"8"},
"8", {\ar"7"},
\end{xy}};
\node at (0,-2.05) {$Q_5^{4}$};
\node at (4.25,-2.05) {$Q_5^{3}$};
\node at (8.5,-2.05) {$Q_5^{2}$};
\node at (12.75,-2.05) {$Q_5^{1}$};
\end{tikzpicture}

Note that $Q_k = Q^{k-1}_k$, and that each row of $Q_k^s$ has no more than $s+1$ vertices.

\subsection{Initial seeds for $\mathcal{A}_{G,\mathbb{S}}$}\label{subsec:FGSeeds}
Fock and Goncharov described, for each choice of regular triangulation $\Delta$ of $\mathbb{S}$, a seed $\Sigma_k(\Delta) = (Q_k(\Delta),\mathbf{x}_k(\Delta))$ in the field of rational functions on $\mathcal{M} = \mathcal{A}_{G,\mathbb{S}}$. They proved that seeds from different triangulations are mutation-equivalent, so that the seeds $\Sigma_k(\Delta)$ together give rise to a cluster algebra $\mathscr{A}(G,\mathbb{S})$.

The initial quiver $Q_k(\Delta)$ is obtained as follows. To each triangle $\delta \in \Delta$ we associate a copy~$Q_k(\delta)$ of the quiver $Q_k$ from above. We identify the three vertices $p_1,p_2,p_3$ of the triangle $\delta$ with the three coordinates $(k,0,0)$, $(0,k,0)$, $(0,0,k)$ in $\hat{\mathbb{H}}_k$. We associate $k-1$ vertices of $Q_k$ to each side of the triangle $\delta$, associating e.g. 
the vertices $(a,k-a,0)$ for $a \in [k-1]$ to the side $\overline{p_1p_2}$ of $\delta$.  
 
If triangles $\delta$ and $\delta'$ have a common side $\gamma \in \Delta$ then we can glue the quiver fragments $Q_k(\delta)$ and $Q_k(\delta')$ by identifying the $k-1$ vertices associated to this shared side. The initial  quiver $Q_k(\Delta)$ is the result of carrying out all such gluings: 
\begin{equation}\label{eq:FGquiver}
Q_k(\Delta) = \coprod_{\delta \in \Delta}Q_k(\delta) / \sim,
\end{equation}
where $\sim$ here indicates the gluing. The sides of triangles corresponding to boundary intervals are not glued during this process. We declare the $k-1$ vertices associated to such a boundary side as frozen vertices of $Q_k(\Delta)$, for a total of $(k-1)|\mathbb{M}_\partial|$ many frozen vertices. If any oriented two-cycles are created in the gluing process then they should be deleted from the quiver. 

The initial cluster $\mathbf{x}_k(\Delta)$ is obtained as follows. A choice of point $z \in \mathcal{M}$ determines an affine flag $F_p$ at every marked point $p \in \mathbb{M}$.  Given a triangle $\delta \in \Delta$, the cluster variable $x_\delta(a,b,c)$ indexed by vertex $(a,b,c) \in Q_k(\delta)$ evaluates on $z$ as follows. First, we parallel transport the tensors $F_{p_1,(a)}$,  $F_{p_2,(b)}$, and $F_{p_3,(c)}$ to a common point $p'$ in the interior of $\delta$, and then we set $x_\delta(a,b,c)(z) := F_{p_1,(a)}\wedge F_{p_2,(b)}\wedge F_{p_3,(c)}\in \bigwedge^k(V_{p'}) \cong \mathbb{C}$. 


\subsection{Initial seeds for $\mathcal{A}'_{G,\mathbb{S}}$}\label{subsec:GrSeeds}
We now carry out a similar procedure for the alternative version of the moduli space. As noted in Remark~\ref{rmk:Grkn}, when $\mathbb{S}$ is an $n$-gon, the moduli space is a Grassmannian and one should use the well-known initial seeds coming from {\sl plabic graphs}. Excluding this case, we can assume that $\mathbb{S}$ admits a taut triangulation~$\Delta$. 

Given such a taut $\Delta$, we construct an initial seed $\Sigma'_k(\Delta) = (Q_k'(\Delta),\mathbf{x}'_k(\Delta))$. We subsequently check that seeds from different taut triangulations are related by mutations, so that we obtain in this way a cluster algebra $\mathscr{A}'({\rm SL}_k,\mathbb{S})$.

Recall that a choice of point $z \in \mathcal{A}'_{G,\mathbb{S}}$ provides us with an affine flag at each puncture and a vector at each boundary point. We associate a ```proxy affine flag'' $F_p$ to boundary points~$p$ by the following construction. Let $C$ be a component of $\partial \mathbb{S}$ carrying boundary points $p_1,\dots,p_m$ in counterclockwise order. A choice of~$z$ yields vectors $v_{i}:= v_{p_i}$. We define $v_i$ for $i \in \mathbb{Z}$ using the monodromy matrix $M_C$ around the component $C$: $v_{i+jm}:= M_C^j(v_i)$ where $i \in [m]$ and $j \in \mathbb{Z}$. Then the proxy affine flag $F_{p_i}$ which we associate to the boundary point $p_i$ has steps $F_{p_i,(a)} := v_i \wedge v_{i+1} \wedge \cdots v_{i+a-1}$ for $a \in [k-1]$. (This recipe only works for sufficiently generic $v_i$'s, but that is all we need to specify our initial cluster variables.)

Since $\Delta$ is a taut triangulation, each triangle $\delta$ has at most one side which is a boundary interval. If $\delta$ has no such sides, we assign to it the the same quiver fragment $Q_k(\delta)$ and initial cluster variables $x_\delta(a,b,c)$ as in the previous section, using the proxy flags $F_p$ at boundary points $p$. If $\delta$ has a boundary side, we assign to it the quiver fragment $Q_k^1$ (cf.~Section~\ref{subsec:Qks}) with the bottom row of $Q_k^1$ identified with the boundary side and with the bottom left corner of $Q_k^1$ following the bottom right corner of $Q_k^1$ in the counterclockwise order around $\partial \mathbf{S}$. The cluster variables on the right
and left sides are of the form  $x_\delta(a,k-a,0)$ and $x_\delta(a,0,k-a)$. The vertex on the bottom side is frozen, and it carries the cluster variable $\det(v_i,\dots,v_{i+k-1})$ in the notation of the preceding paragraph. There are $|\mathbb{M}_\partial|$ many frozen vertices in total. 
The quiver $Q'_k(\Delta)$ is obtained by gluing the quiver fragments associated to triangles along their shared edges.



This completes our description of initial seeds. To complete our definition of the corresponding cluster algebra, our next three lemmas will establish that seeds from different taut triangulations are related by mutations.

We denote by $d(\Delta,\Delta')$ the distance between triangulations in the flip graph, i.e. the minimal number of flips needed to transform $\Delta$ into $\Delta'$. 

We let $b(\Delta)$ denote the number of boundary arches present in the the triangulation $\Delta$. We write $\Delta \sim \Delta'$ to indicate that $\Delta'$ can be obtained from $\Delta$ by a flip and use the notation $\gamma \mapsto \gamma'$ to indicate that a given flip move replaces the arc $\gamma \in \Delta$ with the arc $\gamma' \in \Delta'$. 
We classify the flip $\Delta \sim \Delta'$ as either a 0-move, a 1-move, or a $-1$-move according to the value of $b(\Delta)-b(\Delta')$. 

Clearly, if $d(\Delta,\Delta') = 1$ and both $\Delta$ and $\Delta'$ are taut, then the flip between them is a 0-move.


\begin{lemma}\label{lem:efficient}
Let $\Delta$ and $\Delta'$ be taut triangulations satisfying $d(\Delta,\Delta') \geq 2$. Then any minimal-length sequence of flips connecting $\Delta$ to $\Delta'$ has at least two 0-moves.
\end{lemma}

\begin{proof}
We will assume that there are either no 0-moves in the flip sequence, or that there is exactly one 0-move in the flip sequence, and eventually contradict the length-minimality of the flip sequence. 

When we perform a $1$-move, we flip a spanning arc $\gamma$ to obtain a boundary arch $\gamma'$. Any boundary arches which $\gamma'$ nests will remain boundary arches until we perform another flip at $\gamma'$. That is, boundary arches are removed in a 
``last in first out'' order. 

Let us focus on the {\sl last} $1$-move performed in our minimal flip sequence. The local picture before such move is drawn in Figure~\ref{fig:1move}. Vertices $L,M,R \subset \mathbb{M}$ reside on some boundary component~$C$ while $p \in \mathbb{M} \setminus C$. We write $\gamma = pM$ as a shorthand to say that $\gamma$ has endpoints $p$ and $M$. The below calculations take place in a polygon so we lose nothing by labeling arcs by endpoints in this way. We have $\gamma' = LR$, have arcs $\beta_1 = LM$ and $\beta_2 = MR$ nested by $\gamma'$  (either boundary arches or boundary intervals, depending on whether the vertices $L$, $M$, and $R$ are adjacent vertices along $C$), and have spanning arcs 
$\alpha_1 = pL$ and $\alpha_2 = pR$ forming part of the quadrilateral containing $\gamma$. 

If there are no 0-moves in our flip sequence, then after the flip $\gamma \mapsto \gamma'$, we do not flip at either of the arcs $\alpha_i$ because these would not be $-1$-moves.  And by the ``last in first out'' property, we do not flip at either of the arcs $\beta_i$ until we have first performed the inverse flip $\gamma' \mapsto \gamma$. At some point we {\sl must} perform the $-1$-move at $\gamma'$, but after appropriately commuting terms in the flip sequence, we see that the given flip sequence contains the subsequence $\gamma \mapsto \gamma' \mapsto \gamma$ contradicting the length-minimality of the sequence. 

If there {\sl is} a 0-move in the flip sequence, then the preceding argument shows that this 0-move comes {\sl after} the last 1-move, and moreover that this 0-move must be a flip at one of the $\alpha_i$'s or one of the $\beta_i$'s. 

In each of these four cases, one can use the flip identity in a pentagon (together with commutation moves) to obtain a shorter sequence of flips, contradicting the supposed minimality of the sequence. For example, consider the case that the 0-move is performed at $\alpha_1$ to obtain a new arc $\alpha_1' = qR$. Since this is a 0-move, we have $q \notin C$. In the given flip sequence, we see the flips $\gamma \mapsto \gamma'$,  
followed by the flip $\alpha_1 \mapsto \alpha_1'$ followed by $\gamma' \mapsto \gamma'' = qM$. But up to commutation moves, we 
can replace these three flips by the shorter sequence $\alpha_1 \mapsto \gamma''$ followed by $\gamma \mapsto \alpha'$ contradicting minimality. (Note that the flip at $qL$ is forbidden because it is not a $-1$-move. Flips at $qp$ are possible in the case that $q$ and $p$ are on the same boundary component, but the resulting arc will never be flipped again in the sequence. The contradiction from the preceding sentences still works in this case.) The case that we flip at a $\beta_i$ is similar. 
\end{proof}

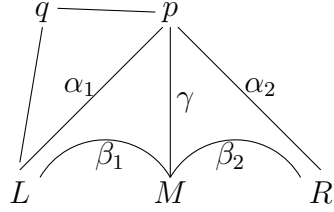
\begin{figure}
\begin{tikzpicture}
\node at (-2,-.2) {$L$};
\node at (0,-.2) {$M$};
\node at (2,-.2) {$R$};
\node at (0,2.2) {$p$};
\node at (-1.7,2.2) {$q$};
\node at (.2,1.0) {$\gamma$};
\node at (-.8,.28) {$\beta_1$};
\node at (.8,.28) {$\beta_2$};
\node at (-1.2,1.2) {$\alpha_1$};
\node at (1.2,1.2) {$\alpha_2$};
\draw (-2,.2)--(-1.7,2.0);
\draw (-.2,2.2)--(-1.5,2.25);
\draw (0,0) arc(30: 150:1cm);
\draw (0,0) arc(150: 30:1cm);
\draw (0,0)--(0,2);
\draw (.1,2)--(2,.1);
\draw (-.1,2)--(-2,.1);
\end{tikzpicture}
\caption{The local picture before the last 1-move in the proof of Lemma~\ref{lem:efficient}. Here $L$, $M$, $R$, $q$, $p$ indicate marked points whereas $\alpha_i$, $\beta_i$, and $\gamma$ indicate arcs. The flip at $\gamma$ is a 1-move. 
}\label{fig:1move}
\end{figure}


\begin{lemma}\label{lem:flipconnectedness}
The set of taut triangulations of~$\mathbb{S}$ is connected by flips. \end{lemma}

By a {\sl taut} flip sequence we will mean a set of sequence of flips which never leaves the set of taut triangulations. 

\begin{proof}
Observe that any triangulation $\Delta$ is exactly $b(\Delta)$ many flips away from a taut triangulation in the flip graph: one can reach such a triangulation by performing $-1$-moves at boundary arches in ``last in first out'' order. 

We now prove by induction on $d(\Delta,\Delta')$ that any two taut triangulations $\Delta,\Delta'$ are connected by a taut flip sequence. 

Let $\underline{\Delta} := \Delta= \Delta_0 \sim \Delta_1 \sim \cdots \sim \Delta_d = \Delta'$ be a shortest walk from $\Delta$ to $\Delta'$ in the flip graph. Abbreviate $b_i := b(\Delta_i)$. 

Define $i \geq 1$ by the property that the first $i-1$ many moves in the flip sequence are 1-moves and the $i$th move is either a 0-move or a $-1$-move. Thus, $b_{i-1} = i-1$ and $b_i < i$. Altogether, the flip sequence has at least $i-1$ many 1-moves, thus at least $i-1$ many $-1$-moves, and also at two many 0-moves by Lemma~\ref{lem:efficient}. Thus, $2i \leq d$.

By the observation in the first paragraph, we can choose a taut triangulation $\Delta''$ whose distance from $\Delta_i$ in the flip graph equals $b_i$. Then 
$$d(\Delta'',\Delta_d) \leq d(\Delta'',\Delta_i)+d(\Delta_i,\Delta_d) = b_i+(d-i) <d$$
 and similarly 
$$d(\Delta_0,\Delta'') \leq d(\Delta_0,\Delta_i)+d(\Delta_i,\Delta'') = i+b_i < 2i \leq d.$$

By induction, $\Delta''$ is connected to both $\Delta_d$ and $\Delta_0$ by a taut flip sequence, so that $\Delta_d$ and $\Delta_0$ are connected to each other by a taut flip sequence as required. 
\end{proof}

\begin{lemma}\label{lem:flipismutation}
If $\Delta$ and $\Delta'$ are taut triangulations related by a flip, then 
$\Sigma'_k(\Delta)$ and $\Sigma'_k(\Delta')$ are mutation-equivalent seeds. 
\end{lemma}

\begin{proof}
If neither of the two triangles involved in a flip has a boundary side, then the statement is a specialization of Fock and Goncharov's \cite[Proposition 10.1]{FGMod}.

We next consider the case that one of the two triangles involved in the flip has a boundary side. We indicate the mutation sequence which implements the flip in the case $G = {\rm SL}_4$ in the below picture. The triangle $pb_1b_2$ has an exposed side $b_1b_2$. The quiver $pqb_1$ has no exposed sides, and we wish to perform the flip at the arc $pb_1$ obtaining the non-exposed triangle $pqb_2$ and the exposed triangle $qb_1b_2$. The mutation sequence which simulates this flip of trinagulations is to mutate at the vertices labeled 1,2,3,4,5,6 (in that order) in the left picture below. The result is the right picture, which is the quiver we expect after performing the flip $pb_1 \mapsto qb_2$. 
\begin{center}
\begin{equation}\label{eq:thinfat}
\begin{tikzpicture}[scale = .75]
\coordinate (AAAC) at (-.75,3);
\coordinate (AAAB) at (.75,3);
\coordinate (AACC) at (-1.5,2);
\coordinate (AABB) at (1.5,2);
\coordinate (ACCC) at (-2.25,1);
\coordinate (ABBB) at (2.25,1);
\coordinate (BBCC) at (0,0);
\coordinate (Q1) at (1.5,3.75);
\coordinate (Q21) at (2.25,2.75);
\coordinate (Q22) at (3.0,3.5);
\coordinate (Q31) at (2.25+.66,1.66);
\coordinate (Q32) at (2.25+2*.66,1+2*.66);
\coordinate (Q33) at (2.25+3*.66,1+3*.66);
\coordinate (Q41) at (3+.66,0.66);
\coordinate (Q42) at (3+2*.66,0+2*.66);
\coordinate (Q43) at (3+3*.66,0+3*.66);

\node at (AAAC) {$\bullet$};
\node at (AACC) {$\bullet$};
\node at (ACCC) {$\bullet$};
\node at (BBCC) {$\boxed{\bullet}$};

\node at (0,4) {$p$};
\node at (-3,0) {$b_2$};
\node at (3,0) {$b_1$};
\node at (AAAB) {$1$};
\node at (AABB) {$2$};
\node at (ABBB) {$4$};

\node at (1.5,3.75) {$\bullet$};
\node at (3.0,3.5) {$\bullet$};
\node at (2.25+3*.66,1+3*.66) {$\bullet$};  
\node at (3+.66,0.66) {$\bullet$};  
\node at (3+2*.66,0+2*.66) {$\bullet$}; 
\node at (3+3*.66,0+3*.66) {$\bullet$}; 
\node at (2.25,2.75) {$3$};
\node at (2.25+.66,1.66) {$5$}; 
\node at (2.25+2*.66,1+2*.66) {$6$};  
\node at (3+4*.66,0+4*.66) {$q$};

\draw [shorten >=0.25cm,shorten <=0.25cm,->] (AAAB)--(AAAC);
\draw [shorten >=0.25cm,shorten <=0.25cm,->] (AABB)--(AACC);
\draw [shorten >=0.25cm,shorten <=0.25cm,->] (ABBB)--(ACCC);
\draw [shorten >=0.35cm,shorten <=0.25cm,->] (AAAC)--(AABB);
\draw [shorten >=0.25cm,shorten <=0.25cm,->] (AABB)--(AAAB);
\draw [shorten >=0.45cm,shorten <=0.55cm,->] (AACC)--(ABBB);
\draw [shorten >=0.25cm,shorten <=0.25cm,->] (ABBB)--(AABB);
\draw [shorten >=0.35cm,shorten <=0.35cm,->] (ACCC)--(BBCC);
\draw [shorten >=0.35cm,shorten <=0.35cm,->] (BBCC)--(ABBB);
\draw [shorten >=0.15cm,shorten <=0.15cm,->] (Q1)--(AAAB);
\draw [shorten >=0.15cm,shorten <=0.15cm,->] (Q22)--(Q21);
\draw [shorten >=0.15cm,shorten <=0.15cm,->] (Q21)--(AABB);
\draw [shorten >=0.15cm,shorten <=0.15cm,->] (Q33)--(Q32);
\draw [shorten >=0.15cm,shorten <=0.15cm,->] (Q32)--(Q31);
\draw [shorten >=0.15cm,shorten <=0.15cm,->] (Q31)--(ABBB);
\draw [shorten >=0.15cm,shorten <=0.15cm,->] (AAAB)--(Q21);
\draw [shorten >=0.15cm,shorten <=0.15cm,->] (Q21)--(Q1);
\draw [shorten >=0.15cm,shorten <=0.15cm,->] (AABB)--(Q31);
\draw [shorten >=0.15cm,shorten <=0.15cm,->] (Q31)--(Q21);
\draw [shorten >=0.15cm,shorten <=0.15cm,->] (Q21)--(Q32);
\draw [shorten >=0.15cm,shorten <=0.15cm,->] (Q32)--(Q22);
\draw [shorten >=0.15cm,shorten <=0.15cm,->] (ABBB)--(Q41);
\draw [shorten >=0.15cm,shorten <=0.15cm,->] (Q41)--(Q31);
\draw [shorten >=0.15cm,shorten <=0.15cm,->] (Q31)--(Q42);
\draw [shorten >=0.15cm,shorten <=0.15cm,->] (Q42)--(Q32);
\draw [shorten >=0.15cm,shorten <=0.15cm,->] (Q32)--(Q43);
\draw [shorten >=0.15cm,shorten <=0.15cm,->] (Q43)--(Q33);

\begin{scope}[xshift = 11.5cm]
\coordinate (AAAC) at (-.75,3);
\coordinate (AAAB) at (.75,3);
\coordinate (AACC) at (-1.5,2);
\coordinate (AABB) at (1.5,2);
\coordinate (ACCC) at (-2.25,1);
\coordinate (ABBB) at (2.25,1);
\coordinate (BBCC) at (0,0);
\coordinate (Q1) at (1.5,3.75);
\coordinate (Q21) at (2.25,2.75);
\coordinate (Q22) at (3.0,3.5);
\coordinate (Q31) at (2.25+.66,1.66);
\coordinate (Q32) at (2.25+2*.66,1+2*.66);
\coordinate (Q33) at (2.25+3*.66,1+3*.66);
\coordinate (Q41) at (3+.66,0.66);
\coordinate (Q42) at (3+2*.66,0+2*.66);
\coordinate (Q43) at (3+3*.66,0+3*.66);

\node at (AAAC) {$\bullet$};
\node at (AACC) {$\bullet$};
\node at (ACCC) {$\bullet$};
\node at (BBCC) {$\boxed{\bullet}$};

\node at (0,4) {$p$};
\node at (-3,0) {$b_2$};
\node at (3,0) {$b_1$};
\node at (AAAB) {$1'$};
\node at (AABB) {$2'$};
\node at (ABBB) {$4'$};

\node at (1.5,3.75) {$\bullet$};
\node at (3.0,3.5) {$\bullet$};
\node at (2.25+3*.66,1+3*.66) {$\bullet$};  
\node at (3+.66,0.66) {$\bullet$};  
\node at (3+2*.66,0+2*.66) {$\bullet$}; 
\node at (3+3*.66,0+3*.66) {$\bullet$}; 
\node at (2.25,2.75) {$3'$};
\node at (2.25+.66,1.66) {$5'$}; 
\node at (2.25+2*.66,1+2*.66) {$6'$};  
\node at (3+4*.66,0+4*.66) {$q$};

\draw [shorten >=0.25cm,shorten <=0.25cm,->] (Q1)--(AAAC);
\draw [shorten >=0.25cm,shorten <=0.25cm,->] (AAAC)--(AAAB);
\draw [shorten >=0.25cm,shorten <=0.25cm,->] (AAAB)--(Q1);
\draw [shorten >=0.25cm,shorten <=0.25cm,->] (Q22)--(AAAB);
\draw [shorten >=0.25cm,shorten <=0.25cm,->] (AAAB)--(Q21);
\draw [shorten >=0.25cm,shorten <=0.25cm,->] (Q21)--(Q22);

\draw [shorten >=0.25cm,shorten <=0.25cm,->] (AAAB)--(AACC);
\draw [shorten >=0.25cm,shorten <=0.25cm,->] (AACC)--(AABB);
\draw [shorten >=0.25cm,shorten <=0.25cm,->] (AABB)--(AAAB);

\draw [shorten >=0.25cm,shorten <=0.25cm,->] (AABB)--(ACCC);
\draw [shorten >=0.25cm,shorten <=0.25cm,->] (ACCC)--(ABBB);
\draw [shorten >=0.25cm,shorten <=0.25cm,->] (ABBB)--(AABB);

\draw [shorten >=0.25cm,shorten <=0.25cm,->] (Q21)--(AABB);
\draw [shorten >=0.25cm,shorten <=0.25cm,->] (AABB)--(Q31);
\draw [shorten >=0.25cm,shorten <=0.25cm,->] (Q31)--(Q21);

\draw [shorten >=0.25cm,shorten <=0.25cm,->] (Q33)--(Q21);
\draw [shorten >=0.25cm,shorten <=0.25cm,->] (Q21)--(Q32);
\draw [shorten >=0.25cm,shorten <=0.25cm,->] (Q32)--(Q33);

\draw [shorten >=0.25cm,shorten <=0.25cm,->] (ABBB)--(BBCC);
\draw [shorten >=0.25cm,shorten <=0.25cm,->] (BBCC)--(Q41);
\draw [shorten >=0.25cm,shorten <=0.25cm,->] (Q41)--(ABBB);

\draw [shorten >=0.25cm,shorten <=0.25cm,->] (Q31)--(Q41);
\draw [shorten >=0.25cm,shorten <=0.25cm,->] (Q41)--(Q42);
\draw [shorten >=0.25cm,shorten <=0.25cm,->] (Q42)--(Q31);

\draw [shorten >=0.25cm,shorten <=0.25cm,->] (Q32)--(Q42);
\draw [shorten >=0.25cm,shorten <=0.25cm,->] (Q42)--(Q43);
\draw [shorten >=0.25cm,shorten <=0.25cm,->] (Q43)--(Q32);
\end{scope}
\end{tikzpicture}
\end{equation}
\end{center}
We have used primes to denote the result of performing the mutation, so that e.g. $1'$ is the result of mutating at vertex 1. The asserted mutation sequence readily extends to the ${\rm SL}_k$ case. 

One should check that the cluster variables transform as expected. 
To see this, observe that each vertex in the mutation sequence is 4-valent at the moment it is mutated. The corresponding exchange relation is an instance of the three-term Pl\"ucker relations. Suppose the proxy affine flag at $b_1$ is represented by vectors $v_1,\dots,v_k$, the 
proxy affine flag at $b_2$ is represented by $v_2,\dots,v_{k+1}$, 
the affine flag at $p$ is represented by $u_1,\dots,u_k$, and the affine flag at $q$ is represented by  $w_1,\dots,w_k$. (See Remark~\ref{rmk:representatives} for the notion of vector representatives of an affine flag.) Let $(\alpha,\beta,\gamma,\eta)$ denote the function 
$$\det(u_1,\dots,u_\alpha,w_1,\dots,w_\beta,v_1,\dots,v_\gamma,v_2,\dots,v_{1+\eta}),$$
 then the mutation sequence sends the variable $(\alpha,\beta,\gamma,0)$ to the variable $(\alpha-1,\beta+1,0,\gamma)$ via the Pl\"ucker relation
$$ (\alpha,\beta,\gamma,0) \times (\alpha-1,\beta+1,0,\gamma) =  (\alpha,\beta,0,\gamma) \times (\alpha-1,\beta+1,\gamma,0)
+
(\alpha,\beta+1,0,\gamma-1) \times (\alpha-1,\beta,\gamma+1,0).$$

For example, in the above example, mutation at vertex 1 (whose variable is $\det(u_1 u_2 u_3 v_1)$) yields the new variable given by $\det(u_1u_2w_1v_2)$. The right hand side of the Pl\"ucker relation is
$$
\det(u_1 u_2 v_1 v_2)\det(u_1 u_2 u_3 w_1)+
\det(u_1 u_2 w_1 v_1)\det(u_1 u_2 u_3 v_2).$$
Note that $1'$ is a mutable variable in the seed fragment $\Sigma_r(\delta)$ where $\delta$ is the triangle $pqb_2$. 
This completes our discussion of the case in which we have one boundary side.

Finally we consider the case that both triangles has a boundary side.
Since the flip results in a taut triangulation, these two sides must be opposite each other in the quadilateral in which the flip is performed. There are $k-1$ vertices of $Q_k(\delta)$ on the edge which needs to be flipped. These vertices have no arrows between them and each vertex is four-valent in the quiver. Mutating at each of them (in any order, since these mutations commute) realizes the flip. As above, each exchange relation is a three-term Pl\"ucker relation, and the variable that results is the one which we would expect to see when we perform the flip. 
\end{proof}


\begin{remark}
We do not prove here that our initial clusters provide rational coordinate systems on~$\mathcal{A}'_{{\rm SL}_k,\mathbb{S}}$ although we believe this should be true. One might be able to prove these statements by mimicking the proofs of these statements given in \cite{FGMod}, or by establishing the cluster fibration property from Remark~\ref{rmk:subtle}. 

We content ourselves here by showing that the size of our initial cluster matches the dimension of the moduli space. Indeed, the difference in the dimensions of the two moduli spaces 
is $|\mathbb{M}_\partial|(\dim G/U - \dim V) = |\mathbb{M}_\partial|(\binom k 2 -1)$ using 
\eqref{eq:flagparametercount} and \eqref{eq:vectorparametercount}.
By inspection, the quiver $Q^1_k$ has $\binom k 2-1$ fewer vertices than $Q_k$. The cardinality $|\mathbb{M}_\partial|$ is the number of boundary intervals, which is the number of times $Q_k^1$ rather than $Q_k$ is used when creating the quiver $Q'_k(\Delta)$. The claim follows. 
\end{remark}

\begin{example}
When $\mathbb{S}$ is an annulus with one point on each boundary component, the cluster algebra $\mathscr{A}'(\mathbb{S},{\rm SL}_k)$ has the same cluster type as the $Q$-system of type $A_{k-1}$ see e.g. \cite{Kedem}.

Recall that $D_{n,1}$ has a unique taut triangulation, hence $\mathscr{A}'({\rm SL}_k,D_{n,1})$ has a canonical choice of initial seed. The mutable part of this seed is the cylindrical triangulated grid quiver $Q_{m,n}$ considered in \cite{ILP}. Thus Lemma~\ref{lem:GS} below provides a different proof of Theorem 3.2 therein. 
\end{example}

\begin{remark}\label{rmk:recipe}
We believe that it is possible to assign a seed in the cluster algebra $\mathscr{A}'_{G,\mathbb{S}}$ to any regular triangulation $\Delta$, not only to the taut ones. We explain this recipe now without proof. 

Suppose that $\delta \in \Delta $ has exactly one side $\gamma$ which is a boundary arch contractible to 
$s$ many consecutive boundary intervals. Then one should associate to this triangle the quiver $Q_k^{\min(s,k-1)}$. If $s \leq k-1$, then  the rightmost vertex on the bottom side is  frozen. 

If $\delta$ has two such sides, one should apply this recipe to {\sl both} of the exposed sides (deleting and gluing vertices as in the definition of $Q_k^s(\delta)$). See \cite[Figure 21]{FP} for an illustration of this recipe when $G = {\rm SL}_3$.  
\end{remark}

\subsection{Weights of cluster variables}
Let $\mathcal{M}$ be one of the moduli spaces either $\mathcal{A}_{G,\mathbb{S}}$ or $\mathcal{A}'_{G,\mathbb{S}}$. Recall the right $T$-action \eqref{eq:Taction} on affine flags by rescaling the $F_{(i)}$'s. It determines a right action $ \mathcal{M} \curvearrowleft T^{\mathbb{M}_\circ}$ by separately rescaling the decorations at each puncture. 

We let $P^{\mathbb{M}_\circ}$ denote the direct sum of weight lattices indexed by the punctures in $\mathbb{S}$, and denote by 
$$\pi_p \colon P^{\mathbb{M}_\circ} \to P$$
the projection onto the copy of the weight lattice indexed by a given puncture $p \in \mathbb{M}_\circ$.

A rational function $f \in \mathbb{C}(\mathcal{M})$ is {\sl homogeneous} if there exists a weight vector $\lam\in P^{\oplus \mathbb{M}_\circ}$ such that 
$$f(z \cdot t) = t^{\lam}f(z), \text{ for } z \in  \mathcal{M}, t \in T^{\mathbb{M}_\circ}.$$
For such an $f$, define  
\begin{equation*}
{\rm wt}(f):= \lambda \in P^{\oplus \mathbb{M}_\circ}, \text{ the {\sl weight} of $f$} \hspace{1cm} {\rm wt}_p(f):= \pi_p(\lambda) \in P, \text{ the {\sl weight of $f$ at} $p$.}
\end{equation*}

Each of the initial cluster variables in $\mathscr{A}(\mathcal{M})$ is a homogeneous function with nonzero weight at three or fewer punctures. For example, let $\delta \in \Delta$ be a triangle with three distinct vertices and with ``top vertex'' $p_1$ a puncture rather than a boundary point. Then the initial cluster variable $x_\delta(a,b,c)$ in this triangle has weight $\omega_a$ at $p_1$. One can see that the two cluster monomials on the right hand side of its exchange are homogeneous functions of the same weight, i.e. they satisfy the balancing condition \eqref{eq:exchangemonomials} necessary to determine an initial $P$-seed. The frozen variables have weight zero, so one can delete them with no effect on the $P$-cluster variables. 

Henceforth when we refer to $P$-clusters, $P$-cluster variables, etc., we always have in mind that we have fixed a choice of moduli space $\mathcal{M}$ and are discussing those $P$-clusters which are reachable from our initial  $P$-seeds by mutations.

For a $P$-cluster $\mathcal{C} \subset P^{\mathbb{M}_\circ}$ and puncture $p$, the projection $\pi_p(\mathcal{C})\subset P$ is the $P$-{\sl cluster  at} $p$.

\begin{example}
When $k=2$, a plain (resp. notched) end of a tagged arc ending at a puncture~$p$ contributes a copy of the standard basis vector $e_1$ (resp. $e_2$) to the weight of the corresponding cluster variable at $p$. Thus, the weight of any cluster variable at $p$ lies in $\{0,e_1,e_2,2e_1,2e_2\}$. The $P$-cluster at $p$ is either a multiset drawn from $\{0,e_1,2e_1\}$ (all arcs plain at $p$), a multiset drawn from $\{0,e_2,2e_2\}$ (all arcs notched at $p$), or it has exactly one $P$-cluster variable of weight $e_1$, one of weight $e_2$, and all others of weight zero at $p$. \end{example}

\section{Weyl group action at punctures}\label{secn:WActs}
For any puncture $p$, Goncharov and Shen identified a birational action of the Weyl group of~$G$ on the space of decorated $G$-local systems by changing the affine flag at $p$. They proved, in most cases, that the resulting Weyl group action on $\mathscr{A}({\rm SL}_k,\mathbb{S})$ is by cluster automorphisms. These Weyl group symmetries play an important role in the cluster combinatorics of $\mathscr{A}(\mathcal{M})$. We review the construction and show that the action is cluster in some new cases. 

\subsection{Weyl group action on affine flags}
\begin{definition}\label{defn:Weylgrouponflags}
Let $u$ be a unipotent matrix and $F$ be an affine flag stabilized by~$u$. Choose a vector $v$ with the property that $F_{(i+1)} = F_{i} \wedge v$. Define an affine flag $s_i\cdot F$ whose $i$th step is given by the formula $(s_i \cdot F)_{(i)} := F_{(i-1)} \wedge (u(v)-v)$ and whose other steps agree with $F$. 
\end{definition}
This recipe only yields an affine flag when $u(v) - v \notin {\rm span}\, F_{(i-1)}$, in which case 
it is independent of the chosen $v$. The resulting affine flag is again stabilized by~$u$. The operations $s_i$ are involutions satisfying the relations \eqref{eq:braidrelations}, thus they determine a (rational) $W$-action on the set of affine flags stabilized by~$u$.

\begin{remark}
Representing $F$ by vectors $v_1,\dots,v_k$ satisfying $F_{(a)} = v_1 \wedge \cdots \wedge v_a$, the flag $s_i\cdot F$ can be represented by vectors 
\begin{equation}\label{eq:changetwosteps}
v_1,\dots,v_{i-1},\mathcal{W}_i v_i, \frac{1}{\mathcal{W}_i} v_{i+1},v_{i+2},\dots,v_k \text{ where } \mathcal{W}_i := \frac{F_{(i-1)} \wedge (u(v_{i+1})-v_{i+1})}{F_{(i)}} \in \mathbb{C}.
\end{equation}
Thus, applying $s_i$ changes the vector representatives in two positions 
$i$ and $i+1$ even though it only changes the affine flag in step~$i$. Expressing the linear transformation $u$ in the ordered basis $v_1,\dots,v_k$, the quantity $\mathcal{W}_i$ is the subdiagonal matrix entry in row $i$ and column $i+1$. Fixing $v_1,\dots,v_k$ and letting $u$ vary, the quantity $\mathcal{W}_i$ is logarithmic in $u$: multiplication of matrices $u$ corresponds to addition of the quantity $\mathcal{W}_i$.
\end{remark}

The following lemma follows from the definition and is used later in the paper. 
\begin{lemma}\label{lem:parabolic}
Let $F$ be an affine flag with a choice of $u \in {\rm stab}(F)$. Then the tensor $(w\cdot F)_{(a)} \in \bigwedge^a(V)$ obtained by Weyl group action along $w$ depends only on the first $a$ symbols of $w$, i.e. on $w|_{[a]}$.
\end{lemma}

Now, following Goncharov and Shen, we relate the  Weyl group action on affine flags with the cluster algebra $\mathscr{A}(\mathcal{M})$ for one of the moduli spaces $\mathcal{M}$ considered in this paper. We denote by~${\rm Bir}(\mathcal{M})$ the group of birational automorphisms of the moduli space~$\mathcal{M}$.

Fix $p \in \mathbb{M}_\circ$.  A choice of point in $\mathcal{M}$ determines an affine flag $F_p$ at $p$, stabilized by the monodromy $u_p$ of the local system around $p$. Applying Definition~\ref{defn:Weylgrouponflags} when $u$ is taken to be this $u_p$, we have a map 
\begin{equation}\label{eq:psiip}
\psi_{i,p} \in {\rm Bir}(\mathcal{M}) \text{ sending } F_p \mapsto s_i\cdot F_p, 
\end{equation}
changing the decoration at the puncture~$p$ while changing neither the underlying local system nor the the decorations at 
$\mathbb{M} \setminus \{p\}$. 

The actions at various punctures commute, yielding a group homomorphism $\psi \colon W^{\mathbb{M}_\circ}\to {\rm Bir}(\mathcal{M})$. This action is compatible with the $W$-action on the weight lattice: for homogeneous~$f$, 
\begin{equation}\label{eq:actiononweights}
{\rm wt}(f \circ \psi(\underline{w})) = \underline{w} \cdot {\rm wt}(f) \in P^{\oplus \mathbb{M}_\circ}, \text{ for } \underline{w} \in W^{\mathbb{M}_\circ}.
\end{equation}
In particular the homomorphism $\psi$ is injective. 

\begin{theorem}[{\cite[Theorem 1.2]{GoncharovShenDT}}]\label{thm:GS}
Let $\mathcal{M} = \mathcal{A}_{{\rm SL}_k,\mathbb{S}}$ where $\mathbb{S}$ is neither an $S_{g,1}$ (for any $k$) when $k=2$. Then $\psi(W^{\mathbb{M}_\circ}) \subset {\rm Bir}(\mathcal{M})$
acts by cluster automorphisms of the cluster algebra $\mathscr{A}(\mathcal{M})$.
\end{theorem}

For an extension of this theorem to other Lie types see~\cite{IIO}.

In the remainder of this section we extend Theorem~\ref{thm:GS} to two new settings, namely to $\mathcal{A}'_{{\rm SL}_k,\mathbb{S}}$ and also the cases $(k,S_{g,1})$ when $k>2$. The statement is false 
for $S_{g,1}$ and for $S_{0,3}$ when $k=2$. 

The Weyl group action at punctures is a crucial ingredient to the cluster combinatorics of $\mathscr{A}(\mathcal{M})$, so we felt it was important to treat these cases here. Our proof is a modification of the ideas already present in Goncharov and Shen's proof of the above theorem. The reader who is not interested in these two special cases can safely skip to Section~\ref{secn:Dosps}.

\subsection{Reminders on the proof of {Theorem~\ref{thm:GS}}}  
We fix a puncture~$p \in \mathbb{M}_\circ$ and a regular triangulation $\Delta$ of $\mathbb{S}$ with the property that no arc in $\Delta$ is a loop based at~$p$. (The existence of such triangulations when $\mathbb{S} \neq S_{g,1}$ is asserted in \cite{GoncharovShenDT}.) In the special case that $\mathbb{S}	= S_{g,1}$, {\sl every} arc is a loop based at the puncture, which is why this case is excluded from Theorem~\ref{thm:GS}. 

For an initial cluster variable $x \in \mathbf{x}_k(\Delta)$, we will compute $\psi^*_{i,p}(x)$ as a Laurent polynomial in the initial cluster variables $\mathbf{x}_k(\Delta)$. Then we show that we can reach this set of Laurent polynomials
$\{\psi^*_{i,p}(x) \colon x \in \mathbf{x}_k(\Delta)\}$ by mutations, proving in this way that $\psi^*_{i,p}$ sends clusters to clusters.

Consider a triangle $\delta \in \Delta$ with $p$ as one of its vertices. Identify $p$ with the 
``top vertex'' $(k,0,0) \in \hat{\mathbb{H}}_k(\delta)$ and recall that the initial cluster variable 
$x_\delta(a,b,c)$ has ${\rm wt}_p(x_\delta(a,b,c)) = \omega_a \in P$. From this and the definition \eqref{eq:psiip}, we get a formula for the Weyl group action: \begin{equation}\label{eq:partialpotential1}
\psi^*_{i,p}(x_\delta(a,b,c)) = x_\delta(a,b,c) \text{ if $a \neq i$, and } \hspace{.5cm} \psi^*_{i,p}(x_\delta(i,b,c)) = \mathcal{W}_i x_\delta(i,b,c),
\end{equation}
where $\mathcal{W}_i$ is the quantity \eqref{eq:changetwosteps} in the case that $u = u_p$ is the monodromy around~$p$. That is, $\psi^*_{i,p}$ only affects the variables in row $i$ around $p$ and it rescales these variables by the factor $\mathcal{W}_i$.

An edge $e = (i,b,c) - (i,b+1,c-1)$ in a triangle $\delta$ with $p$ as one of its vertices 
determines a {\sl rhombus monomial} $\textnormal{RM}(e) = \frac{x_\delta(i+1,b,c)x_\delta(i-1,b+1,c)}{x_\delta(i,b,c)x_\delta(i,b+1,c)}$. Goncharov and Shen show that the quantity $\mathcal{W}_i$ is a Laurent polynomial in initial cluster variables: 
\begin{equation}\label{eq:partialpotential2}
\mathcal{W}_i = \sum_{e \text{ in row $i$ of $\delta$, $\delta$ incident to $p$}} \textnormal{RM}(e).
\end{equation}
Combining formulas \eqref{eq:partialpotential1} and \eqref{eq:partialpotential2} gives us a Laurent polynomial formula for the action of $\psi^*_{i,p}$.

It remains to show that this formula amounts to a sequence of mutations. As a preliminary step, consider a quiver which is an oriented $m$-cycle on the vertex set $[m]$ with arrows $j \to j+1 \mod m$. Add to this quiver $2m$ frozen vertices $j^\pm$ as well as arrows $j+1 \to j^\pm \to j$. We denote the mutable initial cluster variable at vertex~$j$ by $x_j $and the frozen variables at $j^\pm$ by $x_j^\pm $. Note that the mutable edge $j \to j+1$ is part of two three-cycles $j \to j+1 \to j^+ \to j$ and $j \to j+1 \to j^- \to j$ together forming a ``rhombus.''

Consider the permutation-mutation sequence 
\begin{equation}\label{eq:permutationmutation}
\mu_{1} \circ \mu_2 \circ \cdots \mu_{m-1} \circ (m-1,m) \circ \mu_{m-1} \circ \mu_{m-2} \circ \cdots \mu_1
\end{equation}
where the middle term is a transposition written in cycle notation and $\mu_i$ denotes mutation at vertex~$i$. It is known that this permutation-mutation sequence determines an involutive automorphism of the cluster algebra which is independent of the choice of which vertex is called~$1$. The automorphism rescales initial mutable variables by the Laurent polynomial
\begin{equation}\label{eq:partialpotential3}
\sum_{j \in [m]}\frac{x_j^+x_j^-}{x_jx_{j+1}},
\end{equation}
cf.~\cite[Theorem 3.1]{ILP} or \cite[Theorem 7.7]{GoncharovShenDT}.

Returning to the setting of Theorem~\ref{thm:GS}, note that when we restrict the quiver $Q_k(\Delta)$ to the vertices which are in rows~$i$ in the triangles surrounding $p$, we get an oriented $m$-cycle for some $m$. The subquiver in rows $i-1,i,i+1$ is a specialization of the quiver from \eqref{eq:partialpotential3}:
the frozen variables $x_j^\pm$ from \eqref{eq:partialpotential3} should be specialized to the variables $x_\delta(i+1,b,c)$ and $x_\delta(i-1,b+1,c)$ from \eqref{eq:partialpotential2}. After such specialization, the Laurent polynomials 
\eqref{eq:partialpotential3} and \eqref{eq:partialpotential2} coincide, completing the proof.

Finally, for the purposes of our generalization below, let us recall how the formula \eqref{eq:partialpotential2} is proved. By a well-known recipe a generic pair $(F,G)$, with $F$ an affine flag and $G$ a complete flag, determines an ordered basis in~$V$. (Intersecting ${\rm span }\, F_i$ with ${\rm span }\, G_{k-i+1}$ determines an ordered basis up to rescaling each of the basis vectors; the scalars are fixed using the extra information provided by the affine flag $F$.)

Let $\delta$ be a triangle with a puncture $p$ as its top vertex and with $p_2$ and $p_3$ its bottom left and bottom right vertices. A choice of point in the moduli space $\mathcal{M}$ determines affine flags $F_p$, $F_{p_2}$, and $F_{p_3}$. The preceding paragraph gives ordered bases associated to the pairs $(F_p,F_{p_2})$ and also to $(F_p,F_{p_3})$, i.e. to the left and right sides of~$\delta$. There is a unique matrix $u(\delta)$ which sends the left ordered basis to the right ordered basis, and this matrix stabilizes $F_p$. By \cite[Section 3.1]{GoncharovShenCanonical}, the quantity $\mathcal{W}_i$ \eqref{eq:changetwosteps} when calculated on the matrix $u(\delta)$ is the Laurent polynomial
$\sum_{e \text{ in row $i$ of $\delta$}} \textnormal{RM}(e)$. 
If $\delta_1,\dots,\delta_s$ are the triangles surrounding $p$ in cyclic order, it follows that $u_p = \prod_{i=1}^su(\delta_i)$. The formula \eqref{eq:partialpotential2} follows because the quantity $\mathcal{W}_i$ is logarithmic in the matrix~$u$.

\subsection{The action on $\mathcal{A}'_{G,\mathbb{S}}$ is cluster.}
Now we give the analogue of Theorem~\ref{thm:GS} in the case that boundary points carry vector decorations. 
\begin{lemma}\label{lem:GS}
The $W^{\mathbb{M}_\circ}$-action on $\mathscr{A}'({\rm SL}_k,\mathbb{S})$ is by cluster automorphisms. 
\end{lemma}

\begin{proof}We can assume that $\mathbb{S}$ has boundary because otherwise we are dealing with the Fock-Goncharov version of the moduli space. Fix a puncture $p$ and consider a taut triangulation $\Delta$ which has the property that no arc is a loop based at $p$. (To construct such, start with any regular triangulation with no loops and mutate at boundary arches in last in first out fashion as in the proof of Lemma~\ref{lem:efficient}.) It still makes sense to consider vertices which are in row~$i$ around $p$. These will reside in $Q_k(\delta)$ or in $Q_k^1(\delta)$ according to whether or not $\delta$ has a boundary side and these again have weight $\omega_i$ at $p$ so that the formula \eqref{eq:partialpotential1} holds. 

We can mimic the argument from \cite[Section 3.1]{GoncharovShenCanonical} factoring the matrix $u_p$ into matrices $u(\delta)$ indexed by triangles $\delta$ surrounding $p$ and get the corresponding formula \eqref{eq:partialpotential2}. When $\delta$ has a boundary side, there is only one rhombus which contributes to the sum  \eqref{eq:partialpotential2} because there are fewer steps required to move the ordered basis associated to the left side of $\delta$ to that for the right side of $\delta$.

The induced subquiver on the vertices in row $i$ around $p$ is again an oriented $m$-cycle for some $m$, and each edge in this cycle is part of a ``rhombus.'' The formula \eqref{eq:partialpotential3} again specializes to the rhombus Laurent monomial \eqref{eq:partialpotential2}, proving the assertion. 
\end{proof}

\subsection{The action on $\mathscr{A}({\rm SL}_k,S_{g,1})$ is cluster.}
When $k=2$, the Weyl group action at punctures is the tag-changing transformation and is not cluster in the special case that $\mathbb{S} = S_{g,1}$ for some $g \geq 1$. See Example~\ref{eg:taggedasskein} for a detailed comparison between the $s_1$ action at punctures and plain versus notched tagging of arcs. 

Goncharov and Shen do not discuss whether or not the $W$-action on $\mathscr{A}({\rm SL}_k,S_{g,1})$ is cluster when $k>2$. We resolve this affirmatively: 
\begin{proposition}\label{prop:Sg1iscluster}
When $k>2$, the $W$-action on $\mathscr{A}({\rm SL}_k,S_{g,1})$ is by cluster automorphisms.
\end{proposition}

As an ingredient of our proof, we recall that the cluster algebra $\mathscr{A}_{{\rm SL}_k,\mathbb{S}} $ admits a cluster automorphism $\ast$, the {\sl duality map}, which implements the outer automorphism of $i \mapsto k-i$ of the type $A_{k-1}$ Dynkin diagram \cite[Section 9]{GoncharovShenDT}. In particular, we have 
\begin{equation}\label{eq:cutinhalf}
\ast \circ \psi_{i,p} \circ \ast = \psi_{k-i,p},
\end{equation}
where $\psi_{i,p} \in {\rm Bir}(\mathcal{M})$ denotes Weyl group action at $p$.

In terms of points in the moduli space $\mathcal{M}$, the automorphism $\ast$ amounts to replacing the vector space $V$ and its volume form  with their duals. Importantly for us, Goncharov and Shen show that $\ast$ is a cluster automorphism even for $\mathcal{M} = \mathcal{A}_{{\rm SL}_k,S_{g,1}}$.

\begin{proof}
We denote by $p$ the puncture and abbreviate $\psi_{i,p} = \psi_i$. We will explain that $\psi_{i}$ is a cluster automorphism when $i \geq \frac{k}{2}$. This implies the statement for all $i \in [k-1]$ since $\ast$ is a cluster automorphism, via 
\eqref{eq:cutinhalf}. 

Consider a regular triangulation $\Delta$ of $S_{g,1}$ (cf.~Figure~\ref{fig:Sg1} for the case $g=1$). All three vertices of any triangle $\delta \in \Delta$ are the puncture $p$. The cluster variable 
$x(a,b,c)$ having coordinates $(a,b,c) \in \mathbb{H}_k$ in this triangle has weight $\omega_a+\omega_b+\omega_c$ at $p$, regardless of how we identify the three corners of $\delta$ with the vertices $(k,0,0)$,$(0,k,0)$ and $(0,0,k)$ in $\hat{\mathbb{H}}_k$. 

It still makes sense to consider the vertices in row $i$, except that these vertices will now be a union of three line segments parallel to the three sides of the triangle $\delta$. Each of the three segments is contractible to one of the three corners of $\delta$. The union of these three segments, taken over all of the various triangles $\delta$, is therefore a closed curve contractible to $p$. If we temporarily view one of the three corners as the top vertex of a triangle, it makes sense to move the ``left side'' of $\delta$ to the ``right side'' as in the proof of Theorem~\ref{thm:GS}. Doing this for each of the three corners, we again factorize  the monodromy $u_p$ around $p$ into monodromies associated to each of the corners of the various triangles $\delta$. Thus the rhombus monomial formula \eqref{eq:partialpotential2} still holds.

When $i > \frac{k}{2}$ the three segments are disjoint and the aforementiond closed curve has no selfintersections. For an example see the blue cycle in Figure~\ref{fig:Sg1}. In this case, the cluster variable $x(i,b,c)$ has weight $\omega_i+\omega_b+\omega_c$ with $\omega_b,\omega_c \neq \omega_i$, hence \eqref{eq:partialpotential1} holds. The restriction of the quiver ~$Q_k(\Delta)$ to the vertices which are in row $i$ (with respect to any corner of any triangle) is an oriented $m$-cycle for some $m$. We can apply the mutation sequence \eqref{eq:permutationmutation} and the formula \eqref{eq:partialpotential3} still holds. This completes the proof when $i > \frac{k}{2}$.

What remains is to argue the case $i = \frac{k}{2}$ when $k$ is even and $k >2$. In this case the three segments in row $i$ intersect each other at the points $(i,i,0)$, $(i,0,i)$ and $(0,i,i)$ in each triangle. See the red segments in Figure~\ref{fig:Sg1}. We call the corresponding vertices of $Q_k(\Delta)$ the {\sl midpoint vertices} and call the remaining vertices in row $i$ the {\sl nonmidpoint vertices}. The midpoint vertices have weight $2 \omega_i$, so that the correct version of \eqref{eq:partialpotential1} rescales the cluster variable $x(i,i,0)$ by $\mathcal{W}_i^2$. On the other hand, \eqref{eq:partialpotential1} holds for the nonmidpoint vertices. 

Provided $k>2$, there are no edges between the various midpoint vertices. The mutation $\mu_{\rm midpoint}$ at all the midpoint vertices (in any order) is well-defined. We use primes $x'(i,i,0)$ to denote the result of such mutation. Note that each midpoint vertex has degree 4 with 2 incoming and two outgoing neighbors. Each neighbor has weight 
$\omega_i+\omega_{i-1}+\omega_1$, so that ${\rm wt}_p(x'(i,i,0)) = 2(\omega_1+\omega_{i-1})$. In particular, $\psi^*_{i,p}$ fixes $x'(i,i,0)$. 

The induced subquiver of $\mu_{\rm midpoint}(Q_k(\Delta))$ on the nonmidpoint vertices is an oriented $m$-cycle for some~$m$. We claim that performing the sequence \eqref{eq:permutationmutation} to the seed  $\mu_{\rm midpoint}(\Sigma_k(\Delta))$ is the Weyl group automorphism $\psi^*_{i}$.

The action of $\psi^*_{i,p}$ on  $\mu_{\rm midpoint}(\Sigma_k(\Delta))$ is to rescale each 
nonmidpoint cluster variable by the quantity $\mathcal{W}_i$ (fixing all other cluster variables). The expression \eqref{eq:partialpotential2} for $\mathcal{W}_i$
as  a Laurent polynomial in the initial cluster variables is still valid. What needs to be checked is that when we rewrite the quantity  \eqref{eq:partialpotential2} in terms of the variables in $\mu_{\rm midpoint}(\Sigma_k(\Delta))$, it matches the formula \eqref{eq:partialpotential3} for the mutation sequence \eqref{eq:permutationmutation} applied to the cycle on the nonmidpoint variables. This is a straightforward check which we illustrate in Example~\ref{eg:Sg1proof}.
\end{proof}

\begin{example}\label{eg:Sg1proof}
We illustrate the ideas in the proof of Proposition~\ref{prop:Sg1iscluster} in the case that $G = {\rm SL}_4$ and $\mathbb{S} = S_{1,1}$. The initial quiver is in Figure~\ref{fig:Sg1} with rows 3 and 2 drawn in blue and red. The induced subquiver on the red vertices is not an oriented cycle. To rectify this, one should mutate at vertices $2$,$4$ and $10$ to obtain the quiver on the right. In this quiver, the red form an oriented 6-cycle. The formula \eqref{eq:partialpotential2} takes the form 
\begin{align*}
&\frac{x_0x_6}{x_2x_3}+\frac{x_1x_7}{x_3x_4}+\frac{x_1x_{13}}{x_4x_{12}}+
\frac{x_9x_{14}}{x_{10}x_{12}}+\frac{x_7x_9}{x_{6}x_{10}}+\frac{x_3x_5}{x_{2}x_{6}}
+\frac{x_5x_{14}}{x_{2}x_{13}}+\frac{x_8x_{12}}{x_{4}x_{13}}
+\frac{x_3x_{8}}{x_{4}x_{7}}+\frac{x_6x_{11}}{x_{7}x_{10}}+
\frac{x_{11}x_{12}}{x_{10}x_{14}}+\frac{x_0x_{13}}{x_{2}x_{14}}
 \\
=& \frac{x'_4x_1}{x_3x_{12}}+\frac{x'_2x_0}{x_3x_{14}}+
\frac{x'_{10}x_{11}}{x_7x_{14}}+\frac{x'_4x_8}{x_7x_{13}}
+\frac{x'_2x_5}{x_6x_{13}}+\frac{x'_{10}x_9}{x_6x_{12}}.
\end{align*}
In passing from the first to the second line, we e.g. grouped the terms $\frac{x_1}{x_4}(\frac{x_7}{x_3}+\frac{x_{13}}{x_{12}})$ and then used the mutation formula $x'_4 = \frac{x_7x_{12}+x_3x_{13}}{x_4}$. The result is the formula 
\eqref{eq:partialpotential3} applied to the red oriented 6-cycle on the right of Figure~\ref{fig:Sg1} with the frozen variables appropriately specialized. 
\end{example}

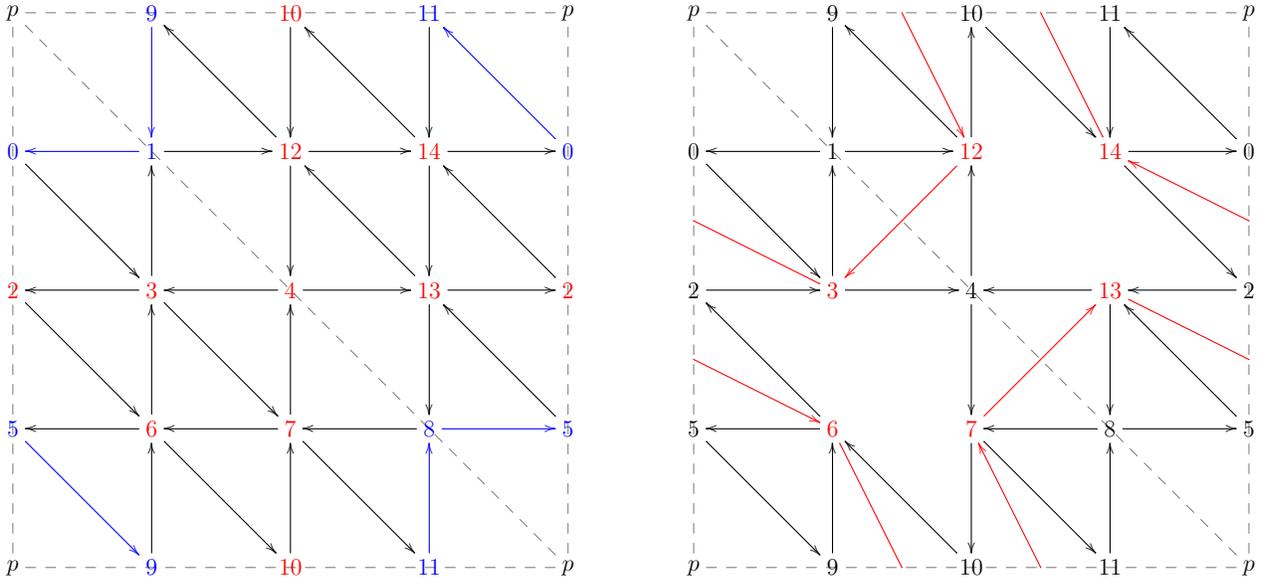
\begin{figure}[ht]
\begin{tikzpicture}
\node at (0,0) {\scalebox{.75}{
\begin{xy} 0;<1pt,0pt>:<0pt,-1pt>:: 
(0,0) *+{p} = "NW",
(0,280) *+{p} = "SW",
(280,280) *+{p} = "SE",
(280,0) *+{p} = "NE",
"NW",{\textcolor{gray}{\ar@{--}"SW"}},
"NW",{\textcolor{gray}{\ar@{--}"SE"}},
"SW",{\textcolor{gray}{\ar@{--}"SE"}},
"SE",{\textcolor{gray}{\ar@{--}"NE"}},
"NW",{\textcolor{gray}{\ar@{--}"NE"}},
(0,70) *+{\textcolor{blue}{0}} ="0",
(70,70) *+{\textcolor{blue}{1}} ="1",
(0,140) *+{\textcolor{red}{2}} ="2",
(70,140) *+{\textcolor{red}{3}} ="3",
(140,140) *+{\textcolor{red}{4}} ="4",
(0,210) *+{\textcolor{blue}{5}} ="5",
(70,210) *+{\textcolor{red}{6}} ="6",
(140,210) *+{\textcolor{red}{7}} ="7",
(210,210) *+{\textcolor{blue}{8}} ="8",
(70,280) *+{\textcolor{blue}{9}} ="9",
(140,280) *+{\textcolor{red}{10}} ="10",
(210,280) *+{\textcolor{blue}{11}} ="11",
(70,0) *+{\textcolor{blue}{9}} ="12",
(140,70) *+{\textcolor{red}{12}} ="13",
(140,0) *+{\textcolor{red}{10}} ="14",
(210,140) *+{\textcolor{red}{13}} ="15",
(210,70) *+{\textcolor{red}{14}} ="16",
(210,0) *+{\textcolor{blue}{11}} ="17",
(280,210) *+{\textcolor{blue}{5}} ="18",
(280,140) *+{\textcolor{red}{2}} ="19",
(280,70) *+{\textcolor{blue}{0}} ="20",
"1", {\textcolor{blue}{\ar"0"}},
"0", {\ar"3"},
"3", {\ar"1"},
"12", {\textcolor{blue}{\ar"1"}},
"1", {\ar"13"},
"3", {\ar"2"},
"2", {\ar"6"},
"4", {\ar"3"},
"6", {\ar"3"},
"3", {\ar"7"},
"7", {\ar"4"},
"13", {\ar"4"},
"4", {\ar"15"},
"6", {\ar"5"},
"5", {\textcolor{blue}{\ar"9"}},
"7", {\ar"6"},
"9", {\ar"6"},
"6", {\ar"10"},
"8", {\ar"7"},
"10", {\ar"7"},
"7", {\ar"11"},
"11", {\textcolor{blue}{\ar"8"}},
"15", {\ar"8"},
"8", {\textcolor{blue}{\ar"18"}},
"13", {\ar"12"},
"14", {\ar"13"},
"15", {\ar"13"},
"13", {\ar"16"},
"16", {\ar"14"},
"16", {\ar"15"},
"18", {\ar"15"},
"15", {\ar"19"},
"17", {\ar"16"},
"19", {\ar"16"},
"16", {\ar"20"},
"20", {\textcolor{blue}{\ar"17"}},
\end{xy}
}};
\begin{scope}[xshift = 9cm]
\node at (0,0) {\scalebox{.75}{
\begin{xy} 0;<1pt,0pt>:<0pt,-1pt>:: 
(0,0) *+{p} = "NW",
(0,280) *+{p} = "SW",
(280,280) *+{p} = "SE",
(280,0) *+{p} = "NE",
"NW",{\textcolor{gray}{\ar@{--}"SW"}},
"NW",{\textcolor{gray}{\ar@{--}"SE"}},
"SW",{\textcolor{gray}{\ar@{--}"SE"}},
"SE",{\textcolor{gray}{\ar@{--}"NE"}},
"NW",{\textcolor{gray}{\ar@{--}"NE"}},
(0,70) *+{0} ="0",
(70,70) *+{\textcolor{black}{1}} ="1",
(0,140) *+{\textcolor{black}{2}} ="2",
(70,140) *+{\textcolor{red}{3}} ="3",
(140,140) *+{\textcolor{black}{4}} ="4",
(0,210) *+{\textcolor{black}{5}} ="5",
(70,210) *+{\textcolor{red}{6}} ="6",
(140,210) *+{\textcolor{red}{7}} ="7",
(210,210) *+{\textcolor{black}{8}} ="8",
(70,280) *+{\textcolor{black}{9}} ="9",
(140,280) *+{\textcolor{black}{10}} ="10",
(210,280) *+{\textcolor{black}{11}} ="11",
(70,0) *+{\textcolor{black}{9}} ="12",
(140,70) *+{\textcolor{red}{12}} ="13",
(140,0) *+{\textcolor{black}{10}} ="14",
(210,140) *+{\textcolor{red}{13}} ="15",
(210,70) *+{\textcolor{red}{14}} ="16",
(210,0) *+{\textcolor{black}{11}} ="17",
(280,210) *+{\textcolor{black}{5}} ="18",
(280,140) *+{\textcolor{black}{2}} ="19",
(280,70) *+{\textcolor{black}{0}} ="20",
"1", {\ar"0"},
"0", {\ar"3"},
"3", {\ar"1"},
"12", {\ar"1"},
"1", {\ar"13"},
"2", {\ar"3"},
"6", {\ar"2"},
"3", {\ar"4"},
"13", {\textcolor{red}{\ar"3"}},
"4", {\ar"7"},
"4", {\ar"13"},
"15", {\ar"4"},
"6", {\ar"5"},
"5", {\ar"9"},
"9", {\ar"6"},
"10", {\ar"6"},
"8", {\ar"7"},
"7", {\ar"10"},
"7", {\ar"11"},
"7", {\textcolor{red}{\ar"15"}},
"11", {\ar"8"},
"15", {\ar"8"},
"8", {\ar"18"},
"13", {\ar"12"},
"13", {\ar"14"},
"14", {\ar"16"},
"18", {\ar"15"},
"19", {\ar"15"},
"17", {\ar"16"},
"16", {\ar"19"},
"16", {\ar"20"},
"20", {\ar"17"},
"6",{\textcolor{red}{\ar@{-}(105,280)}},
(105,0),{\textcolor{red}{\ar"13"}},
"16",{\textcolor{red}{\ar@{-}(175,0)}},
(175,280),{\textcolor{red}{\ar"7"}},
"3",{\textcolor{red}{\ar@{-}(0,105)}},
(280,105),{\textcolor{red}{\ar"16"}},
"15",{\textcolor{red}{\ar@{-}(280,175)}},
(0,175),{\textcolor{red}{\ar"6"}},
\end{xy}}};
\end{scope}
\end{tikzpicture}
\caption{From the proof of Proposition~\ref{prop:Sg1iscluster}. The gray dashed lines describe a triangulation of the once-punctured torus $S_{1,1}$ (with top/bottom and left/right sides identified). The left picture is the initial quiver when $G = {\rm SL}_4$. The blue vertices comprise row $3$; the induced subquiver on these vertices is an oriented 6-cycle contractible to the puncture $p$. The red vertices comprise row 2. Mutation at the midpoint vertices $2$, $4$, and $10$ yields the quiver on the right. The nonmidpoint vertices (drawn red in the right figure) form an oriented 6-cycle contractible to the puncture.}\label{fig:Sg1}
\end{figure}

\begin{remark}
In \cite{GoncharovShenDT}, Goncharov and Shen give a formula for the DT-transformation of $\mathscr{A}({\rm SL}_k,\mathbb{S})$. It is a composition of several commuting maps, one of which is the action of the longest element $w_0 \in W$ at each puncture. We expect that this formula works for $S_{g,1}$'s, so that our Theorem~\ref{thm:GS} shows that the DT transformation of $\mathcal{A}_{{\rm SL}_k,S_{g,1}}$ is cluster once $k>2$. We believe this would follow from Section 8.2 in {\sl loc. cit.} provided one checks the analogue of Theorem~\ref{thm:GS} on the $\mathcal{X}$-space. When $k=3$, we have checked that the 
candidate map is indeed the DT transformation and that the standard sequence of mutations which realizes this map is a maximal green sequence.
\end{remark}

 
\section{Compatibility of weight vectors}\label{secn:Dosps}
With the above preliminaries established, the next several sections investigate properties of $P$-clusters in the cluster algebras $\mathscr{A}(\mathcal{M})$. The present section introduces a compatibility notion for elements of the weight lattice which we expect models compatibility of cluster variables. We use freely here the Coxeter-theoretic terminology of roots, weights, ordered set partitions, etc. from Section~\ref{subsec:G}.

The maximal faces of the Coxeter complex are indexed by permutations. The closure $\overline{C}_w$ of the facet $C_w$ indexed by the permutation $w$ consists of weight vectors $\lambda$ whose coordinates satisfy 
$\lam_{w(1)} \geq \cdots \geq \lam_{w(k)}$. 
\begin{definition}
A pair of weight vectors $\lambda,\mu \in P$ is $w$-{\sl sortable} if both these vectors lie in $\overline{C_w}$. The pair is $W$-sortable if it is $w$-sortable for some $w \in W$. 

The pair is {\sl root-conjugate} if the difference 
 $\lambda-\mu$ lies in the root system $\Phi$ and moreover, the vectors $\lambda$ and $\mu$ are conjugate by reflection in this root vector. 
 
The pair is {\sl compatible} if it is either $W$-sortable or root-conjugate. 
\end{definition} 

A compatible pair of vectors is either $W$-sortable or root-conjugate, never both. 

Whenever weight vectors $\lambda$ and $\mu$ are conjugate by reflection in the root $e_a-e_b$, one has $\lambda - \mu \in {\rm span}_{\mathbb{Z}}(e_a-e_b)$. Root-conjugacy requires more strongly that $\lambda - \mu = \pm (e_a-e_b)$. This condition is equivalent to requiring that  there exists a weight vector $\nu \in P$ whose coordinates satisfy $\nu_a = \nu_b$, and a choice of sign $\varepsilon \in \{\pm 1\}$, such that $\lambda = \nu+\varepsilon e_a$ and $\mu = \nu+\varepsilon e_b$. In fact, one readily sees that if this property holds for $\varepsilon$, then it also holds for $-\varepsilon$, so the choice of sign is immaterial. This sign ambiguity disappears once we have three vectors which are pairwise root-conjugate, see Lemma~\ref{lem:flag} below.

\begin{example}\label{eg:signambiguity}
The vectors $(7,3,2)$ and  $(3,7,2)$ are $s_1$-conjugate but not root-conjugate. The vectors $(7,3,2)$ and $(6,4,2)$ differ by the root $e_1-e_2$  but they are not root-conjugate. The vectors $(7,3,2)$ and $(7,2,3)$ are root-conjugate in the direction of the root $e_2-e_3$. Applying the previous paragraph to this pair, we could take either $\nu = (7,3,3)$ and $\varepsilon = -1$ or take $\nu = (7,2,2)$ and $\varepsilon = +1$. 
\end{example}

If $\mathcal{C}$ is a set of pairwise root-conjugate vectors, then every difference of two vectors in $\mathcal{C}$ takes the form $e_a-e_b$ for appropriate indices $a,b \in [k]$. Define the {\sl ground set} of such a $\mathcal{C}$ as 
\begin{equation}\label{eq:groundsetofsimplex}
\{a \in [k] \colon \text{ there exists $\lambda,\mu \in \mathcal{C}$ and $b \in [k]$ such that $\lambda-\mu = e_a-e_b$} \}.
\end{equation}

As an example, the collection $\mathcal{C} = \{(6,7,6,7),(7,6,6,7),(7,7,6,6)\}$ has ground set $\{1,2,4\}$.

\begin{lemma}\label{lem:flag}
Let $\mathcal{C} \subset P$ be a collection of weight vectors with $|\mathcal{C}|>2$. 

If the elements of $\mathcal{C}$ are pairwise $W$-sortable, then there is a unique osp $\Pi$ such that $\mathcal{C} \subset \overline{C_\Pi}$ and 
$\overline{C_\Pi}$ is minimal by inclusion with this property. 

If instead the elements of $\mathcal{C}$ are pairwise root-conjugate, let $B \subset [k]$ be their corresponding ground set \eqref{eq:groundsetofsimplex}. Then there exists a unique weight vector~$\nu$ with constant $B$-coordinates (i.e., $\nu_a = \nu_b$ when $a,b \in B$), and a unique sign~$\varepsilon \in \{\pm 1\}$, such that 
$\mathcal{C} = \{\nu +\varepsilon e_a \colon a \in B\}$. 
\end{lemma}

The ordered set partition $\Pi$ alluded to in the first part of this lemma is the least upper bound of the ordered set partitions $\Pi(\lambda) \colon \lambda \in \mathcal{C}$. To illustrate the second part of the lemma, we would have uniquely that $\nu = (7,7,6,7)$ and $\varepsilon = -1$ for the collection $\mathcal{C}$ introduced just before the lemma statement.

\begin{proof}
The first statement is an instance of the classical statement that the Coxeter complex is a flag complex. Given $\lambda$, the corresponding closed face $\overline{C_{\Pi(\lambda)}}$ of the Coxeter complex is a simplex on certain vertices $v_i(\lambda)$. If $\mu \in \mathcal{C}$, then the $w$-sortability hypothesis says that $\overline{C_{\Pi(\lambda)}} \cup \overline{C_{\Pi(\mu)}} \subset \overline{C_{\Pi(w)}}$. Thus, any pair of $v_i(\lambda)$ and $v_j(\mu)$ form a 2-simplex in the Coxeter complex, hence the union of the $v_i(\lambda)$'s and $v_j(\mu)$'s determines a simplex in the Coxeter complex (because the latter is a flag complex). Continuing in this way, the minimal simplex which contains all of the vectors $\lambda \in \mathcal{C}$ is the one spanned by the various vertices $v_i(\lambda)$. This minimal by inclusion simplex is the desired $\overline{C_\Pi}$. 

For the second statement, let $\lambda$, $\lambda'$, and $\lambda''$ be three elements of $\mathcal{C}$. Observe that if $a \in [k]$ and $\lambda,\mu$ are root-conjugate, then 
$\lambda_a - \mu_a \in \{0,\pm 1\}$. Suppose that $\lambda-\lambda'= e_a-e_b$. From the definition of root-conjugacy we can write $\lambda_a = \lambda'_b =  x+1$ and $\lambda_b = \lambda'_a = x$ for some $x \in \mathbb{Z}$. 

From the observation two sentences previous, we have $\lambda''_a \in \{x,x+1\}$ and $\lambda''_b \in \{x,x+1\}$. 

In the first case, we have $\lambda''_a = x$. Then $\lambda-\lambda'' = e_a - e_c$ for some $c \neq b$. It follows that $\lambda'_b = \lambda_b = x$ and also follows that $\lambda''_c = x+1$ and $\lambda_c = x$. Letting $\nu = \lambda-e_a$, we see that $\lambda$, $\lambda'$, and $\lambda''$ take the form $\nu+e_a$, $\nu+e_b$, and $\nu+e_c$ where $\nu_a = \nu_b = \nu_c = x$. It is clear that the vector $\nu \in P$ is unique and the sign $\varepsilon = +1$ is determined. 

The second case is that $\lambda''_a = x+1$. Then we have $\lambda''_b - \lambda'_b = e_a-e_c$ for some $c \neq b$. It follows that $\lambda''_b = x+1$ and that $\lambda'_c = \lambda_c = x+1$. Letting $\nu = \lambda+e_b$, we can write our three vectors as $\nu-e_b$, $\nu-e_a$, and $\nu-e_c$. As above, everything (including the sign $\varepsilon = -1$) is determined. 

To complete the proof we suppose there is a fourth vector $\lambda'''$ and repeat the argument with $\lambda$, $\lambda'$, and $\lambda'''$. 
If we were in the first case with $\lambda$, $\lambda'$, and $\lambda''$, one can check that must also be in the first case with $\lambda'''_a$ and so on. So the argument persists and the claims about the existence and uniqueness of $\nu$ and $\varepsilon$ follow.

\end{proof}

Let $\mathcal{C} \subset P$ be a multiset of pairwise compatible weight vectors. We call $\lambda \in \mathcal{C}$ a {\sl root-conjugate vector in $\mathcal{C}$} if it is root-conjugate to some $\mu \in \mathcal{C}$. Otherwise we say that $\lambda$ is a $W$-sortable vector in $\mathcal{C}$. We call the collection $\mathcal{C}$ {\sl basic} if every root-conjugate vector in $\mathcal{C}$ appears with multiplicity one in the multiset.

The following is our main conjecture concerning weights of cluster variables in the cluster algebra $\mathscr{A}(\mathcal{M})$.
\begin{conjecture}\label{conj:Pclusterconjecture}
Let $\mathcal{C} \subset P^{\mathbb{M}_\circ}$ be a $P$-cluster for $\mathscr{A}(\mathcal{M})$ and $p \in \mathbb{M}_\circ$ be a puncture. Then the projection $\pi_p(\mathcal{C}) \subset P$ is a basic multiset of pairwise compatible vectors. 
\end{conjecture}

This conjecture is purely about $P$-clusters at $p$, so one might be able to prove it by a careful analysis of $P$-seeds and their mutations. 

We denote by $\mathbb{E} =\mathbb{E}(\mathcal{M})$ the exchange graph of the cluster algebra $\mathscr{A}(\mathcal{M})$. 
\begin{definition}
We denote by $\mathbb{E}_{\rm good} \subset \mathbb{E}$ the maximal connected subgraph of the exchange graph containing the initial seed and consisting of seeds whose $P$-cluster at any puncture $p$ is a basic multiset of pairwise compatible vectors. We call $\mathbb{E}_{\rm good}$ the {\sl good part} of the exchange graph.
\end{definition}

Thus, a cluster is in the good part of the exchange graph if it is connected to the initial seed by mutations which only pass through clusters whose underlying $P$-cluster satisfy Conjecture~\ref{conj:Pclusterconjecture}. Conjecture~\ref{conj:Pclusterconjecture} is the assertion that $\mathbb{E}_{\rm good} = \mathbb{E}$. Since we do not have a proof of this conjecture, our results in the next two sections  address clusters which are in $\mathbb{E}_{\rm good}$. 

Our first result along these lines is as follows. We say that vertices $v$ and $v'$ are {\sl symmetrical} in a quiver $Q$ if the vertex transposition $(v v')$ induces an isomorphism of~$Q$. 

\begin{proposition}\label{prop:goodimpliessymmetrical}
Let $\mathcal{C} = (Q,{\rm wt}(v)_{v \in V(Q)})$ be a $P$-seed in the good part of the exchange graph.  Let $v,v'$ be vertices of its underlying quiver~$Q$ whose weights $\pi_p({\rm wt}(v))$ and  
$\pi_p({\rm wt}(v'))$ are root-conjugate at some puncture~$p$. Then $v$ and $v'$ are symmetrical vertices in~$Q$. Moreover, $\pi_{p'}({\rm wt}(v)) = 
\pi_{p'}({\rm wt}(v'))$ for punctures $p' \neq p$.
\end{proposition}

\begin{proof}
Put $\lambda = \pi_p({\rm wt}(v))$ and 
$\lambda' = \pi_p({\rm wt}(v'))$. By assumption we have $\lambda = \nu+e_a$ and $\lambda' = \nu+e_b$ where $\nu_a = \nu_b$. By the goodness of the $P$-cluster, all other vertices $u \neq v,v'$ in $Q$ have the property that $\pi_p({\rm wt}(u))_a = \pi_p({\rm wt}(u))_b$.

Consider a vertex $u \in Q$. Let $\mu \in P$ be the result of applying $\pi_p$ to the weight vector 
\eqref{eq:exchangemonomials} encoding the weight of the right hand side of the exchange relation at $u$. 
By considering the left and right hand sides of \eqref{eq:exchangemonomials} in turn and using the observations from the previous paragraph, we see that

\begin{align*}
\mu_a - \mu_b &=  \sum_{u \to v}1-\sum_{u \to v'}1 \\
&=  \sum_{v \to u}1-\sum_{v'\to u}1. 
\end{align*}
On the other hand, only one of the cases $u \to v$ or $v \to u$ is possible and likewise for $u \to v'$ and $v' \to u$. If, for example, the case $u \to v$ occurs, we conclude that the number of arrows $u \to v$ must equal the number of arrows $u \to v'$ and also conclude that there are no arrows $v \to u$ or $v' \to u$. That is, the vertices $v$ and $v'$ are symmetrical in the quiver. 

For the second claim about weights at punctures $p' \neq p$, note by weight considerations that the automorphism $\psi_{(a,b),p}$ must send ${\rm wt}(v)$ to ${\rm wt}(v')$, and it follows that these two weight vectors coincide away from~$p$.
\end{proof}

\begin{example}
The cluster algebra $\mathscr{A}'({\rm SL}_3,D_{2,1})$ has finite cluster type $D_4$. We indicate the initial mutable $P$-seed indicated below (drawn inside the digon with puncture $p$). There are finitely many $P$-clusters which we computed exhaustively on a computer (code available at the first author's web site). The first column in the table below lists a set of representatives for the $W$-equivalence classes of $P$-clusters:
\begin{equation}\label{eq:Pclusters}
\begin{tikzpicture}
\node at (0,0.00) {$p$};
\node at (0,-1.5) {$\omega_1$};
\node at (0,-.85) {$\omega_2$};
\node at (0,.85) {$\omega_2$};
\node at (0,1.5) {$\omega_1$};
\node at (0,2) {$\bullet$};
\node at (0,-2) {$\bullet$};
\draw [gray, out = 170, in = -90] (0,-2) to (-1.3,0);
\draw [gray, out = 90, in = 190] (-1.3,0) to (0,2);
\draw [gray, out = 10, in = -90] (0,-2) to (1.3,0);
\draw [gray, out = 90, in = -10] (1.3,0) to (0,2);
\draw [->,out = 70, in = -70] (.3,-.8) to (.3,1.3);
\draw [->,out = 250, in = 110] (-.3,.8) to (-.3,-1.3);
\draw [->] (0,1.3)--(0,1.1);
\draw [->] (0,-1.3)--(0,-1.1);

\node at (7,0) {\begin{tabular}{l|l|l}
$P$-cluster & class sums & dosp\\ \hline 
$e_1\times 2,e_1+e_2 \times 2$ & $e_1,e_1+e_2$&  $1|2|3$ \\ 
$e_1,e_2,e_1+e_2 \times 2$ &$e_1+e_2$ & $12|3$ \\ 
$e_1\times 2,e_1+e_2,e_1+e_3$ & $e_1$&  $1|23$ \\ 
$e_1,e_2,e_1+e_2,0$ &$0,e_1+e_2$& $12|3$ \\ 
$e_1,e_1+e_2,e_1+e_3,0$ &$ 0,e_1$& $1|23$ \\ 
$e_1,e_2,e_3,0$ & $0$& $123^+$ \\ 
$e_1+e_2,e_1+e_3,e_2+e_3,0$ &$0$ &  $123^-$ 
\end{tabular}
};
\end{tikzpicture}
\end{equation}
We write $e_1+e_2 \times 2$ to indicate that the vector $e_1+e_2$ appears twice in the multiset, and so on. The initial cluster occupies the first row of the table. The second and third columns of the table are discussed in the next section. We give similar tables for $\mathscr{A}'({\rm SL}_3,D_{3,1})$ and $\mathscr{A}'({\rm SL}_4,D_{2,1})$, each of which also has only finitely many $P$-clusters, in \eqref{eq:Pclusters2} and \eqref{eq:Pclusters3} as part of our study of finite mutation types. 	By inspecting these tables, one can see that Conjecture~\ref{conj:Pclusterconjecture} holds in all three of these examples. 
\end{example}

\section{From compatible collections to dosps}\label{secn:CtoPiEps}
We explain a way of associating a certain combinatorial object (a ``dosp'' at every puncture) to $P$-clusters in the good part of the exchange graph. We identify the operation on these objects which is induced by mutations in the cluster algebra. 

\subsection{An osp at every puncture}
Let $\mathcal{C} \subset P$ be a basic multiset of pairwise compatible weight vectors. We get an equivalence relation on the ground set $\mathcal{C}$ by declaring two vectors $\lambda \neq \mu \in \mathcal{C}$ to be equivalent when they are root-conjugate. This relation is transitive by Lemma~\ref{lem:flag}. Each $W$-sortable vector in $\mathcal{C}$ lies in its own equivalence class. 

Each equivalence class $\mathcal{O} \subset \mathcal{C}$ gives rise to a {\sl class sum}, defined as ${\rm sum }\, \mathcal{O} := \sum_{\lambda \in \mathcal{O}} \lambda \in P$. This construction is indicated in the passage from the first to second column in \eqref{eq:Pclusters}. For example, the fifth row has an equivalence class $\mathcal{O} =\{e_1+e_2,e_1+e_3\}$ with ${\rm sum}\, \mathcal{O} = 2e_1+e_2+e_3 = e_1 \in P$.

\begin{lemma}\label{lem:orbitsums}
If $\mathcal{C}$ is a basic multiset of pairwise compatible weight vectors, then its collection of class sums consists of pairwise $W$-sortable vectors. 
\end{lemma}

We have listed the $P$-clusters up to $W$-action in \eqref{eq:Pclusters} in such a way that the class sums in each row of the table are sortable by the identity permutation.  
\begin{proof}
If $\mathcal{C}$ is a pairwise $W$-sortable collection then the claim is clear. Otherwise let 
$\mathcal{O}$ be a class of size at least two. We will show that the collection $\mathcal{O}' := (\mathcal{C} \setminus \mathcal{O}) \cup \{{\rm sum}\, \mathcal{O}\}$ is again a basic multiset of pairwise compatible weight vectors and that ${\rm sum}\, \mathcal{O}$ is a $W$-sortable vector in $\mathcal{O}'$. Repeating this process proves the claim. 

Let $B \subset [k]$ be the ground set of the class $\mathcal{O}$ \eqref{eq:groundsetofsimplex}. By Lemma~\ref{lem:flag}, there is a vector $\nu \in P$ with constant $B$-coordinates and a sign $\varepsilon$ such that $\mathcal{O} = \{\nu+\varepsilon e_a \colon a \in B\}$. Thus ${\rm sum}\, \mathcal{O} = |B|\nu+\varepsilon \iota_B$ has constant $B$-coordinates where $\iota_B$ is the indicator vector of $B$. If $\mu \in \mathcal{C} \setminus \mathcal{O}$, then since $\mu$ is not root-conjugate to the vectors in $\mathcal{O}$, the definition of compatibility requires that $\mu$ is $W$-sortable with each element of $\mathcal{O}$. In particular, $\mu$ must have constant $B$-coordinates. It follows that $\mu$ is pairwise $W$-sortable with ${\rm sum}\, \mathcal{O}$.
\end{proof}

By this lemma and Lemma~\ref{lem:flag}, any basic multiset $\mathcal{C}$ of pairwise compatible weight vectors gives rise to an ordered set partition $\Pi(\mathcal{C})$ via  
\begin{equation}\label{eq:Ctoosp}
\Pi(\mathcal{C}) := \bigvee_{i=1}^s \Pi({\rm sum} \, \mathcal{O}_i),
\end{equation}
where we recall that $\bigvee$ denotes the least upper bound operation on osp's. This osp records the global coordinate equalities and strict inequalities which hold for all class sums in $\mathcal{C}$.

This construction is illustrated in the passage from the first column to the third column in \eqref{eq:Pclusters} provided the reader ignores for the moment the $+,-$ signs appearing as exponents in the third column. (These signs are the extra information of a {\sl decorated} osp and will be explained shortly.) For the collection $\mathcal{C}$ appearing in the fifth row 
of \eqref{eq:Pclusters}, we have class sums $e_1$ and $0$, with corresponding osp's $\Pi(e_1) = 1|23$ and $\Pi(0) = |123|$ so that $\Pi(\mathcal{C}) = 1|23$. 

In this example, note that $\{2,3\}$ is a ground set of the root conjugacy class $\{e_1+e_2,e_1+e_3\} \in \mathcal{C}$ and appears as a block of $\Pi(\mathcal{C}) = 1|23$. This is in fact always the case:

\begin{lemma}\label{lem:Bisblock}
Continuing the setup of the previous lemma, let $B_j$ be the ground set of an equivalence class $\mathcal{O}_j \subset \mathcal{C}$. Then $B_j$ is a block of $\Pi(\mathcal{C})$.
\end{lemma}

\begin{proof}
Let $\mathcal{O}_1,\dots,\mathcal{O}_s$ be a listing of the classses in $\mathcal{C}$.
Let us focus on a particular $B_j$. It is easy to see that $B_j$ is a subset of a block of $\Pi(\mathcal{C})$, i.e. that each of the vectors in the set $S:= \{{\rm sum}\, \mathcal{O}_i: \,  i \in [s]\}$ has constant $B_j$-coordinates. We need to argue that no other coordinate is globally equal to these coordinates on this set $S$.  As in the proof of the previous lemma, we have 
${\rm sum}(\mathcal{O}_j) = |B_j|\nu+\varepsilon \iota_{B_j}$. We see that the $a$th coordinate ${\rm sum}(\mathcal{O}_j)_a = 0 \mod |B_j|$ when $a \notin B_j$ while 
${\rm sum}(\mathcal{O}_j)_a = 1 \mod |B_j|$ when $a \in B_j$, hence these coordinates are not equal. \end{proof}

So far the discussion has focused on a particular copy of the weight lattice $P$. If we are instead given a $P$-cluster $\mathcal{C} \subset P^{\oplus \mathbb{M}_\circ}$ appearing in the good part of the exchange graph, we use the projection $\pi_p\colon P^{\oplus \mathbb{M}_\circ} \to P$  to assign to this $P$-cluster an ordered set partition at each puncture: 
\begin{equation}\label{eq:Ctoosps}
\Pi_p(\mathcal{C}) := \Pi(\pi_p(\mathcal{C})).
\end{equation}

We have the following relationship between this recipe and the mutation rule \eqref{eq:exchangemonomials} for $P$-cluster variables. Given a $P$-seed and mutable vertex $v$ in the underlying quiver, 
we denote by $\kappa(v) := \sum_{u \to v}{\rm wt}(u) \in P^{\mathbb{M}_\circ}$ the weight of the right hand side of the exchange relation describing mutation in direction $v$ out of the current seed.   
\begin{proposition}\label{prop:goodimpliesosp}
Let $\mathcal{C}$ be the $P$-cluster of a $P$-seed in the good part of the exchange graph determining an osp $\Pi_p(\mathcal{C})$ at each puncture $p$. Then each of the weight vectors $\pi_p(\kappa(v))$, as $v$ varies of the vertices of the $P$-seed, lies in the closed Weyl region $\overline{C_{\Pi_p(\mathcal{C})}}$.
\end{proposition}

That is, while the coordinate equalities described by $\Pi_p(\mathcal{C})$ need not hold for the $P$-cluster variables at $p$, they do hold for the weights of each exchange relation out of this $P$-cluster. In particular, the exchange relations for each $P$-cluster have pairwise $W$-sortable weights. This is a nontrivial assertion which might possibly be proved inductively by a careful analysis of the mutation rule \eqref{eq:exchangemonomials} and might be easier to prove than Conjecture~\ref{conj:Pclusterconjecture} itself.

\begin{proof}
By Proposition~\ref{prop:goodimpliessymmetrical}, vertices in each equivalence class $\mathcal{O}_i \subset \pi_p(\mathcal{C})$ are symmetrical in the corresponding quiver, so their contribution to \eqref{eq:exchangemonomials} is symmetrical. So the weight vector $\pi_p(\kappa)$ is a nonnegative linear combination of the vectors ${\rm sum}\, \mathcal{O}_i$ as $i$ ranges over the root-conjugacy classes in $\pi_p(\mathcal{C})$. Each vector 
${\rm sum}\, \mathcal{O}_i$ lies in the closure $\overline{C_{\Pi_p(\mathcal{C})}}$ by Lemma~\ref{lem:orbitsums} (this was the definition of 
$\Pi_p(\mathcal{C})$), hence so does $\pi_p(\kappa)$ since the closed Weyl region is 
a cone. 
\end{proof}

\subsection{A dosp at every puncture}
The previous section associated an osp at each puncture to a $P$-cluster in the good part of the exchange graph. We now give a more refined version of this construction in which we associate a choice of sign to certain blocks of such an osp. 

Our starting point is the following lemma which says informally that the $P$-cluster at a puncture $p$ cannot be ``too small.''

A choice of weight vector $\lambda \in P$ determines a Laurent monomial map $T \to \mathbb{C}^*$  sending $t \mapsto t^\lambda$. If $\mathcal{C} \subset P$ is a multiset of weight vectors of cardinality~$N$, then it determines in this way a map of algebraic tori $T \to (\mathbb{C}^*)^\mathcal{C} \cong (\mathbb{C}^*)^N$. Our next lemma addresses what happens when we remove a vector $\lambda \in \mathcal{C}$, getting a map from $T$ to a torus of rank $N-1$.

\begin{lemma}\label{lem:goodimpliesinjective}
Let $\mathcal{C}$ be a $P$-cluster, $\lambda \in \mathcal{C}$ one of its $P$-cluster variables, and $p$ a puncture. Then $T \to (\mathbb{C}^*)^{\pi_p(\mathcal{C}) \setminus \pi_p(\lambda)}$ is injective. 
\end{lemma}

That is, we can recover $t \in T$ from the monomials $t^{\pi_p(\lambda)}\colon \lambda \in \mathcal{C}$, and more strongly we can recover $t$ from any $N-1$ of these monomials. Note that here $N$ is the number of mutable variables in the cluster algebra $\mathscr{A}(\mathcal{M})$, which is much larger than the rank $k-1$ of the algebraic torus $T$, so the lemma is not so surprising.

\begin{proof}
In an initial seed, there are least two vectors of weight $\omega_i$ at $p$ for $i=1,\dots,k-1$ (because the quiver fragments $Q_k(\delta)$ and $Q_k^1(\delta)$ have at least two variables in each row.) In particular, regardless of which cluster variable we remove, we still have at least one copy of the monomial $t^{\omega_i}$ for $i=1,\dots,k-1$, and these monomials are enough to recover $t$. 

To complete the proof, we should check that the claimed injectivity propagates under mutation. That is, we assume the injectivity statement for a $P$-cluster $\mathcal{C}$ and consider a mutation $\lambda \mapsto \lambda'$ resulting in a new $P$-cluster $\mathcal{C}'$. The monomial map indexed by the vectors in $\mathcal{C}' \setminus \lambda'$ is clearly injective since it coincides with the map indexed by $\mathcal{C} \setminus \lambda$. If $\mu \in \mathcal{C}' \setminus \lambda'$, we need to argue that the monomial map indexed by
$\mathcal{C}' \setminus \mu$ is injective. Note that $\mu$ appears in at most one of the two monomials appearing in the exchange relation \eqref{eq:exchangemonomials} between $\lambda$ and $\lambda'$. Suppose that $t \in {\rm ker}(T \to (\mathbb{C}^*)^{\mathcal{C}' \setminus \mu})$. By considering the monomial in the exchange relation that does not involve $\mu$, we see that $t^{\pi_p(\lambda+\lambda')} = 1$. And on the other hand $t^{\pi_p(\lambda')} = 1$ since $\lambda' \in \mathcal{C}' \setminus \mu$. Thus $t^{\pi_p(\lambda)} = 1$ and $t \in {\rm ker}(T \to (\mathbb{C}^*)^{\pi_p(\mathcal{C} \setminus \mu)})$. Thus $t = 1 \in T$ using the injectivity statement for the $P$-cluster $\mathcal{C}$.
\end{proof}

\begin{corollary}\label{cor:brokentwice}
Let $\mathcal{C}$ be a $P$-cluster which is in the good part of the exchange graph and $p$ be a puncture.
For any coordinates $a \neq b \in [k]$, there is either a pair of vectors $\lam,\mu \in \pi_p(\mathcal{C})$ which are root-conjugate in the direction $e_a-e_b$, or one of the following global directions of inequality holds: 
\begin{itemize}
\item for all $\lambda \in \pi_p(\mathcal{C})$ one has $\lambda_a \geq \lambda_b$ with strict inequality attained at least twice, or
\item the same as above but with the direction of inequality reversed. 
\end{itemize}
\end{corollary}

As an illustration of this corollary, we could not possibly see the $P$-cluster multiset $\{e_1 \times 3, e_1+e_2\}$ in table \eqref{eq:Pclusters} because the global coordinate inequality between coordinates 2 and 3 is attained only once, not twice.

\begin{proof}
The argument only involves weights at $p$ so we can abbreviate $\mathcal{C} = \pi_p(\mathcal{C})$. For $a \neq b \in [k]$ denote by $T_{ab} \subset T$ the rank one subtorus of matrices of the form $(1,\dots,z,\dots,z^{-1},\dots,1)$ for $z \in \mathbb{C}^*$, with the $z$'s in positions $a$ and $b$. 

Consider a pair of coordinates $a \neq b \in [k]$. Suppose $\mathcal{C}$ does not contain vectors which are root-conjugate by the root $e_a-e_b$. By the definition of pairwise compatibility, there is a global inequality $\lam_a \geq \lam_b$ or $\lam_a \leq \lam_b$ holding for all vectors $\lambda$. By considering the action of the torus $T_{ab}$ and using Lemma~\ref{lem:goodimpliesinjective}, we see that the strict version of this inequality must be strict for at least two vectors $\lambda \in \mathcal{C}$.
\end{proof}

\begin{definition}\label{defn:dosp}
A {\sl decorated ordered set partition} (a {\sl dosp} for short) is a pair $(\Pi,\varepsilon)$ consisting of an osp $\Pi$ together with a choice of sign $\varepsilon(B)$ for each block $B$ of $\Pi$ of cardinality at least three. 
\end{definition}


Recall the construction \eqref{eq:Ctoosps} assigning an osp at each puncture to  
any $P$-cluster $\mathcal{C}$ in the good part of the exchange graph.

\begin{lemma}\label{lem:dosprecipeiswelldefined}
Let $\mathcal{C}$ be a $P$-cluster in the good part of the exchange graph, $p$ be a puncture, and $B$ be a non-singleton block of $\Pi_p(\mathcal{C})$. Then $B$ is the ground set of a root-conjugacy class in $\pi_p(\mathcal{C})$. 
\end{lemma}

That is, every non-singleton block of $\Pi_p(\mathcal{C})$ ``comes from'' a root-conjugacy class in $\pi_p(\mathcal{C})$. 

\begin{proof}
Suppose that $a,b$ reside in a block of $\Pi_p(\mathcal{C})$ and are not in part of a ground set. By the corollary, there is a vector $\lam \in \Pi_p(\mathcal{C})$ satisfying a strict inequality $\lam_a > \lam_b$ (or the opposite strict inequality, but this will not change the argument). On the other hand, passing from the collection $\Pi_p(\mathcal{C})$ to the collection of class sums, we only turn a strict inequality between coordinates into an equality of coordinates when the two coordinates are in the ground set of a root-conjugacy class. This contradicts the fact that the coordinates $a$ and $b$ are supposed to be equal on class sums. 
\end{proof}

Using Lemma~\ref{lem:dosprecipeiswelldefined}, we upgrade the recipe \eqref{eq:Ctoosps} to a ``$P$-cluster to dosp'' map
\begin{equation}\label{eq:Pclustertodosp}
\mathcal{C} \mapsto ((\Pi_p(\mathcal{C}),\varepsilon_p(\mathcal{C}))_{p \in \mathbb{M}_\circ},
\end{equation}
defined for any cluster in the exchange graph whose $P$-cluster satisfies Conjecture~\ref{conj:Pclusterconjecture}. Namely, we can define the sign $\varepsilon(B)$ for any block of cardinality at least three to be the sign associated to its underlying root-conjugacy class by the construction in Lemma~\ref{lem:flag}.

\subsection{Dosp mutations}
By the previous section, we can assign to any vertex of the good part of the exchange graph a tuple of dosps indexed by punctures. In this section, we show that every mutation of $P$-clusters changes these dosps in a predictable way.

\begin{definition}\label{defn:adjacent}
A {\sl mutation}  of dosps is the following local move on decorated blocks: 
\begin{equation}\label{eq:dospadjcency}
(L \cup \{a\})^+ \,|\, R^- \leftrightarrow L^+ \, | \, (\{a\} \cup R)^- , 
\end{equation}
where $L,R \subset [k]$ are adjacent blocks, possibly the empty block. 
\end{definition}

When one of $L$ or $R$ is the empty block, the mutation move amounts to merging a singleton block with a neighboring block or the inverse of this move. If a block has cardinality one or two, then for the purposes of mutation one can regard the block as both signs $+$ and $-$ in \eqref{eq:dospadjcency}. See the below example for further clarification.

\begin{example}\label{eg:unravelfurther}
The following is a complete list of dosps which can be obtained from the dosp $12|345^+|6$ by a dosp mutation, with terms grouped if they are related by $W$-action: i) $12|3456^+$, ii)  $1|2|345^+|6$ or $2|1|345^+|6$, iii) $12|34|56$ or $12|35|46$ or $12|45|36$, and finally iv) $12|34|5|6$ or  $12|35|4|6$ or $12|45|3|6$. 
\end{example}

We now justify our use of the word mutation. 
\begin{lemma}\label{lem:adjacent}
Let $\mathcal{C}$ and $\mathcal{C}'$ be mutation-adjacent $P$-clusters in the good part of the exchange graph and consider a puncture $p \in \mathbb{M}_\circ$. Then the corresponding dosps  $(\Pi_p(\mathcal{C}),\varepsilon(\mathcal{C}))$ and $(\Pi_p(\mathcal{C}'),\varepsilon(\mathcal{C}'))$ either coincide or are related by a dosp mutation. 
\end{lemma}


\begin{proof}
The argument is local in the puncture~$p$ so we abbreviate $\mathcal{C} = \pi_p(\mathcal{C})$, $\Pi = \Pi_p$, 
$\lambda = \pi_p(\lambda)$ etc. 

Using Corollary~\ref{cor:brokentwice}, a given pair of coordinates $a,b$ either satisfies a global inequality on $\mathcal{C}$ with strictness attained at least twice, or these coordinates are part of a ground set of a root-conjugacy class (in which case there is a strict inequality between these coordinates in both directions, attained exactly once). Using the data of the direction of these inequalities, and how many times they are attained, we clearly can recover the dosp. 

Fixing $a$ and $b$, let us analyze how the direction of inequality can change when we perform a $P$-cluster mutation removing $\lambda$ and inserting some $\lambda'$. We use primes to denote the result of mutation, e.g. denoting the new $P$-cluster by $\mathcal{C}'$.

First we suppose there is no global inequality between coordinates $a$ and $b$ in $\mathcal{C}$. That is, these two coordinates reside in the same block $B \ni a,b$ of $\Pi_1$, or equivalently, in the ground set of a root-conjugacy class in $\mathcal{C}$. We can write the vectors in this root-conjugacy class as $\nu+\varepsilon(B)e_s$ as $s$ varies over the block $B$, where $\nu$ has constant $B$-coordinates. The inequality between $a$ and $b$ will only change if $\lambda$ is either 
$\nu+\varepsilon(B)e_a$ or  $\nu+\varepsilon(B)e_b$, and up to symmetry we can assume 
$\lambda = \nu+\varepsilon(B)e_a$. Removing $\lambda$ removes $a$ from the block $B$. The coordinates $a$ and $b$ must reside in different blocks of $\Pi'$, i.e. there is a global inequality between these coordinates which holds on $\mathcal{C}'$ and is attained at least twice. Let $\mu:= (ab)\lambda$; this vector lies in $\mathcal{C}$ hence also in $\mathcal{C}'$. When the block $B$ has cardinality at least three, the sign $\varepsilon(B)$ determines the inequality between the coordinates $\mu_a$ and $\mu_b$, hence determines the direction of inequality between coordinates $a$ and $b$ in $\Pi'$. (When $\varepsilon(B) = +1$, $a$ must reside to the right of the coordinate $b$; when $\varepsilon(B) = -1$, $a$ must reside to the left of the coordinate $b$.) 

If there {\sl is} a global inequality between coordinates $a$ and $b$ in $\mathcal{C}$, then because the strict inequality is attained at least twice, $\mathcal{C}'$ either has the same global inequality, or $a$ and $b$ are in the same block of $\mathcal{C}'$ and we are in the case of the previous paragraph. 

After possibly interchanging $\lambda$ and $\lambda'$, then, we can assume that the removed vector $\lambda = \nu+\varepsilon(B)e_a$ is a root-conjugate vector residing in a root-conjugacy class $\mathcal{O}$ with ground set $B$, where $\nu$ has constant $B$-coordinates. Removing $\lambda$ only changes the coordinate inequalities involving $a$. (All other coordinate inequalities are unchanged since they are witnessed by the vectors in $\mathcal{O} \setminus \lambda$.) If $\varepsilon(B) = +1$ (resp. $=-1$), we know that the coordinate $a$ moves right (resp. left), and since no other coordinate inequalities are affected, $a$ must move to the block immediately right of the block $B \setminus\{a\}$ in $\Pi'$. (If $\{a\}$ is a singleton block, then this block will be immediately right of $B \setminus \{a\}$.) To summarize: only the $a$-coordinate ``moves'' in passing from $\Pi$ to $\Pi'$; we know the direction that this coordinate must move provided the block $B$ has cardinality at least three, and we know that it must move into an immediately adjacent block. We get the dosp mutation rule \eqref{eq:dospadjcency}. 
\end{proof}

Our next lemma is a continuation of the previous one, but now considering all punctures at once.


\begin{lemma}\label{lem:adjacent2} Let $\mathcal{C}$ and $\mathcal{C}'$ be mutation-adjacent $P$-clusters in the good part of the exchange graph. Then the dosps $(\Pi_p(\mathcal{C}),\varepsilon_p(\mathcal{C}))$ and  $(\Pi_p(\mathcal{C}'),\varepsilon_p(\mathcal{C}'))$ disagree at at most one puncture. 
\end{lemma}

Our proof relies in part on the assertion that the quiver for our moduli space has no isolated vertices if we exclude the special case of a once-punctured digon and $G = {\rm SL}_2$. This special case presents no counterexample to our lemma so we can ignore it. The assertion itself can be argued by a painful case analysis.

\begin{proof}
By Lemma~\ref{lem:adjacent}, the dosps at each puncture either coincide or are related by a dosp mutation. We suppose that the mutation removes a vector $\lambda$ and inserts a vector $\lambda'$. Our reasoning below relies on the observation at the construction of a dosp from a collection $\mathcal{C}$ depends only on $\mathcal{C}$ as a set, not as a multiset.  

Suppose that a dosp mutation occurs at a puncture $p$. We will show it does not happen at punctures $q \in \mathbb{M}_\circ$ with $q \neq p$. As in the last paragraph of the previous proof, we can assume (after possibly interchanging roles) that $\pi_{p}(\lambda) \in \pi_{p}(\mathcal{C})$ is a root-conjugate vector, root-conjugate to some $\mu \in \mathcal{C}$ at $p$. Using Proposition~\ref{prop:goodimpliessymmetrical}, we know that the vectors $\lambda$ and $\mu$ coincide at $q$. Thus, in removing $\lambda$ and inserting $\lambda'$, we are removing a vector which is present at least twice in the multiset $\pi_q(\mathcal{C})$ and perhaps adding a new vector. In particular, $\pi_q(\lambda) \in \pi_q(\mathcal{C})$ is not a root-conjugate vector. By the same reasoning, if 
$\pi_p(\lambda')\in \pi_{p}(\mathcal{C}')$ is {\sl also} a root-conjugate vector, then we have 
$\pi_q(\mathcal{C}) = \pi_q(\mathcal{C'})$ as sets (not as multisets). So in this case, the dosp at $q$ does not change. 

Suppose seeking contradiction that adding the vector $\lambda'$ changes the dosp at $q$. Then $\pi_q(\lambda')$ must be a root-conjugate vector in $\pi_q(\mathcal{C}')$. (We argued in the previous proof that when a dosp mutation occurs, one of the two vectors involved in the mutation is a root-conjugate vector, and we know that $\pi_q(\lambda)$ is not a root-conjugate vector.) By the reasoning in the previous proof, there is a strict inequality between a pair of coordinates, say $\nu_c \geq \nu_d$, which holds 
on $\pi_q(\mathcal{C})$ and which is strictly attained exactly twice in the multiset $\pi_q(\mathcal{C})$. Moreover, $\pi_q(\lambda)$ must be one of the two vectors which attains this strict inequality because removing $\lambda$ and inserting $\lambda'$ changes the global inequality between $c$ and $d$ into a root-conjugacy. And since $\lambda$ and $\mu$ coincide at $q$, the vector $\pi_q(\mu)$  is the other vector in $\pi_q(\mathcal{C})$ attaining the strict inequality. Finally,
by Proposition~\ref{prop:goodimpliessymmetrical}, the $P$-cluster variables $\lambda$ and $\mu$ are symmetrical in the quiver $Q$ underlying the $P$-seed of $\mathcal{C}$. Let $v$ be a vertex of the $P$-seed which is adjacent to the vertices $\lambda$ and $\mu$ in $Q$; such vertex exists because $Q$ has no isolated vertices. For simplicity assume that the arrows point from $\lambda,\mu$ to $v$ in $Q$, so that $\lambda$ and $\mu$ contribute to the first term in \eqref{eq:exchangemonomials} describing mutation at vertex $v$ in the $P$-seed. By our assumptions, $\pi_q$ applied to this first sum $(\sum_{u \to v \text{ in Q}}{\rm wt}(u))$ has $c$-coordinate strictly greater than its $d$-coordinate because $\lambda$ and $\mu$ both appear in this sum and all other summands have equal $c$ and $d$ coordinates. On the other hand, $\pi_q$ applied to the second sum $(\sum_{v \to u \text{ in Q}}{\rm wt}(u))$ has equal $c$ and $d$ coordinates for the same reasons. But by the balancing condition for $P$-seeds, these two sums must equal one another, yielding the desired contradiction.  
\end{proof}

\section{Realizing dosp mutations via cluster mutations}\label{secn:contraction}
We explore in this section the relation between the exchange graph of the cluster algebra $\mathscr{A}(\mathcal{M})$ and a simpler finite graph.

First we make some standard graph theoretical definitions. Let $G = (V,E)$ and $G' = (V',E')$ be simple undirected graphs without isolated vertices. A map $f\colon V \to V'$ on their vertex sets is {\sl simplicial} if for each edge $(u, v) \in E$, one either has $f(u) = f(v)$ or has $(f(u), f(v)) \in E'$. A simplicial map is an {\sl edge contraction map} if every $e' \in E'$ takes the form $f(e)$ for {\sl some} $e \in E$. (This implies that $f$ is surjective on vertices.) Viewing graphs as 1-dimensional CW complexes, an edge contraction map determines a homotopy from $G$ to $G'$ while a 
simplicial map induces a homotopy onto a subgraph of $G'$.

\begin{definition}\label{defn:thegraph}
The {\sl dosp mutation graph} $H_{\rm dosp} =H_{\rm dosp}(k) $ is the graph  with vertex set consisting of the dosps on $[k]$ and with edges the dosp mutations. 

When we have fixed in our minds a marked surface $\mathbb{S}$, we associate the
graph $\Box_{\mathbb{M}_\circ}H_{\rm dosp}$, the Cartesian product of 
copies of $H_{\rm dosp}$ indexed by the punctures in $\mathbb{S}$.
\end{definition}

Thus, a vertex of $\Box_{\mathbb{M}_\circ}H_{\rm dosp}$ corresponds to a choice of at each puncture, and an edge corresponds to a dosp mutation at one of the punctures.

\begin{example}
When $k=2$, the dosp mutation graph is the path graph on three vertices
\begin{equation}\label{eq:dospadjacencysl2}
H_{\rm dosp}(2) =  1|2\hspace{.25cm}  \text{------} \hspace{.25cm} |12| \hspace{.25cm}  \text{------} \hspace{.25cm} 2|1.
\end{equation}
The left and right vertices correspond to tagged triangulations which are plain versus notched at a puncture $p$. The middle vertex corresponds to tagged triangulations containing one plain and one notched arc arc at $p$. When $|\mathbb{M}_\circ| = 2$, the Cartesian square $\Box_{\mathbb{M}_\circ}H_{\rm dosp}$ of this graph is the $3 \times 3$ grid graph. 

The graph $H_{\rm dosp}(3)$ is the right graph drawn in Figure~\ref{fig:D4}. Its relationship with the left graph in the figure, i.e. the exchange graph of a type~$D_4$ cluster algebra, is explained in Section~\ref{secn:fmt}.

The graph $H_{\rm dosp}(4)$ has 84 vertices and seems unwieldy to draw. We draw instead its quotient by the $W$-action:
\begin{equation}\label{eq:dospadjacencysl4}
\begin{tikzpicture}
\node at (-6,.5) {$H_{\rm dosp}(4) / W = $};
\node (A) at (0,.25) {$1|2|3|4$};
\node (B) at (140:1.65cm) {$12|3|4$};
\node (C) at (270:1.25cm) {$1|23|4$};
\node (D) at (40:1.65cm) {$1|2|34$};
\node (E) at (90:2.05cm) {$12|34$};
\node (F) at (210:1.85cm) {$123^-|4$};
\node (G) at (330:1.85cm) {$1|234^+$};
\node (H) at (140:2.95cm) {$123^+|4$};
\node (I) at (40:2.95cm) {$1|234^-$};
\node (J) at (160:3.25cm) {$1234^+$};
\node (K) at (20:3.25cm) {$1234^-$};
\draw (B)--(C)--(D);
\draw (B)--(A);
\draw (C)--(A);
\draw (D)--(A);
\draw (B)--(E)--(D);
\draw (F)--(C)--(G);
\draw (E)--(H)--(B);
\draw (E)--(I)--(D);
\draw (K)--(I);
\draw (H)--(J);
\end{tikzpicture}
\end{equation}
\end{example}

\begin{figure}[t]
\begin{tikzpicture}[scale = .75]
\node (ADBC) at (90:6cm) {};
\node (A2D2BC) at (150:6cm) {};
\node (A2D2B23C23) at (210:6cm) {};
\node (A3D3B23C23) at (270:6cm) {};
\node (A3D3B13C13) at (330:6cm) {};
\node (ADB13C13) at (30:6cm) {};

\node (AA2BC) at (110:4.8cm) {};
\node (DD2BC) at (130:4.8cm) {};
\node (A2D2BB23) at (170:4.8cm) {};
\node (A2D2CC23) at (190:4.8cm) {};
\node (A2A3B23C23) at (230:4.8cm) {};
\node (D2D3B23C23) at (255:4.8cm) {};
\node (A3D3B13B23) at (290:4.8cm) {};
\node (A3D3C13C23) at (310:4.8cm) {};
\node (AA3B13C13) at (350:4.8cm) {};
\node (DD3B13C13) at (10:4.8cm) {};
\node (ADBB13) at (50:4.8cm) {};
\node (ADCC13) at (70:4.8cm) {};

\node (AA2B) at (105:4.05cm) {};
\node (AA2C) at (115:4.05cm) {};
\node (DD2B) at (125:4.05cm) {};
\node (DD2C) at (135:4.05cm) {};

\node (A2BB23) at (165:4.05cm) {};
\node (D2BB23) at (175:4.05cm) {};
\node (A2CC23) at (185:4.05cm) {};
\node (D2CC23) at (195:4.05cm) {};

\node (A2A3B23) at (225:4.05cm) {};
\node (A2A3C23) at (235:4.05cm) {};
\node (D2D3B23) at (245:4.05cm) {};
\node (D2D3C23) at (255:4.05cm) {};

\node (A3B13B23) at (285:4.05cm) {};
\node (D3B13B23) at (295:4.05cm) {};
\node (A3C13C23) at (305:4.05cm) {};
\node (D3C13C23) at (315:4.05cm) {};

\node (AA3B13) at (345:4.05cm) {};
\node (AA3C13) at (355:4.05cm) {};
\node (DD3B13) at (5:4.05cm) {};
\node (DD3C13) at (15:4.05cm) {};

\node (ABB13) at (45:4.05cm) {};
\node (DBB13) at (55:4.05cm) {};
\node (ACC13) at (65:4.05cm) {};
\node (DCC13) at (75:4.05cm) {};


\node (Ac) at (85:1.35cm) {};
\node (Ab) at (115:1.35cm) {};
\node (Ba) at (260:1.35cm) {};
\node (Bd) at (295:1.35cm) {};
\node (Ca) at (245:2.15cm) {};
\node (Cd) at (300:2.15cm) {};
\node (Db) at (120:2.15cm) {};
\node (Dc) at (55:2.25cm) {};

\draw (ADBC)--(AA2BC)--(A2D2BC)--(A2D2BB23)--(A2D2B23C23)--(A2A3B23C23)--(A3D3B23C23)--(A3D3B13B23)--(A3D3B13C13)--(AA3B13C13)--(ADB13C13)--(ADBB13)--(ADBC);
\draw (ADBC)--(DD2BC)--(A2D2BC)--(A2D2CC23)--(A2D2B23C23)--(D2D3B23C23)--(A3D3B23C23)--(A3D3C13C23)--(A3D3B13C13)--(DD3B13C13)--(ADB13C13)--(ADCC13)--(ADBC);
\draw (AA2C)--(AA2BC)--(AA2B);
\draw (DD2C)--(DD2BC)--(DD2B);
\draw (A2BB23)--(A2D2BB23)--(D2BB23);
\draw (A2CC23)--(A2D2CC23)--(D2CC23);
\draw (A2A3B23)--(A2A3B23C23)--(A2A3C23);
\draw (D2D3B23)--(D2D3B23C23)--(D2D3C23);
\draw (A3B13B23)--(A3D3B13B23)--(D3B13B23);
\draw (A3C13C23)--(A3D3C13C23)--(D3C13C23);
\draw (AA3B13)--(AA3B13C13)--(AA3C13);
\draw (DD3B13)--(DD3B13C13)--(DD3C13);
\draw (ABB13)--(ADBB13)--(DBB13);
\draw (ACC13)--(ADCC13)--(DCC13);
\draw (DD2B)--(D2BB23)--(D2D3B23)--(D3B13B23)--(DD3B13)--(DBB13)--(DD2B);
\draw (DD2C)--(D2CC23)--(D2D3C23)--(D3C13C23)--(DD3C13)--(DCC13)--(DD2C);
\draw (AA2C)--(A2CC23)--(A2A3C23)--(A3C13C23)--(AA3C13)--(ACC13)--(AA2C);
\draw (AA2B)--(A2BB23)--(A2A3B23)--(A3B13B23)--(AA3B13)--(ABB13)--(AA2B);
\draw (AA2B)--(Ab);
\draw (AA2C)--(Ac);
\draw (A2A3B23)--(Ab);
\draw (A2A3C23)--(Ac);
\draw (AA3B13)--(Ab);
\draw (AA3C13)--(Ac);
\draw (DD2B)--(Db);
\draw (DD2C)--(Dc);
\draw (D2D3B23)--(Db);
\draw (D2D3C23)--(Dc);
\draw (DD3B13)--(Db);
\draw (DD3C13)--(Dc);
\draw (ABB13)--(Ba);
\draw (A2BB23)--(Ba);
\draw (A3B13B23)--(Ba);
\draw (ACC13)--(Ca);
\draw (A2CC23)--(Ca);
\draw (A3C13C23)--(Ca);
\draw (DBB13)--(Bd);
\draw (D2BB23)--(Bd);
\draw (D3B13B23)--(Bd);
\draw (DCC13)--(Cd);
\draw (D2CC23)--(Cd);
\draw (D3C13C23)--(Cd);

\draw (Ac)--(Ab);
\draw (Dc)--(Db);
\draw (Ca)--(Cd);
\draw (Ba)--(Bd);

\node at (ADBC) {\textcolor[rgb]{0,.5,0}{$\bullet$}};
\node at (A2D2BC) {$\bullet$};
\node at (A2D2B23C23) {$\bullet$};
\node at (A3D3B23C23)  {$\bullet$};
\node at (A3D3B13C13) {$\bullet$};
\node at (ADB13C13) {$\bullet$};

\node at (AA2BC)  {\textcolor{red}{$\bullet$}};
\node at (DD2BC)  {\textcolor{red}{$\bullet$}};
\node at (A2D2BB23) {$\bullet$};
\node at (A2D2CC23) {$\bullet$};
\node at (A2A3B23C23)  {$\bullet$};
\node at (D2D3B23C23)  {$\bullet$};
\node at (A3D3B13B23) {$\bullet$};
\node at (A3D3C13C23)  {$\bullet$};
\node at (AA3B13C13) {$\bullet$};
\node at (DD3B13C13) {$\bullet$};
\node at (ADBB13) {$\bullet$};
\node at (ADCC13)  {$\bullet$};

\node at (AA2B)  {\textcolor{red}{$\bullet$}};
\node at (AA2C)  {\textcolor{red}{$\bullet$}};
\node at (DD2B)  {\textcolor{red}{$\bullet$}};
\node at (DD2C)  {\textcolor{red}{$\bullet$}};

\node at (A2BB23)  {$\bullet$};
\node at (D2BB23)  {$\bullet$};
\node at (A2CC23)  {$\bullet$};
\node at (D2CC23)  {$\bullet$};

\node at (A2A3B23)  {$\bullet$};
\node at (A2A3C23)  {$\bullet$};
\node at (D2D3B23)  {$\bullet$};
\node at (D2D3C23)  {$\bullet$};

\node at (A3B13B23)  {$\bullet$};
\node at (D3B13B23)  {$\bullet$};
\node at (A3C13C23)  {$\bullet$};
\node at (D3C13C23) {$\bullet$};

\node at (AA3B13)  {$\bullet$};
\node at (AA3C13)  {$\bullet$};
\node at (DD3B13)  {$\bullet$};
\node at (DD3C13)  {$\bullet$};

\node at (ABB13) {$\bullet$};
\node at (DBB13)  {$\bullet$};
\node at (ACC13)  {$\bullet$};
\node at (DCC13)  {$\bullet$};

\node at (Ab)  {\textcolor{blue}{$\bullet$}};
\node at (Ac) {\textcolor{blue}{$\bullet$}};
\node at (Ba)  {$\bullet$};
\node at(Bd)  {$\bullet$};
\node at (Ca)  {$\bullet$};
\node at (Cd) {$\bullet$};
\node at(Db)  {\textcolor{blue}{$\bullet$}};
\node at (Dc){\textcolor{blue}{$\bullet$}};
\begin{scope}[xshift = 10.5cm]
\node (A) at (90:4cm) {\textcolor[rgb]{0,.5,0}{$1|2|3$}};
\node (B) at (150:4cm) {$2|1|3$};
\node (C) at (120:2cm) {\textcolor{red}{$12|3$}};
\node (D) at (60:2cm) {$1|23$};
\node (E) at (30:4cm) {$1|3|2$};
\node (F) at (180:2cm) {$2|13$};
\node (G1) at (-.45,-.45) {\textcolor{blue}{$123^+$}};   
\node (G2) at (.35,.35) {$123^-$};   
\node (H) at (0:2cm) {$13|2$};
\node (BB) at (150+180:4cm) {$3|1|2$};
\node (CC) at (120+180:2cm) {$3|12$};
\node (DD) at (60+180:2cm) {$23|1$};
\node (EE) at (30+180:4cm) {$2|3|1$};
\node (AA) at (-90:4cm) {$3|2|1$};
\draw (A)--(C)--(D)--(A);
\draw (B)--(C)--(F)--(B);
\draw (C)--(G1)--(H);
\draw (DD)--(G1);
\draw (D)--(G2)--(F);
\draw (CC)--(G2);
\draw (D)--(H)--(E)--(D);
\draw (AA)--(CC)--(DD)--(AA);
\draw (BB)--(CC)--(H)--(BB);
\draw (DD)--(F)--(EE)--(DD);
\end{scope}
\end{tikzpicture}
\caption{The exchange graph for a $D_4$ cluster algebra and its edge contraction on to the dosp mutation graph, illustrating Theorem~\ref{thm:edgecontraction} when $\mathcal{M} = {\rm Gr}_{{\rm SL}_3,D_{2,1}}$. The green (resp. red, resp. blue) colors indicate the fibers of the edge contraction map. 
\label{fig:D4}}
\end{figure}
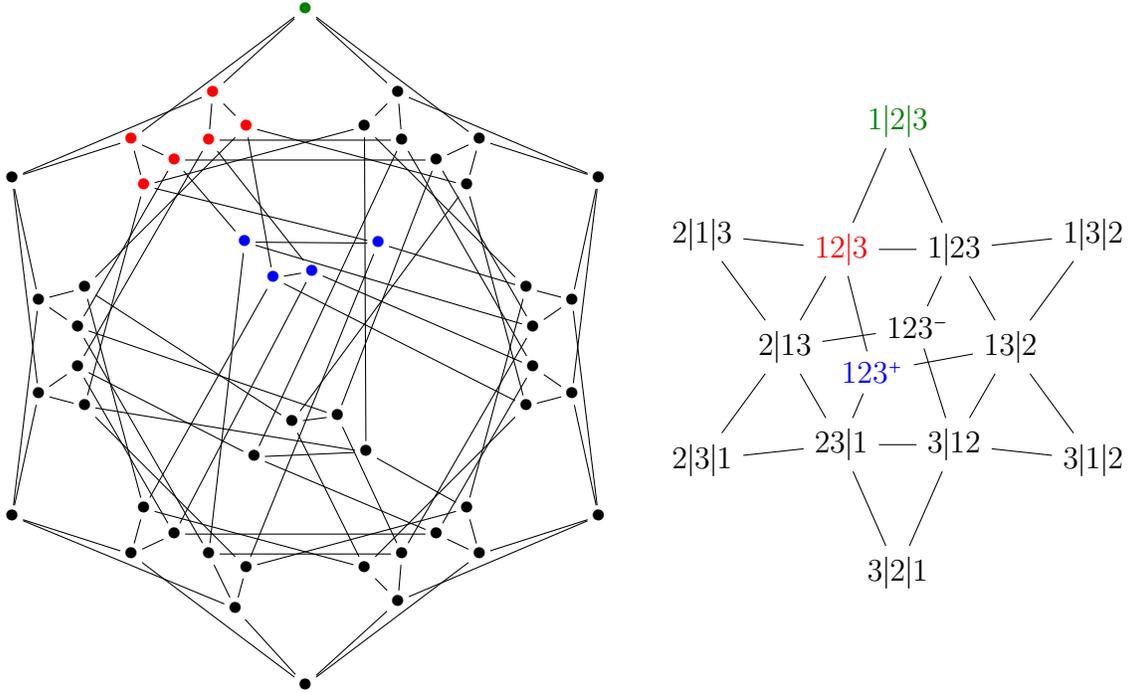

We now identify a relationship between the good part of the exchange graph and the graph $\Box_{\mathbb{M}_\circ}H_{\rm dosp}$.

Consider a $P$-cluster in the good part of the exchange graph. By the recipe \eqref{eq:Pclustertodosp}, it determines a dosp at each puncture, i.e. a vertex of $\Box_{\mathbb{M}_\circ}H_{\rm dosp}$. By Lemmas \ref{lem:adjacent} and ~\ref{lem:adjacent2}, the map $\mathbb{E}_{\rm good} \to \Box_{\mathbb{M}_\circ}H_{\rm dosp}$ is a simplicial map of graphs. Our next theorem addresses the image of this simplicial map.
That is, we ask which dosps and which dosp mutations are ``witnessed'' by clusters and their mutations.

We call a dosp {\sl oneblock} if its underlying osp has only one block. When $k>2$ there are two oneblock dosps. We call a vertex of $\Box_{\mathbb{M}_\circ}H_{\rm dosp}$ {\sl everywhere oneblock} if its dosp at each puncture is oneblock; there are $2^{|\mathbb{M}_\circ|}$  many such vertices when $k>2$. 

\begin{theorem}\label{thm:edgecontraction}
Suppose that $\mathbb{S} \neq D_{1,h}$ for some $h \geq 1$. Then the simplicial map \eqref{eq:Pclustertodosp} has the following properties: 
\begin{enumerate}
\item If $\mathbf{S}$ has boundary, then $\mathbb{E}_{\rm good} \to \Box_{\mathbb{M}_\circ}H_{\rm dosp}$ is an edge contraction map. 
\item If $\mathbf{S}$ has no boundary (hence when $\mathbb{S} = S_{g,h}$), then the map 
$\mathbb{E}_{\rm good} \to \Box_{\mathbb{M}_\circ}H_{\rm dosp}$ misses the everywhere oneblock vertices. 
\end{enumerate}
\end{theorem}

\begin{conjecture}\label{conj:oneblock}
When $\mathbb{S} = S_{g,h}$ and $(k,\mathbb{S}) \neq (2,S_{0,3})$, the map \eqref{eq:Pclustertodosp} is an edge contraction onto the induced subgraph on the vertices which are not everywhere oneblock. 
\end{conjecture}

Let $H'$ be the result of deleting the everywhere oneblock vertices from the graph $\Box_{\mathbb{M}_\circ}H_{\rm dosp}$. We will prove momentarily that the image of \eqref{eq:Pclustertodosp} is contained in this smaller graph $H'$ when $\mathbb{S}$ has no boundary. A verification of the above conjecture would involve showing that every edge in $H'$ has a preimage under the map $P$-cluster to dosp map \eqref{eq:Pclustertodosp}. This is a tractable explicit question which we do not address in this paper.  

\begin{remark}
We expect that (1) from the above theorem holds for punctured monogons $D_{1,h}$ once $h \geq 2$, but there is a technical difficulty with the proof we present. We have excluded once-punctured monogons in this paper, but if one considers them, then the statement is false when $k=2,3$. It most likely holds when $k>3$.
\end{remark}

\begin{remark} We illustrate (2) from Theorem~\ref{thm:edgecontraction} when $k = 2$. A tagged triangulation $T$ is oneblock at $p$ if there exists tagged arcs $\gamma,\gamma^\bowtie \in T$ with $\gamma$ plain and $\gamma^\bowtie$ its notching at $p$. The other endpoint of $\gamma$ is at a puncture $q$ since there are no boundary points. The tagged arcs $\gamma$ and $\gamma^\bowtie$ must be tagged the same way at $q$, hence $T$ is not oneblock at $q$. That is, it is not possible to be oneblock at {\sl every puncture} when $\mathbb{S}$ has no boundary. 

By an explicit argument (constructing approriate tagged triangulations and flips between them), one can check that Conjecture~\ref{conj:oneblock} holds in this case. That is, every vertex which is not oneblock at every puncture has a preimage under \eqref{eq:Pclustertodosp}, as does any edge between two such vertices. 
\end{remark}

We now prove part (2) of Theorem~\ref{thm:edgecontraction} based on the ideas of the preceding remark:
\begin{proof}[Proof of (2) from Theorem~\ref{thm:edgecontraction}]
For a weight vector $\lambda =(\lambda_1,\dots,\lambda_k) \in P$, ${\rm sum}\, \lambda := \sum_i \lambda_i \mod k$ iswell defined. By induction on mutations, every $P$-cluster variable $\lambda$ has the property that $\sum_p {\rm sum}\, \pi_p(\lambda) = 0 \mod k$. There are only two options for a collection $\mathcal{C} \subset P$ to have a oneblock dosp: either $\mathcal{C}$ consists of the vectors $e_1,\dots,e_k$ and multiple copies of the zero vector, or it consists of $-e_1,\dots,-e_k$ and multiple copies of the zero vector. By Proposition~\ref{prop:goodimpliessymmetrical}, if a $P$-cluster variable $\lambda$ is a root-conjugate vector at $p$ then it is not root-conjugate at any other puncture. If the $P$-cluster containing $\lambda$ has everywhere oneblock dosps, then $\lambda$ has zero weight at every puncture $q \neq p$. But then ${\rm sum }\, \lambda = \pm 1 \mod k$ which is a contradiction.
\end{proof}

Our proof of (1) from Theorem~\ref{thm:edgecontraction} is more involved. First, we prove the statement for the special cases of $\mathscr{A}'({\rm SL}_k,D_{2,1})$ and then $\mathscr{A}'({\rm SL}_k,D_{2,1})$ by providing explicit mutations. Then we reduce the assertion for an arbitrary version of the moduli space to this case by simple ``cutting and gluing'' type arguments. 

The reader should consult the following example, which contains the heart of the whole argument, before reading the proof of Lemma~\ref{lem:digoncase}.

\begin{example}\label{eg:everydosp}
The cluster $\mathscr{A}'({\rm SL}_7,D_{2,1})$ has a canonical initial seed since the 
once-punctured digon has a unique taut triangulation. Its $P$-cluster (with weight vectors written in multiplicative notation) 
is below:
\begin{equation}\label{eq:ladderquiver}
\begin{tikzpicture}[scale = .9]
\node at (2.5,-1.5) {\boxed{}};
\node at (-1.0,-1.5) {\boxed{}};
\node at (0,-1) {$a$};
\node at (1.5,-1) {$a$};
\node at (0,-2) {$ab$};
\node at (1.5,-2) {$ab$};
\node at (0,-3) {$abc$};
\node at (1.5,-3) {$abc$};
\node at (0,-4) {$abcd$};
\node at (1.5,-4) {$abcd$};
\node at (0,-5) {$abcde$};
\node at (1.5,-5) {$abcde$};
\node at (0,-6) {$abcdef$};
\node at (1.5,-6) {$abcdef$};
\draw [shorten >=0.25cm,shorten <=0.25cm,->] (0,-1)--(0,-2);
\draw [shorten >=0.25cm,shorten <=0.25cm,->] (0,-2)--(0,-3);
\draw [shorten >=0.25cm,shorten <=0.25cm,->] (0,-3)--(0,-4);
\draw [shorten >=0.25cm,shorten <=0.25cm,->] (0,-4)--(0,-5);
\draw [shorten >=0.25cm,shorten <=0.25cm,->] (1.5,-1)--(1.5,-2);
\draw [shorten >=0.25cm,shorten <=0.25cm,->] (1.5,-2)--(1.5,-3);
\draw [shorten >=0.25cm,shorten <=0.25cm,->] (1.5,-3)--(1.5,-4);
\draw [shorten >=0.25cm,shorten <=0.25cm,->] (1.5,-4)--(1.5,-5);
\draw [shorten >=0.35cm,shorten <=0.35cm,->] (0,-2)--(1.5,-1);
\draw [shorten >=0.5cm,shorten <=0.35cm,->] (0,-3)--(1.5,-2);
\draw [shorten >=0.5cm,shorten <=0.35cm,->] (0,-4)--(1.5,-3);
\draw [shorten >=0.5cm,shorten <=0.35cm,->] (0,-5)--(1.5,-4);
\draw [shorten >=0.35cm,shorten <=0.35cm,->] (1.5,-2)--(0,-1);
\draw [shorten >=0.5cm,shorten <=0.35cm,->] (1.5,-3)--(0,-2);
\draw [shorten >=0.5cm,shorten <=0.35cm,->] (1.5,-4)--(0,-3);
\draw [shorten >=0.5cm,shorten <=0.35cm,->] (1.5,-5)--(0,-4);
\draw [shorten >=0.5cm,shorten <=0.35cm,->] (1.5,-6)--(0,-5);
\draw [shorten >=0.5cm,shorten <=0.35cm,->] (0,-6)--(1.5,-5);
\draw [shorten >=0.25cm,shorten <=0.25cm,->] (0,-5)--(0,-6);
\draw [shorten >=0.25cm,shorten <=0.25cm,->] (1.5,-5)--(1.5,-6);
\draw [shorten >=0.25cm,shorten <=0.25cm,->] (0,-1)--(2.5,-1.5);
\draw [shorten >=0.25cm,shorten <=0.25cm,->] (2.5,-1.5)--(1.5,-1);
\draw [shorten >=0.25cm,shorten <=0.25cm,->] (1.5,-1)--(-1.0,-1.5);
\draw [shorten >=0.25cm,shorten <=0.25cm,->] (-1.0,-1.5)--(0,-1);
\node at (.7,-6.65) {initial $P$-seed};

\begin{scope}[xshift = 6.5cm]
\node at (2.5,-1.5) {\tiny \boxed{14}};
\node at (-1.0,-1.5) {\tiny \boxed{13}};
\node at (0,-1) {1};
\node at (1.5,-1) {2};
\node at (0,-2) {3};
\node at (1.5,-2) {4};
\node at (0,-3) {5};
\node at (1.5,-3) {6};
\node at (0,-4) {7};
\node at (1.5,-4) {8};
\node at (0,-5) {9};
\node at (1.5,-5) {10};
\node at (0,-6) {11};
\node at (1.5,-6) {12};
\draw [shorten >=0.25cm,shorten <=0.25cm,->] (0,-1)--(0,-2);
\draw [shorten >=0.25cm,shorten <=0.25cm,->] (0,-2)--(0,-3);
\draw [shorten >=0.25cm,shorten <=0.25cm,->] (0,-3)--(0,-4);
\draw [shorten >=0.25cm,shorten <=0.25cm,->] (0,-4)--(0,-5);
\draw [shorten >=0.25cm,shorten <=0.25cm,->] (1.5,-1)--(1.5,-2);
\draw [shorten >=0.25cm,shorten <=0.25cm,->] (1.5,-2)--(1.5,-3);
\draw [shorten >=0.25cm,shorten <=0.25cm,->] (1.5,-3)--(1.5,-4);
\draw [shorten >=0.25cm,shorten <=0.25cm,->] (1.5,-4)--(1.5,-5);
\draw [shorten >=0.35cm,shorten <=0.35cm,->] (0,-2)--(1.5,-1);
\draw [shorten >=0.5cm,shorten <=0.35cm,->] (0,-3)--(1.5,-2);
\draw [shorten >=0.5cm,shorten <=0.35cm,->] (0,-4)--(1.5,-3);
\draw [shorten >=0.5cm,shorten <=0.35cm,->] (0,-5)--(1.5,-4);
\draw [shorten >=0.35cm,shorten <=0.35cm,->] (1.5,-2)--(0,-1);
\draw [shorten >=0.5cm,shorten <=0.35cm,->] (1.5,-3)--(0,-2);
\draw [shorten >=0.5cm,shorten <=0.35cm,->] (1.5,-4)--(0,-3);
\draw [shorten >=0.5cm,shorten <=0.35cm,->] (1.5,-5)--(0,-4);
\draw [shorten >=0.5cm,shorten <=0.35cm,->] (1.5,-6)--(0,-5);
\draw [shorten >=0.5cm,shorten <=0.35cm,->] (0,-6)--(1.5,-5);
\draw [shorten >=0.25cm,shorten <=0.25cm,->] (0,-5)--(0,-6);
\draw [shorten >=0.25cm,shorten <=0.25cm,->] (1.5,-5)--(1.5,-6);
\node at (.5,-6.65) {vertex numbering};
\draw [shorten >=0.25cm,shorten <=0.25cm,->] (0,-1)--(2.5,-1.5);
\draw [shorten >=0.25cm,shorten <=0.25cm,->] (2.5,-1.5)--(1.5,-1);
\draw [shorten >=0.25cm,shorten <=0.25cm,->] (1.5,-1)--(-1.0,-1.5);
\draw [shorten >=0.25cm,shorten <=0.25cm,->] (-1.0,-1.5)--(0,-1);
\end{scope}
\end{tikzpicture}
\end{equation}
The two frozen vertices are boxed and have weight zero, hence can be deleted. The dosp associated to this $P$-seed is $1|2|3|4|5|6| 7$. We now describe a sequence of $P$-seed mutations in the good part of the exchange graph which, after applying the ``$P$-cluster to dosp map'' \eqref{eq:Pclustertodosp}, 
witnesses the sequence of dosp mutations 
$$1|2|3|4|5|6| 7 \mapsto 1|23|4|5|6|7\mapsto 1|234^+|5|6|7\mapsto 1|2345^+|6|7\mapsto 1|23456^+|7.$$

In the first step, perform mutation $\mu_4$. The result is the leftmost $P$-seed in Figure~\ref{fig:ladder}, drawn with vertices 3 and 4 merged and highlighted in red. The dosp of this $P$-seed is $1|23|4|5|6|7$. In the second step, mutate $\mu_5 \circ \mu_6$ to obtain the next 
$P$-seed in Figure~\ref{fig:ladder}, with vertices 3,4,5 merged and in red. The dosp is $1|234^+|5|6|7$. Third, perform $\mu_6 \circ \mu_7 \circ \mu_8$ to obtain the next $P$-seed, with vertices 3,4,5,6 merged and with dosp $1|2345^+|6|7$. Finally mutate $\mu_{7} \circ \mu_8 \circ \mu_9 \circ \mu_{10}$ to obtain the fourth seed, with dosp $1|23456^+|7$.

We can simulate the dosp mutations 
$$1|23456^+|7 \mapsto 1|2345^+|6|7 \mapsto 1|234^+|56|7 \mapsto 1|23|456^-|7\mapsto 1|2|3456^-|7 \mapsto 1|23456^-|7$$
by mutating at the red vertices in the rightmost $P$-seed one by one. That is, 
performing the mutation $\mu_7$ we get $af \mapsto a^3bcde$ which witnesses the first of these dosp mutations. Performing $\mu_6$ we get $ae \mapsto a^3bcdf$ and the resulting $P$-cluster exhausts $1|234^+|56|7$. 
\end{example}

\begin{figure}[ht]
\begin{tikzpicture}
\begin{scope}[xshift = 0cm]
\node at (.7,-6.6) {$1|23|4|5|6|7$};
\node at (0,-1) {$a$};
\node at (1.5,-1) {$a$};
\node at (.25,-2) {\textcolor{red}{$ab,ac$}};
\node at (0,-3) {$abc$};
\node at (1.5,-3) {$abc$};
\node at (0,-4) {$abcd$};
\node at (1.5,-4) {$abcd$};
\node at (0,-5) {$abcde$};
\node at (1.5,-5) {$abcde$};
\node at (0,-6) {\small $abcdef$};
\node at (1.5,-6) {\small $abcdef$};
\draw [shorten >=0.25cm,shorten <=0.25cm,->] (1.5,-1)--(0,-1);
\draw [shorten >=0.35cm,shorten <=0.35cm,->] (0,-3)--(1.5,-3);
\draw [shorten >=0.25cm,shorten <=0.25cm,->] (.25,-2)--(0,-3);
\draw [shorten >=0.25cm,shorten <=0.25cm,->] (0,-3)--(0,-4);
\draw [shorten >=0.25cm,shorten <=0.25cm,->] (0,-4)--(0,-5);
\draw [shorten >=0.25cm,shorten <=0.25cm,->] (0,-1)--(.25,-2);
\draw [shorten >=0.25cm,shorten <=0.25cm,->] (1.5,-1)--(1.5,-3);
\draw [shorten >=0.25cm,shorten <=0.25cm,->] (1.5,-3)--(1.5,-4);
\draw [shorten >=0.25cm,shorten <=0.25cm,->] (1.5,-4)--(1.5,-5);
\draw [shorten >=0.35cm,shorten <=0.35cm,->] (.25,-2)--(1.5,-1);
\draw [shorten >=0.5cm,shorten <=0.35cm,->] (0,-4)--(1.5,-3);
\draw [shorten >=0.5cm,shorten <=0.35cm,->] (0,-5)--(1.5,-4);
\draw [shorten >=0.5cm,shorten <=0.35cm,->] (1.5,-3)--(0,-2);
\draw [shorten >=0.5cm,shorten <=0.35cm,->] (1.5,-4)--(0,-3);
\draw [shorten >=0.5cm,shorten <=0.35cm,->] (1.5,-5)--(0,-4);
\draw [rounded corners,->] (-.35,-2.85)--(-.5,-2)--(-.25,-1.15);
\draw [shorten >=0.5cm,shorten <=0.35cm,->] (1.5,-6)--(0,-5);
\draw [shorten >=0.5cm,shorten <=0.35cm,->] (0,-6)--(1.5,-5);
\draw [shorten >=0.25cm,shorten <=0.25cm,->] (0,-5)--(0,-6);
\draw [shorten >=0.25cm,shorten <=0.25cm,->] (1.5,-5)--(1.5,-6);
\end{scope}
\begin{scope}[xshift = 3.5cm]
\node at (1.0,-6.6) {$1|234^+|5|6|7$};
\node at (0,-1) {$a$};
\node at (2,-1) {$a$};
\node at (-.25,-3) {\small \textcolor{red}{$ab,ac,ad$}};
\node at (2,-3) {$a^2bcd$};
\node at (0,-4) {$abcd$};
\node at (2,-4) {$abcd$};
\node at (0,-5) {$abcde$};
\node at (2,-5) {$abcde$};
\node at (0,-6) {\small $abcdef$};
\node at (2,-6) {\small $abcdef$};
\draw [shorten >=0.25cm,shorten <=0.25cm,->] (2,-1)--(0,-1);
\draw [shorten >=0.65cm,shorten <=0.65cm,->] (0,-3)--(2,-3);
\draw [shorten >=0.5cm,shorten <=0.5cm,->] (2,-3)--(0,-1);
\draw [shorten >=0.5cm,shorten <=0.5cm,->] (0,-4)--(2,-4);
\draw [shorten >=0.25cm,shorten <=0.25cm,->] (0,-1)--(0,-3);
\draw [shorten >=0.25cm,shorten <=0.25cm,->] (0,-4)--(0,-3);
\draw [shorten >=0.25cm,shorten <=0.25cm,->] (0,-4)--(0,-5);
\draw [shorten >=0.25cm,shorten <=0.25cm,->] (2,-3)--(2,-1);
\draw [shorten >=0.45cm,shorten <=0.45cm,->] (2,-3)--(0,-4);
\draw [shorten >=0.55cm,shorten <=0.55cm,->] (2.1,-3.1)--(.1,-4.1);
\draw [shorten >=0.25cm,shorten <=0.25cm,->] (2,-4)--(2,-3);
\draw [shorten >=0.25cm,shorten <=0.25cm,->] (2,-4)--(2,-5);
\draw [shorten >=0.5cm,shorten <=0.35cm,->] (0,-5)--(2,-4);
\draw [shorten >=0.5cm,shorten <=0.35cm,->] (2,-5)--(0,-4);
\draw [rounded corners,<-] (2+.35,-3.75)--(2.8,-2.5)--(2+.25,-1.15);
\draw [shorten >=0.5cm,shorten <=0.35cm,->] (2,-6)--(0,-5);
\draw [shorten >=0.5cm,shorten <=0.35cm,->] (0,-6)--(2,-5);
\draw [shorten >=0.25cm,shorten <=0.25cm,->] (0,-5)--(0,-6);
\draw [shorten >=0.25cm,shorten <=0.25cm,->] (2,-5)--(2,-6);
\end{scope}
\begin{scope}[xshift = 8cm]
\node at (1.0,-6.6) {$1|2345^+|6|7$};
\node at (0,-1) {a};
\node at (2,-1) {$a$};
\node at (.5,-3) {\small \textcolor{red}{$ab,ac,ad,ae$}};
\node at (-.25,-4) {\small $a^3bcde$};
\node at (2.25,-4) {\small $a^2bcde$};
\node at (0,-5) {$abcde$};
\node at (2,-5) {$abcde$};
\node at (0,-6) {\small $abcdef$};
\node at (2,-6) {\small $abcdef$};
\draw [shorten >=0.25cm,shorten <=0.25cm,->] (2,-1)--(0,-1);
\draw [shorten >=0.5cm,shorten <=0.5cm,->] (0,-4)--(2,-4);
\draw [shorten >=0.5cm,shorten <=0.5cm,->] (0,-4.1)--(2,-4.1);
\draw [shorten >=0.25cm,shorten <=0.25cm,->] (0,-1)--(0,-3);
\draw [shorten >=0.25cm,shorten <=0.25cm,->] (0,-3)--(0,-4);
\draw [shorten >=0.25cm,shorten <=0.25cm,->] (0,-5)--(0,-4);
\draw [shorten >=0.5cm,shorten <=0.25cm,->] (2,-4)--(2,-1);
\draw [shorten >=0.25cm,shorten <=0.25cm,->] (2,-5)--(2,-4);
\draw [shorten >=0.65cm,shorten <=0.65cm,->] (2,-4)--(0,-3);
\draw [shorten >=0.65cm,shorten <=0.65cm,->] (2,-4)--(0,-5);
\draw [shorten >=0.65cm,shorten <=0.65cm,->] (2.1,-4.1)--(.1,-5.1);
\draw [shorten >=0.65cm,shorten <=0.65cm,->] (0,-5)--(2,-5);
\draw [rounded corners,->] (-.35,-3.85)--(-.75,-2.5)--(-.25,-1.15);
\draw [shorten >=0.5cm,shorten <=0.35cm,->] (2,-6)--(0,-5);
\draw [shorten >=0.5cm,shorten <=0.35cm,->] (0,-6)--(2,-5);
\draw [shorten >=0.25cm,shorten <=0.25cm,->] (0,-5)--(0,-6);
\draw [shorten >=0.25cm,shorten <=0.25cm,->] (2,-5)--(2,-6);
\draw [rounded corners,<-] (2+.35,-4.85)--(2.75,-3)--(2.25,-1.15);
\end{scope}
\begin{scope}[xshift = 12.5cm]
\node at (1.0,-6.6) {$1|23456^+|7$};
\node at (0,-1) {$a$};
\node at (2,-1) {$a$};
\node at (-.5,-4) {\small \textcolor{red}{$ab,\dots,af$}};
\node at (2.1,-4) {\small $a^4bcdef$};
\node at (-.25,-5) {\small $a^3bcdef$};
\node at (2.25,-5) {\small $a^2bcdef$};
\node at (0,-6) {\small $abcdef$};
\node at (2,-6) {\small $abcdef$};
\draw [shorten >=0.25cm,shorten <=0.25cm,->] (2,-1)--(0,-1);
\draw [shorten >=0.25cm,shorten <=0.25cm,->] (.25,-4)--(1.5,-4);
\draw [shorten >=0.5cm,shorten <=0.5cm,->] (1.5,-4)--(0,-1);
\draw [shorten >=0.25cm,shorten <=0.25cm,->] (0,-1)--(0,-4);
\draw [shorten >=0.25cm,shorten <=0.25cm,->] (0,-5)--(0,-4);
\draw [shorten >=0.25cm,shorten <=0.25cm,->] (2.25,-5)--(1.75,-4);
\draw [shorten >=0.25cm,shorten <=0.25cm,->] (2.25,-5)--(2,-1);
\draw [shorten >=0.65cm,shorten <=0.45cm,->] (2,-4)--(0,-5);
\draw [shorten >=0.65cm,shorten <=0.55cm,->] (2.1,-4.1)--(.1,-5.1);
\draw [rounded corners,<-] (2+.35,-5.75)--(3.2,-3.5)--(2+.25,-1.15);
\draw [shorten >=0.65cm,shorten <=0.65cm,->] (2,-5)--(0,-6);
\draw [shorten >=0.65cm,shorten <=0.65cm,->] (2.1,-5.1)--(.1,-6.1);
\draw [shorten >=0.25cm,shorten <=0.25cm,->] (0,-6)--(0,-5);
\draw [shorten >=0.25cm,shorten <=0.25cm,->] (2.25,-6)--(2.25,-5);
\draw [shorten >=0.65cm,shorten <=0.65cm,->] (0,-6)--(2,-6);
\draw [shorten >=0.65cm,shorten <=0.65cm,->] (0,-5)--(2,-5);
\draw [shorten >=0.65cm,shorten <=0.65cm,->] (0,-5.1)--(2,-5.1);
\end{scope}
\end{tikzpicture}
\caption{A sequence of $P$-seeds appearing as intermediate steps of the mutation sequence from Example~\ref{eg:everydosp} and the proof of Theorem~\ref{thm:edgecontraction}. Each $P$-seed is written above the dosp that its $P$-cluster exhausts. Red vertices indicate merging of several symmetrical vertices of the $P$-seed (thereby reducing the number of edges in the picture), e.g. $\textcolor{red}{ab,ac}$ is obtained by gluing a vertex of weight $ab$ with one of weight $ac$, creating a vertex of weight $a^2bc$. 
\label{fig:ladder}}
\end{figure}
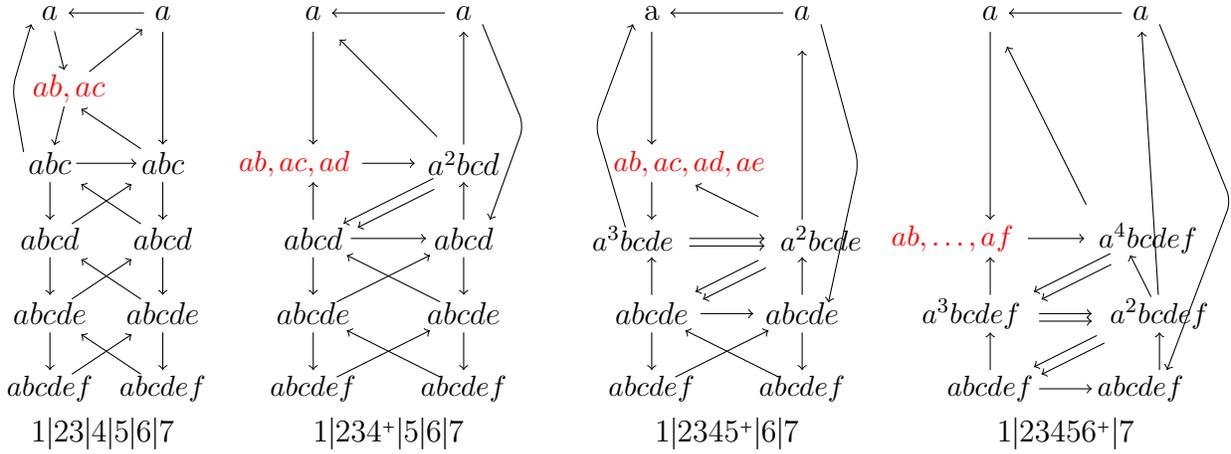

\begin{lemma}\label{lem:digoncase}
Theorem~\ref{thm:edgecontraction} holds for $\mathscr{A}'({\rm SL}_k,D_{2,1})$.\end{lemma}

\begin{proof}
There is only puncture in this case. What we need to show is that every dosp is reachable inside the good part of the exchange graph by a sequence of mutations from the initial $P$-seed, and that every dosp mutation comes from a mutation of $P$-clusters inside the good part of the exchange graph. 

The initial $P$-cluster when $G = {\rm SL}_k$ is the ladder quiver as in 
\eqref{eq:ladderquiver}, but with $k-1$ rows. We number its vertices from left to right and top to bottom as in \eqref{eq:ladderquiver}. Mutation at both vertices in the $i$th row is the Weyl group action $s_i$. Precompose with such mutations as needed, we can prove the desired surjectivity up to $W$-action. 

The proof is based on the mutation sequences given in Example~\ref{eg:everydosp}. If one wants to simulate a dosp mutation $(L\cup \{a\})^+|R^- \mapsto (L)^+|(R \cup \{a\})^-$, one should create the positively decorated block $(L \cup R)^+$ using the first sequence of mutations given  in that example. (Mutating down the rows of the ladder quiver, starting either in the right column or the left column depending on parity.) One can then ``peel off'' the coordinates in $R$ by mutating at the corresponding merged ``red'' vertices, arriving at the dosp $(L\cup \{a\})^+|R^- $. Mutating at one more of these vertices implements the desired dosp mutation. 

The mutation sequences involved are ``local'' in the coordinates $L \cup R$. That is, we can create ``the rest of'' the dosp on the coordinates $[k ] \setminus (L \cup R)$ using the same mutation sequences, because the mutations in distinct blocks commute. To illustrate, in Example~\ref{eg:everydosp}, we create the decorated block $23456^+$ by mutating at rows $2,3,4,5$ only. We would be able to merge $7$ into a block (of any sign we like) by performing the needed mutations starting in row~$7$.
\end{proof}

For the next step in the proof of Theorem~\ref{thm:edgecontraction}, we will show that various $P$-seeds are mutation-equivalent to each other provided we ignore vertices of weight zero. We will use the notation $\Sigma_1 \leadsto \Sigma_2$ to indicate that we can reach the $P$-seed $\Sigma_2$ from $\Sigma_1$  by a sequence of mutations which do not violate Conjecture~\ref{conj:Pclusterconjecture} and deletion of weight zero vertices. 

The once-punctured $m$-gon $D_{m,1}$ has a unique taut triangulation $\Delta = \Delta_m$. It gives rise to initial $P$-seeds $\Sigma_k(\Delta_m)$ and $\Sigma'_k(\Delta_m)$ for $\mathscr{A}({\rm SL}_k,D_{m,1})$ and $\mathscr{A}'({\rm SL}_k,D_{m,1})$ respectively.

We upgrade the intermediate quiver $Q_k^s$ defined in Section~\ref{subsec:Qks} to a $P$-seed $\Sigma_k^s$ by assigning to the vertex indexed by the triple $(a_1,a_2,a_3) \in \mathbb{H}_k$ the weight $\omega_{a_1} \in P$. We denote by $\Sigma^{s}_k \underset{\sim}{\coprod} \Sigma^{t}_k$ the result of gluing two such $P$-seeds along their left and right boundary sides. Thus $\Sigma^{k-1}_k \underset{\sim}{\coprod} \Sigma^{k-1}_k = \Sigma_k(\Delta_{2})$ and $\Sigma^{1}_k \underset{\sim}{\coprod} \Sigma^{1}_k = \Sigma'_k(\Delta_{2})$.

\begin{lemma}\label{lem:leadsto}
We have the following $\leadsto$-relationship between $P$-seeds:
\begin{align}
\Sigma_k(\Delta_{2})&\leadsto \Sigma'_k(\Delta_{k}) \label{eq:leadstoii} \\
\Sigma^{s}_k \underset{\sim}{\coprod} \Sigma^{1}_k & \leadsto \Sigma^{s-1}_k \underset{\sim}{\coprod} \Sigma^{1}_k \text{ for $s \in [2,k-1]$} \text{ for any $k$.}	\label{eq:leadstoiii}
\end{align}
\end{lemma}

\begin{proof}
For the first statement \eqref{eq:leadstoii}, the frozen variables play no role since they have weight zero at the puncture. We claim that the mutable parts of the $P$-seeds $\Sigma_k(\Delta_{2})$  and $\Sigma_k'(\Delta_{k})$ are mutation-equivalent, establishing \eqref{eq:leadstoii}. This is a particular instance of the conjectured quasi cluster isomorphism asserted in Remark~\ref{rmk:subtle}. We argue the statement we need in this case as follows. 

There is a well known recipe that associates to a quiver to any reduced word for $w_0 \in W$. One can upgrade this quiver to a $P$-seed by assigning vertices in row $i$ the weight $\omega_i$. The seed $\Sigma_k^{k-1}$ corresponds to the reduced word $1,\dots,k-1,1,\dots,k-2,\dots,1,2,1$ and the seed $\Sigma_k(\Delta_{2})$ is the gluing of two such seeds (corresponding to the two triangles in the triangulation of the once-punctured digon) 
along their shared edges. One knows that we can modify our reduced word (in either triangle) and obtain a mutation-equivalent seed. Let $o = \prod_{i \text{ odd}} s_i$ and $e = \prod_{i \text{ even}} s_i$ be the product of odd (resp. even) Coxeter generators. Then $o e o e \cdots $ is reduced word for $w_0$ and so is $e o e e\cdots $ (with $k$ terms in either product). Depending on parity of $k$, one can either glue $oeoe\cdots$ to itself, or glue $oeoe\cdots$ with 
$eoeo\cdots$, so that the letters alternate as we read the two triangles in cyclic order. 

Deleting the frozen variables, which have weight 0, gives the $P$-seed $\Sigma'_k(\Delta_{k})$. We illustrate this argument when $k=5$ in Figure~\ref{fig:GrisFG}. The mutation needed to create the upper left quiver from the bottom quiver is $\mu_{19,12,11,6}$ (which creates the word $oeoeo$ in the right triangle) followed by $\mu_{14,16,9,4,8,9}$ (which creates $eoeo$ in the left triangle). Rearranging the vertices, we recognize that this $P$-seed is in fact that of $\Sigma'_k(\Delta_{k})$, proving \eqref{eq:leadstoii}.

The proof of the second statement \eqref{eq:leadstoiii} is illustrated in Figure~\ref{fig:stos--} when $k=8$. The topmost figure is  $\Sigma_8^7 \underset{\sim}{\coprod} \Sigma_8^1 $; moving rightwards and downwards across the page we see $\Sigma_8^6 \underset{\sim}{\coprod} \Sigma_8^1 $ then $\Sigma_8^5 \underset{\sim}{\coprod} \Sigma_8^1 $ etc. Each row of $\Sigma_k^s$ has at most $s$ many variables. To pass from $\Sigma_8^s \underset{\sim}{\coprod} \Sigma_8^1 $ to $\Sigma_8^{s-1} \underset{\sim}{\coprod} \Sigma_8^1 $, one should locate the highest row that has $s$ variables in it. There will be a variable $x$ that has no arrows to higher rows; one should mutate down the column containing~$x$, starting in the row containing~$x$. When a variable in row $i$ is mutated, its weight changes $\omega_i \mapsto \omega_{i-1}$ and one should drag it down a row. One deletes a variable when it has weight zero. After mutating down a column, move one column right and mutate down this column starting one row lower than in the previous column. \end{proof}

\begin{figure}
\begin{tikzpicture}
\node at (5.7,-3.0) {
\begin{xy} 0;<.13pt,0pt>:<0pt,-.13pt>:: 
(466,464) *+{1} ="0",
(455,347) *+{2} ="1",
(456,580) *+{\textcolor{red}{3}} ="2",
(338,396) *+{\textcolor{red}{4}} ="3",
(459,251) *+{\textcolor{red}{5}} ="4",
(585,421) *+{\textcolor{red}{6}} ="5",
(468,678) *+{\textcolor{blue}{7}} ="6",
(191,536) *+{\textcolor{blue}{8}} ="7",
(167,298) *+{\textcolor{blue}{9}} ="8",
(472,151) *+{\textcolor{blue}{10}} ="9",
(717,278) *+{\textcolor{blue}{11}} ="10",
(762,520) *+{\textcolor{blue}{12}} ="11",
(471,793) *+{\textcolor{green}{13}} ="12",
(191,721) *+{\textcolor{green}{14}} ="13",
(0,377) *+{\textcolor{green}{15}} ="14",
(209,97) *+{\textcolor{green}{16}} ="15",
(472,0) *+{\textcolor{green}{17}} ="16",
(709,102) *+{\textcolor{green}{18}} ="17",
(949,407) *+{\textcolor{green}{19}} ="18",
(711,752) *+{\textcolor{green}{20}} ="19",
"3", {\ar"0"},
"0", {\ar"5"},
"1", {\ar"3"},
"5", {\ar"1"},
"2", {\ar"3"},
"5", {\ar"2"},
"7", {\ar"2"},
"2", {\ar"11"},
"3", {\ar"4"},
"3", {\ar"7"},
"8", {\ar"3"},
"4", {\ar"5"},
"4", {\ar"8"},
"10", {\ar"4"},
"5", {\ar"10"},
"11", {\ar"5"},
"6", {\ar"7"},
"11", {\ar"6"},
"13", {\ar"6"},
"6", {\ar"19"},
"7", {\ar"8"},
"7", {\ar"13"},
"14", {\ar"7"},
"8", {\ar"9"},
"8", {\ar"14"},
"15", {\ar"8"},
"9", {\ar"10"},
"9", {\ar"15"},
"17", {\ar"9"},
"10", {\ar"11"},
"10", {\ar"17"},
"18", {\ar"10"},
"11", {\ar"18"},
"19", {\ar"11"},
"12", {\ar"13"},
"19", {\ar"12"},
"13", {\ar"14"},
"14", {\ar"15"},
"15", {\ar"16"},
"16", {\ar"17"},
"17", {\ar"18"},
"18", {\ar"19"},
\end{xy}
};

\node at (0,0) {\begin{xy} 0;<.28pt,0pt>:<0pt,-.28pt>:: 
(444,346) *+{1} ="0",
(439,268) *+{2} ="1",
(445,418) *+{\textcolor{red}{3}} ="2",
(379,340) *+{4} ="3",
(440,171) *+{\textcolor{red}{5}} ="4",
(494,309) *+{6} ="5",
(444,516) *+{\textcolor{blue}{7}} ="6",
(295,307) *+{\textcolor{red}{8}} ="7",
(380,281) *+{9} ="8",
(437,92) *+{\textcolor{blue}{10}} ="9",
(550,241) *+{\textcolor{red}{11}} ="10",
(563,380) *+{\textcolor{red}{12}} ="11",
(451,585) *+{\textcolor{green}{13}} ="12",
(257,380) *+{\textcolor{blue}{14}} ="13",
(90,287) *+{\textcolor{green}{15}} ="14",
(253,201) *+{\textcolor{blue}{16}} ="15",
(443,0) *+{\textcolor{green}{17}} ="16",
(677,113) *+{\textcolor{green}{18}} ="17",
(650,309) *+{\textcolor{blue}{19}} ="18",
(677,550) *+{\textcolor{green}{20}} ="19",
"2", {\ar"0"},
"0", {\ar"3"},
"5", {\ar"0"},
"0", {\ar"11"},
"1", {\ar"4"},
"1", {\ar"5"},
"8", {\ar"1"},
"10", {\ar"1"},
"3", {\ar"2"},
"2", {\ar"6"},
"2", {\ar"7"},
"11", {\ar"2"},
"13", {\ar"2"},
"7", {\ar"3"},
"3", {\ar"8"},
"7", {\ar"4"},
"4", {\ar"8"},
"9", {\ar"4"},
"4", {\ar"10"},
"4", {\ar"15"},
"5", {\ar"10"},
"11", {\ar"5"},
"6", {\ar"11"},
"12", {\ar"6"},
"6", {\ar"13"},
"18", {\ar"6"},
"6", {\ar"19"},
"8", {\ar"7"},
"7", {\ar"13"},
"15", {\ar"7"},
"10", {\ar"9"},
"15", {\ar"9"},
"9", {\ar"16"},
"17", {\ar"9"},
"9", {\ar"18"},
"10", {\ar"11"},
"18", {\ar"10"},
"11", {\ar"18"},
"13", {\ar"12"},
"12", {\ar"14"},
"19", {\ar"12"},
"14", {\ar"13"},
"13", {\ar"15"},
"15", {\ar"14"},
"14", {\ar"16"},
"16", {\ar"15"},
"16", {\ar"17"},
"18", {\ar"17"},
"17", {\ar"19"},
"19", {\ar"18"},
\end{xy}};

\node at (10,0) {\begin{xy} 0;<.38pt,0pt>:<0pt,-.38pt>:: 
(236,221) *+{1} ="0",
(228,142) *+{2} ="1",
(232,283) *+{\textcolor{red}{3}} ="2",
(166,205) *+{4} ="3",
(112,141) *+{\textcolor{red}{5}} ="4",
(281,174) *+{6} ="5",
(392,188) *+{\textcolor{blue}{7}} ="6",
(116,231) *+{\textcolor{red}{8}} ="7",
(170,157) *+{\textcolor{blue}{9}} ="8",
(69,113) *+{\textcolor{blue}{10}} ="9",
(247,90) *+{\textcolor{red}{11}} ="10",
(338,180) *+{\textcolor{red}{12}} ="11",
(446,194) *+{\textcolor{green}{13}} ="12",
(229,334) *+{\textcolor{blue}{14}} ="13",
(221,402) *+{\textcolor{green}{15}} ="14",
(54,263) *+{\textcolor{blue}{16}} ="15",
(0,292) *+{\textcolor{green}{17}} ="16",
(24,83) *+{\textcolor{green}{18}} ="17",
(263,46) *+{\textcolor{blue}{19}} ="18",
(280,0) *+{\textcolor{green}{20}} ="19",
"2", {\ar"0"},
"0", {\ar"3"},
"5", {\ar"0"},
"0", {\ar"11"},
"1", {\ar"4"},
"1", {\ar"5"},
"8", {\ar"1"},
"10", {\ar"1"},
"3", {\ar"2"},
"2", {\ar"6"},
"2", {\ar"7"},
"11", {\ar"2"},
"13", {\ar"2"},
"7", {\ar"3"},
"3", {\ar"8"},
"7", {\ar"4"},
"4", {\ar"8"},
"9", {\ar"4"},
"4", {\ar"10"},
"4", {\ar"15"},
"5", {\ar"10"},
"11", {\ar"5"},
"6", {\ar"11"},
"12", {\ar"6"},
"6", {\ar"13"},
"18", {\ar"6"},
"6", {\ar"19"},
"8", {\ar"7"},
"7", {\ar"13"},
"15", {\ar"7"},
"10", {\ar"9"},
"15", {\ar"9"},
"9", {\ar"16"},
"17", {\ar"9"},
"9", {\ar"18"},
"10", {\ar"11"},
"18", {\ar"10"},
"11", {\ar"18"},
"13", {\ar"12"},
"12", {\ar"14"},
"19", {\ar"12"},
"14", {\ar"13"},
"13", {\ar"15"},
"15", {\ar"14"},
"14", {\ar"16"},
"16", {\ar"15"},
"16", {\ar"17"},
"18", {\ar"17"},
"17", {\ar"19"},
"19", {\ar"18"},
\end{xy}};
\end{tikzpicture}
\caption{From the proof of Lemma~\ref{lem:leadsto}. The bottom quiver is that of $\Sigma_5(\Delta_{2})$. The vertices $1,3,7,13$ resp. $2,5,10,17$ sit on the two internal edges of the triangulation. The upper left quiver is mutation-equivalent to the bottom quiver. It is obtained by gluing the quiver for the reduced word $1324132413$ with that for $2413241324$ along the edges of the triangulation. Rearranging the vertices of this quiver we get the upper right quiver, which is $\Sigma'_5(\Delta_{5})$. The colors indicate weights of $P$-cluster variables.\label{fig:GrisFG}}
\end{figure}
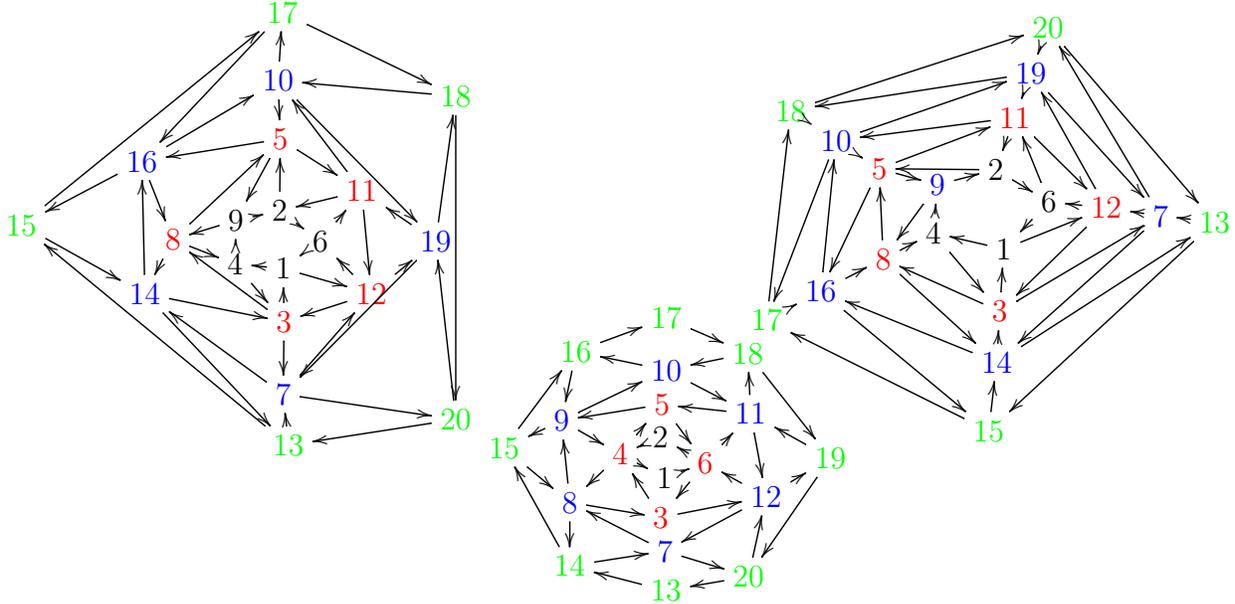

\begin{figure}
\begin{tikzpicture}
\node at (0,0) {
\begin{xy} 0;<.261pt,0pt>:<0pt,-.26pt>:: 
(3,9) *+{1} ="0",
(104,0) *+{2} ="1",
(0,129) *+{3} ="2",
(106,135) *+{4} ="3",
(314,171) *+{5} ="4",
(4,205) *+{6} ="5",
(121,201) *+{7} ="6",
(219,203) *+{8} ="7",
(335,232) *+{9} ="8",
(9,331) *+{10} ="9",
(127,325) *+{11} ="10",
(242,318) *+{12} ="11",
(338,312) *+{13} ="12",
(522,342) *+{14} ="13",
(11,405) *+{15} ="14",
(129,403) *+{16} ="15",
(251,400) *+{17} ="16",
(357,398) *+{18} ="17",
(458,393) *+{19} ="18",
(618,443) *+{20} ="19",
(16,500) *+{21} ="20",
(141,488) *+{22} ="21",
(258,502) *+{23} ="22",
(376,503) *+{24} ="23",
(473,499) *+{25} ="24",
(576,502) *+{26} ="25",
(682,542) *+{27} ="26",
(17,595) *+{28} ="27",
(151,593) *+{29} ="28",
(264,599) *+{30} ="29",
(392,599) *+{31} ="30",
(508,604) *+{32} ="31",
(601,606) *+{33} ="32",
(696,609) *+{34} ="33",
(792,646) *+{35} ="34",
(-80,0) *+{\omega_7},
(-80,100) *+{\omega_6},
(-80,200) *+{\omega_5},
(-80,300) *+{\omega_4},
(-80,400) *+{\omega_3},
(-80,500) *+{\omega_2},
(-80,600) *+{\omega_1},
"2", {\ar"0"},
"0", {\ar"3"},
"1", {\ar"2"},
"3", {\ar"1"},
"3", {\ar"2"},
"2", {\ar"4"},
"5", {\ar"2"},
"2", {\ar"6"},
"4", {\ar"3"},
"6", {\ar"3"},
"3", {\ar"7"},
"4", {\ar"5"},
"7", {\ar"4"},
"6", {\ar"5"},
"5", {\ar"8"},
"9", {\ar"5"},
"5", {\ar"10"},
"7", {\ar"6"},
"10", {\ar"6"},
"6", {\ar"11"},
"8", {\ar"7"},
"11", {\ar"7"},
"7", {\ar"12"},
"8", {\ar"9"},
"12", {\ar"8"},
"10", {\ar"9"},
"9", {\ar"13"},
"14", {\ar"9"},
"9", {\ar"15"},
"11", {\ar"10"},
"15", {\ar"10"},
"10", {\ar"16"},
"12", {\ar"11"},
"16", {\ar"11"},
"11", {\ar"17"},
"13", {\ar"12"},
"17", {\ar"12"},
"12", {\ar"18"},
"13", {\ar"14"},
"18", {\ar"13"},
"15", {\ar"14"},
"14", {\ar"19"},
"20", {\ar"14"},
"14", {\ar"21"},
"16", {\ar"15"},
"21", {\ar"15"},
"15", {\ar"22"},
"17", {\ar"16"},
"22", {\ar"16"},
"16", {\ar"23"},
"18", {\ar"17"},
"23", {\ar"17"},
"17", {\ar"24"},
"19", {\ar"18"},
"24", {\ar"18"},
"18", {\ar"25"},
"19", {\ar"20"},
"25", {\ar"19"},
"21", {\ar"20"},
"20", {\ar"26"},
"27", {\ar"20"},
"20", {\ar"28"},
"22", {\ar"21"},
"28", {\ar"21"},
"21", {\ar"29"},
"23", {\ar"22"},
"29", {\ar"22"},
"22", {\ar"30"},
"24", {\ar"23"},
"30", {\ar"23"},
"23", {\ar"31"},
"25", {\ar"24"},
"31", {\ar"24"},
"24", {\ar"32"},
"26", {\ar"25"},
"32", {\ar"25"},
"25", {\ar"33"},
"26", {\ar"27"},
"33", {\ar"26"},
"28", {\ar"27"},
"27", {\ar"34"},
"29", {\ar"28"},
"30", {\ar"29"},
"31", {\ar"30"},
"32", {\ar"31"},
"33", {\ar"32"},
"34", {\ar"33"},
\end{xy}
};

\node at (7.5,-.1) {\begin{xy} 0;<.25pt,0pt>:<0pt,-.25pt>:: 
(3,9) *+{1} ="0",
(104,0) *+{2} ="1",
(0,129) *+{3} ="2",
(106,135) *+{4} ="3",
(294,171) *+{5} ="4",
(4,205) *+{6} ="5",
(121,201) *+{7} ="6",
(219,203) *+{8} ="7",
(397,232) *+{9} ="8",
(9,321) *+{10} ="9",
(127,325) *+{11} ="10",
(242,318) *+{12} ="11",
(328,312) *+{13} ="12",
(492,332) *+{14} ="13",
(11,405) *+{15} ="14",
(129,403) *+{16} ="15",
(251,400) *+{17} ="16",
(347,398) *+{18} ="17",
(438,393) *+{19} ="18",
(598,443) *+{20} ="19",
(16,500) *+{21} ="20",
(141,488) *+{22} ="21",
(258,502) *+{23} ="22",
(366,503) *+{24} ="23",
(453,499) *+{25} ="24",
(546,502) *+{26} ="25",
(642,532) *+{27} ="26",
(17,595) *+{28} ="27",
(151,593) *+{29} ="28",
(264,599) *+{30} ="29",
(382,599) *+{31} ="30",
(488,604) *+{32} ="31",
(571,606) *+{33} ="32",
(696,669) *+{34} ="33",
"2", {\ar"0"},
"0", {\ar"3"},
"1", {\ar"2"},
"3", {\ar"1"},
"3", {\ar"2"},
"2", {\ar"4"},
"5", {\ar"2"},
"2", {\ar"6"},
"4", {\ar"3"},
"6", {\ar"3"},
"3", {\ar"7"},
"4", {\ar"5"},
"7", {\ar"4"},
"6", {\ar"5"},
"5", {\ar"8"},
"9", {\ar"5"},
"5", {\ar"10"},
"7", {\ar"6"},
"10", {\ar"6"},
"6", {\ar"11"},
"8", {\ar"7"},
"11", {\ar"7"},
"7", {\ar"12"},
"8", {\ar"9"},
"12", {\ar"8"},
"10", {\ar"9"},
"9", {\ar"13"},
"14", {\ar"9"},
"9", {\ar"15"},
"11", {\ar"10"},
"15", {\ar"10"},
"10", {\ar"16"},
"12", {\ar"11"},
"16", {\ar"11"},
"11", {\ar"17"},
"13", {\ar"12"},
"17", {\ar"12"},
"12", {\ar"18"},
"13", {\ar"14"},
"18", {\ar"13"},
"15", {\ar"14"},
"14", {\ar"19"},
"20", {\ar"14"},
"14", {\ar"21"},
"16", {\ar"15"},
"21", {\ar"15"},
"15", {\ar"22"},
"17", {\ar"16"},
"22", {\ar"16"},
"16", {\ar"23"},
"18", {\ar"17"},
"23", {\ar"17"},
"17", {\ar"24"},
"19", {\ar"18"},
"24", {\ar"18"},
"18", {\ar"25"},
"19", {\ar"20"},
"25", {\ar"19"},
"21", {\ar"20"},
"20", {\ar"26"},
"27", {\ar"20"},
"20", {\ar"28"},
"22", {\ar"21"},
"28", {\ar"21"},
"21", {\ar"29"},
"23", {\ar"22"},
"29", {\ar"22"},
"22", {\ar"30"},
"24", {\ar"23"},
"30", {\ar"23"},
"23", {\ar"31"},
"25", {\ar"24"},
"31", {\ar"24"},
"24", {\ar"32"},
"26", {\ar"25"},
"32", {\ar"25"},
"25", {\ar"33"},
"26", {\ar"27"},
"33", {\ar"26"},
"28", {\ar"27"},
"29", {\ar"28"},
"30", {\ar"29"},
"31", {\ar"30"},
"32", {\ar"31"},
"33", {\ar"32"},
"27", {\ar"33"},
\end{xy}};

\node at (-.45,-6) {\begin{xy} 0;<.25pt,0pt>:<0pt,-.25pt>:: 
(0,9) *+{1} ="0",
(104,0) *+{2} ="1",
(0,109) *+{3} ="2",
(106,109) *+{4} ="3",
(254,161) *+{5} ="4",
(0,205) *+{6} ="5",
(121,201) *+{7} ="6",
(219,203) *+{8} ="7",
(357,222) *+{9} ="8",
(0,321) *+{10} ="9",
(127,325) *+{11} ="10",
(242,318) *+{12} ="11",
(328,312) *+{13} ="12",
(482,330) *+{14} ="13",
(0,405) *+{15} ="14",
(129,403) *+{16} ="15",
(251,400) *+{17} ="16",
(347,398) *+{18} ="17",
(438,393) *+{19} ="18",
(578,417) *+{20} ="19",
(0,500) *+{21} ="20",
(141,488) *+{22} ="21",
(258,502) *+{23} ="22",
(366,503) *+{24} ="23",
(453,499) *+{25} ="24",
(576,542) *+{26} ="25",
(0,596) *+{27} ="26",
(143,588) *+{29} ="28",
(267,594) *+{30} ="29",
(382,599) *+{31} ="30",
(474,599) *+{32} ="31",
(599,652) *+{33} ="32",
"2", {\ar"0"},
"0", {\ar"3"},
"1", {\ar"2"},
"3", {\ar"1"},
"3", {\ar"2"},
"2", {\ar"4"},
"5", {\ar"2"},
"2", {\ar"6"},
"4", {\ar"3"},
"6", {\ar"3"},
"3", {\ar"7"},
"4", {\ar"5"},
"7", {\ar"4"},
"6", {\ar"5"},
"5", {\ar"8"},
"9", {\ar"5"},
"5", {\ar"10"},
"7", {\ar"6"},
"10", {\ar"6"},
"6", {\ar"11"},
"8", {\ar"7"},
"11", {\ar"7"},
"7", {\ar"12"},
"8", {\ar"9"},
"12", {\ar"8"},
"10", {\ar"9"},
"9", {\ar"13"},
"14", {\ar"9"},
"9", {\ar"15"},
"11", {\ar"10"},
"15", {\ar"10"},
"10", {\ar"16"},
"12", {\ar"11"},
"16", {\ar"11"},
"11", {\ar"17"},
"13", {\ar"12"},
"17", {\ar"12"},
"12", {\ar"18"},
"13", {\ar"14"},
"18", {\ar"13"},
"15", {\ar"14"},
"14", {\ar"19"},
"20", {\ar"14"},
"14", {\ar"21"},
"16", {\ar"15"},
"21", {\ar"15"},
"15", {\ar"22"},
"17", {\ar"16"},
"22", {\ar"16"},
"16", {\ar"23"},
"18", {\ar"17"},
"23", {\ar"17"},
"17", {\ar"24"},
"19", {\ar"18"},
"24", {\ar"18"},
"18", {\ar"25"},
"19", {\ar"20"},
"25", {\ar"19"},
"21", {\ar"20"},
"20", {\ar"25"},
"26", {\ar"20"},
"20", {\ar"28"},
"22", {\ar"21"},
"28", {\ar"21"},
"21", {\ar"29"},
"23", {\ar"22"},
"29", {\ar"22"},
"22", {\ar"30"},
"24", {\ar"23"},
"30", {\ar"23"},
"23", {\ar"31"},
"25", {\ar"24"},
"31", {\ar"24"},
"24", {\ar"32"},
"25", {\ar"26"},
"32", {\ar"25"},
"28", {\ar"26"},
"26", {\ar"32"},
"29", {\ar"28"},
"30", {\ar"29"},
"31", {\ar"30"},
"32", {\ar"31"},
\end{xy}};

\node at (6.5,-6.3) {\begin{xy} 0;<.25pt,0pt>:<0pt,-.25pt>:: 
(0,9) *+{1} ="0",
(100,0) *+{2} ="1",
(0,109) *+{3} ="2",
(100,109) *+{4} ="3",
(284,161) *+{5} ="4",
(0,205) *+{6} ="5",
(100,201) *+{7} ="6",
(200,203) *+{8} ="7",
(347,222) *+{9} ="8",
(0,321) *+{10} ="9",
(100,325) *+{11} ="10",
(200,318) *+{12} ="11",
(300,312) *+{13} ="12",
(450,330) *+{14} ="13",
(0,405) *+{15} ="14",
(100,403) *+{16} ="15",
(200,400) *+{17} ="16",
(300,398) *+{18} ="17",
(450,433) *+{19} ="18",
(0,477) *+{20} ="19",
(100,583) *+{21} ="20",
(100,488) *+{22} ="21",
(200,502) *+{23} ="22",
(300,503) *+{24} ="23",
(450,539) *+{25} ="24",
(0,574) *+{26} ="25",
(200,597) *+{30} ="29",
(300,599) *+{31} ="30",
(504,639) *+{32} ="31",
"2", {\ar"0"},
"0", {\ar"3"},
"1", {\ar"2"},
"3", {\ar"1"},
"3", {\ar"2"},
"2", {\ar"4"},
"5", {\ar"2"},
"2", {\ar"6"},
"4", {\ar"3"},
"6", {\ar"3"},
"3", {\ar"7"},
"4", {\ar"5"},
"7", {\ar"4"},
"6", {\ar"5"},
"5", {\ar"8"},
"9", {\ar"5"},
"5", {\ar"10"},
"7", {\ar"6"},
"10", {\ar"6"},
"6", {\ar"11"},
"8", {\ar"7"},
"11", {\ar"7"},
"7", {\ar"12"},
"8", {\ar"9"},
"12", {\ar"8"},
"10", {\ar"9"},
"9", {\ar"13"},
"14", {\ar"9"},
"9", {\ar"15"},
"11", {\ar"10"},
"15", {\ar"10"},
"10", {\ar"16"},
"12", {\ar"11"},
"16", {\ar"11"},
"11", {\ar"17"},
"13", {\ar"12"},
"17", {\ar"12"},
"12", {\ar"18"},
"13", {\ar"14"},
"18", {\ar"13"},
"15", {\ar"14"},
"14", {\ar"18"},
"19", {\ar"14"},
"14", {\ar"21"},
"16", {\ar"15"},
"21", {\ar"15"},
"15", {\ar"22"},
"17", {\ar"16"},
"22", {\ar"16"},
"16", {\ar"23"},
"18", {\ar"17"},
"23", {\ar"17"},
"17", {\ar"24"},
"18", {\ar"19"},
"24", {\ar"18"},
"19", {\ar"20"},
"21", {\ar"19"},
"19", {\ar"24"},
"25", {\ar"19"},
"20", {\ar"21"},
"20", {\ar"25"},
"29", {\ar"20"},
"22", {\ar"21"},
"21", {\ar"29"},
"23", {\ar"22"},
"29", {\ar"22"},
"22", {\ar"30"},
"24", {\ar"23"},
"23", {\ar"31"},
"24", {\ar"25"},
"31", {\ar"24"},
"25", {\ar"31"},
"30", {\ar"29"},
"31", {\ar"30"},
\end{xy}};

\node at (-1.3,-12) {\begin{xy} 0;<.25pt,0pt>:<0pt,-.25pt>:: 
(0,9) *+{1} ="0",
(106,0) *+{2} ="1",
(0,109) *+{3} ="2",
(106,109) *+{4} ="3",
(286,161) *+{5} ="4",
(0,205) *+{6} ="5",
(106,201) *+{7} ="6",
(211,203) *+{8} ="7",
(369,222) *+{9} ="8",
(0,311) *+{10} ="9",
(106,311) *+{11} ="10",
(211,311) *+{12} ="11",
(370,332) *+{13} ="12",
(0,413) *+{14} ="13",
(106,492) *+{15} ="14",
(106,403) *+{16} ="15",
(211,400) *+{17} ="16",
(370,428) *+{18} ="17",
(0,497) *+{19} ="18",
(106,583) *+{20} ="19",
(211,583) *+{22} ="21",
(211,502) *+{23} ="22",
(370,523) *+{24} ="23",
(0,583) *+{25} ="24",
(370,623) *+{31} ="30",
"2", {\ar"0"},
"0", {\ar"3"},
"1", {\ar"2"},
"3", {\ar"1"},
"3", {\ar"2"},
"2", {\ar"4"},
"5", {\ar"2"},
"2", {\ar"6"},
"4", {\ar"3"},
"6", {\ar"3"},
"3", {\ar"7"},
"4", {\ar"5"},
"7", {\ar"4"},
"6", {\ar"5"},
"5", {\ar"8"},
"9", {\ar"5"},
"5", {\ar"10"},
"7", {\ar"6"},
"10", {\ar"6"},
"6", {\ar"11"},
"8", {\ar"7"},
"11", {\ar"7"},
"7", {\ar"12"},
"8", {\ar"9"},
"12", {\ar"8"},
"10", {\ar"9"},
"9", {\ar"12"},
"13", {\ar"9"},
"9", {\ar"15"},
"11", {\ar"10"},
"15", {\ar"10"},
"10", {\ar"16"},
"12", {\ar"11"},
"16", {\ar"11"},
"11", {\ar"17"},
"12", {\ar"13"},
"17", {\ar"12"},
"13", {\ar"14"},
"15", {\ar"13"},
"13", {\ar"17"},
"18", {\ar"13"},
"14", {\ar"15"},
"14", {\ar"18"},
"19", {\ar"14"},
"14", {\ar"21"},
"22", {\ar"14"},
"16", {\ar"15"},
"15", {\ar"22"},
"17", {\ar"16"},
"22", {\ar"16"},
"16", {\ar"23"},
"17", {\ar"18"},
"23", {\ar"17"},
"18", {\ar"19"},
"18", {\ar"23"},
"24", {\ar"18"},
"21", {\ar"19"},
"19", {\ar"24"},
"21", {\ar"22"},
"30", {\ar"21"},
"23", {\ar"22"},
"22", {\ar"30"},
"23", {\ar"24"},
"30", {\ar"23"},
"24", {\ar"30"},
\end{xy}};

\node at (4.1,-12.2) {\begin{xy} 0;<.25pt,0pt>:<0pt,-.25pt>:: 
(0,9) *+{1} ="0",
(113,0) *+{2} ="1",
(0,109) *+{3} ="2",
(113,109) *+{4} ="3",
(250,143) *+{5} ="4",
(0,205) *+{6} ="5",
(113,201) *+{7} ="6",
(250,243) *+{8} ="7",
(0,305) *+{9} ="8",
(113,386) *+{10} ="9",
(113,305) *+{11} ="10",
(250,358) *+{12} ="11",
(0,398) *+{13} ="12",
(113,482) *+{14} ="13",
(250,607) *+{15} ="14",
(250,532) *+{16} ="15",
(250,438) *+{17} ="16",
(0,488) *+{18} ="17",
(113,573) *+{19} ="18",
(0,573) *+{23} ="22",
"2", {\ar"0"},
"0", {\ar"3"},
"1", {\ar"2"},
"3", {\ar"1"},
"3", {\ar"2"},
"2", {\ar"4"},
"5", {\ar"2"},
"2", {\ar"6"},
"4", {\ar"3"},
"6", {\ar"3"},
"3", {\ar"7"},
"4", {\ar"5"},
"7", {\ar"4"},
"6", {\ar"5"},
"5", {\ar"7"},
"8", {\ar"5"},
"5", {\ar"10"},
"7", {\ar"6"},
"10", {\ar"6"},
"6", {\ar"11"},
"7", {\ar"8"},
"11", {\ar"7"},
"8", {\ar"9"},
"10", {\ar"8"},
"8", {\ar"11"},
"12", {\ar"8"},
"9", {\ar"10"},
"9", {\ar"12"},
"13", {\ar"9"},
"9", {\ar"15"},
"16", {\ar"9"},
"11", {\ar"10"},
"10", {\ar"16"},
"11", {\ar"12"},
"16", {\ar"11"},
"12", {\ar"13"},
"12", {\ar"16"},
"17", {\ar"12"},
"13", {\ar"14"},
"15", {\ar"13"},
"13", {\ar"17"},
"18", {\ar"13"},
"14", {\ar"15"},
"14", {\ar"18"},
"22", {\ar"14"},
"15", {\ar"16"},
"17", {\ar"15"},
"15", {\ar"22"},
"16", {\ar"17"},
"17", {\ar"18"},
"22", {\ar"17"},
"18", {\ar"22"},
\end{xy}};

\node at (8.6,-12.4) {\begin{xy} 0;<.25pt,0pt>:<0pt,-.25pt>:: 
(0,0) *+{1} ="0",
(100,0) *+{2} ="1",
(0,100) *+{3} ="2",
(100,100) *+{4} ="3",
(0,200) *+{5} ="4",
(100,300) *+{6} ="5",
(100,200) *+{7} ="6",
(0,300) *+{8} ="7",
(100,400) *+{9} ="8",
(0,500) *+{10} ="9",
(0,400) *+{11} ="10",
(100,500) *+{12} ="11",
(0,600) *+{13} ="12",
(100,600) *+{17} ="16",
"2", {\ar"0"},
"0", {\ar"3"},
"1", {\ar"2"},
"3", {\ar"1"},
"4", {\ar"2"},
"2", {\ar"6"},
"3", {\ar"4"},
"6", {\ar"3"},
"4", {\ar"5"},
"7", {\ar"4"},
"5", {\ar"6"},
"8", {\ar"5"},
"5", {\ar"10"},
"6", {\ar"7"},
"7", {\ar"8"},
"10", {\ar"7"},
"8", {\ar"9"},
"11", {\ar"8"},
"9", {\ar"10"},
"12", {\ar"9"},
"9", {\ar"16"},
"10", {\ar"11"},
"11", {\ar"12"},
"16", {\ar"11"},
\end{xy}};

\node at (3.5,0) {$\overset{\mu_{35}}{\leadsto}$};
\node at (-3.9,-5.75) {$\overset{\mu_{27,34}}{\hspace{.01cm}}$};
\node at (-3.9,-6.1) {$\overset{\mu_{28}}{\leadsto}$};

\node at (2.85,-5.4) {$\overset{\mu_{20,26,33}}{\hspace{.01cm}}$};
\node at (2.85,-5.75) {$\overset{\mu_{21,27}}{\hspace{.01cm}}$};
\node at (2.85,-6.1) {$\overset{\mu_{29}}{\leadsto}$};

\node at (9.15,-5.25) {$\overset{\mu_{14,19,25,32}}{\hspace{.01cm}}$};
\node at (9.15,-5.6) {$\overset{\mu_{15,20,26}}{\hspace{.01cm}}$};
\node at (9.15,-5.95) {$\overset{\mu_{22,21}}{\hspace{.01cm}}$};
\node at (9.15,-6.3) {$\overset{\mu_{30}}{\leadsto}$};

\node at (1.55,-11.85) {$\overset{\mu_{9,13,18,24,31}}{\hspace{.01cm}}$};
\node at (1.55,-12.10) {$\overset{\mu_{10,14,19,25}}{\hspace{.01cm}}$};
\node at (1.55,-12.35) {$\overset{\mu_{16,15,20}}{\hspace{.01cm}}$};
\node at (1.55,-12.6) {$\overset{\mu_{23,22}}{\hspace{.01cm}}$};
\node at (1.55,-12.85) {$\overset{\mu_{24}}{\leadsto}$};

\node at (6.65,-11.45) {$\overset{\mu_{5,8,12,17,16,15}}{\hspace{.01cm}}$};
\node at (6.65,-11.8) {$\overset{\mu_{6,9,13,18,23}}{\hspace{.01cm}}$};
\node at (6.65,-12.15) {$\overset{\mu_{11,10,14,19}}{\hspace{.01cm}}$};
\node at (6.65,-12.5) {$\overset{\mu_{12,17,16}}{\hspace{.01cm}}$};
\node at (6.65,-12.85) {$\overset{\mu_{13,18}}{\hspace{.01cm}}$};
\node at (6.65,-13.2) {$\overset{\mu_{14}}{\leadsto}$};
\end{tikzpicture}
\caption{From the proof of Lemma~\ref{lem:leadsto}. The mutations and vertex deletions proving that $\Sigma_8^s \underset{\sim}{\coprod} \Sigma_8^1  \leadsto \Sigma_8^{s-1} \underset{\sim}{\coprod} \Sigma_8^1$ for $s \in [2,7]$. 
\label{fig:stos--}}
\end{figure}

\begin{lemma}\label{lem:ngoncase} Theorem 7.3 holds when $\mathcal{M} = \mathcal{A}_{{\rm SL}_k,D_{m,1}}$ for any~$m \geq 2$.
\end{lemma}

\begin{proof}
Cyclically number the arcs in the taut triangulation $\Delta_m$ of $D_{m,1}$ as $1,\dots,m$. Flipping at arc 1 yields a new arc which is a boundary arch contractible to two boundary intervals. After performing this flip, we see a copy of $\Delta_{m-1}$ surrounding the puncture plus an extra ``exterior triangle'' not incident to the puncture. We will argue momentarily that we can simulate this flip of triangulations at the level of $P$-seeds in such a way that the $P$-cluster variables in the exterior triangle have zero weight and can be ignored. Deleting these variables of zero weight, the resulting $P$-seed is an $m-1$-cycle consisting of $m-2$ copies of $Q_k^1$ followed by one copy of $Q_k^2$ (corresponding to the triangle which has the flipped arc 1), with consecutive fragments glued along shared edges. 

We next mutate at arc $2$ to get $m-2$ triangles surrounding the puncture, with the new arc a boundary arch contractible to three boundary intervals. The corresponding $P$-seed (after deleting vertices of zero weight) is an $m-2$-cycle consisting of $m-3$ copies of $Q_k^1$ followed by one copy of $Q_k^3$. We continue flipping in this way, flipping arc $3,\dots,m-2$, until the punctures sits in a once-punctured digon. One of the triangles surrounding the puncture has an exposed side (corresponding to the seed fragment $\Sigma_k^1$) and the other has a side which is contractible to $m-1$ boundary intervals (corresponding to the seed fragment $\Sigma_k^{\min(m-1,k-1)}$). In any case, the seed fragment is $\Sigma_k^s \underset{\sim}{\coprod}\Sigma_k^1$ for some $s$. Since $\Sigma_k^s \underset{\sim}{\coprod}\Sigma_k^1 \leadsto \Sigma_k^1 \underset{\sim}{\coprod}\Sigma_k^1$, the claim follows from Lemma~\ref{lem:digoncase}. 

What remains is to explain the mutations of $P$-cluster variables which realize the flips. We start with two adajacent $Q_k^1$'s and claim that this $\leadsto Q_k^2$. Then we glue a copy of $Q_k^1$ to this $Q_k^2$ (simulating adding the next triangle), and claim that this 
$\leadsto Q_k^3$. The argument is indicated in Figure~\ref{fig:fliptriangle} in the case of $Q_5^1$ glued to $Q_5^2$. In general, one should mutate down all of the ``interior'' diagonals from top to bottom, starting with the leftmost interior diagonal and ending with the rightmost diagonal. One should slide a variable one unit down its diagonal once it is mutated, and delete vertices of weight zero. The result in the above example is $Q_5^3$, as in the figure. 
\end{proof}

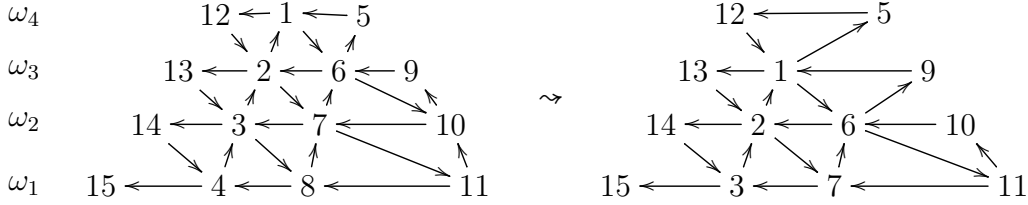
\begin{figure}
\begin{tikzpicture}
\node at (-3.5,1.1) {$\omega_4$};
\node at (-3.5,.4) {$\omega_3$};
\node at (-3.5,-.3) {$\omega_2$};
\node at (-3.5,-1.15) {$\omega_1$};

\node at (0,0) 
{\begin{xy} 0;<.45pt,0pt>:<0pt,-.45pt>:: 
(156,0) *+{1} ="0",
(137,49) *+{2} ="1",
(116,94) *+{3} ="2",
(99,146) *+{4} ="3",
(221,3) *+{5} ="4",
(200,49) *+{6} ="5",
(184,94) *+{7} ="6",
(174,146) *+{8} ="7",
(261,49) *+{9} ="8",
(294,94) *+{10} ="9",
(314,146) *+{11} ="10",
(96,2) *+{12} ="11",
(65,49) *+{13} ="12",
(38,94) *+{14} ="13",
(0,146) *+{15} ="14",
"1", {\ar"0"},
"4", {\ar"0"},
"0", {\ar"5"},
"0", {\ar"11"},
"2", {\ar"1"},
"5", {\ar"1"},
"1", {\ar"6"},
"11", {\ar"1"},
"1", {\ar"12"},
"3", {\ar"2"},
"6", {\ar"2"},
"2", {\ar"7"},
"12", {\ar"2"},
"2", {\ar"13"},
"7", {\ar"3"},
"13", {\ar"3"},
"3", {\ar"14"},
"5", {\ar"4"},
"6", {\ar"5"},
"8", {\ar"5"},
"5", {\ar"9"},
"7", {\ar"6"},
"9", {\ar"6"},
"6", {\ar"10"},
"10", {\ar"7"},
"9", {\ar"8"},
"10", {\ar"9"},
\end{xy}};
\node at (3.5,0) {$\leadsto$};
\node at (7,0){\begin{xy} 0;<.45pt,0pt>:<0pt,-.45pt>:: 
(139,49) *+{1} ="0",
(120,94) *+{2} ="1",
(102,146) *+{3} ="2",
(226,0) *+{5} ="4",
(196,94) *+{6} ="5",
(184,146) *+{7} ="6",
(263,49) *+{9} ="8",
(290,94) *+{10} ="9",
(334,146) *+{11} ="10",
(96,1) *+{12} ="11",
(65,49) *+{13} ="12",
(38,94) *+{14} ="13",
(0,146) *+{15} ="14",
"1", {\ar"0"},
"0", {\ar"4"},
"0", {\ar"5"},
"8", {\ar"0"},
"11", {\ar"0"},
"0", {\ar"12"},
"2", {\ar"1"},
"5", {\ar"1"},
"1", {\ar"6"},
"12", {\ar"1"},
"1", {\ar"13"},
"6", {\ar"2"},
"13", {\ar"2"},
"2", {\ar"14"},
"4", {\ar"11"},
"6", {\ar"5"},
"5", {\ar"8"},
"9", {\ar"5"},
"5", {\ar"10"},
"10", {\ar"6"},
"10", {\ar"9"},
\end{xy}
};
\end{tikzpicture}
\caption{From the proof of Lemma~\ref{lem:ngoncase}. Vertices $1,\dots,11$ in the left figure form a copy of $Q_5^2$. We glue a copy of $Q_5^1$ to this quiver on the left with vertices $1,2,3,4$ the glued vertices. Mutating at $1,2,3,4,6,7,8$ and deleting vertices $4,8$ (which have zero weight) yields the right quiver, which is $Q_5^3$. 
\label{fig:fliptriangle}
}
\end{figure}

\begin{proof}[Proof of part (1) of Theorem~\ref{thm:edgecontraction}]
By Lemmas~\ref{lem:adjacent} and \ref{lem:adjacent2}, the map \eqref{eq:Pclustertodosp} is an edge contraction onto some subgraph of $\Box_{\mathbb{M}_\circ} H_{\rm dosp}$. We need to argue the surjectivity on edges (hence on vertices). Since we are arguing part (1) of the theorem, $\mathbf{S}$ has boundary.

If there are at least two boundary components, then it is possible to choose a regular triangulation of $\mathbb{S}$ in which every puncture $p \in \mathbb{M}_\circ$ sits inside a once-punctured digon, with the two endpoints of the digon residing on different boundary components of $\mathbf{S}$. In particular, the two sides of such digon are not boundary arches, so that the seed near $p$ is that of $\Sigma_k(\Delta_2)$ (in either version of the moduli space). Since $\Sigma_k(\Delta_2) \leadsto \Sigma'_k(\Delta_k)$, 
the claim follows from Lemma~\ref{lem:ngoncase}. 

The final case to consider is that there is only one boundary component but more than one puncture. As in the theorem statement, we also assume that the boundary carries at least two boundary points. Then we can enclose all of the punctures inside once-punctured digons as before (this fails if there is only one boundary point, which is why we exclude it from Theorem~\ref{thm:edgecontraction}). For the Fock-Goncharov case, the preceding argument works. For the Grassmannian case, the leftmost and rightmost punctures will be enclosed in digons which have an exposed triangle. The corresponding seed will be $\Sigma_k^{k-1} \coprod \Sigma_k^1/ \sim$ and the claim follows from Lemma~\ref{lem:leadsto} and the digon case Lemma~\ref{lem:digoncase}. 
\end{proof}

\section{Tensor diagrams}\label{secn:Webs}
We provide background on tensor diagrams on a surface, skein relations, and arborization moves following \cite{CKM,FPII,FP}. Then we introduce tagged and pseudotagged diagrams and state our cluster combinatorics conjectures. 
 
\subsection{Plain tensor diagrams}\label{subsec:plaindiagrams}
Let $\mathcal{M}$ be one of the moduli spaces considered in this paper, i.e. either $\mathcal{M} = \mathcal{A}_{{\rm SL}_k,\mathbb{S}}$ or $\mathcal{M} = \mathcal{A}'_{{\rm SL}_k,\mathbb{S}}$. 
\begin{definition}\label{defn:tensordiagram}
An $\mathcal{M}$-{\sl tensor diagram} is an immersed edge-labeled acylic directed graph $T = (V(T),E(T)) \subset \mathbf{S}$ satisfying the following: 
\begin{itemize}
\item Every marked point is a vertex of $T$ and is a source vertex of $T$. The remaining vertices reside in ${\rm int} \, \mathbf{S} \setminus \mathbb{M}_\circ$, and each of these either has a unique outgoing or unique incoming edge. 
\item Edges have their endpoints at vertices but otherwise do not intersect $\mathbb{M}$ nor $\partial \mathbf{S}$. Any self-crossings in $T$ consist of two edges crossing transversely and there are finitely many of these. 
\item Every edge $e \in E(T)$ carries an integer label ${\rm wt}(e) \in [k]$. When $\mathcal{M} = \mathcal{A}'_{{\rm SL}_k, \mathbb{S}}$ we require that edges $e$ incident to boundary points satisfy ${\rm wt}(e) = 1$.
\item The local picture surrounding any vertex which is not a marked point is one of the following four pictures, each of which depicts such a vertex and the oriented and labeled half-edges incident to it: 
\begin{equation}
\label{eq:fundamentalmorphisms}
\begin{tikzpicture}[scale = 1]
\draw [thick,decoration={markings,mark=at position 1 with {\arrow[scale=1.7]{>}}},
    postaction={decorate},
    shorten >=0.4pt] (0,0)--(0,.6);
\draw [thick] (0,.6)--(0,1.2);
\draw [thick,decoration={markings,mark=at position 1 with {\arrow[scale=1.7]{>}}},
    postaction={decorate},
    shorten >=0.4pt] (-1.2,-1.2)--(-.6,-.6);
\draw [thick] (-.6,-.6)--(0,0);
\draw [thick,decoration={markings,mark=at position 1 with {\arrow[scale=1.7]{>}}},
    postaction={decorate},
    shorten >=0.4pt] (-.6,-1.2)--(-.3,-.6);
\draw [thick] (-.3,-.6)--(0,0);
\draw [thick,decoration={markings,mark=at position 1 with {\arrow[scale=1.7]{>}}},
    postaction={decorate},
    shorten >=0.4pt] (.6,-1.2)--(.3,-.6);
\draw [thick] (.3,-.6)--(0,0);
\draw [thick,decoration={markings,mark=at position 1 with {\arrow[scale=1.7]{>}}},
    postaction={decorate},
    shorten >=0.4pt] (1.2,-1.2)--(.6,-.6);
\draw [thick] (.6,-.6)--(0,0);

\node at (-1.25,-1.4) {$a_1$};
\node at (-.65,-1.4) {$a_2$};
\node at (0,-1.4) {$\dots$};
\node at (0,-.85) {$...$};
\node at (.65,-1.4) {$a_{m-1}$};
\node at (1.35,-1.4) {$a_m$};
\node at (1.1,.7) {$a_1+\cdots+ a_m$};
\draw [fill= white] (0,0) circle [radius = .1];
\node at (0,-2) {\small wedge product};

\begin{scope}[xshift = 4.2cm,yshift = -.5cm]
\draw [thick,decoration={markings,mark=at position 1 with {\arrow[scale=1.7]{>}}},
    postaction={decorate},
    shorten >=0.4pt] (0,-1.2)--(0,-.6);
\draw [thick] (0,-.6)--(0,0);
\draw [thick] (-1.2,1.2)--(-.6,.6);
\draw [thick,decoration={markings,mark=at position 1 with {\arrow[scale=1.7]{>}}},
    postaction={decorate},
    shorten >=0.4pt] (0,0)--(-.6,.6);
\draw [thick] (-.6,1.2)--(-.3,.6);
\draw [thick,decoration={markings,mark=at position 1 with {\arrow[scale=1.7]{>}}},
    postaction={decorate},
    shorten >=0.4pt] (0,0)--(-.3,.6);
\draw [thick] (.6,1.2)--(.3,.6);
\draw [thick,decoration={markings,mark=at position 1 with {\arrow[scale=1.7]{>}}},
    postaction={decorate},
    shorten >=0.4pt] (0,0)--(.3,.6);
\draw [thick] (1.2,1.2)--(.6,.6);
\draw [thick,decoration={markings,mark=at position 1 with {\arrow[scale=1.7]{>}}},
    postaction={decorate},
    shorten >=0.4pt] (0,0)--(.6,.6);

\node at (-1.25,1.4) {$a_1$};
\node at (-.65,1.4) {$a_2$};
\node at (0,1.4) {$\dots$};
\node at (0,.85) {$...$};
\node at (.65,1.4) {$a_{m-1}$};
\node at (1.35,1.4) {$a_m$};
\node at (1.1,-.4) {$a_1+\cdots +a_m$};
\draw [fill= black] (0,0) circle [radius = .1];
\node at (0,-1.5) {\small dual wedge product};
\end{scope}
\begin{scope}[xshift = 8.0cm]
\draw [thick,decoration={markings,mark=at position 1 with {\arrow[scale=1.7]{>}}},
    postaction={decorate},
    shorten >=0.4pt] (0,-1.5)--(0,-.9);
\draw [thick] (0,-.9)--(0,-.3);
\draw [fill= white] (0,-1.5) circle [radius = .1];
\node at (.5,-1.0) {$k$};
\node at (0,-2) {\small coevaluation leaf};
\end{scope}
\begin{scope}[xshift = 11.5cm]
\draw [thick,decoration={markings,mark=at position 1 with {\arrow[scale=1.7]{>}}},
    postaction={decorate},
    shorten >=0.4pt] (0,-1.5)--(0,-.9);
\draw [thick] (0,-.9)--(0,-.3);
\draw [fill= black] (0,-.3) circle [radius = .1];
\node at (.5,-1.0) {$k$};
\node at (0,-2) {\small evaluation leaf};
\end{scope}
\end{tikzpicture}.
\end{equation}
\end{itemize}
We consider such tensor diagrams $T$ up to $\mathbb{S}$-isotopy fixing the marked points. 
\end{definition}

We  use the adjectives {\sl plain} or {\sl underlying} to refer to the diagrams just defined when contrasting them with the {\sl tagged} and {\sl psuedotagged} diagrams to be defined below. The extra words wedge product, etc, appearing in \eqref{eq:fundamentalmorphisms} are explained in the next section.

We use the notations $s(e)$ and $t(e)$ for the source and target of an oriented edge $e$ in a tensor diagram. We color vertices white or black according to whether they have a unique outgoing or incoming edge. Every coevaluation leaf has an outgoing edge whose endpoint is a black vertex and every evaluation leaf has an incoming edge whose source is a white vertex. We will depict marked points as black vertices (and as a black-white vertex pair when $k=3$); these vertices have no incoming edges and may have any number of outgoing edges. We refer to the first two types of vertices 
in \eqref{eq:fundamentalmorphisms} merely as {\sl interior vertices}, using always the special word ``leaf'' for the third and fourth types of vertices. At interior vertices, the sum of incoming labels equals the sum of outgoing labels, but this property fails at leaves. This latter failure is a combinatorial encoding of the identification $\bigwedge^k(V) \cong \mathbb{C}$ determined by the chosen volume form $\xi$, as we explain in more detail in the next section. 

In examples, we omit the integer label for edges $e$ satisfying ${\rm wt}(e) = 1$. We depict evaluation leaves and coevaluation leaves, with their incident edge, as tiny ``hairs'' emanating from an interior vertex. See \eqref{eq:Plucker} for an example of a hair which is drawn as a standin for an evaluation leaf. These hairs are called tags in \cite{CKM} but we avoid this technology as it clashes with the meaning of tagging arcs in the sense of Fomin, Shapiro, and Thurston. 

In our tensor diagrammatical formulas below (e.g. in \eqref{eq:crossingremoval}), edges labeled by $0$ should be deleted and any tensor diagram with at least one label in $\mathbb{Z} \setminus [0,k]$ should be discarded or treated as zero in our formulas. An edge of weight~$k$ can be chopped in thirds with the middle third discarded, creating an evaluation leaf at the source of the edge and a coevaluation leaf at the target of the edge.

\begin{definition}
The {\sl legs} of a tensor diagram $T$ are its edges incident to punctures. We denote these by ${\rm Legs}(T)$ and denote by  ${\rm Legs}_p(T)$ those legs incident to a given puncture~$p$. We define 
\begin{equation}\label{eq:weightofT}
{\rm wt}_p(T) := \sum_{e \in {\rm  Legs}_p(T)} \omega_{{\rm wt}(e)} \in P,
\end{equation}
and define similarly ${\rm wt}(T) \in P^{\mathbb{M}_\circ}$.
\end{definition}
We give no special name to edges emanating from boundary points.

The weight vector $\lambda := {\rm wt}_p(T)$ always has weakly decreasing coordinates. We can translate it in the $(1,\dots,1)$ direction until $\lambda_k = 0$. Once this is done, we have the following useful translation between legs at $p$ and the weight of $[T]$ at $p$:
\begin{equation}\label{eq:nooflegs}
(\text{number of legs of weight $a$ of $T$ at $p$}) = \lambda_a-\lambda_{a+1}.
\end{equation}

\begin{example}\label{eg:SL2plainarcs}
An arc $\gamma \colon [0,1] \to \mathbb{S}$ with endpoints $p,q \in \mathbb{M}$ determines an ${\rm SL}_2$ tensor diagram~$T(\gamma)$ described as follows. The diagram has an oriented edge $\gamma(0)\to \gamma(\frac{1}{2})$ and an edge $\gamma(1)\to \gamma(\frac{1}{2})$, both traveling along~$\gamma$. These edges meet up at a white vertex at which has an outgoing edge of weight 2 ending at an evaluation leaf. 
There are two directions in which this outgoing edge can point, but these only change our recipes below by a sign. We suppress this sign throughout.

In a similar way, any simple closed curve $\alpha$ gives rise to a tensor diagram $T(\alpha)$ consisting of  a coevaluation leaf edge whose target is a black vertex at $\alpha(0)$. This black vertex has two outgoing edges, one oriented along $\alpha$ from time 0 to time $\frac{1}{2}$ and the other from time $1$ to time $\frac{1}{2}$. These two oriented edges meet at a white vertex at $\alpha(\frac{1}{2})$ whose outgoing edge is an evaluation leaf edge.  
\end{example}

\subsection{Plain tensor diagrams as functions}\label{subsec:plaininvariants}
The four types of interior vertices in tensor diagrams -- namely white and black interior vertices and the two types of leaves -- encode certain ${\rm SL}(V)$-equivariant morphisms between tensor products of exterior powers of $V$. Tensor diagrams are a graphical way of encoding compositions of these maps. 

We elaborate on this point now. White interior vertices encode the exterior product morphism $\wedge \colon \bigotimes_{j=1}^m \bigwedge^{a_j}V \to \bigwedge^{\sum_j a_j}V$. An evaluation leaf encodes the isomorphism $\bigwedge^k\mathbb{C}^k \overset{\cong}{\to} \mathbb{C}$ and a coevaluation leaf encodes its inverse $\mathbb{C} \overset{\cong}{\to} \bigwedge^k\mathbb{C}^k $. Black vertices encode the dual exterior product morphism
\begin{align}
\bigwedge^{\sum_j a_j}V &\to \bigwedge^{a_1}V \otimes \cdots \otimes  \bigwedge^{a_m}V \label{eq:dualexterior}\\
x_1 \wedge \cdots \wedge x_b &\mapsto \sum \pm (x_{i_1} \wedge x_{i_2} \wedge \cdots \wedge x_{i_{a_m}}) \otimes \cdots \otimes  (x_{i_{b-a_1+1}} \wedge x_{i_{b-a_1+2}} \wedge \cdots \wedge x_{i_{b}}) 
\end{align}
where the summation is over permutations $(i_1,i_2,\ldots,i_b)$ of $b := a_1+\cdots +a_m$ which are increasing in each block, i.e. which satisfy $i_1 < \cdots < i_{a_1}$, $i_{a_1+1} < \cdots < i_{a_1+a_2}$, etc. The sign $\pm$ is the sign of the permutation $(i_1,i_2,\ldots,i_b)$, multiplied by the ``global sign'' of the permutation $(b-a_1+1,\dots,b,\dots,1,\dots,a_m)$. Thus, there is ``no sign'' associated with shuffling the first $a_m$ vectors to the last tensor factor, shuffling the next $a_{m-1}$ vectors to the penultimate factor, and so on.

Given a tensor diagram $T$, we have sets of source and sink edges 
$$\overset{\to}{E}(T) = \{e \colon s(e) \in \mathbb{M} \text{ or is a coevaluation leaf}\} \text{ and } \overset{\leftarrow}{E}(T) = \{e \colon t(e) \text{ is an evaluation leaf}\}.$$
They determine spaces 
$$\overset{\to}{W}(T):= \otimes_{e \in \overset{\to}{E}(T)}\bigwedge^{{\rm wt}(e)} V_{s(e)} \text{ and } \overset{\leftarrow}{W}(T):= \otimes_{e \in \overset{\leftarrow}{E }(T)}\bigwedge^{{\rm wt}(e)} V_{t(e)}.$$ The second of these spaces has a canonical isomorphism with $\mathbb{C}$ given by the tensor product of the evaluation maps. 

A choice of decorated local system $z \in \mathcal{M}$ provides us with a tensor $v_e(z) \in 
\bigwedge^{{\rm wt}(e)} V_{s(e)}$ for each $e \in \overset{\to}{E}(T)$: one should take the volume form $\chi$ at each coevaluation leaf and take the decoration (either the vector assigned to boundary points or the ${\rm wt}(e)$th step of the chosen affine flag) which the chosen~$z$ assigns at marked points. These give rise to  a tensor $v_{{\rm } initial}(z) := \otimes_{e \in \overset{\to }{E}(T)} v_e(z) \in \overset{\to}{W}(T)$.

The local system also provides us with parallel transport isomorphisms $\Xi_e \colon V_{s(e)} \to V_{t(e)}$, hence an isomorphism $\bigwedge^{{\rm wt}(e)}V_{s(e)} \to \bigwedge^{{\rm wt}(e)}V_{t(e)}$. Combining these isomorphisms with the exterior product, dual exterior product maps, and evaluation maps, the chosen $z$ determines a map $\overset{\to}{W}(T) \to \mathbb{C}$. 

\begin{definition}\label{defn:evaluation}
The invariant $[T]$ encoded by an $\mathcal{M}$-tensor diagram $T$ is the function on $\mathcal{M}$ whose value on a point $z \in \mathcal{M}$ is the image of the vector $v_{\rm initial}(z) \in \overset{\to}{W}(T)$ under the morphism $\overset{\to}{W}(T) \to \mathbb{C}$ encoded by $z$.
\end{definition}

It can be seen that any invariant $[T]$ is a regular function on $\mathcal{M}$. The word invariant is used because regular functions on the moduli space are the $G$-invariant regular functions on the space of decorated local systems. 

The weight of a tensor diagram, defined combinatorially in \eqref{eq:weightofT}, matches the weight of its corresponding invariant $[T]$ with respect to the right $T$-action at punctures \eqref{eq:Taction}. In particular, if we know the weight of the invariant $[T]$ then we can recover how many legs of each weight the tensor diagram $T$ has at each puncture using the translation \eqref{eq:nooflegs}.

In our algebraic calculations in Section~\ref{secn:flattening}, we typically forego the dual exterior product symbol and work with a related algebraic construct. Namely, combining the dual exterior product map with an evaluation map, we have a map 
$\bigwedge^{k-a}V \otimes \bigwedge^{k-b}V \to \bigwedge^{k-a-b}V \otimes \bigwedge^bV \otimes \bigwedge^{k-b}V \to \bigwedge^{k-a-b}V \otimes \bigwedge^kV \cong \bigwedge^{k-a-b}V$. We denote this composition by 
\begin{equation}\label{eq:GrassmannCayley}
\cap \colon \bigwedge^{k-a}V \otimes \bigwedge^{k-b}V \to \bigwedge^{k-a-b}V. 
\end{equation}
One computes $\cap$ by shuffling $b$ vectors from the first tensor factor over to the second factor (for a total of $b$ plus $k-b$ vectors), taking the determinant of the $k$ vectors in the second factor after such shuffle, and using it to rescale the $k-a-b$ leftover vectors in the first factor. The value of $\cap$ is the signed sum of such terms. The operation $\cap$ is associative and is commutative up to a predictable sign; see \cite[Section 3.3]{Sturmfels} for an exposition. 

\begin{remark}
One can make Definition~\ref{defn:evaluation} more explicit: the value of $[T]$ on $z \in \mathcal{M}$ can be computed as a signed sum of complex numbers. Each term in this sum corresponds to a process of ``flowing'' the decorations along edges using the parallel transport isomorphisms, applying the exterior product map at white vertices and splitting up the incoming tensor at black vertices using the right hand side of  
\eqref{eq:dualexterior}. The terms in this sum correspond to a choice of term from the right hand side of \eqref{eq:dualexterior} at each black vertex. The sign of a term is its product of signs from the various terms at the black vertices, and the complex number associated with a term is the product of the evaluation maps at the evaluation leaves. For more details on this perspective when $k=3$ see \cite[Section 4]{FPII}.
\end{remark}

\begin{example}\label{eg:determinantsasdiagrams} As a simplest example of the invariant by a tensor diagram, consider the case that $\mathbb{S} = D_{n,0}$ is an $n$-gon with boundary points $p_1,\dots,p_n$ labeled in counterclockwise order. Consider the moduli space $\mathcal{M} = \mathcal{A}'_{{\rm SL}_k,\mathbb{S}}$. Any two paths joining a pair of points in $p,q \in \mathbb{S}$ are isotopic, so the fibers $V_p,V_q$ of the local system are canonically identified by the transport isomorphisms. A choice of $z \in \mathcal{M}$ determines a vector $v_i \in P$ at the boundary point $p_i$ and the vector $v_{\rm initial}$ is the tensor product of these vectors. 
Consider the invariant $[T]$ of the following diagram: 
\begin{equation}\label{eq:Plucker}
\begin{tikzpicture}
\node at (-3,0) {$T = $};
\node at (90:1.5) {$p_{i_k}$};
\node at (0:1.5) {$p_{i_{k-1}}$};
\node at (-30:1.5) {$p_{i_{k-2}}$};
\node at (-90:1.5) {$\cdots$};
\node at (210:1.5) {$p_{i_{2}}$};
\node at (180:1.5) {$p_{i_1}$};
\node at (0,0) {$\circ$};
\node at (-.4,.37) {\tiny $k-1$};
\draw (-.05,.75)--(-.25,.75);
\draw [->](90:.2)--(90:.7);
\draw [->](90:1.25)--(90:.8);
\draw [->](0:1.2)--(0:.2) ;
\draw [->](-30:1.2)--(-30:.2) ;
\draw [->](-90:1.25)--(-90:.2) ;
\draw [->](210:1.25)--(210:.2) ;
\draw [->](180:1.25)--(180:.2) ;
\end{tikzpicture}
\end{equation}
Applying the parallel transport isomorphisms we get a copy of $v_{i_1} \wedge \cdots \wedge v_{i_{k-1}}$ exiting the white vertex, and we get $v_{i_1} \wedge \cdots \wedge v_{i_{k-1}} \otimes v_{i_k}$ entering the ``hair'' (which is a visual shortcut for an evaluation leaf, see above). When we exit the hair, we get the tensor $v_{i_1} \wedge \cdots \wedge  v_{i_k}$. Applying the evaluation map, we get the number $\det (v_{i_1} \wedge \cdots \wedge  v_{i_k}) \in \mathbb{C}$.

When a similar tensor diagram is drawn on a surface, one transports the tensors to the evaluation leaf using the parallel transport isomorphisms and then takes a similar determinant. 
\end{example}

\begin{definition} A {\sl web} is an embedded (i.e. planar) tensor diagram.
A {\sl forest diagram} is a tensor diagram whose underlying undirected graph has no cycles on interior vertices. A {\sl tree diagram} is a forest diagram which is not a superposition of two forest diagrams.

A function on the moduli space $\mathcal{M}$ is a {\sl diagram invariant} if it can be expressed as $[T]$ for some tensor diagram~$T$. One has similarly {\sl web invariants},  {\sl forest invariants}, and {\sl tree invariants}. 
\end{definition}


The product of diagram invariants $[T] \cdot [T']$ is again a diagram invariant encoded by the  superposition of diagrams $T \cup T'$. 

\begin{definition}
The {\sl skein algebra} ${\rm Sk}(\mathcal{M})$ is the algebra generated by $\mathcal{M}$-tensor diagram invariants, thought of as a subalgebra of the algebra of regular functions on $\mathcal{M}$.
\end{definition}

We stress that we do not define the skein algebra as the algebra generated by certain diagrams and modulo certain linear relations, which is the typical usage of this terminology. We discuss partial steps towards such a definition in the next section.

\begin{remark}
We expect that in many cases, the skein algebra ${\rm Sk}(\mathcal{M})$ coincides with the algebra of regular functions on $\mathcal{M}$, but we do not have a sense of precisely when this holds. 
\end{remark}

\begin{example}\label{eg:SL2plainarcsasfunctions}
Recall the tensor diagram $T(\gamma)$ associated to an arc $\gamma$  as in Example~\ref{eg:SL2plainarcs}. As in Example~\ref{eg:determinantsasdiagrams}, the invariant $[T(\gamma)] \in {\rm Sk}({\rm SL}_2,\mathbb{S})$ evaluates on a point in the moduli space by first transporting the chosen vectors at the endpoints of $\gamma$ to the midpoint of $\gamma$ and then taking their determinant. The two possible signs for such a determinant correspond to the two possible directions in which the evaluation edge can point. As we have said, we gloss over this sign in what follows. 

The invariant $[T(\alpha)] \in {\rm Sk}({\rm SL}_2,\mathbb{S})$ encoded by a simple closed curve $\alpha$ is the trace of the monodromy matrix $M_\alpha$ assigned to $\alpha$. Indeed, the volume form $\xi = e_1 \wedge e_2$ at the coevaluation stub is sent to $e_1 \otimes e_2 - e_2 \otimes e_1$ by the coevaluation map. The two ways of proceeding from the coevaluation leaf to the evaluation leaf differ by $M_\alpha$, so that the parallel transport of the above tensor is 
$e_1 \otimes M_\alpha(e_2) - e_2 \otimes M_\alpha(e_1)$. Applying the exterior product map we get 
$$[T(\alpha)] = e_1 \wedge M_\alpha(e_2) - e_2 \wedge M_\alpha(e_1) = {\rm tr}(M_\alpha)  \in \mathbb{C}.$$
\end{example}

\subsection{Skein relations between diagram invariants}
There are interesting linear relations between diagram invariants known as {\sl skein relations}. Each skein relation $\sum a_i [T_i] = 0$ is local in the sense that one can choose an open neighborhood $U \subset {\rm int}\, \mathbf{S}$ such that the involved diagrams 
$T_i$ coincide on $\mathbf{S} \setminus U$. 

As an example, whenever a tensor diagram $T$ has a crossing, one can apply the following {\sl crossing removal relation} \cite[Corollary 6.2.3]{CKM}: 
\begin{equation}\label{eq:crossingremoval}
\begin{tikzpicture}
\draw [thick,    decoration={markings,mark=at position 1 with {\arrow[scale=1.7]{>}}},
    postaction={decorate},
    shorten >=0.4pt] (-1,-1)--(-.55,-.2);
\draw [thick] (-.55,-.2)--(.2,1);
\node at (-1.2,-.5) {$a$};
\draw [thick,    decoration={markings,mark=at position 1 with {\arrow[scale=1.7]{>}}},
    postaction={decorate},
    shorten >=0.4pt] (.2,-1)--(-.25,-.2);
    \draw [thick] (-.25,-.2)--(-1.0,1);
\node at (-1.2,-.5) {$a$};
\node at (.4,-.5) {$b$};
\node at (1.5,0) {$=$};
\node at (4.0,-.8) {\small $a$};
\node at (4.,.8) {\small $b$};
\node at (3.7,.05) {\small $a-c$};
\node at (6.4,-.8) {\small $b$};
\node at (6.4,.8) {\small $a$};
\node at (6.7,.05) {\small $b+c$};
\node at (5.25,-.5) {\small $c$};
\node at (5.2,.7) {\tiny $b+c-a$};

\node at (2.5,0) {$\displaystyle \sum_{c }$};
\draw [thick,    decoration={markings,mark=at position 1 with {\arrow[scale=1.7]{>}}},
    postaction={decorate},
    shorten >=0.4pt] (4.4,-1)--(4.4,-.4);
\draw [fill= black] (4.4,-.3) circle [radius = .08];
\draw [thick,    decoration={markings,mark=at position 1 with {\arrow[scale=1.7]{>}}},
    postaction={decorate},
    shorten >=0.4pt] (4.4,-.25)--(4.4,.4);
\draw [fill= white] (4.4,.5) circle [radius = .08];
\draw [thick,    decoration={markings,mark=at position 1 with {\arrow[scale=1.7]{>}}},
    postaction={decorate},
    shorten >=0.4pt] (4.4,.58)--(4.4,1);
    \draw [thick] (4.4,1)--(4.4,1.2);
\draw [thick,    decoration={markings,mark=at position 1 with {\arrow[scale=1.7]{>}}},
    postaction={decorate},
    shorten >=0.4pt] (6.1,-1)--(6.1,-.4);
\draw [fill= white] (6.1,-.3) circle [radius = .08];
\draw [thick,    decoration={markings,mark=at position 1 with {\arrow[scale=1.7]{>}}},
    postaction={decorate},
    shorten >=0.4pt] (6.1,-.25)--(6.1,.4);
\draw [fill= black] (6.1,.5) circle [radius = .08];
\draw [thick,    decoration={markings,mark=at position 1 with {\arrow[scale=1.7]{>}}},
    postaction={decorate},
    shorten >=0.4pt] (6.1,.58)--(6.1,1);
    \draw [thick] (6.1,1)--(6.1,1.2);
\draw [thick,    decoration={markings,mark=at position 1 with {\arrow[scale=1.7]{>}}},
    postaction={decorate},
    shorten >=0.4pt] (4.4,-.3)--(6.0,-.3);
\draw [thick,    decoration={markings,mark=at position 1 with {\arrow[scale=1.7]{>}}},
    postaction={decorate},
    shorten >=0.4pt] (6.1,.5)--(4.5,.5);
\end{tikzpicture},
\end{equation}
expressing a diagram invariant which has a crossing as a sum of diagram invariants in which this crossing has been removed. It follows that web invariants 
span the algebra ${\rm Sk}(\mathcal{M})$. The Ptolemy relations familiar in the theory of cluster algebras from surfaces are the $k=2$ instance of this relation.

Although web invariants span ${\rm Sk}(\mathcal{M})$, they are not linearly independent. The most mysterious linear relation between them is the {\sl square switch relation} \cite[(2.10)]{CKM}:
\begin{equation}\label{eq:squaremove}
\begin{tikzpicture}
\begin{scope}[xshift = -7.85cm]
\node at (4.0,-.8) {\small $a$};
\node at (3.6,.8) {\tiny $a+d-c$};
\node at (3.7,.05) {\small $a-c$};
\node at (6.4,-.8) {\small $b$};
\node at (6.8,.8) {\tiny $b+c-d$};
\node at (6.7,.05) {\small $b+c$};
\node at (5.25,-.5) {\small $c$};
\node at (5.25,.75) {\small $d$};
\draw [thick,    decoration={markings,mark=at position 1 with {\arrow[scale=1.7]{>}}},
    postaction={decorate},
    shorten >=0.4pt] (4.4,-1)--(4.4,-.4);
\draw [fill= black] (4.4,-.3) circle [radius = .08];
\draw [thick,    decoration={markings,mark=at position 1 with {\arrow[scale=1.7]{>}}},
    postaction={decorate},
    shorten >=0.4pt] (4.4,-.25)--(4.4,.4);
\draw [fill= white] (4.4,.5) circle [radius = .08];
\draw [thick,    decoration={markings,mark=at position 1 with {\arrow[scale=1.7]{>}}},
    postaction={decorate},
    shorten >=0.4pt] (4.4,.58)--(4.4,1);
    \draw [thick] (4.4,1)--(4.4,1.2);
\draw [thick,    decoration={markings,mark=at position 1 with {\arrow[scale=1.7]{>}}},
    postaction={decorate},
    shorten >=0.4pt] (6.1,-1)--(6.1,-.4);
\draw [fill= white] (6.1,-.3) circle [radius = .08];
\draw [thick,    decoration={markings,mark=at position 1 with {\arrow[scale=1.7]{>}}},
    postaction={decorate},
    shorten >=0.4pt] (6.1,-.2)--(6.1,.4);
\draw [fill= black] (6.1,.5) circle [radius = .08];
\draw [thick,    decoration={markings,mark=at position 1 with {\arrow[scale=1.7]{>}}},
    postaction={decorate},
    shorten >=0.4pt] (6.1,.58)--(6.1,1);
    \draw [thick] (6.1,1)--(6.1,1.2);
\draw [thick,    decoration={markings,mark=at position 1 with {\arrow[scale=1.7]{>}}},
    postaction={decorate},
    shorten >=0.4pt] (4.4,-.3)--(6.0,-.3);
\draw [thick,    decoration={markings,mark=at position 1 with {\arrow[scale=1.7]{>}}},
    postaction={decorate},
    shorten >=0.4pt] (6.1,.5)--(4.5,.5);\end{scope}
\node at (0.0,-.2) {$=$};
\node at (1.0,0) {$\displaystyle \sum_e $};
\node at (2.0,0) {\small $\binom{a+d-b-c}e$};
\node at (4.0,-.8) {\small $a$};
\node at (3.6,.8) {\tiny $a+d-c$};
\node at (3.7,.05) {\tiny $a+d-e$};
\node at (6.4,-.8) {\small $b$};
\node at (6.8,.8) {\tiny $b+c-d$};
\node at (6.9,.05) {\tiny $b-d+e$};
\node at (5.25,-.5) {\tiny $d-e$};
\node at (5.25,.75) {\tiny $c-e$};
\draw [thick,    decoration={markings,mark=at position 1 with {\arrow[scale=1.7]{>}}},
    postaction={decorate},
    shorten >=0.4pt] (4.4,-1)--(4.4,-.4);
\draw [fill= white] (4.4,-.3) circle [radius = .08];
\draw [thick,    decoration={markings,mark=at position 1 with {\arrow[scale=1.7]{>}}},
    postaction={decorate},
    shorten >=0.4pt] (4.4,-.2)--(4.4,.4);
\draw [fill= black] (4.4,.5) circle [radius = .08];
\draw [thick,    decoration={markings,mark=at position 1 with {\arrow[scale=1.7]{>}}},
    postaction={decorate},
    shorten >=0.4pt] (4.4,.58)--(4.4,1);
    \draw [thick] (4.4,1)--(4.4,1.2);
\draw [thick,    decoration={markings,mark=at position 1 with {\arrow[scale=1.7]{>}}},
    postaction={decorate},
    shorten >=0.4pt] (6.1,-1)--(6.1,-.4);
\draw [fill= black] (6.1,-.3) circle [radius = .08];
\draw [thick,    decoration={markings,mark=at position 1 with {\arrow[scale=1.7]{>}}},
    postaction={decorate},
    shorten >=0.4pt] (6.1,-.25)--(6.1,.4);
\draw [fill= white] (6.1,.5) circle [radius = .08];
\draw [thick,    decoration={markings,mark=at position 1 with {\arrow[scale=1.7]{>}}},
    postaction={decorate},
    shorten >=0.4pt] (6.1,.58)--(6.1,1);
    \draw [thick] (6.1,1)--(6.1,1.2);
\draw [thick,    decoration={markings,mark=at position 1 with {\arrow[scale=1.7]{>}}},
    postaction={decorate},
    shorten >=0.4pt] (5.9,-.3)--(4.5,-.3);
\draw [thick,    decoration={markings,mark=at position 1 with {\arrow[scale=1.7]{>}}},
    postaction={decorate},
    shorten >=0.4pt] (4.5,.5)--(5.9,.5);\end{tikzpicture},
\end{equation}

Besides this square switch relation, there are additional skein relations which reflect ``obvious'' properties of the underlying exterior algebra gadgets. Among these are 
the {\sl leaf migration} relations \cite[(2.3), (2.7),(2.8)]{CKM} and also the associativity relation \cite[(2.6)]{CKM}. The former relations capture signs associated with permuting vectors in an exterior product and the latter relation expresses associativity of exterior product. The composition of a black vertex (splitting up a tensor into several factors) with a white vertex (wedging these vectors back together) is a scalar multiple of the identity map and this is encoded by another skein relation \cite[(2.4)]{CKM}. There are also the duals of all of these relations. 

The above named linear relations are the {\sl only} relations between diagram invariants  provided we make two simplifying assumptions \cite{CKM}. First, we should work only with the $n$
-gon $\mathbb{S } = D_{n,0}$ and second, we should only consider tensor diagrams with at most one edge at each marked point. In our setting, marked points can have arbitrary degree and further diagrammatic relations are needed. For example, any diagram with a 2-cycle based at a marked point will define the zero invariant.  As a separate issue, we do not think that it has been proved that the relations in \cite{CKM} are the only linear relations between diagrams when $\mathbb{S}$ is not an $n$-gon.

\subsection{Tagged and pseudotagged diagrams}
We now modify the calculus of tensor diagram invariants in the presence of punctures, generalizing the tagged arc calculus valid when $k=2$.

Recall that a leg is an edge of a tensor diagram which is incident to a puncture. Denote by $2^{k}$ the power set of $[k]$.

\begin{definition}\label{defn:tagging}
A {\sl pseudotagged} tensor diagram is a pair $(T,\varphi)$ consisting of a tensor diagram~$T$ and a {\sl tagging function} 
\begin{equation}\label{eq:varphitagging}
\varphi \colon \,{\rm Legs}(T) \to 2^{[k]} \hspace{.5cm } \text{ subject to } \hspace{.5cm } |\varphi(e)| = {\rm wt}(e) \text{ for all } e \in {\rm Legs}(T). 
\end{equation}

A {\sl tagged tensor diagram} is a pseudotagged tensor diagram $(T,\varphi)$ which satisfies the following additional condition:  for any puncture $p$ and any pair of legs $e,e' \in {\rm Legs}_p(T)$, one has a containment of subsets either $\varphi(e) \subseteq \varphi(e')$ or $\varphi(e') \subseteq \varphi(e)$. \end{definition}

We depict a pseudotagged tensor diagram by labeling each leg of $T$ by the subset $\varphi(e)$ rather than by  its integer label ${\rm wt}(e) \in [k]$. Edges which are not legs continue to carry their integer labels. 

The construction of invariants for pseudotagged tensor diagrams is obtained by modifying Definition~\ref{defn:evaluation} only in the choice of vector~$v_e(z)$ associated to edges~$e$ whose source is a puncture. For such a leg $e \in {\rm Legs}_p(T)$ choose an arbitrary permutation $w_e$ satisfying $w_e([{\rm wt}(e)]) = \varphi(e)$. Acting by $w_e$ on the affine flag $F_p$ decoration which $z$ assigns to $p$, we get a tensor $(w_e \cdot F_p)_{({\rm wt}(e))}$. Then $v_{\rm initial}(\varphi,z)$ is the tensor product of these vectors over all $e \in {\rm Legs}(T)$, together with the vectors assigned to edges at boundary points and those assigned to coevaluation leafs. The choice of permutation $w_e$ in this recipe is immaterial thanks to Lemma~\ref{lem:parabolic}. 

\begin{definition}\label{defn:evaluatetagged}
The invariant $[(T,\varphi)]$ encoded by a tagged tensor diagram $(T,\varphi)$ is the rational function whose value on a point $z \in \mathcal{M}$ is the image of the vector $v_{\rm initial}(\varphi,z) \in \overset{\to}{W}(T)$ under the morphism $\overset{\to}{W}(T) \to \mathbb{C}$ encoded by $z$.\end{definition}

The weight of a pseudotagged diagram invariant is given by ${\rm wt}_p([(T,\varphi)]) = \sum_{e \in {\rm Legs}_p(T)}\iota_{\varphi(e)}$, a sum of indicator vectors. 

We use {\sl pseudotagged} and {\sl tagged} as adjectives modifying the nouns 
diagram invariant, web invariant, etc. For example, a {\sl tagged web invariant} is a function on $\mathcal{M}$ which can be encoded as $[(T,\varphi)]$ with $T$ a planar tensor diagram and $\varphi$ a tagging (not merely a pseudotagging) of $T$. 

\begin{definition}
We denote by ${\rm Sk}^\bowtie({\rm SL}_k,\mathbb{S}) $ the algebra generated by pseudotagged invariants, thought of as a subalgebra of the field of rational functions on $\mathcal{M}$.
\end{definition}
Since the superposition of pseudotagged diagrams is again pseudotagged, this algebra is spanned by pseudotagged diagram invariants.

\begin{remark}
Tagged tensor diagram invariants are exactly the pullbacks of diagram invariants along the Weyl group action at punctures. Indeed, if $(T,\varphi)$ is a tagged tensor diagram, then the containment condition ensures that one can choose a permutation $w_p \in W$ for each puncture (as opposed to a permutation $w_e$ for each leg~$e$)  with the property 
that $\varphi(e) = w_p([{\rm wt}(e)])$ for all legs $e$ at $p$. One has then that 
$[(T,\varphi)] = \prod_p \psi(w_p)^*([T])$. And conversely, the pullback of any diagram invariant $[T]$ along $\prod_p \psi(w_p)$ is the invariant of a tagged diagram $(T,\varphi)$ whose tagging function is $\varphi(e):= w_p([{\rm wt}(e)])$ whenever $e$ is a leg at $p$.
\end{remark}

By the preceding remark, we have the equality of algebras ${\rm Sk}^\bowtie({\rm SL}_k,\mathbb{S}) = W^{\mathbb{M}_\circ} \cdot {\rm Sk}({\rm SL}_k,\mathbb{S})$, i.e. the former algebra is the orbit of the skein algebra under the birational Weyl group action at punctures.

\begin{remark}
If a rational function $f$ is the invariant of a tagged diagram $(T,\varphi)$, then using the translation \eqref{eq:nooflegs} we can deduce how many legs $(T,\varphi)$ has at each puncture and how each of these is tagged. Such a deduction is not possible in general within the class of pseudotagged diagram invariants. For example, we cannot deduce from the knowledge ${\rm wt}_p(f) = e_1+e_2$ whether we have two legs tagged by the subsets $\{1\}$ and $\{2\}$ or one leg tagged by $\{1,2\}$. In a tagged diagram, only the second of these is possible. 
\end{remark}

\begin{remark}\label{rmk:warningtoreader}
Let $v_1,\dots,v_k$ be a tuple of vectors representing an affine flag $F_p$ assigned to a puncture $p$ by a decorated local system $z \in \mathcal{M}$. We stress that, for example, a leg $e \in {\rm Legs}_p(T)$ appearing in a pseudotagged diagram $(T,\varphi)$ and with tagging $\varphi(e) = \{2\}$ does not represent the ``second vector'' $v_2$, but rather the vector $u_p(v_2)-v_2$. In a similar vein, one should not confuse the following two pseudotagged tensor diagram fragments: 
\begin{equation}
\begin{tikzpicture}
\node at (-.25,-.25) {$p$};
\node at (0,0) {$\bullet$};
\draw [->] (0,0)--(0,1.25);
\node at (.5,.4) {\tiny $\{1,2\}$};

\node at (2.75,-.25) {$p$};
\node at (3,0) {$\bullet$};
\draw [->, out = 150, in = 210] (3,0) to (2.9,.7);
\draw [->, out = 30, in = -30] (3,0) to (3.1,.7);
\node at (3.0,.75) {$\circ$};
\node at (2.3,.4) {\tiny $\{1\}$};
\node at (3.7,.4) {\tiny $\{2\}$};
\node at (3.2,1.1) {\tiny $2$};
\draw [->] (3,.85)--(3,1.25);
\end{tikzpicture}
\end{equation}
The left fragment represents $v_1 \wedge v_2$ while the right one represents $v_1 \wedge (u(v_2)-v_2) = 0$. Any $(T,\varphi)$ containing the right fragment has $[(T,\varphi)] = 0$. 
\end{remark}

\begin{example}\label{eg:taggedasskein}
We explain in detail that when $k=2$, the choice of tagging a leg by the subset $\{1\}$ or $\{2\}$ is exactly the choice of plain versus notched tagging of~$e$.

We use the notation of \cite[Figure 14]{CATSII} for arcs in a once-punctured digon with boundary points $p',q$ and with puncture~$p$. We consider an arc $\gamma$ connecting $p$ to $p'$ in this digon and denote by $\gamma^\bowtie$ the result of notching the end of $\gamma$ at $p$. We denote by $\alpha$ and $\beta$ be the arcs forming the two sides of the digon, by $\theta$ the arc connecting $q$ to $p$ in the digon, and by $\eta$ the loop based $p'$ and enclosing $p$ as a once-punctured monogon. Each of the plain arcs $\gamma$, $\alpha$, $\beta$, and $\eta$ determines an element of the skein algebra which we denote by $[T(\gamma)]$ etc. as in Example~\ref{eg:SL2plainarcsasfunctions}. 

We denote by $x(\gamma^\bowtie)$ the cluster variable indexed by $\gamma^\bowtie$. It is defined by the exchange relation 
$[T(\theta)]x(\gamma^\bowtie) =[T(\alpha)]+[T(\beta)]$. On the other hand, we have a skein relation $[T(\theta)][T(\eta)] = [T(\gamma)]\left([T(\alpha)]+[T(\beta)]\right) \in {\rm Sk}({\rm SL}_2,\mathbb{S})$. It follows that
\begin{equation}\label{eq:bowtie}
x(\gamma^\bowtie) = \frac{[T(\eta)]}{[T(\gamma)]}.
\end{equation}

We need to reconcile this expression for $x(\gamma^\bowtie)$ with the result of pulling back along $s_1$ action at $p$. We have vector decorations $v_p$ at $p$ and $v_{p'}$ at $p'$ respectively. In the notation of Example~\ref{eg:SL2plainarcsasfunctions}, 
$\psi_{s_1,p}^*([T(\gamma)])$ is the function 
$\det\begin{pmatrix}
v_{p'} & u_p(w)-w
\end{pmatrix}
,$
where $u_p$ is the monodromy around $p$ and $w$ is {\sl any} choice of vector with the property that $\det(v_{p} w) = 1$. Generically, we can take $w = \frac{v_{p'}}{\det(v_p v_{p'})}$. Thus 
$$\psi_{s_1,p}^*([T(\gamma)])= \det\begin{pmatrix}
v_{p'} & \frac{1}{\det(v_p v_{p'})}\left(u_p(v_{p'}) - v_{p'}\right) \\
\end{pmatrix} = \frac{\det\begin{pmatrix}
v_{p'} & u_p(v_{p'})
\end{pmatrix}}{\det \begin{pmatrix}
v_{p} & v_{p'}
\end{pmatrix}} = \frac{[T(\eta)]}{[T(\gamma)]} = x(\gamma^\bowtie)
$$
as claimed.
\end{example}

\subsection{Cluster compatibility conjectures}\label{secn:conjectures} Let $\mathcal{M}$ be one of the two versions of moduli space considered in this paper. We now state our two main 
conjectures phrasing cluster compatibility for $\mathscr{A}(\mathcal{M})$ in terms of tagged tensor diagram calculus. 

Our first conjecture extends the philosophy of Fomin and the second author \citep{FPII,FP} in two directions: to the case of surfaces $\mathbb{S}$ with punctures, and also to the setting $k>3$.
\begin{conjecture}\label{conj:yespunctures}
Every cluster monomial $f \in \mathscr{A}(\mathcal{M})$ is a tagged web invariant and also a tagged forest invariant. Every cluster variable $f$ is moreover a tagged tree invariant. 
\end{conjecture}

\begin{remark}
It would be desirable to formulate a move on tagged diagrams which does not change the tagged invariant and with the property that a planar diagram represents a cluster variable if and only if it can be turned into a tree diagram by repeated applications of this move. This would be a higher rank analogue of the arborization move introduced when $k=3$ in \cite{FP}. We found generalizations of this arborization move in higher rank, but we were not confident that the generalization we found has the desired properties. Certainly, this deserves further study. 
\end{remark}

Our second  conjecture extends the $k=2$ philosophy of Fomin, Shapiro, and Thurston to the $k>2$ setting:
\begin{conjecture}\label{conj:clusterconjecture}
Let $T_1,\dots,T_s$ be tensor diagrams with each $[T_i] \in \mathscr{A}(\mathcal{M})$ a cluster variable. Let $\varphi_i$ be a tagging of $T_i$ with corresponding invariant $f_i = [(T_i,\varphi_i)]$. Assume for simplicity that all the $f_i$'s are distinct. Then~$\prod_i f_i \in\mathscr{A}(\mathcal{M}) $ is a cluster monomial if and only if the following three conditions hold. 

First, the underlying plain monomial $\prod_i [T_i] = [\cup _i T_i]\in \mathscr{A}(\mathcal{M})$ is a cluster monomial. Second, if $[T_i] \neq [T_j]$, then $(T_i,\varphi_i) \cup (T_j,\varphi_j)$ is a tagged tensor diagram. And third, if 
$[T_i] =[T_j]$, then the weight vectors ${\rm wt}(f_i)$ and 
${\rm wt}(f_j)$ are root-conjugate at some puncture $p$ and coincide at all other punctures $p'$.
\end{conjecture}

Note that Conjecture~\ref{conj:yespunctures} is only a necessary condition on cluster monomials, while Conjecture~\ref{conj:clusterconjecture} claims to be both necessary and sufficient.

Translating from the language of weight vectors to the combinatorics of legs as in \eqref{eq:nooflegs}, we can  restate Conjecture~\ref{conj:clusterconjecture} in a way that is more obviously consonant with \cite{CATSI}. Recall that we depict a tagged tensor diagram $(T,\varphi)$ as a plain diagram each of whose legs~$e$ is decorated by a subset $\varphi(e)$. Then the second condition in Conjecture~\ref{conj:clusterconjecture} says that 
if $[T_i] \neq [T_j]$, then either 
$\varphi_i(e) \subset \varphi_j(e')$ or 
$\varphi_j(e') \subset \varphi_i(e)$ for any pair of legs $e,e' \in {\rm Legs}_p(T)$ and any puncture $p$, which generalizes the requirement that distinct arcs must be tagged the same way at punctures. The 
third condition says that if $[T_i] = [T_j]$, then the taggings $\varphi_i$ and $\varphi_j$ disagree in exactly one leg just as in the $k=2$ case.

\begin{lemma}\label{lem:clusterconjimpliesgood}
If a cluster $\{f_i\}$ satisfies the three conditions in Conjecture~\ref{conj:clusterconjecture}, then its underlying $P$-cluster $\{\pi_p({\rm wt}f_i)\} \subset P$  at any puncture $p$ is a basic multiset of pairwise compatible weight vectors. 
\end{lemma}

Thus, if the conditions stated in Conjecture~\ref{conj:clusterconjecture} are indeed necessary conditions, and if every cluster variable is indeed a tagged invariant, then every cluster is in the good part of the exchange graph, and the dosp mutation results of Section~\ref{secn:contraction} apply to the exchange graph itself.

\begin{proof}
The second and third conditions in Conjecture~\ref{conj:clusterconjecture} 
together imply that the $P$-cluster at any puncture consists of pairwise compatible vectors. We must also show that the $P$-cluster at each puncture is basic. Suppose that $f_i$ and $f_j$ have root-conjugate weights by the transposition $t = (ab)$ at some puncture $p$. Then from the definition of evaluation of tagged tensor diagrams, the corresponding functions 
are related by birational Weyl group action: $f_j = \psi(t)^* f_i$ with $t$ acting at the puncture $p$. Since we can compute $f_i$ uniquely from $f_j$, it follows that $f_j$ must be the unique cluster variable with this weight vector. 
\end{proof}

\begin{remark}\label{rmk:easycase}
Suppose in Conjecture~\ref{conj:clusterconjecture} that the collection $\{f_i\}$ has the property that the vectors ${\rm wt}_p(f_i)$ and ${\rm wt}_p(f_j)$  are $W$-sortable for every puncture $p$ and every pair $i,j$. (That is, the second condition in Conjecture~\ref{conj:clusterconjecture} holds.) Then the third condition plays no role, and the first condition is clearly both necessary and sufficient for $\prod f_i$ to be a cluster monomial: the monomials $\prod_i f_i$ and $\prod_i [T_i]$ are related by pullback along the $W^{\mathbb{M}_\circ}$ action, so the two products are cluster monomials simultaneously. To emphasize: the ``interesting'' behavior in Conjecture~\ref{conj:clusterconjecture} arises when at least one pair of weights at some puncture $p$ are root-conjugate, provided we believe Conjecture~\ref{conj:Pclusterconjecture} holds. We did not see how to prove either the necessity or sufficiency of the third condition in this case. 
\end{remark}

\section{The flattening and spiral theorems}\label{secn:flattening}
We state and prove the flattening theorem and the spiral theorem. Both theorems concern the relationship between pseudotagged diagram invariants and tagged diagram invariants. The flattening theorem says that every pseudotagged invariant is a linear combination of tagged invariants. The flattening relation which underlies the proof is a useful algebraic tool for computations in the cluster algebra $\mathscr{A}(\mathcal{M})$. For example, certain exchange relations in $\mathscr{A}(\mathcal{M})$ are instances of the flattening theorem. The spiral theorem establishes a setting in which certain specific pseudotagged invariants are in fact tagged invariants as predicted by our cluster compatibility conjectures. We also discuss more refined statements, e.g. whether or not the pseudotagged invariant has a tagged forest form.

\subsection{The flattening theorem}\label{subsec:flattening}
\begin{theorem}\label{thm:flattening}
Any pseudotagged diagram invariant is a linear combination of tagged diagram invariants. 
\end{theorem}
That is, ${\rm Sk}^\bowtie(\mathcal{M})$ is spanned by tagged diagram invariants. Figure~\ref{fig:flattenproduct} illustrates Theorem~\ref{thm:flattening}, flattening a pair of legs tagged by the subsets $\{1,2,3\}$ and $\{1,4,5\}$ as a linear combination of five tagged diagrams, 
each which with one leg tagged by $[5]$ and the other by $\{1\}$. 

See \cite[Section 8.4]{MSW} for a discussion of similar flattening relations when $k=2$. 

\begin{figure}[ht]
\begin{tikzpicture}
\node at (0,0) {$\bullet$};
\node at (-.2,-.2) {\tiny $p$};
\draw [->, out = 15, in = 270] (0,0) to (1.25,1.1);
\draw [->, out = 7, in = 270](0,0) to (1.75,1.1);
\draw [->](1.25,1.7)--(1.25,1.2);
\draw [->](1.75,1.7)--(1.75,1.2);
\draw (1.25,1.13)--(1.05,1.13);
\draw (1.75,1.13)--(1.55,1.13);
\node at (1.25,2) {$\zeta$};
\node at (1.75,2) {$\eta$};
\node [rotate = 5]at (.3,.45) {\tiny $\{1,4,5\}$};
\node [rotate = 25] at (1.6,.15) {\tiny $\{1,2,3\}$};
\node at (2.8,0) {$=$};
\begin{scope}[xshift = 5cm]
\node at (0,0) {$\bullet$};
\node at (-.2,-.2) {\tiny $p$};
\node at (1,.5) {$\bullet$};
\node at (1,-.5) {$\circ$};
\draw [->](0,0)--(.9,.45);
\draw [->](0,0)--(.85,-.425);
\draw [->] (1,.5)--(1,-.35);
\node [rotate = -30]at (.45,-.6) {\tiny $[1]$};
\node [rotate = 30]at (.45,.6) {\tiny $[5]$};
\draw [->, out = 15, in = 270](1,.5)to (1.75,1.1);
\draw [->, out = 7, in = 270](1.1,-.5)to (3.05,1.1);
\node at (1.2,0) {\tiny $2$};
\node at (1.4,.9) {\tiny $3$};
\node [rotate  = 10]at (1.7,-.6) {\tiny $3$};
\node at (2.6,2) {$\zeta$};
\node at (3.05,2) {$\eta$};
\draw [->](3.05,1.8)--(3.05,1.2);
\draw (3.05,1.15)--(2.85,1.15);
\draw [out = -80, in =0 ] (2.65,1.6) to (1.8,-1.9);
\draw [out = 180, in =-90 ] (1.8,-1.9) to (-1.2,0);
\draw [out = 90, in =180 ] (-1.2,0) to (1.2,2.3);
\draw [out = 0, in =0 ] (1.2,2.3) to (1.75,-1.7);
\draw [out = 180, in =-90 ] (1.75,-1.7) to (-1.0,0);
\draw [out = 90, in =180 ] (-1.0,0) to (1.25,2.1);
\draw [out = 0, in =0 ] (1.25,2.1) to (1.55,-1.55);
\draw [out = 180, in =-90 ] (1.55,-1.55) to (-.75,0);
\draw [out = 90, in =180 ] (-.75,0) to (1.3,1.8);
\draw [out = 0, in =0 ] (1.3,1.8) to (1.45,-1.4);
\draw [out = 180, in =-90 ] (1.45,-1.4) to (-.5,0);
\draw [out = 90, in =180 ] (-.5,0) to (1.15,1.6);
\draw [->,out = 0, in =100 ] (1.15,1.6) to (1.75,1.2);
\draw (1.75,1.15)--(1.55,1.15);
\node at (-1.5,0) {$\frac{1}{2}$};
\end{scope}
\begin{scope}[xshift = 11.0cm]
\node at (0,0) {$\bullet$};
\node at (-.2,-.2) {\tiny $p$};
\node at (1,.5) {$\bullet$};
\node at (1,-.5) {$\circ$};
\draw [->](0,0)--(.9,.45);
\draw [->](0,0)--(.85,-.425);
\draw [->] (1,.5)--(1,-.35);
\node [rotate = -30]at (.45,-.6) {\tiny $[1]$};
\node [rotate = 30]at (.45,.6) {\tiny $[5]$};
\draw [->, out = 15, in = 270](1,.5)to (1.75,1.1);
\draw [->, out = 7, in = 270](1.1,-.5)to (3.05,1.1);
\node at (1.2,0) {\tiny $2$};
\node at (1.4,.9) {\tiny $3$};
\node [rotate  = 10]at (1.7,-.6) {\tiny $3$};
\node at (2.6,2) {$\zeta$};
\node at (3.05,2) {$\eta$};
\draw [->](3.05,1.8)--(3.05,1.2);
\draw (3.05,1.15)--(2.85,1.15);
\draw [out = -80, in =0 ] (2.6,1.6) to (1.6,-1.8);
\draw [out = 180, in =-90 ] (1.6,-1.8) to (-1.0,0);
\draw [out = 90, in =180 ] (-1.0,0) to (1.2,2.1);
\draw [out = 0, in =0 ] (1.2,2.1) to (1.5,-1.65);
\draw [out = 180, in =-90 ] (1.5,-1.65) to (-.8,0);
\draw [out = 90, in =180 ] (-.8,0) to (1.25,1.9);
\draw [out = 0, in =0 ] (1.25,1.9) to (1.3,-1.5);
\draw [out = 180, in =-90 ] (1.3,-1.5) to (-.6,0);
\draw [out = 90, in =180 ] (-.6,0) to (1.4,1.7);
\draw [->,out = 0, in =100 ] (1.4,1.7) to (1.75,1.2);
\draw (1.75,1.15)--(1.55,1.15);
\node at (-1.5,0) {$-\frac{4}{2}$};
\end{scope}
\begin{scope}[yshift = -4.5cm]
\node at (0,0) {$\bullet$};
\node at (-.2,-.2) {\tiny $p$};
\node at (1,.5) {$\bullet$};
\node at (1,-.5) {$\circ$};
\draw [->](0,0)--(.9,.45);
\draw [->](0,0)--(.85,-.425);
\draw [->] (1,.5)--(1,-.35);
\node [rotate = -30]at (.45,-.6) {\tiny $[1]$};
\node [rotate = 30]at (.45,.6) {\tiny $[5]$};
\draw [->, out = 15, in = 270](1,.5)to (1.75,1.1);
\draw [->, out = 7, in = 270](1.1,-.5)to (2.75,1.1);
\node at (1.2,0) {\tiny $2$};
\node at (1.4,.9) {\tiny $3$};
\node [rotate  = 10]at (1.7,-.6) {\tiny $3$};
\node at (2.3,2) {$\zeta$};
\node at (2.75,2) {$\eta$};
\draw [->](2.75,1.8)--(2.75,1.2);
\draw (2.75,1.15)--(2.55,1.15);
\draw [out = -80, in = 0] (2.35,1.6) to (1.75,-1.8);
\draw [out = 180, in = -90] (1.75,-1.8) to (-.8,0);
\draw [out = 90, in = 180] (-.8,0) to (1,2.1);
\draw [out = 0, in = 0] (1,2.1) to (1.65,-1.6);
\draw [out = 180, in = -90] (1.65,-1.6) to (-.6,0);
\draw [->,out = 90, in = 100] (-.6,0) to (1.75,1.2);
\draw (1.75,1.15)--(1.55,1.15);
\node at (-1.4,0) {$+\frac{6}{2}$};
\end{scope}

\begin{scope}[yshift = -4.5cm,xshift = 5cm]
\node at (0,0) {$\bullet$};
\node at (-.2,-.2) {\tiny $p$};
\node at (1,.5) {$\bullet$};
\node at (1,-.5) {$\circ$};
\draw [->](0,0)--(.9,.45);
\draw [->](0,0)--(.85,-.425);
\draw [->] (1,.5)--(1,-.35);
\node [rotate = -30]at (.45,-.6) {\tiny $[1]$};
\node [rotate = 30]at (.45,.6) {\tiny $[5]$};
\draw [->, out = 15, in = 270](1,.5)to (1.75,1.1);
\draw [->, out = 7, in = 270](1.1,-.5)to (2.5,1.1);
\node at (1.2,0) {\tiny $2$};
\node at (1.3,.8) {\tiny $3$};
\node [rotate  = 10]at (1.7,-.55) {\tiny $3$};
\node at (2.0,2) {$\zeta$};
\node at (2.5,2) {$\eta$};
\draw [->](2.5,1.8)--(2.5,1.2);
\draw (2.5,1.15)--(2.3,1.15);
\draw [out = -85, in = 0] (2.1,1.6) to (1.5,-1.6);
\draw [out = 180, in = -90] (1.5,-1.6) to (-.6,0);
\draw [->, out = 90, in = 100] (-.6,0) to (1.75,1.2);
\draw (1.75,1.15)--(1.55,1.15);
\node at (-1.1,0) {$-\frac{4}{2}$};
\end{scope}

\begin{scope}[yshift = -4.5cm,xshift = 9.5cm]
\node at (0,0) {$\bullet$};
\node at (-.2,-.2) {\tiny $p$};
\node at (1,.5) {$\bullet$};
\node at (1,-.5) {$\circ$};
\draw [->](0,0)--(.9,.45);
\draw [->](0,0)--(.85,-.425);
\draw [->] (1,.5)--(1,-.35);
\node [rotate = -30]at (.45,-.6) {\tiny $[1]$};
\node [rotate = 30]at (.45,.6) {\tiny $[5]$};
\draw [->, out = 15, in = 270](1,.5)to (1.75,1.1);
\draw [->, out = 7, in = 270](1.1,-.5)to (2.25,1.1);
\node at (1.2,0) {\tiny $2$};
\node at (1.4,.9) {\tiny $3$};
\node [rotate  = 10]at (1.8,-.4) {\tiny $3$};
\node at (1.65,2) {$\zeta$};
\node at (2.25,2) {$\eta$};
\draw [->](2.25,1.8)--(2.25,1.2);
\draw (2.25,1.15)--(2.05,1.15);
\draw [->](1.75,1.8)--(1.75,1.2);
\draw (1.75,1.15)--(1.55,1.15);
\node at (-.7,0) {$+\frac{1}{2}$};
\end{scope}
\end{tikzpicture}
\caption{Expressing a pseudotagged invariant as a $\mathbb{Q}$-linear combination of tagged invariants using Lemma~\ref{lem:flattenproduct} with $a=1$ and $b=c=2$. The two legs in the leftmost diagram violate the containment condition from Definition~\ref{defn:tagging}. The tensors $\eta,\zeta \in \bigwedge^{k-3}V$ schematically indicate ``the rest'' of the diagram. The numerators on the right hand side form the fourth row of Pascal's triangle. The denominator counts SYTs of $b\times c = 2 \times 2$ shape. \label{fig:flattenproduct}}
\end{figure}
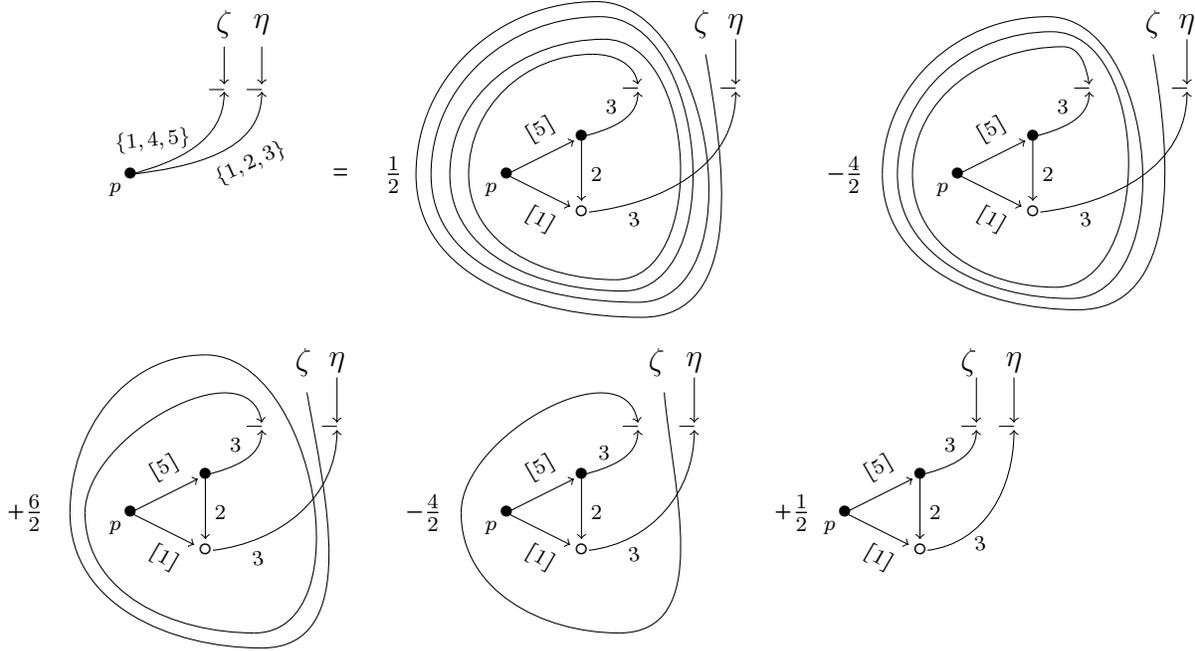

\begin{remark}
We expect, but are not sure how to prove, that ${\rm Sk}^\bowtie(\mathcal{M})$ is the (upper) cluster algebra $\mathscr{A}^{\rm up}(\mathcal{M})$. For a similar statement for the ``$\mathcal{X}$-space'' see \cite{Linhui}. We have not investigated the relationship between ordinary and upper cluster algebras for $\mathscr{A}(\mathcal{M})$. This has been thoroughly explored when $k=2$, see e.g. \cite[Section 3.2]{Mills} and references therein. 
\end{remark}

\subsection{Proof of the flattening theorem}\label{secn:flatteningproof}
For the next several paragraphs we fix the following data. First, a choice of vectors $v_1,\dots,v_k$ determining an affine flag with steps $F_{(a)} := v_1 \wedge \cdots \wedge v_a$. Second, we choose a matrix $u \in {\rm stab}(F)$ in the stabilizer of this flag.

Recall that a subset $S\in \binom{[k]}{a}$ determines a Grassmannian permutation $w_S \in W$ and a Young diagram $\lambda(S)$ whose number of standard Young tableaux is denoted $f^{\lambda(S)}$.

We partially order $\binom{[k]}a$ via the {\sl termwise order}: an $m$-subset $T$ is less than or equal to another such $S$ provided, when we list $T= \{t_1 < t_2 < \cdots <t_a\}$ and $S = \{s_1 < s_2 < \cdots <s_a\}$, we have $t_i \leq s_i$ for all $i$. This partial order is graded by the length function $\ell(S) = \ell(w_S)$.

The above data provides us with a tensor
\begin{equation}\label{eq:Stotensorii}
v_S := \bigwedge_{s \in S}v_s \in \bigwedge^aV,
\end{equation}
with the exterior product taken in increasing order from left to right. We also get a complex number $\mathcal{W}_S$ defined by 
\begin{equation}\label{eq:Stotensor}
(w_s\cdot F)_{(a)} = \mathcal{W}_S F_{(a)} \in \bigwedge^aV, 
\end{equation}
where $F \mapsto w_s \cdot F$ denotes Weyl group action using the permutation $w_S$ and the chosen matrix~$u$.

We let $\mathbb{C}\langle G\rangle$ be the group algebra of $G$. We view $\bigwedge^m(V)$ as a $\mathbb{C}\langle G \rangle$-module via the natural action $G \curvearrowright V$. For $M \in G$ and a complex number $c$, we distinguish between the group algebra element $c \cdot M \in \mathbb{C}\langle G\rangle$ and the matrix $cM$ in which each entry of $M$ is rescaled by $c$. This is important because, e.g., $(c \cdot M)(x \wedge y) = c Mx \wedge My$ whereas $(cM) \cdot x \wedge y = c^2 Mx \wedge My$ inside $\bigwedge^2(V)$. Our next lemma concerns elements of the form $(u^{-1}-{\rm Id}_V)^m \in \mathbb{C}\langle G\rangle$ for $m \in \mathbb{N}$.

\begin{lemma}
For any $S$ and any $T \leq S$, we have   
\begin{equation}\label{eq:killeq}
(u-1)^{\ell(S)} v_{T} = 
\begin{cases}
f^{\lambda(S)}\mathcal{W}_SF_{(a)} & \text{ if $T = S$} \\
0 & \text{ if $T<S$} 
\end{cases} \in \bigwedge^aV.
\end{equation}
\end{lemma}

\begin{proof}
Define numbers $z_{ji} \in \mathbb{C}$ for $1 \leq j\leq i \leq k$ by $u(v_i) = \sum_j z_{ji}v_j$. 

For any $T \in \binom{[k]}a$ it is easy to see that $(u-1)v_T$ is a linear combination of tensors $v_{T'}$ satisfying $T' < T$ in termwise order. (The coefficients of such a linear combination are polynomial functions in the $z_{ji}$'s.) Thus, $(u-1)^{\ell(S)}$ annihilates the tensor $v_T$ whenever 
$T<S$ as claimed. 

It remains to compute $(u-1)^{\ell(S)}v_S$. We use $\lessdot$ to denote cover relations with respect to the order $<$. We expand $(u-1)v_S$ as a linear combination of tensors $v_T$ with $T<S$. By the same reasoning as in the previous paragraph, if $T$ does not satisfy $T \lessdot S$ then $v_T$ is annihilated by $(u-1)^{\ell(S)-1}$, so we can ignore such terms for the rest of the computation. In a cover relation $T \lessdot S$, we have $T = S \setminus \{j+1\} \cup \{j\}$ for some $j+1 \in S$, and the coefficient of $v_T$ in $(u-1)v_S$ is the number $z_{j,j+1}$. 

Consider the set of saturated chains 
$$
E(S) = \{[a] =:S_0 \lessdot S_1 \lessdot \cdots \lessdot S_{\ell(S)} := S\} $$
ending at $S$ in termwise order. The tensor $v_{S_i}$ appears in the tensor $(u-1)v_{S_{i+1}}$ with coefficient $z_{j_i,j_i+1}$ where $S_{i+1} = S_i \setminus (j_i+1) \cup j_i$. Note also that $v_{S_0} = F_{(a)}$. 
It follows that 
\begin{equation}\label{eq:loweringop}
(u-1)^{\ell(S)}v_S = \left(\sum_{S_0 \lessdot S_1 \lessdot \cdots \lessdot S_{\ell(S)}}\prod_{i=1}^{\ell(s)}z_{j_i,j_i+1}\right)F_{(a)}
\end{equation}
where $j_i$ are as defined above. 

We view the parenthesized scalar on the right hand side as a weighted generating function for elements of $E(S)$, with the weight of such an element defined as the monomial in $z_{j_i}$'s coming from the steps of the chain. To complete the proof, we will show that this weight function is constant on $E(S)$, and moreover is equal to the number $\mathcal{W}_S$ defined in \eqref{eq:Stotensor}. The right hand side of \eqref{eq:loweringop} then becomes $f^{\lambda(S)}\mathcal{W}_SF_{(a)}$, completing the proof. 

To argue that $\prod_{i=1}^{\ell(s)}z_{j_i,j_i+1} = \mathcal{W}_S$ is independent of the chain, begin by observing that if $S_i \lessdot S_{i+1}$ is a cover relation in such a chain, then the permutations $w_{S_i}$ and $w_{S_{i+1}}$ are related by left action of a simple transposition swapping the values $(j_i,j_i+1)$. Multiplying these simple transpositions $(s_{j_{\ell(S)}},\dots,s_{j_1})$  from right to left, we get the sequence of permutations $w_{S_1},w_{S_2},\dots,w_{S_{\ell(S)}} =w_S$. In particular, $(s_{j_{\ell(S)}},\dots,s_{j_1})$ is a reduced word for $w_S^{-1}$ (since reduced words are multiplied from left to right). This determines a bijection between saturated chains ending at $S$ and reduced words for $w_{S}^{-1}$. The indices $(j_{\ell(s)},\dots,j_{1})$ appearing in the reduced word 
are what is needed to compute the 
weight 
$\prod_{i=1}^{\ell(s)}z_{j_i,j_i+1}$ in \eqref{eq:loweringop}. So we can view the parenthesized scalar factor in \eqref{eq:loweringop} as a weight-generating function for reduced words of $w_S^{-1}$, or equivalently of $w_S$.

The set of reduced words for $w_S$ is connected by commutation moves. Clearly, commutation moves do not change the weight appearing in \eqref{eq:loweringop}, so that the weight is in fact independent of the reduced word (or equivalently, of the chain $S_0 \lessdot \cdots \lessdot S_{\ell(S)}$) as claimed. 

Finally we prove that the weight of a certain reduced word is $\mathcal{W}_S$, completing the proof. We write $S = \{x_1 < \cdots < x_a\}$. We have the following reduced word for $w_S$
\begin{equation}\label{eq:choiceofword}
\prod_{j=1}^as_{j}\cdots s_{x_j-1}
\end{equation} 
(the product is taken from left to right). 

This word transforms the identity permutation to $w_S$ by the following sequence of swaps in adjacent positions. First consider the smallest element $x_j$ such that $x_j > j$. We can swap the number $x_j$ leftwards (past the numbers $j,j+1,\dots,x_j-1$) until it is in position $j$. Then we swap the number $x_{j+1}$ leftwards until it is in position $j+1$, etc. 

We can compute the tensor $w_S(F)$, hence the number $\mathcal{W}_S$ using the formula \eqref{eq:psiip}. The effect of swapping the number $x_j$ leftwards into position $j$ is to take the $j$th vector $v_j$ of the affine flag and replace it with the vector $(u-1)^{x_j-j}v_{x_j}$. (The vectors in positions $j+1,\dots,x_j$ are also rescaled, but this will not affect the rest of the calculation.)

It is easy to see that 
$$F_{(j-1)} \wedge (u-1)^{x_j-j}v_{x_j} = z_{j,j+1}z_{j+1,j+2}\cdots z_{x_j-1,x_j}F_{(j)}.$$
In the next step, we swap the number $x_{j+1}$ leftwards to position $j+1$, which corresponds to replacing the $j+1$st step of the affine flag by $(u-1)^{x_{j+1}-j-1}v_{x_{j+1}}$. Again, we have 
$$F_{(j)} \wedge (u-1)^{x_{j+1}-j-1}v_{x_{j+1}} = z_{j+1,j+2}\cdots z_{x_{j+1}-1,x_{j+1}}F_{(j+1)}.$$ 
Continuing in this way, we arrive at 
\begin{equation}
\mathcal{W}_S = \prod_{j=1}^az_{j,j+1}\cdots z_{x_j-1,x_j}
\end{equation} 
which is the weight \eqref{eq:loweringop} we would associate to the reduced word \eqref{eq:choiceofword}. 
\end{proof}

Our next lemma uses the operation $\cap$ defined in \eqref{eq:GrassmannCayley}.

\begin{lemma}\label{lem:flattenproduct}
Consider tensors $\eta \in \bigwedge^{k-a-b}(V)$ and $\zeta \in \bigwedge^{k-a-c}(V)$ where $a,b,c \in \mathbb{N}$ satisfy $a+b+c \leq k$. Let $S = [a]\cup [a+b+1,a+b+c]$. We have the following equality of complex numbers in $\bigwedge^k(V) \cong \mathbb{C}$:
\begin{equation}\label{eq:flattenproduct}
(F_{(a+b)} \wedge \eta)  ((w_S\cdot F)_{(a+c)}\wedge  \zeta) = \frac{1}{f^{\lam(S)}}\left(F_{(a+b+c)} \cap F_{(a)} \eta \cap (u^{-1}-1)^{\ell(S)}(\zeta)\right).
\end{equation}
\end{lemma}

Figure~\ref{fig:flattenproduct} illustrate this lemma in the language of tensor diagrams. The quantity on the left hand side is the product of two elements of $\bigwedge^k(V)$, each of which is identified with a complex number using the volume form. The parenthesized term on the right hand side can be similarly interpreted. 

\begin{proof}
We compute the right hand side by shuffling $b$ vectors from the tensor $F_{(a+b+c)} = v_1 \wedge \cdots \wedge v_{a+b+c}$ to the second factor and shuffling the complementary $a+c$ vectors to the third factor, summing over all such shuffles with an appropriate sign. If we shuffle any of the vectors $v_1,\dots,v_a$ into the first factor then the resulting term vanishes since $v_i \wedge F_{(a)} = 0$. 
Thus, the terms that are shuffled into the third factor correspond to subsets $T \in \binom{[a+b+c]}{a+c}$ satisfying $[a] \subset T$. The given subset $S = [a]\cup [a+b+1,a+b+c]$ is maximal among these in the termwise partial order. For such a $T$, we have 
\begin{align*}
v_T \wedge (u^{-1}-1)^{\ell(S)}(\zeta) &= \sum_i (-1)^i \binom{\ell(S)}{i} v_T \wedge u^{-i}(\zeta) \\
&= \sum_i (-1)^i \binom{\ell(S)}{i} u^iv_T \wedge \zeta \\
&= (u-1)^{\ell(S)}v_T \wedge \zeta \\
&= \delta_{S,T}f^{\lam(S)}\mathcal{W}_S F_{(a+c)} \wedge \zeta \\
&= \delta_{S,T}f^{\lam(S)}w_S(F)_{(a+c)} \wedge \zeta.
\end{align*}
The first and third equalities are the binomial theorem. In the second, we move 
$u$ from $\zeta$ to $v_T$ using $\det u =1$. The fourth equality is \eqref{eq:killeq} where $\delta$ denotes Kronecker delta and the last equality is the definition of $\mathcal{W}_S$ \eqref{eq:Stotensorii}. By this calculation, the unique $T$ which can be shuffled to the third factor when computing the right hand side of \eqref{eq:flattenproduct} is $T = S$, leaving $[a+1,b]$ to be shuffled to the second factor (with a sign of $+1$). We get the two determinants appearing on the left hand side after canceling the $f^{\lambda(S)}$. 
\end{proof}

Recall the indicator vector $\iota_S = \sum_{s \in S} e_s\in \mathbb{Z}^k$ of a subset $S \subset [k-1]$. 

\begin{lemma}\label{lem:rewritingrule}
Let $\mathcal{S}$ be a multiset of indicator vectors $\iota_S$ with the property that $\sum_{v \in \mathcal{S}} v$ lies in the dominant Weyl chamber closure $\overline{C_{\rm id}}$. By a {\sl rewriting move}, we will mean the transformation 
$$\{\iota_S,\iota_T\} \mapsto \{\iota_{S\cap T},\iota_{S \cup T}\} $$
removing indicator vectors of sets $S$ and $T$ and replacing them with that of their intersection and union. Then by a finite sequence of such rewriting moves, we can transform the multiset $\mathcal{S}$ to a new multiset $\mathcal{S}'$ consisting entirely of fundamental weights. 
\end{lemma}

\begin{proof}
Pick an enumeration $\iota_{S_1},\iota_{S_2},\dots,\iota_{S_m}$ of the vectors in the initial multiset. We claim for $2 \leq i \leq m$ that by a sequence of rewrites one can transform the multiset $\iota_{S_1},\dots,\iota_{S_i}$ into the multiset of indicator vectors for the subsets 
\begin{equation}\label{eq:rewritten}
\cup_{T \in \binom{[i]}j} \cap_{t \in T}S_t \text{ for $j = 1,\dots,i$}.
\end{equation}

The base case $(i=2)$ is a single rewriting
\begin{equation}
\{\iota_{S_1},\iota_{S_2}\} \mapsto \{\iota_{S_1 \cup S_2},  \iota_{S_1 \cap S_2}\}, 
\end{equation}
with the first term the $j=1$ term and the second the $j=2$ term. 

Assuming we can rewrite the first $i$ indicator vectors as in \eqref{eq:rewritten} , we perform several rewrites to obtain the $i+1$ version of \eqref{eq:rewritten}. First we rewrite the pair consisting of the $j=i$ term 
$\cup_{T \in \binom{[i]}i} \cap_{t \in T}S_t = S_1 \cap \cdots \cap S_i$ and the new term $S_{i+1}$. The result of this rewriting is to obtain 
$S_1 \cap \cdots \cap S_{i+1}$, which is a desired term in the $i+1$ version of \eqref{eq:rewritten}, and 
$(S_1 \cap \cdots \cap S_i) \cup S_{i+1}$, which is not. Then we rewrite the undesired term from the previous step with the $i-1$ term in \eqref{eq:rewritten}. One of the resulting terms is 
$$(\cup_{T \in \binom{[i]}{i-1}} \cap_{t \in T} S_t)\cap ((S_1 \cap \cdots \cap S_i) \cup S_{i+1})  = \cup_{T \in \binom{[i+1]}{i}} \cap_{t \in T} S_t,$$
a desired term in the $i+1$ version of \eqref{eq:rewritten}, and the other term is not. The rest of the proof of the inductive step proceeds similarly using the identity 
$$(\cup_{T \in \binom{[i]}{j-1}} \cap_{t \in T} S_t)\cap ((\cup_{T \in \binom{[i]}{j}} \cap_{t \in T} S_t) \cup S_{i+1})  = \cup_{T \in \binom{[i+1]}{j}} \cap_{t \in T} S_t,$$
in each step. The first term of these terms appears in \eqref{eq:rewritten} and the second of these is the undesired term created in the previous step using the identity $(\cup_{T \in \binom{[i]}{j+1}}\cap_{t \in T}S_t) \cup( \cup_{T \in \binom{[i]}{j}}\cap_{t\in T}S_T) = \cup_{T \in \binom{[i]}{j}}\cap_{t\in T}S_T$.

Let $\eta$ be the sum $\iota_{S_1}+ \cdots +\iota_{S_m}$ of the indicator vectors of our initial multiset. Since the rewriting rules do not affect the sum of the indicator vectors in the collection, we can evaluate $\eta$ by summing the indicator vectors in \eqref{eq:rewritten} in the case $i=m$. It is clear that the $j$th term in \eqref{eq:rewritten} contains the $j+1$st term, for $j=1,\dots,m-1$. Thus, the largest coordinates of $\eta$ are supported on 
the $j=m$ term in \eqref{eq:rewritten}, i.e. on the intersection $S_1 \cap S_2 \cap \cdots S_m$. Since $\eta$ lies in the dominant Weyl chamber by assumption, we conclude that $S_1 \cap S_2 \cap \cdots S_m$ is an initial interval $[a]$. Continuing in this way, we conclude that each term in \eqref{eq:rewritten} is an initial interval, completing the proof. 
\end{proof}

Now we can prove Theorem~\ref{thm:flattening}.
\begin{proof}
Let $(T,\varphi)$ be a pseudotagged tensor diagram and $p$ a puncture. The weight of the function $[(T,\varphi)]$ at $p$ is the sum of indicator vectors $\iota_{\varphi([{\rm wt}(e)])}$. Composing with the birational Weyl group action, we may assume that  ${\rm wt}_p([T])$ lies in the dominant Weyl chamber closure for all punctures~$p$. We want to show that $[(T,\varphi)]$ is a linear combination of diagram invariants. 

For any pair of legs $e,e'$ incident to the same puncture, Lemma~\ref{lem:flattenproduct} says that we can replace the tagged tensor diagram $T$ by a linear combination of tensor diagrams in which the legs $e,e'$ are replaced by a pair of legs whose weights are indicator vectors $\iota_{\varphi(e) \cap \varphi(e')}$ and  
$\iota_{\varphi(e) \cup \varphi(e')}$. Moreover, each of the tensor diagrams in this linear combination agrees with $T$ outside of a small neighborhood of the puncture~$p$. 

This transformation on legs is exactly the rewriting rule from Lemma~\ref{lem:rewritingrule}. Thus, after sufficiently many applications of Lemma~\ref{lem:flattenproduct} we will be able to express $[(T,\varphi)]$ as a linear combination of tagged tensor diagrams $T'$ whose legs are tagged by initial subsets $\varphi(e) = [{\rm wt}(e)]$. That is, each of the tensor diagrams $T'$ is plainly tagged at $p$. As this recipe only changes the tensor diagrams nearby~$p$, we can continue this reduction process until we are plainly tagged at every puncture.
\end{proof}

\subsection{The spiral theorem}
Consider an $r$-tuple of tagged tensor diagrams $(T,\varphi_i)$ for $i=1,\dots,r$, different taggings of the same underlying tensor diagram~$T$. Suppose that the weights of the invariants $f_i:= [(T,\varphi_i)]$ are pairwise root-conjugate at some puncture~$p$ and coincide at all other punctures $p' \neq p$. Thus, there is a unique leg $e \in {\rm Legs}_p(T)$ such that $\varphi_i(e') = \varphi_j(e')$ for all $i,j \in [r]$ and all $e' \in {\rm Legs}(T) \setminus e$. The taggings $\varphi_i(e)$ and $\varphi_j(e)$ do not coincide.

Next, suppose that the underlying plain diagram invariant $[T]$ is a cluster variable. Then Conjecture~\ref{conj:clusterconjecture} predicts that the product $\prod_if_i$ is a cluster monomial. We have $\prod_i f_i = [\cup_i (T,\varphi_i)]$ but the diagram on the right hand side is pseudotagged rather than tagged. Conjecture~\ref{conj:yespunctures} predicts that this monomial should be a tagged invariant, moreover a tagged forest and tagged planar invariant. The spiral theorem partially confirms this. 

\begin{theorem}\label{thm:clusterflattening}
Let $f_i = [(T,\varphi_i)]$, for $i=1,\dots,r$, be taggings which are root-conjugate at exactly one puncture as above and let $e \in T$ be the unique leg at which these tagged diagrams disagree. Suppose that $T$ is a tree diagram. Then the following statements hold: 
\begin{itemize}
\item The monomial $\prod_{i=1}^r f_i$ is a tagged tree invariant. 
\item The monomial $\prod_{i=1}^r f_i$ can be computed by a tagged diagram which is the result of gluing $r$ many copies of the tensor diagram $T \setminus e$ to a tagged and planar diagram fragment. 
\end{itemize}
\end{theorem}

We elaborate on the second part of this theorem, which is slightly technical. We view the superposition diagram $\cup_{i=1}^r (T,\varphi_i)$ as $r$ many root-conjugate tagged legs plugged into several copies of the fragment $\eta := T \setminus e$. The second part of this theorem asserts that we can replace these root-conjugate legs by a more complicated tagged tensor diagram fragment which is planar. We think of this result as saying that the corresponding monomial is given by a diagram which is tagged and ``planar near $p$,'' although this is of course silly since {\sl any} diagram is planar in a very small neighborhood of a puncture. 

\begin{remark}
We would be happiest if we could strengthen the second part of this theorem to conclude that $\prod_i f_i$ is a tagged planar invariant by adding an appropriate planarization hypothesis on the underlying plain diagram $T$. There are two issues here: first, the diagram $T\setminus e$ might have legs at $p$, and these legs will cross the planar fragment near $p$ appearing in the second part of the theorem statement. Second, the union of several copies of $T \setminus e$ will typically have many self-crossings, and we need to argue that these self-crossing can be planarized. The first of these issues is not hard to handle: a straightforward generalization of the arborization move from \cite{FPII} is sufficient to planarize all crossings between the copies of $T \setminus e$ and the planar fragment introduced below. The second of these issues is a bit more subtle; we remove it in the case $k=3$ in Proposition~\ref{prop:kis3}.
\end{remark}

We illustrate Theorem~\ref{thm:clusterflattening} in Figure~\ref{fig:twoforms}, which schematically depicts three diagrams $D_{\rm super}$, $D_{\rm maelstorm}$, and $D_{\rm snail}$ appearing in the upper left, upper right, and bottom of the figure. Several copies of the tensor diagram fragment $T \setminus e$ enter from the top of each of these pictures. The first diagram $D_{\rm super}$ is pseudotagged and depicts the union $\cup_i (T,\varphi_i)$. The latter two diagrams are tagged. We will prove below that $[D_{\rm pseudo}] = [D_{\rm maelstrom}] = [D_{\rm snail}]$.

\begin{figure}[htbp]
\begin{tikzpicture}
\node at (-3,10.5) {\includegraphics[scale=.55]{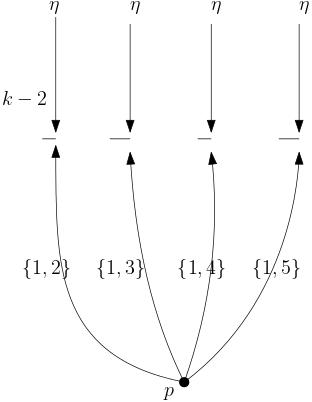}};
\node at (7.0,10) {\includegraphics[scale=.42]{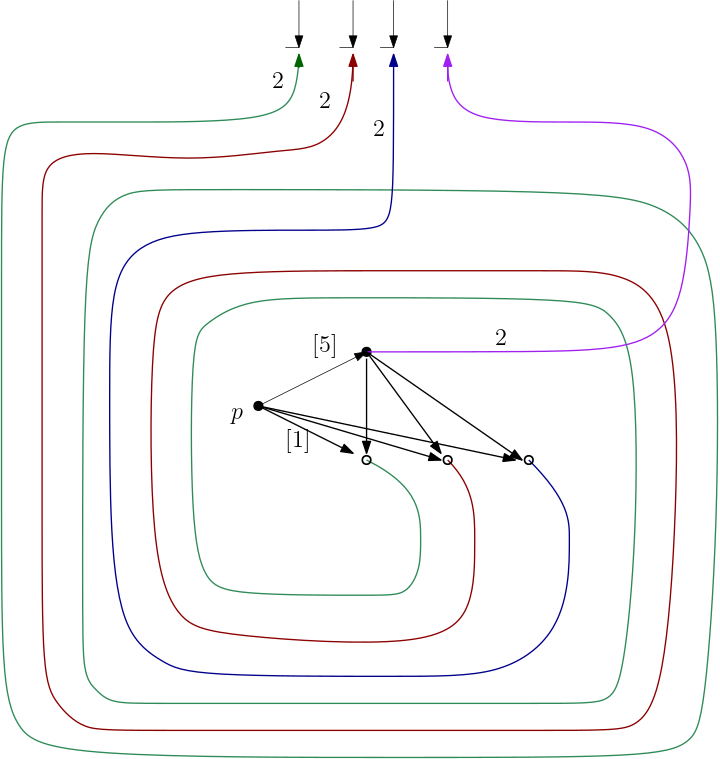}};
\node at (1.0,-.5) {\includegraphics[scale=.44]{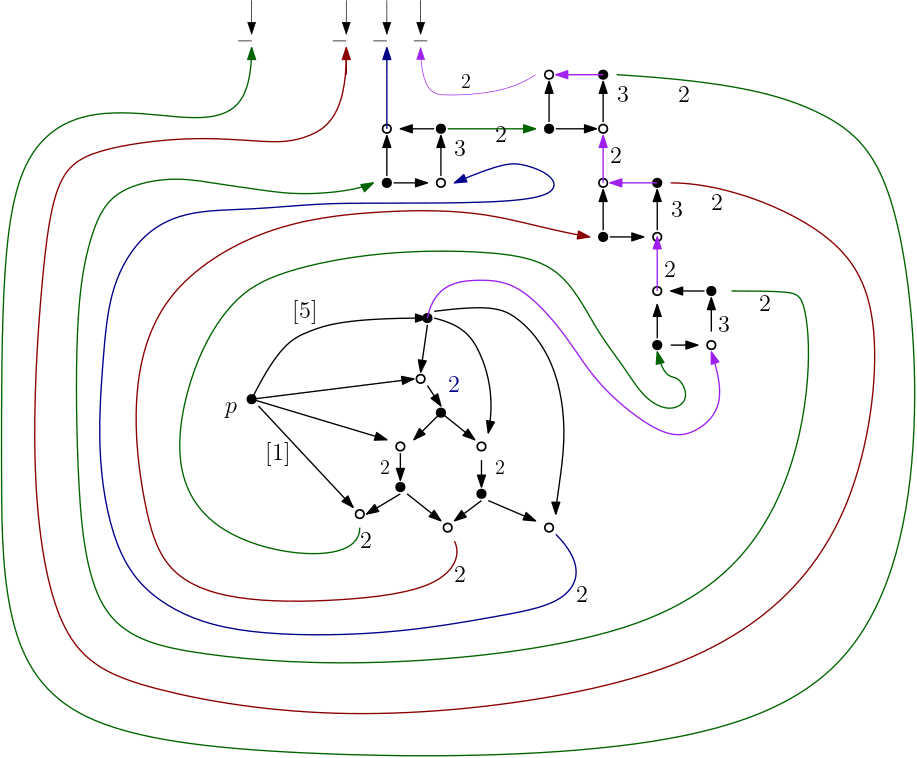}};
\end{tikzpicture}
\caption{The superposition, maelstrom, and snail shell fragmentsdiagrams from the proof of the spiral theorem. All three diagrams determine the same invariant when plugged into several copies of the tensor diagram fragment $\eta := T \setminus e$. \label{fig:twoforms}}
\end{figure}

\subsection{Proof of {Theorem~\ref{thm:clusterflattening}}}\label{subsec:webflattening}
Fix an affine flag $F$ and a matrix $u \in {\rm stab}(F)$. The following algebraic identity 
underlies the passage from $D_{\rm super}$ to $D_{\rm maelstrom}$ in Figure~\ref{fig:twoforms}. 
\begin{lemma}\label{lem:flattensplit}
Let $a,r \in \mathbb{N}$ with $a+r \leq k$ and let $\eta \in \bigwedge^{k-a-1}V$. Then 
\begin{equation}\label{eq:flattensplit}
\prod_{j=a}^{a+r-1}((s_{a+1} s_{a+2}\cdots s_j \cdot F)_{(a+1)} \wedge \eta )= F_{(a+r)} \cap (F_{(a)}\wedge\eta )\cap (F_{(a)}\wedge u^{-1}\eta )\cap \cdots \cap (F_{(a)} \wedge u^{2-r}\eta) \cap u^{1-r}\eta. 
\end{equation} 
And dually, when $\eta \in \bigwedge^{k-a-r+1}(V)$, one has 
\begin{equation}\label{eq:flattensplitdual}
\prod_{j=a+1}^{a+r}((s_{r-1} \cdots s_{j+1}s_j \cdot F)_{(a+r-1)} \wedge \eta)= F_{(a)} \wedge (F_{(a+r)} \cap \eta) \wedge (F_{(a+r)} \cap u^{-1}\eta) \wedge \cdots \wedge (F_{(a+r)} \cap u^{2-r}\eta) \wedge u^{1-r} \eta.
\end{equation} 
\end{lemma}

The left hand side of these equalities represents an $r$-tuple of root-conjugate legs in a pseudotagged superposition of tagged diagrams $\cup_i (T_i,\varphi)$.  The ground set of the root-conjugacy class is $[a+1,\dots,a+r]$. The two versions of the equality correspond to the two choices of sign $\varepsilon = \pm 1$ of the root-conjugacy class. 

The right hand side of either equality, when interpreted in terms of tensor diagrams, is computed by a tensor diagram fragment whose legs satisfy the containment condition for tagged diagrams. Thus, these equalities allow us to replace an $r$-tuple of root-conjugate legs at a puncture with an $r$-tuple of legs satisfying the containment condition, provided each of these legs is connected to the same tensor diagram fragment $T \setminus e$.

The application of the first equality in this lemma is illustrated in the passage from the upper left to the upper right diagram in Figure~\ref{fig:twoforms}, in the case $a=1$ and $r=4$.

\begin{proof} We prove \eqref{eq:flattensplit} by induction on $r \geq 1$; the proof of \eqref{eq:flattensplitdual} is very similar. For the base case, the right hand side should be interpreted as $F_{(a+1)} \cap u^{r-1}\eta = F_{(a+1)} \wedge u^0\eta$, so the identity is trivial. We 
abbreviate 
$$\zeta_r  = (F_{(a)}\wedge\eta )\cap (F_{(a)}\wedge u^{-1}\eta) \cap \cdots \cap (F_{(a)}\wedge u^{2-r}\eta)  \in \bigwedge^{k-r+1}(V)$$
so that the right hand side of \eqref{eq:flattensplit} takes the form $F_{(a+r)} \wedge (\zeta_r \cap u^{1-r}\eta)$. It is easy to argue that $F_{(a)} \wedge (\zeta_r \cap u^{1-r}\eta) = \zeta_{r+1}$. 

Assuming the formula for a given $r \geq 1$, we have
\begin{align*}
\prod_{j=a}^{a+r}(s_{a+1} \cdots s_j\cdot F)_{(a+1)} \wedge \eta &= 
\left(F_{(a+r)} \wedge (\zeta_r \cap u^{1-r}\eta) \right)\left((s_{a+1} \cdots s_{a+r}\cdot F)_{(a+1)} \wedge \eta\right)\\
&= F_{(a+r+1)} \cap( F_{(a)}\wedge (\zeta_r \cap u^{1-r}\eta) )\cap (u^{-1}-1)^r\eta \\
&= F_{(a+r+1)} \cap \zeta_{r+1}\cap (u^{-1}-1)^r\eta \\
&= F_{(a+r+1)} \cap \zeta_{r+1}\cap u^{-r} \eta
\end{align*}
establishing the inductive claim. The first of these equalities is the inductive hypothesis, the second is Lemma~\ref{lem:flattenproduct} when $S = [a] \cup \{r+1\}$,  the third is the relation between $\zeta_r$ and $\zeta_{r+1}$ we noted above, and the fourth can be argued as follows. We claim that if $i < r$, then $\zeta_{r+1} \cap u^{-i}\eta = 0$. Indeed, by reordering terms in the definition of  $\zeta_{r+1}$ we can compute this quantity by first performing $F_{(a)}u^{-i}\eta \cap u^{-i}\eta$ (since the first of these factors appears in $\zeta_{r+1}$). To compute $F_{(a)}u^{-i}\eta \cap u^{-i}\eta$ we need to shuffle $a+1$ vectors from the first factor over to the second factor, hence we must shuffle at least one vector from $u^{-i}\eta $ over, hence $F_{(a)}u^{-i}\eta \cap u^{-i}\eta = 0$. This completes the proof of the fourth equality. 
\end{proof}

Next we state a lemma which underlies the passage from the $D_{\rm maelstrom}$ to 
$D_{\rm snail}$ in Figure~\ref{fig:twoforms}. 

A type +1 {\sl maelstrom fragment} is a fragment of a plain tensor diagram which generalizes the upper right diagram in the figure and computes the right hand side of \eqref{eq:flattensplit}. It has a leg $e_0$ of weight $a+r$ and legs $e_1,\dots,e_r$ of weight $a$ at the puncture~$p$. It has edges $e_0 \to t(e_i)$ for $i \in [r]$, and dangling edges $e_i'$ of weight $a+1$, with $e'_i$ emanating from $t(e_i)$ and swirling around the puncture $i-1$ many times. Each of these dangling edges is then paired with a copy of the fragment $T \setminus e$ to create a tensor diagram whose invariant we call the {\sl maelstrom invariant}. There is a similar type $-1$ maelstrom fragment computing the right hand side of \eqref{eq:flattensplitdual}. 

The maelstrom fragment has two types of crossings: crossing between legs $e_i$ and $e_j$ and crossings between the dangling edges $e'_i$ and $e'_j$ swirling around the puncture. 
The former type of crossings can be planarized using the crossing removal relation; all but one of the terms on the right hand side of such a crossing relation vanish. A {\sl snail shell fragment} is the result of removing these crossings and also replacing each crossing between edges $e'_i$ and $e'_j$ by a square fragment as indicated by example in the passage from the right diagram to the bottom diagram in the figure. As above, we pair the dangling edges of the snail shell fragment with several copies of the fragment $T \setminus e$ to get a tensor diagram and corresponding {\sl snail shell invariant}.

\begin{lemma}\label{lem:seashelliseddy}
The snail shell invariant coincides with its corresponding maelstrom invariant. 
\end{lemma}
That is, the snail shell invariant also computes the right hand side of \eqref{eq:flattensplit} or \eqref{eq:flattensplitdual}.
\begin{proof}
We first argue that one can replace each of the crossings between dangling edges $e'_i$ and $e'_j$ by a square. One can subsequently planarize the crossings between legs using the inverse arborization moves, as we have claimed.  

Let $z \in \mathcal{M}$ determine an affine flag $F = F_p$ and monodromy $u = u_p$. Since $u$ fixes $F_{(a)}$, when we evaluate the maelstrom diagram on~$z$, the tensors flowing in to any of the spiraling crossings are of the form $v \wedge F_{(a)}$ and $w \wedge F_{(a)}$ for appropriate vectors $v,w \in V$. In such an instance, the right hand side of the crossing removal relation \eqref{eq:crossingremoval} only has only two terms because the other terms define the zero invariant: 
\begin{equation}\label{eq:crossingremoval2}
\begin{tikzpicture}
\draw [thick,    decoration={markings,mark=at position 1 with {\arrow[scale=1.7]{>}}},
    postaction={decorate},
    shorten >=0.4pt] (-1,-1)--(-.55,-.2);
\draw [thick] (-.55,-.2)--(.2,1);
\draw [thick,    decoration={markings,mark=at position 1 with {\arrow[scale=1.7]{>}}},
    postaction={decorate},
    shorten >=0.4pt] (.2,-1)--(-.25,-.2);
    \draw [thick] (-.25,-.2)--(-1.0,1);
\node at (-1.4,-.2) {\tiny $a+1$};
\node at (-1.2,-1.3) {\small $v \wedge F_{(a)} $};
\node at (.6,-.2) {\tiny $a+1$};
\node at (.4,-1.3) {\small $w \wedge F_{(a)} $};
\node at (1.8,-.2) {$=$};

\begin{scope}[xshift = 4.5cm]
\draw [thick,    decoration={markings,mark=at position 1 with {\arrow[scale=1.7]{>}}},
    postaction={decorate},
    shorten >=0.4pt] (-1,-1)--(-1,-.2);
\draw [thick] (-1,-.2)--(-1,1);
\draw [thick,    decoration={markings,mark=at position 1 with {\arrow[scale=1.7]{>}}},
    postaction={decorate},
    shorten >=0.4pt] (.2,-1)--(.2,-.2);
    \draw [thick] (.2,-.2)--(.2,1);
\node at (-1.4,-.1) {\tiny $a+1$};
\node at (-1.2,-1.3) {\small $v \wedge F_{(a)}$};
\node at (.6,-.1) {\tiny $a+1$};
\node at (.4,-1.3) {\small $w \wedge F_{(a)} $};
\node at (1.8,-.2) {$+$};
\end{scope}

\begin{scope}[xshift = 3cm]
\draw [thick,    decoration={markings,mark=at position 1 with {\arrow[scale=1.7]{>}}},
    postaction={decorate},
    shorten >=0.4pt] (4.4,-1)--(4.4,-.4);
\draw [fill= black] (4.4,-.3) circle [radius = .08];
\draw [thick,    decoration={markings,mark=at position 1 with {\arrow[scale=1.7]{>}}},
    postaction={decorate},
    shorten >=0.4pt] (4.4,-.25)--(4.4,.4);
\draw [fill= white] (4.4,.5) circle [radius = .08];
\draw [thick,    decoration={markings,mark=at position 1 with {\arrow[scale=1.7]{>}}},
    postaction={decorate},
    shorten >=0.4pt] (4.4,.58)--(4.4,1);
    \draw [thick] (4.4,1)--(4.4,1.2);
\draw [thick,    decoration={markings,mark=at position 1 with {\arrow[scale=1.7]{>}}},
    postaction={decorate},
    shorten >=0.4pt] (6.1,-1)--(6.1,-.4);
\draw [fill= white] (6.1,-.3) circle [radius = .08];
\draw [thick,    decoration={markings,mark=at position 1 with {\arrow[scale=1.7]{>}}},
    postaction={decorate},
    shorten >=0.4pt] (6.1,-.25)--(6.1,.4);
\draw [fill= black] (6.1,.5) circle [radius = .08];
\draw [thick,    decoration={markings,mark=at position 1 with {\arrow[scale=1.7]{>}}},
    postaction={decorate},
    shorten >=0.4pt] (6.1,.58)--(6.1,1);
    \draw [thick] (6.1,1)--(6.1,1.2);
\draw [thick,    decoration={markings,mark=at position 1 with {\arrow[scale=1.7]{>}}},
    postaction={decorate},
    shorten >=0.4pt] (4.4,-.3)--(6.0,-.3);
\draw [thick,    decoration={markings,mark=at position 1 with {\arrow[scale=1.7]{>}}},
    postaction={decorate},
    shorten >=0.4pt] (6.1,.5)--(4.5,.5);
\node at (3.8,-.8) {\tiny $a+1$};
\node at (3.8,.8) {\tiny $a+1$};
\node at (4.0,.05) {\tiny $a$};
\node at (6.6,-.8) {\tiny $a+1$};
\node at (6.6,.8) {\tiny $a+1$};
\node at (6.6,.05) {\tiny $a+2$};
\node at (5.25,-.5) {\tiny $1$};
\node at (5.2,.7) {\tiny $1$};  
\node at (4.2,-1.3) {\small $v \wedge F_{(a)}$};
\node at (6.4,-1.3) {\small $w \wedge F_{(a)} $};  
    \end{scope}
\end{tikzpicture}.
\end{equation}
In the rightmost diagram in this equation, the only nonzero contributions occur when the tensor $F_{(a)}$ flows to the edge of weight $a$ and the vector $v$ flows rightward along the edge of weight 1. That is, there are no choices on where the tensor $F_{(a)}$ must flow, and we can reduce to the case that $a=0$. 
Once we set $a=0$, the crossing removal relation becomes 
\begin{equation}\label{eq:crossingremoval3}
\begin{tikzpicture}
\draw [thick,    decoration={markings,mark=at position 1 with {\arrow[scale=1.7]{>}}},
    postaction={decorate},
    shorten >=0.4pt] (-1,-1)--(-.55,-.2);
\draw [thick] (-.55,-.2)--(.2,1);
\draw [thick,    decoration={markings,mark=at position 1 with {\arrow[scale=1.7]{>}}},
    postaction={decorate},
    shorten >=0.4pt] (.2,-1)--(-.25,-.2);
    \draw [thick] (-.25,-.2)--(-1.0,1);
\node at (-1.4,-.2) {\tiny $1$};
\node at (-1.2,-1.3) {\small $v  $};
\node at (.6,-.2) {\tiny $1$};
\node at (.4,-1.3) {\small $w $};
\node at (1.8,-.2) {$=$};

\begin{scope}[xshift = 4.5cm]
\draw [thick,    decoration={markings,mark=at position 1 with {\arrow[scale=1.7]{>}}},
    postaction={decorate},
    shorten >=0.4pt] (-1,-1)--(-1,-.2);
\draw [thick] (-1,-.2)--(-1,1);
\draw [thick,    decoration={markings,mark=at position 1 with {\arrow[scale=1.7]{>}}},
    postaction={decorate},
    shorten >=0.4pt] (.2,-1)--(.2,-.2);
    \draw [thick] (.2,-.2)--(.2,1);
\node at (-1.4,-.1) {\tiny $1$};
\node at (-1.2,-1.3) {\small $v $};
\node at (.6,-.1) {\tiny $1$};
\node at (.4,-1.3) {\small $w$};
\node at (1.8,-.2) {$+$};
\end{scope}

\begin{scope}[xshift = 3cm]
\draw [thick,    decoration={markings,mark=at position 1 with {\arrow[scale=1.7]{>}}},
    postaction={decorate},
    shorten >=0.4pt] (4.4,-1)--(4.85,-.75);
\draw [thick,    decoration={markings,mark=at position 1 with {\arrow[scale=1.7]{>}}},
    postaction={decorate},
    shorten >=0.4pt] (6.2,-1.0)--(5.75,-.75);
\draw [thick] (4.85,-.75)--(5.3,-.5);
\draw [thick] (5.75,-.75)--(5.3,-.5);
\draw [thick,    decoration={markings,mark=at position 1 with {\arrow[scale=1.7]{>}}},
    postaction={decorate},
    shorten >=0.4pt] (5.3,-.45)--(5.3,-.05);
\draw [thick] (5.3,-.05)--(5.3,.25);
\draw [thick,    decoration={markings,mark=at position 1 with {\arrow[scale=1.7]{>}}},
    postaction={decorate},
    shorten >=0.4pt] (5.3,.25)--(4.85,.5);
\draw [thick,    decoration={markings,mark=at position 1 with {\arrow[scale=1.7]{>}}},
    postaction={decorate},
    shorten >=0.4pt] (5.3,.25)--(5.75,.5);
\draw [thick] (5.75,.5)--(6.2,.75);
\draw [thick] (4.85,.5)--(4.4,.75);
    \draw [fill= white] (5.3,-.5) circle [radius = .08];
\draw [fill= black] (5.3,.25) circle [radius = .08];
\node at (4.2,-.8) {\tiny $1$};
\node at (4.2,.8) {\tiny $1$};
\node at (6.4,-.8) {\tiny $1$};
\node at (6.4,.8) {\tiny $1$};
\node at (5.5,.05) {\tiny $2$};
\node at (4.2,-1.3) {\small $v $};
\node at (6.4,-1.3) {\small $w $};  
    \end{scope}
\end{tikzpicture}.
\end{equation}
We refer to the three terms appearing in this relation as the crossing term, the resolve term, and the bridge term. The maelstrom versus snail shell diagrams amount to choosing either the crossing term or the bridge term at every crossing. We can express the snail shell invariant as a linear combination of tensor diagrams in which we either take the crossing term or the resolve term. Regardless of how we make these binary choices, the resulting tensor diagram will have $b$ edges which emanate from the leg of weight $[r]$ and spiral the puncture some number of times before meeting up with a copy of $T \setminus e$. The total number of times these edges spiral, namely  $0 + 1 + \cdots +r-1$, is conserved. 

In any of these terms, the resulting invariant is antisymmetric in the $r$ spiraling edges since they all emanate from the same black vertex. On the other hand, if any two of these edges spiral the same number of times around the puncture, then the corresponding invariant is also symmetric in these edges, which are all plugged into the same fragment $T \setminus e$. Thus, such an invariant is the zero invariant. Since the total amount of spiral around the puncture is conserved, one can see that in order for the $r$ legs to spiral a distinct number of times, we must take the crossing term (not the resolve term) at {\sl every} crossing, i.e. the maelstrom and snail shell diagrams determine the same invariant. 
\end{proof}

With these lemmas in hand, we can prove Theorem~\ref{thm:clusterflattening}.

\begin{proof}
Translating as needed by the birational Weyl group action, we may assume that ${\rm wt}_p \prod_i f_i$ lies in the closed dominant Weyl chamber at every puncture $p$.  Moreover, if $p$ is the puncture at which the ${\rm wt}_pf_i$'s are root-conjugate we may assume that the ground set \eqref{eq:groundsetofsimplex} of the pairwise root-conjugate vectors ${\rm wt}_pf_i$ takes the form $[a+1,a+r]$, and that $[a+1,c]$ is a block of the ordered set partition $\Pi_{{\rm wt}_p\prod_i f_i}$ for some $c \geq a+r$. Let $\varepsilon$ be the sign of the root-conjugacy class as in Lemma~\ref{lem:flag}. Using the translation \eqref{eq:nooflegs}, when $\varepsilon = +1$, one can see that $T$ has a unique leg $e$ of weight $a+1$ and all its other legs at $p$ either have weight at most~$a$ or at least $a+r$. A similar statement holds when $\varepsilon = -1$, but in this case the leg $e$ has weight $a+r-1$. Because we have translated to the dominant Weyl chamber, our goal is to show that $\prod_i f_i$ is a planar invariant and a forest invariant. We will argue this when $\varepsilon = +1$; the other case is similar. 

The $r$-fold superposition $\cup_{i=1}^rT$ has $r$ many copies of the leg $e$. We can tag these $r$ legs by the subsets $[a+1],[a]\cup \{2\},\dots, [a] \cup \{a+r\}$ (tagging all other legs by a fundamental weight) to obtain a pseudotagged tensor diagram whose invariant is $\prod_{i=1}^rf_i$. We denote this pseudotagged diagram by $(\cup_{i=1}^rT,\varphi)$. It is the upper left diagram in Figure~\ref{fig:twoforms}.

 We denote by $T_{\rm maelstrom}$ (resp. $T_{\rm snail}$) the plainly tagged diagram which results by replacing the $r$ many root-conjugate legs in $(\cup_{i}T,\varphi)$ by the maelstrom (resp. snail shell) fragment and plugging this fragment into $r$ many copies of the fragment $\eta := T \setminus e$. By Lemmas~\ref{lem:flattensplit} and \ref{lem:seashelliseddy} we have $[(\cup_iT,\varphi)] = [T_{\rm maelstrom}] = [T_{\rm snail}]$.

The diagram $T_{\rm maelstrom}$ is not a superposition of smaller diagrams and is without interior cycles. Thus, it is a tree diagram which computes the monomial $\prod_i f_i$. This proves the first assertion in the theorem. The snail shell fragment is planar and the diagram $T_{\rm maelstrom}$ is obtained by gluing this planar fragment to $r$ many copies  of $T \setminus e$. This proves the second assertion in the theorem. 
\end{proof}

\subsection{Undoing the spiral when $k=3$}\label{secn:FPconjecturesII}
We conclude this section by strengthening the above results to the case $k=3$, thereby extending the conjectures from \cite{FPII,FP} to surfaces with punctures. When $k=3$, the weight of a leg at a puncture is either 1 or 2, and any such leg can be tagged in three possible ways. 

We refer to \cite{FPII,FP} for a discussion of the arborization algorithm converting a planar ${\rm SL}_3$ web  diagram into its {\sl arborized form}. Conjecturally, this algorithm has the following properties when $\mathbb{S}$ has no punctures: a non-elliptic web diagram is a cluster variable if and only if its arborized form is a tree diagram; it is a cluster monomial if and only if its arborized form is a union of tree diagrams corresponding to the 
irreducible factors of the cluster monomial. We now extend these conjectures in the presence of punctures:

\begin{conjecture}\label{conj:kis3}
Let $[(T,\varphi)] \in \mathscr{A}(\mathcal{M})$ be a tagged non-elliptic ${\rm SL}_3$ web invariant. Then $[(T,\varphi)]$ is a cluster variable if and only if the arborized form of $T$ is a tree diagram. Moreover, $[(T,\varphi)]$ is a cluster monomial if and only $(T,\varphi)$ can 
be transformed to a superposition of tagged tree diagrams $(T_i,\varphi_i)$ by a sequence of arborization moves and the four moves in Figure~\ref{fig:extraarbs}. In this case, the various $[(T_i,\varphi_i)]$ are the irreducible factors appearing in $[(T,\varphi)]$.\end{conjecture} 

The four extra moves in the figure correspond to the passage from the snail shell fragment to the superposition of root-conjugate legs (i.e., from the bottom to the upper left diagram in Figure~\ref{fig:twoforms}). There are four moves corresponding to the four possible dosps  with a non-singleton block, namely $ab|c$, $a|bc$, $|123^+|$, and $|123^-|$ from left to right and top to bottom. 

\begin{figure}[ht]
\begin{tikzpicture}
\node at (0,0) {\includegraphics[scale=.75]{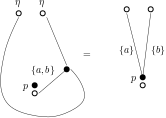}};
\node at (7,0) {\includegraphics[scale=.45]{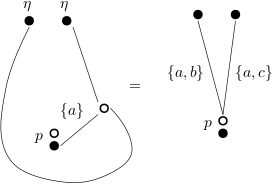}};
\node at (0,-3.7) {\includegraphics[scale=.75]{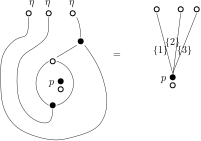}};
\node at (7,-3.7) {\includegraphics[scale=.75]{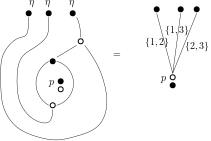}};
\end{tikzpicture}
\caption{Additional diagrammatic moves for tagged tensor diagrams when $G = {\rm SL}_3$. The black and white dots at the puncture~$p$ represent the first and second steps of the affine flag at $p$; the symbol $\eta$ is a tensor diagram fragment. 
 \label{fig:extraarbs}}
\end{figure}

\begin{proposition}\label{prop:kis3}
Continuing the setup of Theorem~\ref{thm:clusterflattening} in the special case that $k=3$, suppose that $T$ is the arborized form of a non-elliptic web invariant. Then $\prod_i f_i$ is a tagged web invariant. 
\end{proposition}

\begin{proof}
Suppose that $T$ is the arborized form a non-elliptic web diagram~$W$. Denote by $W_r$ the $r$-thickening of $W$ as defined in \cite[Definition 10.8]{FP}. Then $W_r$ is again a non-elliptic web invariant. Recall that $T_{{\rm } snail}$ is the result of gluing the $r$ dangling edges of the snail shell fragment to $\cup_i T$. We denote by $W_{{\rm } snail}$ the result of instead gluing the $r$ dangling edges of the snail shell fragment to $W_r$.

By the ``thickening theorem'' \cite[Theorem 10.9]{FP}, the $r$-fold superposition $\cup_i T$ planarizes into $W_r$ by a sequence of inverse arborization moves. We recall the proof of this theorem now and refer heavily to Figure 61 from {\sl loc. cit.} which underlies the proof. By assumption, we can planarize the tree diagram $T$ into the planar diagram $W$ by a sequence of inverse arborization moves (``planarization moves''). The thickening theorem says that we can planarize $\cup_i T$ into $W_r$ by ``thickening'' each of these planarization moves (each step in the sequence converting $T$ into $W$ corresponds to several moves turning $\cup_i T$ into $W_r$), working from the ``inside out.'' 

The aforementioned Figure 61 corresponds to the case $r=3$. The left part of that figure 
has six identical tensor fragments entering from the bottom, three identical fragments entering from the upper left, three more entering from middle left, three more from middle right, and three more from upper right. The six fragments entering the bottom are not connected to any of the other fragments. We take the opportunity here to point out that, contrary to what is written in the proof of the thickening theorem, one can planarize the left figure into the right figure without any square moves. (One should planarize the triple of crossings at the top of the figure, then planarize the nine crossings in the middle, and finally planarize the right triple of crossings.)

Note that $\cup_i T$ coincides with $T_{{\rm } snail}$ outside of a neighborhood of the puncture $p$ and likewise $W_r$ coincides with $W_{{\rm } snail}$ away from $p$. We claim that we can planarize the crossings in $T_{{\rm } snail}$ to obtain $W_{{\rm } snail}$ via the same sequence of steps which planarizes the crossings in $\cup_i T$ to obtain $W_r$. Indeed, note that the distinguished edge $e \in {\rm Legs}_p(T)$ is the unique leg of its weight at $p$. Thus, when we carry out the steps outlined in the previous two paragraphs, the fragment containing $e$ appears in one of the triples entering from the left or right (not from the bottom). Therefore, the planarization moves which turn the left diagram in Figure 61 into the right diagram are unaffected if we replace the three copies of the edge $e$ by the three dangling edges in the seashell fragment. 

Let $T' \subset T$ be a maximal planar subtree containing the distinguished edge $e$ (start with the edge $e$ and grow the tree one edge at a time provided the added edge does not cross any previously added edges). When we superimpose $r$ copies of the tree $T$, we create self-crossings between the copies of $T'$. One may check that it is possible to planarize these self-crossings via arborization moves in which $e$ never appears in the bottom of the arborization fragment. 

If we first perform the planarizaton moves from the previous paragraph and subsequently perform those from two paragraphs previous, then we planarize $\cup_i T$ into $W_r$ without the distinguished edge $e$ ever appearing in the bottom of the arborization fragment. Thus, replacing the $r$ distinguished edges $e$ with the snail shell fragment and performing the same sequence of planarization moves, we transform $T_{{\rm } snail}$ into $W_{{\rm } snail}$.
\end{proof}

\section{Examples of finite mutation type}\label{secn:fmt}
We study the $\mathscr{A}(\mathcal{M})$'s of finite mutation type. Our first proposition classifies these. 
\begin{proposition}
Suppose that $\mathbb{S}$ has at least one puncture and $k \geq 3$. Then the cluster algebra $\mathscr{A}_{{\rm SL}_k,\mathbb{S}}$ has finite mutation type if and only $\mathbb{S}$ is a once-punctured bigon and $k=3$. Likewise, 
$\mathscr{A}'_{{\rm SL}_k,\mathbb{S}}$ has finite mutation type if and only if either $\mathbb{S}$ is a once-punctured bigon and $k=3,4$ or $\mathbb{S}$ is a once-punctured triangle and $k=3$.
\end{proposition}

\begin{proof}
We call a triangle {\sl fully interior} if it has three distinct sides none of which is a boundary arch. When a fully interior triangle appears in a triangulation, it contributes a copy of the quiver $Q_k$ with all of the vertices considered mutable vertices and with none of its vertices identified. 

The quiver $Q_4$ is mutation infinite, and every $Q_k$ when $k \geq 4$ contains this quiver as a  full subquiver. The quiver $Q_3$ is mutation finite, but if we glue this quiver to that for an adjacent triangle (regardless of whether zero, one, or two sides of this adjacent triangle are boundary intervals carrying frozen variables), the resulting quiver is mutation infinite. Any fully interior triangle appearing in a triangulation will have such an adjacent triangle, so we can rule out the $k=3$ cases with a fully interior triangle as well. 

Recall that $\mathbb{S}$ has at least one puncture. One can check that any surface besides a once-punctured $n$-gon has a fully interior triangle, and from the preceding paragraph therefore admits a full subquiver of infinite mutation type. So it remains to consider the cases $\mathbb{S } = D_{n,1}$ for some $n \geq 2$.

One may check that the result of gluing three copies of $Q_3$, with the right edge of the first copy glued to the left edge of the second copy and the right edge of the second copy glued to the left edge of the third copy, has infinite mutation type. The same is true for gluing three copies of $Q_3^1$, and for gluing two copies of $Q_4$ along a shared edge. One can also check that 
$\mathscr{A}({\rm SL}_k,D_{2,1})$ has infinite type when $k=4$ hence when $k \geq 4$. In a similar way one checks that 
$\mathscr{A}'({\rm SL}_k,D_{3,1})$ has infinite mutation type when $k \geq 4$ and $\mathscr{A}'({\rm SL}_k,D_{2,1})$ has infinite mutation type when $k \geq 5$. 
\end{proof}

\begin{remark}
The cluster algebra $\mathscr{A}({\rm SL}_2,\mathbb{S})$ always has finite mutation type. We assumed that $\mathbb{S}$ has at least one puncture because that is our focus in the present paper. If we relax this assumption, there are additional finite mutation type examples arising when $k$ is small and $\mathbb{S}$ is an $n$-gon. For example, $\mathscr{A}'({\rm SL}_3,D_{n,0})$ has finite mutation type when $n \leq 9$. We are not aware of finite mutation type examples outside of those listed in the above proposition and cases when $\mathbb{S}$ is an $n$-gon, but we did not thoroughly rule these out. 
\end{remark}

One may verify that the cluster algebras $\mathscr{A}'({\rm SL}_k,D_{3,1})$ and $\mathscr{A}({\rm SL}_k,D_{2,1})$ are related by a quasi cluster isomorphism as predicted in Remark~\ref{rmk:subtle}. This quasi  cluster isomorphism respects the notions of being an arborizable web invariant in either side, and preserves the weight of cluster variables at punctures. Thus, we only need to understand the three finite mutation $\mathscr{A}'$-version examples listed in the above proposition. We study these three examples in the remainder of this section. 

The cluster algebra $\mathscr{A}'({\rm SL}_3,D_{2,1})$ has finite cluster type~$D_4$. The other two examples both have the same cluster type as 
$\mathscr{A}({\rm SL}_2,S_{0,4})$, also known as elliptic $D_4$ type. 

Our next proposition summarizes our results in this section: 
\begin{proposition}\label{prop:trueinfmt}
Conjectures \ref{conj:Pclusterconjecture},  \ref{conj:yespunctures}, and \ref{conj:clusterconjecture} hold for the cluster algebras $\mathscr{A}'({\rm SL}_3,D_{n,1})$ when $n=2,3$. Conjecture~\ref{conj:Pclusterconjecture} holds for $\mathscr{A}'({\rm SL}_4,D_{2,1})$, and 
Conjectures \ref{conj:yespunctures} and \ref{conj:clusterconjecture} can be verified through an explicit finite computation. 
\end{proposition}

Before outlining the proof of this proposition, we need to introduce the following ingredient.  Both of the elliptic $D_4$ examples have infinitely many clusters. To understand them, we introduce an appropriate group of quasi cluster automorphisms of the cluster algebra, modulo which there are only finitely many clusters. Provided we also check that this group preserves the assertions in our conjectures, this reduces our conjectures to a finite verification.

We number the boundary points of the once-punctured $n$-gon in cyclic order. A choice of point in the moduli space $\mathcal{M}$ affords us a vector $v_i$ at the $i$th boundary point, as well as an affine flag $F$ at the puncture $p$ and a monodromy $u$ around~$p$. 

We denote by $\rho$ denote the {\sl cyclic shift} cluster automorphism induced by rotating boundary points $i \mapsto i+1 \mod n$. We let $\tau$ denote a fixed reflection in the $n$-gon, so that $\langle\rho,\tau \rangle$ is a dihedral group. 

The cluster algebra has a canonical initial seed since there is a unique taut triangulation of a once-punctured $n$-gon. We use the notation $A,B,C$ for the initial cluster variables associated to the arc connecting the first boundary point to the puncture, so that the initial cluster variable $A$ has weight $a$ at the puncture in multiplicative notation, the initial cluster variable $B$ has weight $ab$, and the initial cluster variable $C$ has weight $abc$. Note that $C$ only is present when $k=4$. The cluster variables associated to the arc connecting the second boundary point to $p$ are cyclic shifts of these, namely $\rho^*A$, $\rho^*B$, and $\rho^*C$. And so on. 

\begin{definition}\label{defn:braidgroup}
We define a birational automorphism $\sigma \in {\rm Bir}(\mathcal{A}_{{\rm SL}_3,D_{3,1}})$ as follows. The automorphism changes the decorations at boundary points according to 
\begin{equation}\label{eq:sigmadef}
v_1 \mapsto v_2 \hspace{.5cm} v_2 \mapsto (v_1 \wedge v_2) \cap (v_3 \wedge u^{-1}v_1) \hspace{.35cm} v_3 \mapsto v_3 
\end{equation}
while preserving the local system and the affine flag $F$. 

We have similarly a birational automorphism $\sigma \in {\rm Bir}(\mathcal{A}_{{\rm SL}_4,D_{2,1}})$ defined via 
$$v_1 \mapsto v_2 \hspace{.5cm} v_2 \mapsto (v_1 \wedge v_2) \cap (u^{-1}v_1 \wedge u^{-1}v_2 \wedge u^{-2}v_1) . $$
\end{definition}

\begin{lemma}\label{lem:sigmadef}
The transformation $\sigma$ is a quasi cluster automorphism of the corresponding cluster algebra. It acts on the cluster algebra while preserving the sets of diagram invariants, web invariants, tree invariants, and forest invariants. 
\end{lemma}
\begin{proof}
One checks that $\sigma$ is a quasi cluster automorphism by providing an explicit sequence permutation-mutation sequence which matches Definition~\ref{defn:braidgroup}.  
The required sequences are 
\begin{align}
&(B\,  \rho^2(A) \, \rho( A) \, \rho( B))\circ  \mu_{B}\mu_{\rho A} \text{ for the once-punctured triangle, $k=3$} \\
&(A\,  \rho(A) \, \rho( B) )\circ (C\,  \rho(C) )\circ \mu_{C}\mu_{\rho A}\mu_{\rho B} \mu_C  \text{ for the once-punctured bigon, $k=4$}.
\end{align}
In both cases, we wrote the required permutation of vertices in cycle notation. 

If $T$ is a tensor diagram, then one can compute $\sigma_*([T])$ by ``plugging in'' the boundary edges of $T$ to a tensor-diagrammatical gadget which computes the formulas \eqref{eq:sigmadef}. This plugging in process does not create any interior cycles on boundary vertices, so $\sigma$ preserves the sets of forest and tree invariants.  The crossings which are created by plugging in can be planarized using the crossing removal relation (only one of the terms in the crossing removal relation survives), so $\sigma$ also preserves the set of web invariants. 
\end{proof}

Now we sketch the proof of Proposition~\ref{prop:trueinfmt} with further details filled in on a case by case basis below.

\begin{proof}
In all three cases, there are only finitely many $P$-clusters and the Conjecture \ref{conj:Pclusterconjecture} can be verified by hand and by inspection. We have listed these $P$-clusters up to $W$-action in \eqref{eq:Pclusters}, \eqref{eq:Pclusters2}, and \eqref{eq:Pclusters3}. 

In the finite type example, we compute below all cluster variables and clusters explicitly as tagged tensor diagram invariants. One can see the validity of our two cluster combinatorics conjectures by inspection. 

In the two elliptic $D_4$ examples, consider the group generated by the quasi cluster automorphism $\sigma$ identified above, together with the cyclic shift map $\rho$, the duality map $\ast$, and the Weyl group action at the puncture. This group acts on the cluster algebra by (quasi) cluster automorphisms and preserves the sets of tagged web invariants and tagged forest invariants. 

We argue in both cases elliptic $D_4$ examples that there are only finitely many clusters modulo the action of this group. When $k=3$ example, we identify explicitly these finitely many clusters and argue the following extra steps. First, every cluster monomial in this finite list of clusters is a tagged web invariant and a tagged forest invariant. Second, if $x,y$ are cluster variables in this finite list, and if the weights of $x$ and $y$ at the puncture are not $W$-sortable, then $x$ and $y$ are in fact root-conjugate taggings of the same underlying tensor diagram. Third, if $x$ appears in our finite list, then $x$ is cluster compatible with any of its root-conjugate taggings. (It is not important that a cluster witnessing this fifth assertion appears in our finite list.) The first of these statements 
proves Conjecture~~\ref{conj:yespunctures} while the second and third statements proves Conjecture~\ref{conj:clusterconjecture}.
\end{proof}

\begin{remark}
It would be possible to carry out the same three steps alluded to in the last paragraph of the above proof in the case of $\mathscr{A}'({\rm SL}_4,D_{2,1})$, thereby verifying our cluster combinatorics conjectures for this cluster algebra. The number of clusters which needed to be explicitly computed was fairly large so we did not carry this out. 
\end{remark}

\subsection{The case of $\mathscr{A}'({\rm SL}_3,D_{2,1})$}
This is a type $D_4$ cluster algebra with two frozen variables, $16$ cluster variables, and $50$ clusters. We depict the underlying tensor diagrams (up to dihedral and $W$- action) in Figure~\ref{fig:D4diagrams}. 
The diagrams $B$ and $A$ both can be tagged in three ways and have two dihedral images. The diagram $x$ has weight zero at the puncture and has four dihedral images. The two dihedral images of the last diagram $f$ are the frozen variables. We have drawn $x$ in its tree form but we can obtain its planar form by applying the crossing removal relation and noting that one of the two terms vanishes because it has a boundary 2-cycle. 

Each of the four variables $\{A,B,\rho(A),\rho(B)\}$ is compatible with exactly two out of the four cluster variables $\{x,\rho(x),\tau(x),\tau\rho(x)\}$ in both senses of compatibility (namely, cluster compatibility and also the compatibility notion that two web invariants are compatible when their product is again a web invariant). Specifically $B$ is compatible with $\rho(x)$ and $\rho(\tau(x))$, $\rho(B)$ is compatible with $x$ and $\tau(x)$,
$A$ is compatible with $x$ and $\rho(\tau(x))$, and $\rho(A)$ is compatible with $\rho(x)$ and $\tau(x)$. 

Figure~\ref{fig:D4} shows all 50 clusters in this cluster algebra, grouped in concentric circle ``levels'' with 6, 12, 24, 8 clusters respectively. Equation \eqref{eq:Pclusters} lists the $P$-clusters. 
Clusters in the outermost level have $P$-cluster in the first row of \eqref{eq:Pclusters} up to $W$-action, those in the next level have $P$-cluster in rows two or three of \eqref{eq:Pclusters} up to $W$-action, those in the next level have $P$-cluster in rows four or five, and those in the innermost have $P$-cluster in the last two rows. 

Mutation from the first to second level is an instance of the flattening relation \eqref{eq:flattenproduct} with $\ell(S) = 1$, e.g exchanging $\rho(B)$ for $s_2(B)$. Mutating from the second to the third level is an instance of the skein relations, e.g. exchanging $\rho(A)$ for $\rho \tau (x)$. There are also mutations within the third level which implement the dosp mutation $12|3 \leftrightarrow 12|3$. Mutation from the third to fourth level is an instance of $\ell(S) =2$ flattening relation followed by additional skein relation, e.g. exchanging $A$ for $s_2s_1(B)$. Specifically, the flattening relation yields a linear combination $[T]-2[T']+[T'']$ after applying the binomial theorem to $(u^{-1}-1)^2$. One of the invariants $[T'']$ is equal to 0. The tensor diagram $[T]$ has a self-crossing and can be expressed as $[T] = [T''']+3[T']$ using the crossing removal relation and the fact that ${\rm tr}(u) =3$. The right hand side of the flattening relation then becomes a sum of two terms $[T''']+[T']$
 and these two terms match the two terms appearing the exchange relation. Finally, there are mutations within the fourth level exchanging two dihedral images of $x$ (these are consequences of the skein relations).

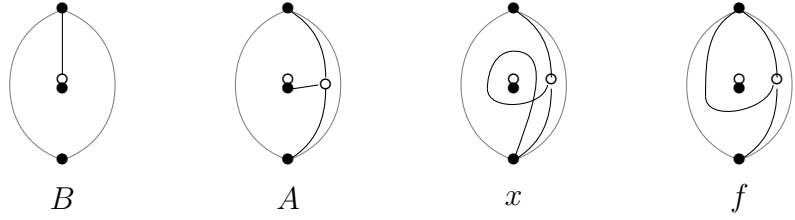
\begin{figure}
\begin{tikzpicture}
\draw [gray, out = 160, in = -90] (0,0) to (-.7,1);
\draw [gray, out = 90, in = 200] (-.7,1) to (0,2);
\draw [gray, out = 20, in = -90] (0,0) to (.7,1);
\draw [gray, out = 90, in = -20] (.7,1) to (0,2);
\draw (0,2)--(0,1.14);
\node at (0,2) {$\bullet$};
\node at (0,0) {$\bullet$};
\node at (0,1.06) {$\circ$};
\node at (0,.94) {$\bullet$};
\node at (0,-.5) {$B$};

\draw [gray, out = 160, in = -90] (3,0) to (2.3,1);
\draw [gray, out = 90, in = 200] (2.3,1) to (3,2);
\draw [gray, out = 20, in = -90] (3,0) to (3.7,1);
\draw [gray, out = 90, in = -20] (3.7,1) to (3,2);
\draw [out = -30, in = 90] (3,2) to (3.5,1.1);
\draw [out = 30, in = -90] (3,0) to (3.5,.95);
\draw (3,.94)--(3.4,1.0);
\node at (3,2) {$\bullet$};
\node at (3,0) {$\bullet$};
\node at (3,1.06) {$\circ$};
\node at (3,.94) {$\bullet$};
\node at (3,-.5) {$A$};
\node at (3.5,1.0) {$\circ$};

\draw [gray, out = 160, in = -90] (6,0) to (5.3,1);
\draw [gray, out = 90, in = 200] (5.3,1) to (6,2);
\draw [gray, out = 20, in = -90] (6,0) to (6.7,1);
\draw [gray, out = 90, in = -20] (6.7,1) to (6,2);
\draw [out = -30, in = 90] (6,2) to (6.5,1.1);
\draw [out = 30, in = -90] (6,0) to (6.5,.95);
\draw [out = -100, in = -90] (6.45,1.0) to (5.65,.95);
\draw [out = 90, in = 180] (5.65,.95) to (6.0,1.45);
\draw [out = 0, in = 70] (6.0,1.45) to (6,0);
\node at (6,2) {$\bullet$};
\node at (6,0) {$\bullet$};
\node at (6,-.5) {$x$};
\node at (6,1.06) {$\circ$};
\node at (6,.94) {$\bullet$};
\node at (6.5,1.06) {$\circ$};
\draw [gray, out = 160, in = -90] (9,0) to (8.3,1);
\draw [gray, out = 90, in = 200] (8.3,1) to (9,2);
\draw [gray, out = 20, in = -90] (9,0) to (9.7,1);
\draw [gray, out = 90, in = -20] (9.7,1) to (9,2);
\draw [out = -30, in = 90] (9,2) to (9.5,1.1);
\draw [out = 30, in = -90] (9,0) to (9.5,.95);
\draw [out = -100, in = -90] (9.45,1.0) to (8.55,.85);
\draw [out = 90, in = 205] (8.55,.85) to (9.0,2);
\node at (9.5,1.06) {$\circ$};
\node at (9,2) {$\bullet$};
\node at (9,0) {$\bullet$};
\node at (9,-.5) {$f$};
\node at (9,1.06) {$\circ$};
\node at (9,.94) {$\bullet$};
\end{tikzpicture}
\caption{The tensor diagrams underlying mutable and frozen variables in $\mathscr{A}'({\rm SL}_3,D_{2,1})$, which has finite cluster type $D_4$.
\label{fig:D4diagrams}}
\end{figure}

\subsection{The case of $\mathscr{A}'({\rm SL}_3,D_{3,1})$}
This cluster algebra has the same cluster type as the cluster algebra associated to a four-punctured sphere $S_{0,4}$. The underlying quiver mutation class consists of four quivers. There are six combinatorial types of tagged triangulations cf.~\cite{BMRV}. We use the nomenclature for these six combinatorial types given in {\sl loc. cit.} and note that 
types $I$ and $VI$ have the same underlying quiver, as do types II and V. We distinguish between 
{\sl unlabeled} and {\sl labeled} triangulations -- in the former we consider the arcs in a triangulation as a set while in the latter we think of these arcs as an ordered tuple. Any two tagged unlabeled triangulations of the same combinatorial type are related by an element of the tagged mapping class group ${\rm MCG}^\bowtie(S_{0,4})$.

We denote by $T_0$ be the following labeled triangulation of $S_{0,4}$, where we use $p_1,\dots,p_4$ to denote the four punctures and think of the sphere as $\mathbb{R}^2 \cup \{\infty\}$. 
\begin{equation}
\begin{tikzpicture}
\node at (-2,0) {$T_0 = $};
\node at (0,0) {\small $p_4$};
\node at (90:1.25cm) {\small $p_1$};
\node at (-30:1.25cm) {\small $p_2$};
\node at (210:1.25cm) {\small $p_3$};
\draw (-30:.2cm)--(-30:1.05cm);
\draw (210:.2cm)--(210:1.05cm);
\draw (90:.2cm)--(90:1.05cm);
\draw (105:1.15cm)--(195:1.15cm);
\draw (225:1.15cm)--(-45:1.15cm);
\draw (75:1.15cm)--(-15:1.15cm);
\node at (35:1.05cm) {\tiny $6$};
\node at (-90:1.05cm) {\tiny $2$};
\node at (145:1.05cm) {\tiny $4$};
\node at (80:.75cm) {\tiny $3$};
\node at (-45:.75cm) {\tiny $5$};
\node at (195:.75cm) {\tiny $1$};
\node at (2.5,.8) {$1: \, A$};
\node at (2.63,.3) {$3: \, \rho A$};
\node at (2.7,-.2) {$5: \, \rho^2 A$};
\node at (4.2,.8) {$2: \, B$};
\node at (4.33,.3) {$4: \, \rho B$};
\node at (4.4,-.2) {$6: \, \rho^2 B$};
\end{tikzpicture}
\end{equation}
To the right, we have identified the 6 initial cluster arcs for $S_{0,4}$ with the 6 initial cluster variables for $\mathscr{A}'({\rm SL}_3,D_{3,1})$  in such a way that the mutable parts of the quivers are identified. We denote by $T_1$ the result of notching all arcs in $T_0$ at the puncture $p_1$.

Fixing the above matching of initial seeds, we obtain a bijection between cluster variables (resp. clusters) for $\mathscr{A}'({\rm SL}_3,\mathbb{D}_{3,1})$ and tagged arcs (resp. tagged triangulations) in $S_{0,4}$.  If $\Sigma$ is a seed for $\mathscr{A}'({\rm SL}_3,D_{3,1})$ we denote by $T(\Sigma)$ its corresponding tagged triangulation. In particular, every cluster for $\mathscr{A}'({\rm SL}_3,D_{3,1})$ has one of the six aforementioned combinatorial types.

Figures~\ref{fig:bijection} and \ref{fig:bijection2} together illustrate this bijection between cluster variables in specific instances. The bijection is equivariant with respect to the group of dihedral symmetries of $D_{3,1}$, where in the case of $S_{0,4}$ we view the punctures  $p_1,p_2,p_3$ as the vertices of a triangle whose puncture is $p_4$. We have broken the two figures up so that the ${\rm SL}_3$ tensor diagrams appearing in Figure~\ref{fig:bijection} have nonzero weight at the puncture (hence, $W$ acts on them), while those in Figure~\ref{fig:bijection2} have zero weight. We computed the tensor diagrams in this table by performing mutations using skein relations and the flattening relation. We omit the details.

\begin{figure}
\begin{tikzpicture}
\draw (210:0cm)--(210:1.25cm);
\draw [green] (-30:1.25cm)--(90:1.25cm);
\draw [red](-.08,.08)--(205:1.23cm);
\node [rotate = 35, green] at (80:1.0cm) {$\bowtie$};
\node [rotate = 35, green] at (-20:1.0cm) {$\bowtie$};
\node [rotate = -60, red] at (-.202,.02) {$\bowtie$};
\node [rotate = 0]at (-.35,-.4) {\tiny $A$};
\node [rotate = 30, red] at (-.7,.1) {\tiny $s_1(A)$};
\node [rotate = -60, green] at (.85,.35) {\tiny $s_1s_2(A)$};
\node at (0,0) {\small $\bullet$};
\node at (90:1.25cm) {\small $\bullet$};
\node at (-30:1.25cm) {\small $\bullet$};
\node at (210:1.25cm) {\small $\bullet$};
\begin{scope}[xshift = 3.7cm]
\draw (-30:1.25cm)--(210:1.25cm);
\draw [green](.1,0)--(.1,1.25);
\draw [red](-.1,0)--(-.1,1.25);
\node [rotate = 0, green] at (.1,.2) {$\bowtie$};
\node [rotate = 0,red] at (-.1,1.1) {$\bowtie$};
\node [rotate = 0,green] at (.1,1.1) {$\bowtie$};
\node [rotate = 0]at (0,-.8) {\tiny $B$};
\node [rotate = 0, red] at (-.6,.4) {\tiny $s_2(B)$};
\node [rotate = 0, green] at (.6,.4) {\tiny $s_2s_1(B)$};
\node at (0,0) {\small $\bullet$};
\node at (90:1.25cm) {\small $\bullet$};
\node at (-30:1.25cm) {\small $\bullet$};
\node at (210:1.25cm) {\small $\bullet$};
\end{scope}
\begin{scope}[xshift = 7.4cm]
\draw [out = 135, in = 90] (90:1.25) to (-1.5,-.5);
\draw [out = -90, in = -90] (-1.5,-.5) to (0,0);
\draw [red, out = 115, in = 90] (90:1.25) to (-1.7,-.5);
\draw [red, out = -90, in = -90] (-1.7,-.5) to (0,-.3);
\draw [red](0,-.3) to (0,0);
\draw [green, out = 90, in = 180] (210:1.25) to (0,.2);
\draw [green, out = 0, in = 90] (0,.2) to (-30:1.25);
\node [rotate = 0, green] at (-1.05,-.4) {$\bowtie$};
\node [rotate = 0, green] at (1.05,-.4) {$\bowtie$};
\node [rotate = 0, red] at (0,-.35) {$\bowtie$};
\node at (0,0) {\small $\bullet$};
\node at (90:1.25cm) {\small $\bullet$};
\node at (-30:1.25cm) {\small $\bullet$};
\node at (210:1.25cm) {\small $\bullet$};
\node at (-.75,.85) {\tiny $A'$};
\node [rotate = 45, red] at (-1.20,1.25) {\tiny $s_1(A')$};
\node [green] at (.8,.35) {\tiny $s_1s_2(A')$};
\end{scope}
\begin{scope}[xshift = 11.1cm]
\draw [out = 210, in = 90] (0,1.25) to (-.34,-.14);
\draw [out = -90, in = 210] (-.34,-.14) to (-30:1.25);
\draw [red,out = 15, in = 0] (0,0) to (.2,1.65);
\draw [red,out = 180, in = 60] (.2,1.65) to (210:1.25);
\draw [green,out = 5, in = 0] (.15,0) to (.3,1.85);
\draw [green,out = 180, in = 80] (.3,1.85) to (210:1.25);
\node [rotate = -45, green] at (.3,.05) {$\bowtie$};
\node [rotate = -5, green] at (-1.05,-.2) {$\bowtie$};
\node [rotate = -10, red] at (-.95,-.45) {$\bowtie$};
\node at (0,0) {\small $\bullet$};
\node at (90:1.25cm) {\small $\bullet$};
\node at (-30:1.25cm) {\small $\bullet$};
\node at (210:1.25cm) {\small $\bullet$};
\node at (-.5,-.65) {\tiny $B'$};
\node [red] at (1.0,.15) {\tiny $s_2(B')$};
\node [rotate = 0,green] at (1.6,.5) {\tiny $s_2s_1(B')$};
\end{scope}
\begin{scope}[yshift = 3.1cm]
\draw (0,.1)--(.45,.45);
\draw [out = 90, in = 0] (.5,.59) to (0,1.25);
\draw [out = 100, in = -10] (-30:1.25) to (.55,.5);
\node at (.75,.75) {\tiny $A$};
\node at (0,.08) {\small $\bullet$};
\node at (0,-.08) {\small $\circ$};
\node at (90:1.25cm) {\small $\bullet$};
\node at (-30:1.25cm) {\small $\bullet$};
\node at (210:1.25cm) {\small $\bullet$};
\node at (.5,.5) {\small $\circ$};
\end{scope}
\begin{scope}[xshift = 3.7cm, yshift = 3.1cm]
\draw (0,1.25)--(0,.12);
\node at (.2,.75) {\tiny $B$};
\node at (0,-.08) {\small $\bullet$};
\node at (0,.08) {\small $\circ$};
\node at (90:1.25cm) {\small $\bullet$};
\node at (-30:1.25cm) {\small $\bullet$};
\node at (210:1.25cm) {\small $\bullet$};
\end{scope}
\begin{scope}[xshift = 7.4cm, yshift = 3.1cm]
\node at (0,.5) {\small $\circ$};
\draw (0,.45)--(0,.12);
\draw [out = 90, in = 180] (210:1.25) to (-.1,.5);
\draw [out = 90, in = 0] (-30:1.25) to (.1,.5);
\node at (0,.08) {\small $\bullet$};
\node at (0,-.08) {\small $\circ$};
\node at (90:1.25cm) {\small $\bullet$};
\node at (-30:1.25cm) {\small $\bullet$};
\node at (210:1.25cm) {\small $\bullet$};
\node at (.5,.7) {\tiny $A'$};
\end{scope}
\begin{scope}[xshift = 11.1cm, yshift = 3.1cm]
\node at (-.5,.3) {\small $\circ$};
\node at (.6,-.35) {\small $\circ$};
\node at (.5,.3) {\small $\bullet$};
\draw (.1,.12)--(.5,.3);
\draw [out = 200, in = 145] (0,1.25) to (-.55,.35);
\draw [out = 110, in = 180] (210:1.25) to (-.58,.28);
\draw [out = 80, in = 90] (-.45,.3) to (.5,.3);
\draw [out = 20, in = -100] (210:1.25) to (.6,-.4);
\draw (.65,-.4)--(-30:1.25);
\draw [out = 80, in = -10] (.62,-.3) to (.55,.25);

\node at (0,-.08) {\small $\bullet$};
\node at (0,.08) {\small $\circ$};
\node at (90:1.25cm) {\small $\bullet$};
\node at (-30:1.25cm) {\small $\bullet$};
\node at (210:1.25cm) {\small $\bullet$};
\node at (.65,.75) {\tiny $B'$};
\end{scope}
\end{tikzpicture}
\caption{The top row shows plain ${\rm SL}_3$ tensor diagrams $A,B,A',B'$ in the once-punctured triangle, each of which has three images under ${\rm SL}_3$-Weyl group action at the puncture. In the bottom row, we depict the three tagged arcs in $S_{0,4}$ which correspond to the each of these $W$-orbits.}
\label{fig:bijection}
\end{figure}
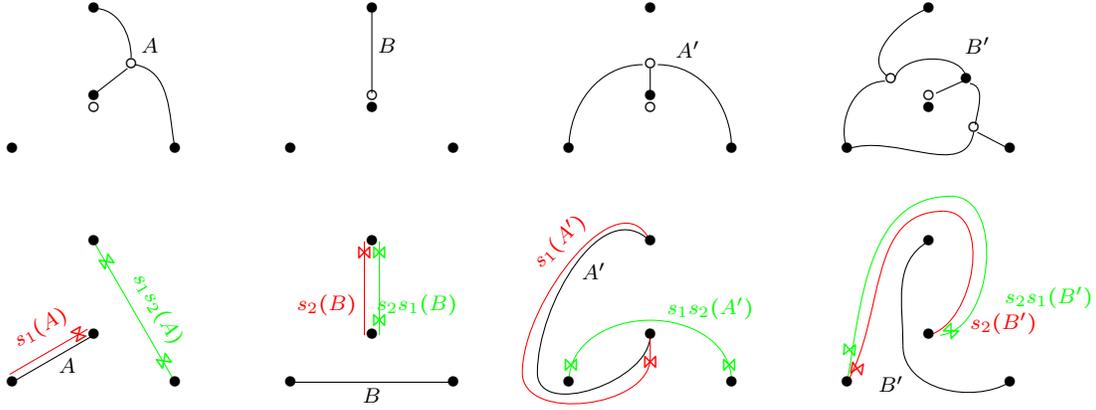

\begin{figure}[hb]
\begin{tikzpicture}
\begin{scope}[xshift = 7.4cm]
\draw (210:1.25)--(-30:1.25);
\node[rotate = 90] at (-.8,-.63) {$\bowtie$};
\node at (0,0) {\small $\bullet$};
\node at (90:1.25cm) {\small $\bullet$};
\node at (-30:1.25cm) {\small $\bullet$};
\node at (210:1.25cm) {\small $\bullet$};
\end{scope}
\begin{scope}[xshift = 11.1cm]
\node [rotate = -30] at (-.3,1.1) {$\bowtie$};
\draw [out = 195, in = 0] (-30:1.25) to (.2,-1);
\draw [out = 180, in = 190] (.2,-1) to (90:1.25);
\node at (0,0) {\small $\bullet$};
\node at (90:1.25cm) {\small $\bullet$};
\node at (-30:1.25cm) {\small $\bullet$};
\node at (210:1.25cm) {\small $\bullet$};
\end{scope}
\begin{scope}[yshift = 0cm]
\node at (-.3,.5) {\small $\circ$};
\draw [out = 210, in = 90] (-.35,.5) to (210:1.25);
\draw [out = 90, in = 225] (-.3,.57) to (90:1.25);
\draw [out = 00, in = 90] (-.25,.5) to (.5,0);
\draw [out = -90, in = 10] (.5,0) to (210:1.25);
\node at (0,.08) {\small $\bullet$};
\node at (0,-.08) {\small $\circ$};
\node at (90:1.25cm) {\small $\bullet$};
\node at (-30:1.25cm) {\small $\bullet$};
\node at (210:1.25cm) {\small $\bullet$};
\end{scope}
\begin{scope}[xshift = 3.7cm, yshift = 0cm]
\draw [out = 210, in = 90] (-.35,.5) to (210:1.25);
\draw [out = 90, in = 225] (-.3,.57) to (90:1.25);
\draw [out = 00, in = 90] (-.25,.5) to (.5,0);
\draw [out = -90, in = -70] (.5,0) to (-.5,-.2);
\draw [out = 110, in = 200] (-.5,-.2) to (90:1.25);
\node at (0,-.08) {\small $\bullet$};
\node at (-.3,.5) {\small $\circ$};
\node at (90:1.25cm) {\small $\bullet$};
\node at (-30:1.25cm) {\small $\bullet$};
\node at (210:1.25cm) {\small $\bullet$};
\end{scope}
\end{tikzpicture}
\caption{The first and second ${\rm SL}_3$ tensor diagrams in the once-punctured triangle correspond respectively to the third and fourth tagged arcs in $S_{0,4}$. Each of these has six dihedral images.}
\label{fig:bijection2}
\end{figure}
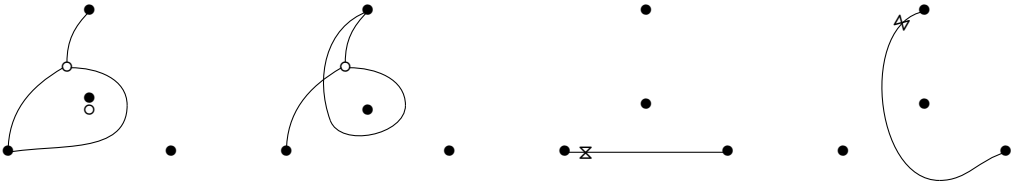

Altogether, there are six plain ${\rm SL}_3$ tensor diagrams appearing in these figures versus 4 plain ${\rm SL}_2$ arcs. Each plain arc has three dihedral images and can be tagged in four ways, for a total of 48 tagged arcs appearing in the two figures. The ${\rm SL}_3$ diagrams in the first figure have 3 dihedral images and can be tagged in 3 ways, while those in the second figure have six dihedral images and cannot be tagged, again for a total of 48 diagrams.

Each quasi cluster automorphism of $\mathscr{A}'({\rm SL}_3,D_{3,1})$ induces an element of the cluster modular group for $S_{0,4}$. We list some instances of this correspondence in the table below: 
\begin{center}
\begin{tabular}{l|l|l}
Aut($\mathcal{A}'$)&${\rm MCG}^\bowtie(S_{0,4})$&Sequence\\\hline 
$\rho$& $120^\circ$ rotation about $p_4$ & (135)(246)\\
$\sigma$& half-twist $\substack{p_2 \\ \leftrightarrows \\ p_1}$  & $(2534)\mu_{23}$\\
$s_1$& tag-change at $p_4$ & $(35)\mu_{1351}$\\
$s_2$& N/A & $(46)\mu_{2462}$\\
$w_0 \ast \rho$& tag-change at all punctures & \\
\end{tabular}
\end{center}
The third column shows the permutation-mutation sequence out of the initial seed which induces the corresponding automorphism. The notation $\mu_{23}$ means to mutate at vertex 2 and then at vertex 3. The quasi automorphism in the last row is the Donaldson-Thomas transformation, whose mutation sequence we have omitted due to its length. The quasi automorphism in the fourth row does not correspond to a tagged mapping class, but rather an ``exotic'' element of the cluster modular group for $S_{0,4}$ matching up a type I triangulation with a type VI triangulation. 

We now carry out the three additional steps from Proposition~\ref{prop:trueinfmt}, thereby verifying all of our conjectures for $\mathscr{A}'({\rm SL}_3,D_{3,1})$.
\begin{proof}
We begin by explaining that there are only finitely many clusters modulo $H'$ action, where $H'$ is the group defined in the proof of the above proposition. Then we carry out the three extra steps from the proof of the proposition. There is a group homomorphism $B_3 \to {\rm MCG}(S_{0,4})$ from the braid group on three strands to the mapping class group of a four-punctured sphere. The image of this homomorphism is the subgroup of mapping classes which fix the fourth puncture, see \cite[Section 9.2]{Farb}. This subgroup has 
index four inside ${\rm MCG}(S_{0,4})$ with cosets distinguished by which puncture is brought to $p_4$. Looking at the above table, the elements $\sigma$ and $\rho \sigma \rho^{-1}$ are the images of the Artin generators under this homomorphism, thus the group $H:= \langle\sigma, \rho \rangle$ has index four in ${\rm MCG}(S_{0,4})$. There are clearly only finitely many clusters modulo the action of this group. There will be even fewer clusters modulo the action of the larger group $H'$.

To carry out the three extra steps from the proof of the proposition, we explicitly identify a finite list of clusters which exhaust all clusters modulo $H'$-action. 

One can check that the ${\rm MCG}(S_{0,4})$-orbit of the unlabeled triangulation $T_0$ coincides with its $H$-orbit. (The mutations needed to bring a given $p_i$, $i=1,2,3$ to the North pole are the same mutations needed to perform a half-twist. The two mapping classes differ only by a permutation of variables.)

One can also check that there are six ways to mutate from a type I triangulation to a type II triangulation, four ways to mutate from a type II to a type IV, one way to mutate from a type IV to a type V, two ways to mutate from a type V to a type VI, and one way to mutate from a type V to a type III. All other mutations move weakly ``backwards'' to an earlier combinatorial type.  Altogether, given a particular type I triangulation $T$, we get $6 \cdot 4 \cdot 1 \cdot 2+6\cdot 4 \cdot 1 = 72$ mutation sequences which efficiently pass from a given type I triangulation to a triangulation of any other type. Note that some steps in these mutation sequences commute so these 72 mutations sequences do not yield 72 distinct clusters. We refer to the clusters reachable from~$T$ by one of these 72 mutations sequences as the {\sl club} of $T$. 

We claim that every tagged triangulation of $S_{0,4}$ is in the club of either $T_0$ or $T_1$ modulo $H'$-action. Indeed, suppose that $T$ has type $I$ and let $T'$ be it underlying plain triangulation. By two paragraphs previous, there exists an element $h_1 \in H$ such that $h_1(T') $ is a relabeling of $T_0$. We can choose $h_2 \in \langle {\rm DT},s_1\rangle $ (i.e., in the group generated by tag-changing at $p_4$ or at all the $p_i$'s) such that $h_2h_1(T) \in \{T_0,T_1\}$ as unlabeled tagged triangulations. Thus, if $T$ has type~I, then we can move it using $H'$ to either $T_0$ or $T_1$. If we consider a tagged triangulation $T'$ of some other type, then $T'$ is in the club of some type~I triangulation $T$. Choosing $h_1$ and $h_2$ as in the previous sentences so that $h_2h_1(T) \in \{T_0,T_1\}$, then $h_2h_1(T')$ is in the club of either $T_0$ or $T_1$ as claimed. 

The clubs of $T_0$ and $T_1$ serves as our list of finitely many clusters modulo $H'$-action. 

Now we carry out the first extra step, i.e. we check that every cluster monomial drawn from the clubs of $T_0$ and $T_1$ is a tagged web invariant and a tagged forest invariant. This is a straightforward but lengthy check using Figures~\ref{fig:bijection} and \ref{fig:bijection2}. Every tagged triangulation in these two clubs is built out of the 48 tagged arcs dihedrally related to these two figures. One must verify that when a collection of such tagged arcs are pairwise compatible then so are their corresponding tensor diagrams, where the notion of compatibility for ${\rm SL}_3$ tensor diagrams is that their product is a tagged web invariant. 
The proof of the second extra step is illustrated in Figure~\ref{fig:bijection}: in the clubs of $T_0$ and $T_1$, the cluster variables which have nonzero weight are listed in the top of the figure, and each of these is compatible with its three taggings (because the three tagged arcs below each diagram are compatible tagged arcs). 

The proof of the third extra step is another straightforward check using Figure~\ref{fig:bijection}. For example, the tensor diagram $A$ has weight $\omega_1$ at the puncture, so it is not $W$-sortable with tensor diagrams of weight $e_2$, $e_3$, or $e_2+e_3$. Looking right in the figure, the diagrams that have this weight and which are not root-conjugate to $A$ are $s_1s_2(A')$, $s_1(A')$, $s_2s_1(B')$, and $s_2s_1(B)$. We can see that the tagged arc corresponding to $A$ is not compatible with any of the four tagged arcs which correspond to these cluster variables.  
\end{proof}

The $P$-clusters in $\mathscr{A}'({\rm SL}_3,D_{3,1})$ are listed below up to $W$-action:
\begin{equation}\label{eq:Pclusters2}
\begin{tikzpicture}
\node at (0,0) {$
\begin{tabular}{l|c}
$P$-cluster & dosp \\ \hline
$e_1 \times 3, e_1+e_2 \times 3$ & $1|2|3$ \\ \hline
$e_1 \times 2, e_1+e_2 \times 4$ &  \\ \hline
$e_1 \times 4, e_1+e_2 \times 2$ &  \\ \hline
$e_1 \times 2, e_1+e_2 \times 3, 0$ &   \\ \hline
$e_1 \times 3, e_1+e_2 \times 2,0$ &  \\ \hline
$e_1 \times 2, e_1+e_2 \times 2,0\times 2$ &  \\ \hline
$e_1,e_2, e_1+e_2 \times 4$ & $12|3$ \\ \hline
$e_1,e_2, e_1+e_2 \times 3,0$ &   \\ \hline
$e_1,e_2, e_1+e_2 \times 2,0 \times 2$ &  \\ \hline
$e_1,e_2, e_1+e_2 ,1\times 3$ &  \\ 
\end{tabular}$};
\node at (7,1) 
{$\begin{tabular}{l|c}
$P$-cluster & dosp \\ \hline
$e_1\times 4,e_1+e_2, e_1+e_3 $ & $1|23$ \\ \hline
$e_1\times 3,e_1+e_2, e_1+e_3,1$ &   \\ \hline
$e_1\times 2,e_1+e_2, e_1+e_3,0\times 2$ &  \\ \hline
$e_1\times 1,e_1+e_2, e_1+e_3,0\times 3$ &  \\ \hline
$e_1,e_2,e_3,0\times 3$ & $123^+$ \\ \hline
$e_1+e_2,e_1+e_3,e_2+e_3,0\times 3$ & $123^-$ \\ \hline
\end{tabular}$};
\end{tikzpicture}
\end{equation}

\begin{remark}Using the above argument, it is not hard to see that every cluster variable in $\mathscr{A}'({\rm SL}_3,D_{3,1})$ is a tagged tree invariant, and that every cluster monomial factors as a product of these using arborization moves and the four extra moves from Figure~\ref{fig:extraarbs}, as predicted in Conjecture~\ref{conj:kis3}. We do not know how to prove the converse directions of these statements, e.g. that a web diagram which arborizes to a tree diagram indeed determines a cluster variable.
\end{remark}

\subsection{The case of $\mathscr{A}'({\rm SL}_4,D_{2,1})$}
This cluster algebra also has the same cluster type as a four-punctured sphere. We follow the style of argument from the previous section making the needed modifications. We take our initial triangulation $T_0$ to be the following labeled tagged triangulation of $S_{0,4}$, which is a type V triangulation:  
\begin{equation}
\begin{tikzpicture}
\node at (-2.7,.9) {$T_0 = $};
\node at (0,2) {$p_1$};
\node at (0,0) {$p_2$};
\node at (1,1) {$p_3$};
\node at (-1,1) {$p_4$};
\draw (-.2,1.8)--(-.8,1.2);
\draw (-.2,1.9)--(-.8,1.3);
\draw (.2,1.8)--(.8,1.2);
\draw (.2,1.9)--(.8,1.3);
\draw (0,.2)--(0,1.8);
\draw [out = 170, in = -90] (-.2,0) to (-1.5,1);
\draw [out = 90, in = 190] (-1.5,1) to (-.2,2);
\node at (.2,1) {\tiny $2$};
\node at (-1.7,1) {\tiny $5$};
\node at (.55,1.85) {\tiny $4$};
\node at (.3,1.5) {\tiny $1$};
\node at (-.65,1.75) {\tiny $3$};
\node at (-.3,1.5) {\tiny $6$};

\node [rotate = -45] at (-.7,1.4) {$\bowtie$};
\node [rotate = 45] at (.7,1.4) {$\bowtie$};
\node at (2.63,1.5) {$1: \, A$};
\node at (2.63,1.0) {$3: \, C$};
\node at (2.73,.5) {$5: \, \rho B$};
\node at (4.3,1.5) {$2: \, B$};
\node at (4.4,1.0) {$4: \, \rho C$};
\node at (4.4,.5) {$6: \, \rho A$};
\end{tikzpicture}
\end{equation}
We identify the initial tagged arcs with initial cluster variables for $\mathcal{A}'_{{\rm SL}_4,D_{2,1}}$ as indicated to the right. This induces the following correspondence between quasi cluster automorphisms of $\mathcal{A}'_{{\rm SL}_4,D_{2,1}}$ and elements of the cluster modular group of $\mathscr{A}(S_{0,4})$:
\begin{center}
\begin{tabular}{l|l|l}
Aut($\mathcal{A}'$)&${\rm MCG}^\bowtie(S_{0,4})$&Sequence\\\hline 
$\rho$& lift of $(34)$ & (16)(34)(25)\\
$\sigma \circ \rho$&  Dehn twist about $p_1,p_4$  & $(152)\mu_{4214}$\\
$\rho \circ \sigma$ & Dehn twist about $p_1,p_3$ & $(625)\mu_{3563}$\\
$s_1 \circ s_3$&  lift of $(12)(34)$& \\
$s_2$&tag change at $p_2$ &  $(25)\mu_{25}$\\
$\ast$& tag change at both $p_3,p_4$ & (14)(36)\\
& tag-change at all punctures & \\
\end{tabular}
\end{center}
There is a group homomorphism from ${\rm  MCG}(S_{0,4})$ to the symmetric group on four symbols keeping track of how the punctures are permuted. In the middle column above, ``lift of $(34)$'' indicates a mapping class whose image under this homomorphism is the transposition $(34)$, etc. The two Dehn twists in this table are pure mapping classes, i.e. their image is the identity permutation.

We now explain that there are only finitely many clusters modulo the action of the group defined in the proof of Proposition~\ref{prop:trueinfmt}.
\begin{proof}
The two compositions $\rho \circ \sigma$ and $\sigma \circ \rho$ correspond to Dehn twists about simple closed curves with geometric intersection number two. Any two such mapping classes generate the pure mapping class group of $S_{0,4}$ , see ~\cite[Section 4.2.4]{Farb}. There are clearly finitely many tagged triangulations  of $S_{0,4}$ modulo the pure mapping class group. 
\end{proof}

We note that we can reduce the size of the finite check needed to verify our conjectures by precomposing with $\rho$,$s_1 \circ s_3$,$s_2$, and $\ast$, as these quasi cluster transformations correspond to certain non-pure mapping classes and tag-changing transformations.

The $P$-clusters for $\mathscr{A}'({\rm SL}_4,D_{2,1})$ are listed below up to $W$ action, this time in multiplicative notation for space-saving purposes:
\begin{equation}\label{eq:Pclusters3}
\begin{tikzpicture}
\node at (0,0) {
\begin{tabular}{l|l}
$P$-cluster & dosp\\\hline \rule{0pt}{0.85\normalbaselineskip}
$a \times 2,ab \times 2,abc \times 2$ & $1|2|3|4$ \\ \hline \rule{0pt}{0.85\normalbaselineskip}
\hspace{-.25cm} $a,b,ab\times 2, abc\times 2 $ & $12|3|4$\\
$a,b,ab,abc\times 3$ &  \\
$a,b,ab,abc\times 2,1$ &  \\ \hline \rule{0pt}{0.85\normalbaselineskip}
\hspace{-.25cm} $a \times 2,ab\times 2 ,abc,abd$ & $1|2|34$  \\
$a\times 3 ,ab,abc,abd$ &  \\
$a\times 2,ab,abc,abd,1$ &  \\ \hline \rule{0pt}{0.85\normalbaselineskip}
\hspace{-.25cm} $a\times 2,ab,ac,abc \times 2$ & $1|23|4$  \\
$a,ab,ac,abc\times 3$ &   \\ 
$a\times 3,ab,ac,abc$ &   \\ 
$a,ab,ac,abc\times 2,1$ & \\ 
$a\times 2,ab,ac,abc,1$ & \\
\hline \rule{0pt}{0.85\normalbaselineskip}
\hspace{-.25cm} $a,b,ab,ab,abc,abd$ & $12|34$ \\
$a,b,ab,abc,abd,1$ &   \\
$a,b,abc,abd,1\times 2$ & \\  \end{tabular}};
\node at (7.7,.45) {\begin{tabular}{l|l}
$P$-cluster & dosp\\\hline \rule{0pt}{0.85\normalbaselineskip}
\rule{0pt}{0.85\normalbaselineskip} \hspace{-.425cm} $a,b,c,abc,abc,1$ &$123^+|4$ \\
$a,b,c,abc,1\times 2$ &  \\ 
$a,b,c,abc\times 3$ &  \\ \hline \rule{0pt}{0.85\normalbaselineskip}
\hspace{-.25cm} $ab,ac,bc,abc \times 3$ & $123^-|4$ \\ 
$ab,ac,bc,abc \times 2, 1$ &  \\ \hline \rule{0pt}{0.85\normalbaselineskip}
\hspace{-.25cm} $a \times 3, ab,ac,ad$ & $1|234^+$  \\
$a\times 2,ab,ac,ad,1$ &  \\ \hline \rule{0pt}{0.85\normalbaselineskip}
\hspace{-.25cm} $a\times 2,abc,abd,acd,1$ & $1|234^-$  \\ 
$a,abc,abd,acd,1\times 2$ &  \\ 
$a\times 3,abc,abd,acd$ &   \\ 
\hline \rule{0pt}{0.85\normalbaselineskip} \hspace{-.35cm}
$abc,abd,acd,bcd,1\times 2$ & $1234^-$\\ \hline
\rule{0pt}{0.85\normalbaselineskip}\hspace{-.15cm} 
$a,b,c,d,1\times 2$ & $1234^+$ \end{tabular}};
\end{tikzpicture}
\end{equation}

\end{document}